\documentclass[brochure,francais,11pt]{smfbourbaki}

\usepackage[T1]{fontenc}
\usepackage[latin1]{inputenc}
\usepackage[francais]{babel}
\usepackage{amssymb}
\usepackage{todonotes}
\usepackage{dsfont}
\usepackage{pinlabel,color}

\definecolor{mygreen}{rgb}{0.0,0.55, 0.0}

\usepackage[
backend=bibtex,
style=authoryear, 
citestyle=authoryear-comp,
maxnames=5
]{biblatex}
\usepackage{csquotes}

\DeclareDelimFormat{nameyeardelim}{\addcomma\space}

\DeclareNameAlias{sortname}{given-family}

\renewcommand{\bibnamedash}{\leavevmode\raise3pt\hbox to3em{\hrulefill}\space}

\AtEveryBibitem{%
  \clearfield{issn} 
  \clearfield{doi} 

  \ifentrytype{online}{}{
    \clearfield{url}
  }
}

\renewbibmacro{in:}{%
    \ifentrytype{article}{}{\printtext{\bibstring{in}\intitlepunct}}}

\DeclareFieldFormat[article,periodical,inreference]{number}{\mkbibparens{#1}}
\DeclareFieldFormat[article,periodical,inreference]{volume}{\mkbibbold{#1}}
\renewbibmacro*{volume+number+eid}{%
    \printfield{volume}%
    \setunit*{\addthinspace}
    \printfield{number}%
    \setunit{\addcomma\space}%
    \printfield{eid}}

\DeclareFieldFormat[article,inbook,incollection]{title}{\enquote{#1}\addcomma} 

\date{Octobre 2021}
\bbkannee{74\textsuperscript{e} ann\'ee, 2021--2022}
\bbknumero{1181}
\title{Groupes de surface dans les r\'eseaux des groupes de Lie semi-simples}
\subtitle{d'apr\`es J.~Kahn, V.~Markovi\'c, U.~Hamenst\"adt, F.~Labourie et S.~Mozes}
\author{Fanny KASSEL}
\address{CNRS et Laboratoire Alexander Grothendieck\\
  Institut des Hautes \'Etudes Scientifiques\\
  Universit\'e Paris-Saclay\\
  35 route de Chartres, 91440 Bures-sur-Yvette, France}
\email{kassel@ihes.fr}

\newcommand{\note}[1]{}

\newcommand*{\longhookrightarrow}{\ensuremath{\lhook\joinrel\relbar\joinrel\rightarrow}}

\makeatletter
\def\namedlabel#1#2{\begingroup
    #2%
    \def\@currentlabel{#2}%
    \phantomsection\label{#1}\endgroup
}
\makeatother

\usepackage{mathtools}
\input{xy}
\xyoption{all}

\newtheorem{theoreme}{Th\'eor\`eme}[section]
\newtheorem*{theo-princ}{Th\'eor\`eme principal}
\newtheorem{proposition}[theoreme]{Proposition}
\newtheorem{corollaire}[theoreme]{Corollaire}

\newtheorem{lemme}[theoreme]{Lemme}
\newtheorem{fact}[theoreme]{Fait}

\theoremstyle{remark}
\newtheorem{definition}[theoreme]{D\'efinition}
\newtheorem{exemple}[theoreme]{Exemple}
\newtheorem{remarque}[theoreme]{Remarque}
\newtheorem{remarques}[theoreme]{Remarques}
\newtheorem{question}[theoreme]{Question}

\newtheorem{cadre}[theoreme]{Cadre}

\numberwithin{equation}{section}

\newcommand{\N}{\mathbb{N}}
\newcommand{\Z}{\mathbb{Z}}
\newcommand{\Q}{\mathbb{Q}}
\newcommand{\R}{\mathbb{R}}
\newcommand{\C}{\mathbb{C}}
\newcommand{\HH}{\mathbb{H}}
\newcommand{\PP}{\mathbb{P}}

\newcommand{\SL}{\mathrm{SL}}
\newcommand{\GL}{\mathrm{GL}}
\newcommand{\SO}{\mathrm{SO}}
\newcommand{\OO}{\mathrm{O}}
\newcommand{\PO}{\mathrm{PO}}
\newcommand{\PSO}{\mathrm{PSO}}
\newcommand{\PSL}{\mathrm{PSL}}
\newcommand{\PGL}{\mathrm{PGL}}
\newcommand{\Sp}{\mathrm{Sp}}
\newcommand{\SU}{\mathrm{SU}}
\newcommand{\U}{\mathrm{U}}
\newcommand{\PU}{\mathrm{PU}}

\newcommand{\g}{\mathfrak{g}}
\newcommand{\ssl}{\mathfrak{sl}}
\newcommand{\psl}{\mathfrak{psl}}

\newcommand{\calG}{\mathcal{G}}

\newcommand{\ad}{\operatorname{ad}}
\newcommand{\id}{\mathrm{id}}
\newcommand{\dd}{\mathrm{d}}
\newcommand{\Hom}{\mathrm{Hom}}

\newcommand{\inv}{\mathrm{inv}}
\newcommand{\rot}{\mathrm{rot}}
\newcommand{\refl}{\mathrm{refl}}
\newcommand{\sym}{\mathrm{sym}}
\newcommand{\Geom}{\mathrm{Geom}}
\newcommand{\Triconn}{\mathrm{Triconn}}
\newcommand{\tildeTriconn}{\widetilde{\mathrm{Tri}}\mathrm{conn}}
\newcommand{\Biconn}{\mathrm{Biconn}}

\newcommand{\cf}{cf.\ }
\newcommand{\ie}{\textit{i.e.}\ }

\newcommand{\resp}{resp.\ }
\newcommand{\attract}{\scriptscriptstyle\text{\rm\textcircled{$+$}}}
\newcommand{\repuls}{\scriptscriptstyle\text{\rm\textcircled{$-$}}}
\newcommand{\bs}[1]{\boldsymbol{#1}}

\begin{document}

\maketitle
\setcounter{tocdepth}{1}
\tableofcontents


Un \emph{r\'eseau} d'un groupe de Lie~$G$ est un sous-groupe discret $\Gamma$ tel que le quotient $G/\Gamma$ soit de volume fini pour la mesure de Haar ; on dit que $\Gamma$ est cocompact si $G/\Gamma$ est compact.

Tout réseau cocompact sans torsion $\Gamma$ de $\PSL(2,\R)$ est un groupe \emph{de surface}, c'est-à-dire isomorphe au groupe fondamental d'une surface compacte $S$ de genre au moins deux.
En effet, on peut prendre pour $S$ le quotient du plan hyperbolique $\HH^2$ par~$\Gamma$.

Le but de cet expos\'e est de pr\'esenter le r\'esultat suivant.

\begin{theo-princ}[Kahn--Markovi\'c, Hamenst\"adt, Kahn--Labourie--Mozes]
Soit $G$ un groupe de Lie simple complexe\footnote{Par exemple $\SL(n,\C)$, $\Sp(2n,\C)$ ou $\SO(n,\C)$ ; \cf \textcite[Ch.\,X, \S\,2]{hel01} pour une description explicite de tous les groupes de Lie simples classiques.}, ou l'un des groupes $\SO(2p-1,1)$, $\SU(p,q)$ ou $\Sp(p,q)$ pour $p>q\geq 1$.
Tout r\'eseau cocompact de~$G$ contient un sous-groupe de surface.
\end{theo-princ}

Le cas $G=\SL(2,\C)$ est d\^u \`a \textcite{km12}, les cas $G=\SO(2p-1,1)$, $\SU(p,1)$ et $\Sp(p,1)$ \`a \textcite{ham15}, et le cas g\'en\'eral \`a \textcite{klm18}.
R\'ecemment, \textcite{ham20} a annonc\'e une nouvelle démonstration du cas g\'en\'eral, qui \'etend les m\'ethodes de son article de 2015.

Dans la partie~\ref{sec:motiv} nous pr\'esentons diverses motivations du th\'eor\`eme, puis dans la partie~\ref{sec:enonces-precis} nous en \'enon\c{c}ons des versions plus pr\'ecises, pour une classe plus g\'en\'erale de groupes de Lie semi-simples~$G$.
La strat\'egie de la preuve est expliqu\'ee dans la partie~\ref{sec:strategie}.
Elle comporte trois \'etapes : g\'eom\'etrique (partie~\ref{sec:geom-KLM}), dynamique (partie~\ref{sec:dyn}) et combinatoire (partie~\ref{sec:conclusion}).

\subsection*{Remerciements}

Je remercie chaleureusement Jonas Beyrer, Pierre-Louis Blayac, Jean-Philippe\linebreak Burelle, Le\'on Carvajales, Balthazar Fl\'echelles, Olivier Glorieux, Jeremy Kahn, Fran\c{c}ois Labourie, Daniel Monclair, Alan Reid, Ilia Smilga, Katie Vokes, ainsi que Nicolas Bourbaki, pour leur aide dans la préparation de cet exposé.
Je suis particulièrement reconnaissante à Jonas Beyrer, Pierre-Louis Blayac, Olivier Glorieux et Fran\c{c}ois Labourie pour de nombreuses discussions autour des articles présentés ici.

\section{Motivations} \label{sec:motiv}

Soit $G$ un groupe de Lie r\'eel semi-simple lin\'eaire non compact, par exemple $\PSL(n,\mathbb{K})$ pour $\mathbb{K}=\R$ ou~$\C$.

\subsection{Comprendre les sous-groupes des r\'eseaux} \label{subsec:intro-ss-gpe-reseau}

Un point de vue fécond en théorie des groupes est de chercher à comprendre certains groupes en étudiant quels types de sous-groupes ils admettent.
Il résulte des travaux de \textcite{tit72} que tout réseau de~$G$ contient un groupe libre non abélien à deux générateurs.
On peut voir les groupes de surface comme les groupes de type fini (non r\'esolubles à indice fini près) les plus \og simples\fg\ apr\`es les groupes libres\footnote{Par exemple en considérant la dimension cohomologique : les groupes sans torsion de dimension cohomologique~$1$ sont les groupes libres \parencite{sta68} ; les groupes de surface sont de dimension cohomologique~$2$.}.
La question suivante est alors naturelle dans le cadre de l'\'etude des r\'eseaux des groupes de Lie semi-simples.

\begin{question} \label{qu:ss-gpe-reseau}
Soit $G$ un groupe de Lie r\'eel semi-simple lin\'eaire non compact.
Tout r\'eseau de~$G$ admet-il des sous-groupes de surface ?
\end{question}

Le th\'eor\`eme principal r\'epond affirmativement \`a cette question pour les r\'eseaux cocompacts des groupes de Lie simples complexes et des groupes $\SO(2p-1,1)$, $\SU(p,q)$, $\Sp(p,q)$ pour $p>q\geq 1$.

Lorsque $G$ est de rang r\'eel un (c'est-à-dire localement isomorphe \`a $\SO(n,1)$, $\SU(n,1)$, $\Sp(n,1)$ ou \`a la forme réelle de rang un du groupe exceptionnel~$F_4$), les r\'eseaux cocompacts de~$G$ sont des groupes hyperboliques au sens de Gromov.
La question~\ref{qu:ss-gpe-reseau} pour ces réseaux est alors un cas particulier d'une question de Gromov (\cf \cite[Q.\,1.6]{bes}) : tout groupe hyperbolique \`a un bout admet-il un sous-groupe de surface ?
Voir par exemple \textcite{glr04} pour une r\'eponse affirmative dans le cadre des groupes de Coxeter, et \textcite{cal08} pour des conditions homologiques suffisantes.
Notons qu'un groupe hyperbolique ne peut contenir qu'un nombre fini de classes de conjugaison de sous-groupes de surface correspondant \`a des surfaces de genre donn\'e, comme suggéré par \textcite{gro87} et démontré par \textcite{del95} ; voir \textcite{thu97} pour les groupes de $3$-variétés hyperboliques.

En rang réel sup\'erieur, il est facile de construire, de mani\`ere arithm\'etique, des exemples de r\'eseaux contenant des groupes de surface.

\begin{exemple} \label{ex:sgs-arithm}
(\cf \cite[\S\,2.1, exemples 5 et~8]{ben09})
Soit $n\geq 3$.
L'automorphisme $\sigma$ d'ordre deux de $\Q[\sqrt{2}]$ d\'efinit un plongement $\iota : g\mapsto (g,g^{\sigma})$ de $\SL(n,\R)$ dans $G:=\SL(n,\R)\times\SL(n,\R)$, et $\Gamma:=\iota(\SL(n,\Z[\sqrt{2}]))$ est un r\'eseau non cocompact de~$G$.
Il contient $\Lambda:=\iota(H\cap\SL(3,\Z[\sqrt{2}]))$, o\`u $H\simeq\SO(2,1)_0\simeq\PSL(2,\R)$ est la composante neutre du groupe orthogonal associ\'e \`a la forme quadratique $x^2 + y^2 - \sqrt{2}\,z^2$ sur~$\R^3$.
Le groupe $\Gamma_0 := H\cap\SL(3,\Z[\sqrt{2}])$ est un r\'eseau cocompact de~$H$.
Ainsi, $\Lambda = \iota(\Gamma_0)$ est un sous-groupe de surface de~$\Gamma$.
\end{exemple}

Afin d'\'etablir que \emph{tout} r\'eseau contient des groupes de surface, nous verrons que la démonstration du th\'eor\`eme principal repose, non pas sur des consid\'erations arithm\'etiques, mais sur des arguments g\'eom\'etriques et dynamiques.

\subsection{Sous-groupes de surface \og bien positionn\'es dans~$G$\fg} \label{subsec:intro-qf}

La mani\`ere peut-\^etre la plus simple d'obtenir des sous-groupes discrets isomorphes \`a des groupes de surface dans des groupes de Lie semi-simples~$G$ est de consid\'erer des r\'eseaux cocompacts $\Gamma_0$ de $\PSL(2,\R)$ et de les voir comme des sous-groupes discrets de~$G$ via un plongement $\tau$ de $\PSL(2,\R)$ dans~$G$.
Autrement dit, on part d'une surface compacte $S$ de genre au moins deux ; on la munit gr\^ace au th\'eor\`eme d'uniformisation d'une structure hyperbolique, ce qui d\'efinit une repr\'esentation injective de $\pi_1(S)$ dans $\PSL(2,\R)$, d'image un r\'eseau cocompact $\Gamma_0$ de $\PSL(2,\R)$ ; puis on applique le plongement~$\tau$ ou l'un de ses conjugu\'es.
Nous appellerons ces sous-groupes \emph{$\tau$-fuchsiens}, par analogie avec la terminologie classique pour $G=\PSL(2,\C)$.

\begin{definition} \label{def:tau-fuchsien}
Soient $G$ un groupe de Lie semi-simple et $\tau : \PSL(2,\R)\hookrightarrow G$ un plongement.
Un sous-groupe de~$G$ est \emph{$\tau$-fuchsien} s'il est l'image d'une \emph{repr\'esentation $\tau$-fuchsienne} d'un groupe de surface $\pi_1(S)$, c'est-\`a-dire d'une repr\'esentation de la forme
$$\rho_0 : \pi_1(S) \overset{\varrho}{\longhookrightarrow} \PSL(2,\R) \overset{\tau}{\longhookrightarrow} G \overset{\mathrm{conj}}{\longrightarrow} G,$$
o\`u $\varrho$ est injective d'image discr\`ete et $\mathrm{conj}$ est la conjugaison par un \'el\'ement de~$G$.
\end{definition}

Il se peut qu'un r\'eseau $\Gamma$ de~$G$ contienne des sous-groupes de surface $\tau$-fuchsiens pour un certain plongement $\tau : \PSL(2,\R)\hookrightarrow G$ : c'est le cas dans l'exemple~\ref{ex:sgs-arithm} pour $\tau : \PSL(2,\R)\simeq H\overset{\iota}{\longhookrightarrow} G$.
En g\'en\'eral, \'etant donn\'e un r\'eseau~$\Gamma$, on pourrait esp\'erer qu'\`a d\'efaut de sous-groupes $\tau$-fuchsiens, il contienne au moins des \emph{d\'eformations} de sous-groupes $\tau$-fuchsiens.
De telles petites d\'eformations sont encore des groupes de surface par la proposition suivante.

\begin{proposition} \label{prop:deform-fuchsien}
Soient $G$ un groupe de Lie semi-simple, $\tau : \PSL(2,\R)\hookrightarrow G$ un plongement et $\rho_0 : \pi_1(S)\to G$ une repr\'esentation $\tau$-fuchsienne d'un groupe de surface $\pi_1(S)$ (d\'efinition~\ref{def:tau-fuchsien}).
Il existe un voisinage ouvert $\mathcal{U}$ de~$\rho_0$ dans $\Hom(\pi_1(S),G)$ form\'e enti\`erement de repr\'esentations injectives d'image discr\`ete.
\end{proposition}

On note ici $\Hom(\pi_1(S),G)$ l'espace des repr\'esentations de $\pi_1(S)$ dans~$G$, muni de sa topologie naturelle (topologie de la convergence sur une partie g\'en\'eratrice finie de~$\pi_1(S)$).
La proposition~\ref{prop:deform-fuchsien} est initialement due \`a \textcite{gui04} ; c'est d\'esormais une cons\'equence facile de la th\'eorie des représentations anosoviennes, \cf paragraphe~\ref{subsec:intro-Anosov}.

Pour un ouvert $\mathcal{U}$ comme ci-dessus, l'image de toute repr\'esentation $\rho\in\mathcal{U}$ est un sous-groupe de surface discret de~$G$ ; par analogie avec le cas classique $G=\PSL(2,\C)$, on dira qu'il est \emph{$\tau$-quasi-fuchsien} d\`es que $\mathcal{U}$ est connexe.

\begin{definition} \label{def:tau-qf}
Soient $G$ un groupe de Lie semi-simple et $\tau : \PSL(2,\R)\hookrightarrow G$ un plongement.
Un sous-groupe de~$G$ est \emph{$\tau$-quasi-fuchsien} s'il est de la forme $\rho(\pi_1(S))$ pour un groupe de surface $\pi_1(S)$ et une repr\'esentation $\rho : \pi_1(S)\to G$ appartenant \`a un ouvert connexe $\mathcal{U}$ de $\Hom(\pi_1(S),G)$ comme dans la proposition~\ref{prop:deform-fuchsien}.
\end{definition}

\subsubsection{Ouverts de d\'eformations de groupes $\tau$-fuchsiens}

Des ouverts $\mathcal{U}$ comme dans la proposition~\ref{prop:deform-fuchsien} ont \'et\'e beaucoup \'etudi\'es dans plusieurs cas.

\begin{exemple} \label{ex:qf-classique}
Soient $G = \PSL(2,\C)$ et $\tau : \PSL(2,\R)\hookrightarrow G$ le plongement standard.
Toute repr\'esentation $\tau$-fuchsienne $\rho_0 : \pi_1(S)\to G$ d'un groupe de surface $\pi_1(S)$ est contenue dans l'ouvert $\mathcal{U}$ des \emph{repr\'esentations quasi-fuchsiennes} de $\pi_1(S)$, c'est-\`a-dire des repr\'esentations injectives de $\pi_1(S)$ dans~$G$ dont l'image est un sous-groupe discret dans lequel tous les éléments non triviaux sont hyperboliques (c'est-à-dire diagonalisables sur~$\C$ et dont les valeurs propres sont de modules différents de~$1$).
Cet ouvert $\mathcal{U}$ joue un r\^ole important dans la th\'eorie des groupes kleiniens.
Il est connexe (Bers a montr\'e que, modulo conjugaison par $G$ au but, il est naturellement paramétré par le produit de deux copies de l'espace de Teichm\"uller de~$S$) et dense dans l'ensemble des repr\'esentations injectives d'image discr\`ete de $\pi_1(S)$ dans~$G$.
\end{exemple}

\begin{exemple} \label{ex:Hitchin}
Soient $G = \PSL(n,\R)$ et $\tau : \PSL(2,\R)\hookrightarrow G$ le plongement irr\'eductible (voir l'exemple~\ref{ex:PSL(n,K)} ci-dessous).
D'apr\`es Choi et Goldman (pour $n=3$), Labourie ($n$~quelconque), Fock et Goncharov ($n$ quelconque), pour toute repr\'esentation $\tau$-fuchsienne $\rho_0 : \pi_1(S)\to G$, la composante connexe de $\rho_0$ dans $\Hom(\pi_1(S),G)$ est enti\`erement form\'ee de repr\'esentations injectives et discr\`etes.
D'apr\`es Hitchin, cette composante connexe est, modulo conjugaison par~$G$ au but, hom\'eomorphe \`a une boule de dimension $(n^2-1)(2\mathtt{g}-2)$, o\`u $\mathtt{g}\geq 2$ est le genre de la surface~$S$.
D\'esormais appelée \emph{composante de Hitchin}, elle joue un r\^ole important en \emph{th\'eorie de Teichm\"uller--Thurston de rang sup\'erieur} (\cf \cite{poz20}).
\end{exemple}

Il est remarquable qu'il existe ainsi des sous-groupes discrets de~$G$ (sous-groupes de surface $\tau$-fuchsiens) avec de gros espaces de d\'eformations continues.
Par contraste, les r\'eseaux de~$G$ ont souvent de fortes propriétés de rigidité, qui ont donné lieu à des travaux célèbres.
Par exemple, pour $G$ non localement isomorphe \`a $\PSL(2,\R)$ (\resp non localement isomorphe à $\PSL(2,\R)$ ni à $\PSL(2,\C)$) et sans facteur compact, les réseaux cocompacts (\resp non cocompacts) irr\'eductibles de~$G$ sont localement rigides, d'après Selberg, Calabi, Weil, Garland et Raghunathan.
Pour $G$ sans facteur localement isomorphe à $\PSL(2,\R)$ et sans facteur compact, la rigidité de Mostow implique que toute représentation injective et discrète d'un réseau irréductible de~$G$ à valeurs dans~$G$ est la restriction d'un automorphisme de~$G$.
Pour $G$ de rang réel supérieur et sans facteur compact, Margulis a montré que les r\'eseaux irr\'eductibles de~$G$ ont de surcro\^it une propri\'et\'e plus forte de super-rigidit\'e, qui implique que ce sont des groupes arithm\'etiques.
Voir \textcite{pan95} pour plus de détails.

\subsubsection{Retour aux sous-groupes des r\'eseaux}

Voici une version plus précise de la question~\ref{qu:ss-gpe-reseau}.

\begin{question} \label{qu:ss-gpe-qf-reseau}
Soient $G$ un groupe de Lie r\'eel semi-simple et $\tau : \PSL(2,\R)\hookrightarrow\nolinebreak G$ un plongement.
Tout r\'eseau de~$G$ admet-il des sous-groupes de surface $\tau$-quasi-fuchsiens (d\'efinition~\ref{def:tau-qf}) ?
\end{question}

Dans une série de papiers, \textcite{lrt11}, \textcite{lr13,lr16}, \textcite{lt18,lt20} montrent que, pour $G=\PSL(n,\R)$ et $\tau$ le plongement irr\'eductible, c'est le cas de certains réseaux de~$G$, à savoir tous les réseaux non cocompacts pour $n=3$, une famille infinie de réseaux cocompacts pour $n=3$, et le réseau non cocompact $\PSL(n,\Z)$ pour $n=4$ et $n\geq 5$ impair.
Pour cela, ils considèrent des groupes discrets $\Delta$ d'isométries de~$\HH^2$ engendrés par les réflexions orthogonales dans les côtés de certains triangles de~$\HH^2$~; ces groupes admettent des sous-groupes d'indice fini sans torsion, qui sont alors des groupes de surface.
L'analogue pour $\Delta$ de la composante de Hitchin de l'exemple~\ref{ex:Hitchin} (\cf \cite{als21}) est une composante connexe de $\Hom(\Delta,\PGL(n,\R))$ formée entièrement de représentations injectives et discrètes, dont les auteurs montrent que certaines prennent leurs valeurs dans des sous-groupes arithmétiques de $\PGL(n,\R)$.
Ceci fournit, pour certains réseaux donnés de~$G$, une infinité de classes de conjugaison de sous-groupes de surface correspondant \`a des surfaces de m\^eme genre, par contraste avec la situation de rang un mentionnée au paragraphe~\ref{subsec:intro-ss-gpe-reseau}.

En allant encore plus loin, on peut poser la question suivante.

\begin{question} \label{qu:ss-gpe-qf-proche-f}
Soient $G$ un groupe de Lie r\'eel semi-simple et $\tau : \PSL(2,\R)\hookrightarrow\nolinebreak G$ un plongement.
Tout r\'eseau de~$G$ admet-il des sous-groupes de surface $\tau$-quasi-fuchsiens qui soient \og arbitrairement proches\fg\ de groupes $\tau$-fuchsiens, dans un sens \`a spécifier ?
\end{question}

Les théorèmes \ref{thm:quantitatif-PSL(n,C)} et~\ref{thm:quantitatif} ci-dessous sugg\`erent une réponse affirmative à cette question dans le cas des r\'eseaux cocompacts, pour certains couples $(G,\tau)$ qui couvrent tous les groupes $G$ du théorème principal.

Les constructions de Long--Reid--Thistlethwaite, Hamenst\"adt et Kahn--Labourie--Mozes permettent, pour nombre de r\'eseaux arithm\'etiques de~$G$, d'obtenir des sous-groupes de surface qui sont Zariski-denses dans~$G$ : ce sont alors des sous-groupes \emph{fins} (\emph{thin} en anglais) de~$G$, \`a savoir des sous-groupes d'indice infini de groupes arithm\'etiques qui sont encore Zariski-denses dans~$G$.
Voir \textcite{kllr19} pour plus de d\'etails sur les groupes fins et leur importance.

\subsection{Motivations sp\'ecifiques en basse dimension} \label{subsec:intro-motiv-cpt}

Les questions \ref{qu:ss-gpe-reseau}, \ref{qu:ss-gpe-qf-reseau} et~\ref{qu:ss-gpe-qf-proche-f} ont donn\'e lieu \`a une riche litt\'erature pour $G=\PSL(2,\C)$ et $G=\PSL(2,\R)\times\PSL(2,\R)$, motiv\'ee par les consid\'erations g\'eom\'etriques suivantes.

\subsubsection{Vari\'et\'es hyperboliques compactes de dimension trois}

Soient $G=\PSL(2,\C)$ et $\tau : \PSL(2,\R)\hookrightarrow G$ le plongement standard.
Tout sous-groupe discret sans torsion $\Gamma$ de~$G$ définit une variété hyperbolique de dimension trois, \`a savoir le quotient de l'espace hyperbolique $\HH^3$ par~$\Gamma$.
La question~\ref{qu:ss-gpe-reseau} pour les r\'eseaux cocompacts de~$G$ appartient \`a une s\'erie de grandes conjectures de Thurston sur les vari\'et\'es hyperboliques compactes de dimension trois (\cf \cite{ber13,ber14}).
Une r\'eponse affirmative à cette question a \'et\'e donn\'ee par \textcite{lac10} dans le cas des r\'eseaux arithm\'etiques de~$G$, puis par \textcite{km12} en g\'en\'eral.
Plus précisément, Kahn et Markovi\'c ont montr\'e que pour tout r\'eseau cocompact de~$G$, la variété hyperbolique correspondante contient des surfaces immergées \emph{$\pi_1$-injectives} (c'est-\`a-dire telles que l'inclusion induise une injection au niveau des groupes fondamentaux) ; les sous-groupes de surface correspondants peuvent être pris \og arbitrairement proches\fg\ de groupes fuchsiens, dans un sens quantitatif pr\'ecis, r\'epondant affirmativement aux questions \ref{qu:ss-gpe-qf-reseau} et~\ref{qu:ss-gpe-qf-proche-f}.
\textcite{ago13} a ensuite utilis\'e ce résultat et les travaux de Wise et ses collaborateurs (\cf \cite{wis21}) pour d\'emontrer la \emph{conjecture de Haken virtuelle}, qui affirme que toute vari\'et\'e hyperbolique compacte orientable de dimension trois poss\`ede un rev\^etement fini qui est de Haken, c'est-à-dire qui contient une surface \emph{plong\'ee} $\pi_1$-injective.
Grâce aux travaux de Perelman, ceci résout une conjecture de \textcite{wal68} affirmant que toute variété compacte, connexe, orientable, irréductible de dimension trois poss\`ede un rev\^etement fini qui est de Haken.

\subsubsection{Conjecture d'Ehrenpreis} \label{subsubsec:Ehrenpreis}

La question~\ref{qu:ss-gpe-qf-proche-f} dans le cas de $G=\PSL(2,\R)\times\PSL(2,\R)$ et du plongement diagonal $\tau : \PSL(2,\R)\hookrightarrow G$, pour les r\'eseaux cocompacts de~$G$ de la forme $\Gamma_1\times\Gamma_2$ o\`u $\Gamma_i \subset \PSL(2,\R)$, est motiv\'ee par une c\'el\`ebre conjecture d'\textcite{ehr70} : pour tout r\'eel $k>1$ et toute paire de surfaces de Riemann compactes de genre au moins deux, on peut trouver des rev\^etements finis des deux surfaces qui sont $k$-quasi-conformes.
\textcite{km15} ont d\'emontr\'e cette conjecture en construisant, pour tout r\'eseau cocompact $\Gamma_i$ de $\PSL(2,\R)$, des sous-groupes d'indice fini \og arbitrairement proches\fg\ de r\'eseaux d'une forme particuli\`ere (dits \emph{$R$-parfaits}), \cf paragraphe~\ref{subsec:strategie-KM}.

\subsubsection{Vari\'et\'es hyperboliques de volume fini de dimension trois} \label{subsubsec:intro-var-hyp-non-cptes}

Revenons \`a $G=\PSL(2,\C)$ et au plongement standard~$\tau$, et consid\'erons \`a pr\'esent des r\'eseaux \emph{non cocompacts} $\Gamma$ de~$G$.
\textcite{clr97} ont r\'epondu affirmativement \`a la question~\ref{qu:ss-gpe-reseau} dans ce cas en montrant que la vari\'et\'e hyperbolique de volume fini $M = \Gamma\backslash\HH^3$ contient une surface compacte immerg\'ee $\pi_1$-injective essentielle de genre au moins deux, qui de plus se rel\`eve en une surface \emph{plong\'ee} non s\'eparante dans un rev\^etement fini de~$M$.
Les sous-groupes de surface de~$\Gamma$ ainsi obtenus contiennent des \'el\'ements paraboliques (c'est-\`a-dire non diagonalisable sur~$\C$), dits accidentels.
D'autres sous-groupes de surface, quasi-fuchsiens au sens de l'exemple~\ref{ex:qf-classique}, ont ensuite \'et\'e construits par \textcite{mz08,mz09} et \textcite{bc15}.

Rappelons que l'ensemble limite d'un sous-groupe discret $\Lambda$ de $G=\PSL(2,\C)$ est l'ensemble des points d'accumulation, dans le bord \`a l'infini $\partial\HH^3 \simeq \PP^1(\C)$ de~$\HH^3$, des $\Lambda$-orbites de~$\HH^3$.
Lorsque $\Lambda$ est fuchsien (contenu dans un conjugu\'e de $\PSL(2,\R)$), son ensemble limite est un cercle (bord \`a l'infini d'une copie de $\HH^2$ dans~$\HH^3$).
Lorsque $\Lambda$ est quasi-fuchsien, son ensemble limite est un \emph{quasi-cercle} (c'est-à-dire l'image d'un cercle par une application quasi-conforme), dont la géométrie permet de mesurer combien $\Lambda$ est proche d'\^etre fuchsien.
Dans ce cas l'enveloppe convexe dans~$\HH^3$ de l'ensemble limite de~$\Lambda$ est un convexe ferm\'e de~$\HH^3$ sur lequel $\Lambda$ agit avec quotient compact : on dit que $\Lambda$ est \emph{convexe cocompact} dans~$\HH^3$.

En s'appuyant sur les travaux de \textcite{km12}, \textcite{cf19} ont r\'ecemment r\'epondu affirmativement \`a la question~\ref{qu:ss-gpe-qf-proche-f} en montrant que pour tout r\'eseau non cocompact $\Gamma$ de~$G$, les sous-groupes de surface quasi-fuchsiens de~$\Gamma$ v\'erifient la propri\'et\'e d'\og ubiquit\'e\fg\ suivante : pour toute paire de cercles disjoints dans $\partial\HH^3$, on peut trouver un sous-groupe de surface quasi-fuchsien de~$\Gamma$ dont l'ensemble limite est contenu dans la région de $\partial\HH^3$ bordée par les deux cercles.
\textcite{kw21} ont obtenu une version plus forte de ce r\'esultat : pour tout $k>1$ on peut choisir les sous-groupes de surface~$\Lambda$ de telle sorte que leur action sur $\partial\HH^3$ soit conjugu\'ee de mani\`ere $k$-quasi-conforme \`a l'action d'un groupe fuchsien.
Ces r\'esultats r\'epondent \`a une question d'Agol \parencite[Q.\,3.5]{dhm15}.
Ils ont r\'ecemment \'et\'e utilis\'es par \textcite{cf19,gm21} pour donner de nouvelles démonstrations de r\'esultats de \textcite{wis21} affirmant que $\Gamma$ agit librement et cocompactement sur un complexe cubique CAT(0) et que le quotient de ce complexe par un certain sous-groupe d'indice fini de~$\Gamma$ est \emph{sp\'ecial}.

\subsubsection{Groupe modulaire}

Rappelons que le groupe modulaire $\mathrm{Mod}(S)$ d'une surface compacte~$S$ de genre $\mathtt{g}\geq 2$ est le groupe des classes d'isotopie de diff\'eomorphismes de~$S$~; il s'identifie au groupe des automorphismes ext\'erieurs de $\pi_1(S)$ par un r\'esultat classique de Dehn, Nielsen et Baer.
Il agit proprement par isom\'etries sur un complexe simplicial hyperbolique au sens de Gromov (bien que localement infini), le \emph{complexe des courbes} $C(S)$ de~$S$, et Ivanov a montr\'e qu'en genre $\mathtt{g}>2$ il s'identifie au groupe tout entier des isom\'etries simpliciales de $C(S)$.
Comme au paragraphe~\ref{subsec:intro-ss-gpe-reseau}, on cherche \`a comprendre $\mathrm{Mod}(S)$ en \'etudiant quels types de sous-groupes il admet.

Un point de vue f\'econd est d'\'etudier $\mathrm{Mod}(S)$ via son action sur $C(S)$ par analogie avec les r\'eseaux non cocompacts de $G=\PSL(2,\C)$ agissant sur~$\HH^3$, \cf \textcite{rei06}.
Il est facile de construire des sous-groupes de $\mathrm{Mod}(S)$ qui sont des groupes libres non ab\'eliens, en faisant \og jouer au ping pong\fg\ des \'el\'ements, dits \emph{pseudo-Anosov}, qui admettent une dynamique analogue \`a celle des \'el\'ements hyperboliques de $G=\PSL(2,\C)$.
Se pose alors la question de l'existence de sous-groupes de $\mathrm{Mod}(S)$ qui soient isomorphes \`a des groupes de surface.
Cette question a \'et\'e r\'esolue affirmativement par \textcite{gh99} pour $\mathtt{g}\geq 4$ et par \textcite{lr06} pour $\mathtt{g}\geq 2$.
Comme pour $G=\PSL(2,\C)$ (\cf paragraphe~\ref{subsubsec:intro-var-hyp-non-cptes}), il existe une notion de sous-groupe \emph{convexe cocompact} de $\mathrm{Mod}(S)$, dont tous les \'el\'ements d'ordre infini sont pseudo-Anosov : \cf \textcite{fm02,ham05,kl08}.
La question de trouver des sous-groupes de surface de $\mathrm{Mod}(S)$ qui soient convexes cocompacts reste ouverte \`a ce jour.
Elle constitue l'une des motivations du travail r\'ecent de \textcite{kw21} sur les r\'eseaux non cocompacts de $G=\PSL(2,\C)$.

\subsection{Lien avec les repr\'esentations anosoviennes} \label{subsec:intro-Anosov}

Soient $G$ un groupe de Lie r\'eel semi-simple lin\'eaire et $\tau : \PSL(2,\R)\hookrightarrow G$ un plongement.
Les repr\'esentations $\tau$-fuchsiennes $\rho_0 : \pi_1(S)\to G$ de la d\'efinition~\ref{def:tau-fuchsien} sont des exemples de \emph{repr\'esentations anosoviennes} de $\pi_1(S)$ dans~$G$.
Ces derni\`eres, introduites par \textcite{lab06}, sont des repr\'esentations injectives et discr\`etes avec de fortes propri\'et\'es dynamiques.
Elles ont \'et\'e beaucoup \'etudi\'ees ces derni\`eres ann\'ees et jouent un r\^ole important en th\'eorie de Teichm\"uller--Thurston de rang sup\'erieur (\cf \cite{poz20}) et dans des d\'eveloppements r\'ecents sur les sous-groupes discrets des groupes de~Lie.

Leur d\'efinition d\'epend du choix\footnote{Il n'y a qu'un nombre fini de tels choix possibles, \`a conjugaison pr\`es.}, \`a conjugaison pr\`es, d'un sous-groupe parabolique de~$G$, c'est-à-dire (disons si $G$ est algébrique) d'un sous-groupe algébrique $P$ de~$G$ tel que l'espace homogène $G/P$ soit compact.
On peut penser à $G/P$ comme à un \og bord\fg\ de~$G$ ou de son espace riemannien symétrique $G/K$, où $K$ est un sous-groupe compact maximal de~$G$.
Pour simplifier, on supposera $P$ symétrique (c'est-à-dire conjugué à ses opposés), une condition technique qui est satisfaite dans l'exemple important suivant.

\begin{exemple} \label{ex:PSL(n,K)-drapeaux}
Soit $G = \PSL(n,\mathbb{K})$ où $\mathbb{K}=\R$ ou~$\C$.
Le groupe des matrices triangulaires supérieures est un sous-groupe parabolique $P$ de~$G$.
L'espace homogène compact $G/P$ correspondant est l'espace des drapeaux complets $(V_1\subset\dots\subset V_{n-1})$ de~$\mathbb{K}^n$.
Pour $n=2$, l'espace $G/P$ est la droite projective $\PP^1(\mathbb{K})$ ; il s'identifie au bord à l'infini de l'espace riemannien symétrique $G/K$ de~$G$, qui est le plan hyperbolique $\HH^2$ si $\mathbb{K}=\R$, et l'espace hyperbolique de dimension trois $\HH^3$ si $\mathbb{K}=\C$.
\end{exemple}

Les repr\'esentations $P$-anosoviennes sont d\'efinies par l'existence de bonnes \og applications de bord\fg, au sens suivant.
(Rappelons que l'holonomie de toute structure hyperbolique sur une surface~$S$ d\'efinit une action du groupe fondamental $\pi_1(S)$ sur le bord \`a l'infini $\PP^1(\R)$ du plan hyperbolique~$\HH^2$.)

\begin{definition} \label{def:applic-bord}
Soient $\pi_1(S)$ un groupe de surface et $\rho : \pi_1(S)\to G$ une repr\'esentation.
Une application $\xi : \PP^1(\R)\to G/P$ est une \emph{application de bord} pour~$\rho$ si elle est \'equivariante relativement \`a l'action de $\pi_1(S)$ sur $\PP^1(\R)$ donn\'ee par une certaine structure hyperbolique sur~$S$ et l'action de $\pi_1(S)$ sur $G/P$ via~$\rho$ : pour tous $\gamma\in\pi_1(S)$ et $x\in\PP^1(\R)$ on a $\xi(\gamma\cdot x)=\rho(\gamma)\cdot\xi(x)$.
\end{definition}

Par d\'efinition, une repr\'esentation $\rho : \pi_1(S)\to G$ est \emph{$P$-anosovienne} si elle admet une application de bord continue $\xi : \PP^1(\R)\to G/P$ qui :
\begin{itemize}
  \item est injective et même \emph{transverse} : toute paire de points distincts de $\PP^1(\R)$ est envoy\'ee sur une paire de points de $G/P$ en position g\'en\'erique ;
  \item pr\'eserve la dynamique : l'image par~$\xi$ du point fixe attractif dans $\PP^1(\R)$ d'un \'el\'ement $\gamma\in\pi_1(S)$ est un point fixe attractif dans $G/P$ de $\rho(\gamma)$ ;
  \item satisfait une condition de contraction uniforme pour le relev\'e, \`a un certain fibr\'e d\'efini par~$\rho$, du flot g\'eod\'esique du fibr\'e unitaire tangent de~$S$.
\end{itemize}
Cette condition de contraction est li\'ee \`a la condition d\'efinissant les \emph{flots d'Anosov} en dynamique, d'o\`u la terminologie.
Elle implique que les repr\'esentations $P$-anosoviennes forment un ouvert de $\Hom(\pi_1(S),G)$.
Nous n'\'enoncerons pas cette condition ici, mais renvoyons \`a \textcite[\S\,4]{kas19} pour plus de d\'etails, et pour diverses caract\'erisations.
La notion de repr\'esentation anosovienne se généralise \`a tous les groupes de type fini qui sont hyperboliques au sens de Gromov, \cf \textcite{gw12}.

Il est facile de voir qu'une repr\'esentation anosovienne $\rho : \pi_1(S)\to G$ est toujours injective.
En effet, si l'image par $\rho$ d'un \'el\'ement $\gamma\in\pi_1(S)$ est triviale dans~$G$, alors $\rho(\gamma)$ agit trivialement sur $\xi(\PP^1(\R))$ ; comme l'application $\rho$-\'equivariante $\xi$ est injective, $\gamma$ agit trivialement sur $\PP^1(\R)$, ce qui implique que $\gamma$ est trivial dans $\pi_1(S)$.
Un raffinement de ce raisonnement (bas\'e sur le fait que l'action de $\pi_1(S)$ sur $\PP^1(\R)$ est une action \emph{de convergence}) montre que l'image d'une repr\'esentation anosovienne est un sous-groupe \emph{discret} de~$G$, dit \emph{anosovien}.

Les sous-groupes anosoviens sont des sous-groupes discrets remarquables de~$G$.
Lorsque $G$ est de rang r\'eel un (par exemple $\PSL(2,\mathbb{K})$), ce sont exactement les sous-groupes convexes cocompacts de~$G$ au sens du paragraphe~\ref{subsubsec:intro-var-hyp-non-cptes}, c'est-\`a-dire les sous-groupes discrets agissant avec quotient compact sur un fermé convexe non vide de l'espace riemannien sym\'etrique $G/K$.
Lorsque $G$ est de rang r\'eel sup\'erieur (par exemple $\PSL(n,\mathbb{K})$ pour $n\geq 3$), les sous-groupes anosoviens sont des sous-groupes discrets de covolume infini avec de bonnes propri\'et\'es g\'eom\'etriques, topologiques et dynamiques, qui en font une bonne g\'en\'eralisation des sous-groupes convexes cocompacts (\cf \cite{gui19}).

La question suivante pr\'ecise la question~\ref{qu:ss-gpe-reseau}.

\begin{question} \label{qu:ss-gpe-Anosov-reseau}
Soit $G$ un groupe de Lie r\'eel semi-simple lin\'eaire non compact.
Tout r\'eseau de~$G$ admet-il des sous-groupes de surface qui soient anosoviens ?
\end{question}

Notons qu'une r\'eponse affirmative \`a la question~\ref{qu:ss-gpe-qf-proche-f} implique une r\'eponse affirmative \`a la question~\ref{qu:ss-gpe-Anosov-reseau}.
En effet, tout plongement $\tau : \PSL(2,\R)\hookrightarrow G$ d\'efinit un sous-groupe parabolique strict de~$G$ qui est sym\'etrique, à savoir le plus petit sous-groupe parabolique $P_{\tau}$ de~$G$ contenant l'image par~$\tau$ du groupe des matrices triangulaires sup\'erieures de $\PSL(2,\R)$ ainsi que le centralisateur dans~$G$ de l'image par~$\tau$ du groupe des matrices diagonales de $\PSL(2,\R)$.
Par construction, l'image par $\tau$ de tout \'el\'ement hyperbolique de $\PSL(2,\R)$ admet un point fixe attractif dans $G/P_{\tau}$.
De plus, $\tau$ induit un plongement \'equivariant de $\PP^1(\R)$ dans $G/P_{\tau}$ (\cf \eqref{eqn:underline-tau} ci-dessous) qui transmet la dynamique de $\PSL(2,\R)$ sur $\PP^1(\R)$ \`a $G/P_{\tau}$.
L'observation suivante en d\'ecoule ais\'ement, \cf \textcite{lab06} et \textcite{gw12}.

\begin{remarque} \label{rem:fuchsien->ano}
Tout sous-groupe $\tau$-fuchsien de~$G$ (définition~\ref{def:tau-fuchsien}) est $P_{\tau}$-anosovien.
\end{remarque}

En particulier, l'inclusion naturelle dans~$G$ de tout sous-groupe $\tau$-fuchsien $\Lambda$ de~$G$ admet un voisinage dans $\Hom(\Lambda,G)$ form\'e enti\`erement de repr\'esentations $P_{\tau}$-anosoviennes, donc injectives et discr\`etes, ce qui prouve la proposition~\ref{prop:deform-fuchsien}.

\section{\'Enonc\'es plus pr\'ecis et quantitatifs} \label{sec:enonces-precis}

Dans cette partie, nous donnons des \'enonc\'es qui r\'epondent affirmativement aux questions \ref{qu:ss-gpe-reseau} et \ref{qu:ss-gpe-Anosov-reseau}, \`a la fois dans le cadre du th\'eor\`eme principal et dans un cadre plus g\'en\'eral (th\'eor\`emes \ref{thm:quantitatif-PSL(n,C)} et~\ref{thm:quantitatif}).

\subsection{\'Enonc\'e pour les groupes $G$ du théorème principal} \label{subsec:quantitatif-PSL(n,C)}

Consid\'erons les exemples suivants de triplets $(G,\tau,P_{\tau})$, o\`u $n\geq 2$.

\begin{exemple} \label{ex:PSL(n,K)}
Soit $G = \PSL(n,\mathbb{K})$ où $\mathbb{K}=\R$ ou~$\C$, et soit $\tau = \tau_n : \PSL(2,\R)\hookrightarrow G$ le plongement irr\'eductible.
Ce plongement, unique \`a conjugaison par $\PGL(n,\mathbb{K})$ pr\`es, est donné concrètement de la manière suivante : identifions $\mathbb{K}^n$ à l'espace vectoriel $\mathbb{K}[X,Y]_{n-1}$ des polynômes à coefficients dans~$\mathbb{K}$ à deux variables $X$ et~$Y$ qui sont homogènes de degré $n-1$.
Le groupe $\SL(2,\mathbb{K})$ agit sur $\mathbb{K}[X,Y]_{n-1}$ par
$$\begin{pmatrix}a & b\\ c & d\end{pmatrix} \cdot P\!\begin{pmatrix}X\\Y\end{pmatrix} = P\!\left(\!\begin{pmatrix}a & b\\ c & d\end{pmatrix}^{-1} \! \begin{pmatrix}X\\Y\end{pmatrix}\!\right).$$
Ceci définit une repr\'esentation irr\'eductible $\SL(2,\mathbb{K}) \to \SL(\mathbb{K}[X,Y]_{n-1})\simeq\SL(n,\mathbb{K})$, qui passe au quotient en un plongement $\tau_n : \PSL(2,\mathbb{K})\hookrightarrow G$.
Pour $\mathbb{K}=\C$, on note encore $\tau_n$ sa restriction à $\PSL(2,\R)$.
Le sous-groupe parabolique $P_{\tau_n}$ du paragraphe~\ref{subsec:intro-Anosov} est le sous-groupe de $G \simeq \PSL(\mathbb{K}[X,Y]_{n-1})$ formé des matrices triangulaires supérieures dans la base $(X^{i-1}Y^{n-i})_{1\leq i\leq n}$ de $\mathbb{K}[X,Y]_{n-1}$.
L'espace $G/P_{\tau_n}$ est l'espace des drapeaux complets de $\mathbb{K}^n \simeq \mathbb{K}[X,Y]_{n-1}$.
\end{exemple}

\begin{exemple} \label{ex:SO(n,1)}
Soit $G=\SO(n,1)$ le sous-groupe de $\SL(n+1,\R)$ formé des matrices préservant la forme quadratique $Q(x) = x_1^2 + \dots x_n^2 - x_{n+1}^2$ sur~$\R^{n+1}$.
Il n'y a \`a conjugaison pr\`es qu'un seul plongement $\tau : \PSL(2,\R)\hookrightarrow G$, obtenu en identifiant $\PSL(2,\R)$ \`a $\SO(2,1)_0$ et en voyant $\SO(2,1)$ comme le sous-groupe de $\SO(n,1)$ agissant trivialement sur les $n-2$ premi\`eres coordonnées.
Le sous-groupe parabolique $P_{\tau}$ est le stabilisateur d'une droite $Q$-isotrope de~$\R^{n+1}$.
L'espace $G/P_{\tau}$ est le fermé de l'espace projectif $\PP(\R^{n+1})$ correspondant aux droites $Q$-isotropes de~$\R^{n+1}$ ; il s'identifie au bord à l'infini de l'espace hyperbolique réel~$\HH^n$, qui est l'espace riemannien symétrique $G/K$~de~$G$.
\end{exemple}

L'exemple suivant est bas\'e sur une observation élémentaire : pour $n$ impair, le plongement irr\'eductible $\tau_n : \PSL(2,\R)\hookrightarrow\PSL(n,\C)$ de l'exemple~\ref{ex:PSL(n,K)} se factorise par $\SL(n,\C)$ et préserve une forme hermitienne sur $\C[X,Y]_{n-1}\simeq\C^n$, dont la matrice dans la base $(X^{i-1} Y^{n-i})_{1\leq i\leq n}$ est anti-diagonale à coefficient strictement positifs ; l'image $\tau_n(\PSL(2,\R))$ s'identifie donc à un sous-groupe de $\SU(q+1,q)$ pour $2q+1=n$.

\begin{exemple} \label{ex:SU(p,q)}
Pour $p>q\geq 1$, soit $G=\SU(p,q)$ (\resp $\Sp(p,q)$), vu comme le groupe des transformations inversibles d'un espace vectoriel $V$ de dimension $p+q$ sur $\C$ (\resp sur les quaternions) qui préservent une forme hermitienne de signature $(p,q)$, et sont de déterminant un.
Soit $\tau : \PSL(2,\R)\hookrightarrow G$ le plongement obtenu en composant $\tau_{2q+1}$ avec l'inclusion naturelle de $\SU(q+1,q)$ dans $\SU(p,q)$ (\resp $\Sp(p,q)$).
Le sous-groupe parabolique $P_{\tau}$ est le stabilisateur d'un drapeau partiel $(V_1\subset\dots\subset V_q)$ de~$V$ formé de sous-espaces totalement isotropes $V_i$ de dimension~$i$ pour $1\leq i\leq q$.
L'espace $G/P_{\tau}$ est l'ensemble de tous les drapeaux partiels de~$V$ de cette forme.
\end{exemple}

Voici une version plus précise du théorème principal pour les groupes $G=\PSL(n,\C)$, $\SO(2p-1,1)$, $\SU(p,q)$ et $\Sp(p,q)$.
L'\'enonc\'e fait intervenir les notions d'\emph{application de bord} et de \emph{sous-groupe anosovien} de~$G$ du paragraphe~\ref{subsec:intro-Anosov}, et d'\emph{application sullivannienne} d\'efinie ci-dessous ; il donne un sens quantitatif au fait que l'on peut trouver, \`a l'int\'erieur de tout r\'eseau cocompact de~$G$, des sous-groupes de surface arbitrairement proches de groupes $\tau$-fuchsiens (d\'efinition~\ref{def:tau-fuchsien}).

\begin{theoreme} \label{thm:quantitatif-PSL(n,C)}
Soit $(G,\tau,P_{\tau})$ comme dans l'exemple~\ref{ex:PSL(n,K)}, l'exemple~\ref{ex:SO(n,1)} avec $n$ impair, ou l'exemple~\ref{ex:SU(p,q)}.
Soit $\Gamma$ un r\'eseau cocompact de~$G$.
Pour tout $\delta>0$, on peut trouver un sous-groupe de surface de~$\Gamma$ qui admet une application de bord \emph{$(\delta,\tau)$-sullivannienne} $\xi : \PP^1(\R)\to G/P_{\tau}$ (définition~\ref{def:sullivan}) ; en particulier, un tel sous-groupe est $P_{\tau}$-anosovien pour $\delta>0$ assez petit.
\end{theoreme}

Le cas $G=\PSL(2,\C)$ résulte des travaux de \textcite{km12}, et le cas g\'en\'eral est d\'emontr\'e par \textcite{klm18}.
\textcite{ham15,ham20} construit elle aussi des sous-groupes de surface de~$\Gamma$ proches de groupes $\tau$-fuchsiens dans ce contexte, mais elle ne consid\`ere pas d'application de bord.

Le th\'eor\`eme~\ref{thm:quantitatif-PSL(n,C)} r\'epond affirmativement aux questions \ref{qu:ss-gpe-reseau} et \ref{qu:ss-gpe-Anosov-reseau} pour les r\'eseaux cocompacts de $G=\PSL(n,\C)$, $\SO(2p-1,1)$, $\SU(p,q)$ ou $\Sp(p,q)$.
Pour $G=\PSL(2,\C)$, ceci implique une r\'eponse affirmative aux questions \ref{qu:ss-gpe-qf-reseau} et~\ref{qu:ss-gpe-qf-proche-f}, car dans ce cas les sous-groupes de surface anosoviens de~$G$ sont exactement les sous-groupes quasi-fuchsiens de l'exemple~\ref{ex:qf-classique}.
En g\'en\'eral, il n'est pas clair que les sous-groupes de surface construits par \textcite{klm18} et \textcite{ham15,ham20} soient $\tau$-quasi-fuchsiens (d\'efinition~\ref{def:tau-qf}) : la proposition~\ref{prop:deform-fuchsien} ne s'applique pas car lorsque $\delta$ tend vers z\'ero dans le th\'eor\`eme~\ref{thm:quantitatif-PSL(n,C)}, le genre de la surface a tendance \`a augmenter.
On s'attend toutefois \`a ce que ces constructions donnent une r\'eponse affirmative aux questions \ref{qu:ss-gpe-qf-reseau} et~\ref{qu:ss-gpe-qf-proche-f} au moins dans le cas de $G=\PSL(n,\C)$ et du plongement irr\'eductible~$\tau$.\footnote{Dans ce cas on pourrait esp\'erer utiliser par exemple les coordonn\'ees de \textcite{fg06} pour des sous-surfaces ouvertes.}

Voir le paragraphe~\ref{subsec:cadre-general} pour une version du th\'eor\`eme~\ref{thm:quantitatif-PSL(n,C)} dans un cadre plus général.

\subsection{Cercles et applications sullivanniennes} \label{subsec:sullivan}

On définit à présent la notion d'application sullivannienne.

Soient $G$ un groupe de Lie r\'eel semi-simple lin\'eaire, $\tau : \PSL(2,\R)\hookrightarrow G$ un plongement et $P_{\tau}$ le sous-groupe parabolique de~$G$ associ\'e comme \`a la fin du paragraphe~\ref{subsec:intro-Anosov}.
L'id\'ee de Kahn, Labourie et Mozes est la suivante : toute repr\'esentation $\tau$-fuchsienne $\rho : \pi_1(S)\to G$ pr\'eserve un \emph{$\tau$-cercle} dans $G/P_{\tau}$ (d\'efinition~\ref{def:tau-cercle}) ; pour $\delta>0$, ils consid\`erent une repr\'esentation $\rho : \pi_1(S)\to G$ comme \og$\delta$-proche de $\tau$-fuchsienne\fg\ si elle pr\'eserve une \og $\delta$-approximation de $\tau$-cercle\fg, qu'ils appellent \emph{application $(\delta,\tau)$-sullivannienne} (d\'efinition~\ref{def:sullivan}).

\subsubsection{$\tau$-cercles} \label{subsubsec:cercles}

Le plongement $\tau$ envoie le sous-groupe $B$ des matrices triangulaires sup\'erieures de $\PSL(2,\R)$ dans~$P_{\tau}$, et passe donc au quotient en un plongement $\tau$-\'equivariant
\begin{equation} \label{eqn:underline-tau}
\underline{\tau} : \PP^1(\R) \simeq \PSL(2,\R)/B \longrightarrow G/P_{\tau}.
\end{equation}

\begin{definition} \label{def:tau-cercle}
Un \emph{$\tau$-cercle} est une application de la forme $g\circ\underline{\tau} : \PP^1(\R)\to G/P_{\tau}$ o\`u $g\in G$.
\end{definition}

Si $g\circ\underline{\tau}$ est un $\tau$-cercle, alors toute reparamétrisation $g\circ\underline{\tau}\circ h : \mathbb{P}^1(\R)\to G/P_{\tau}$, où $h\in\PSL(2,\R)$, est encore un $\tau$-cercle, par \'equivariance de~$\underline{\tau}$.

\begin{exemple}
Soient $G=\PSL(2,\C)$ et $\tau : \PSL(2,\R)\hookrightarrow G$ le plongement standard.
Les images des $\tau$-cercles de $G/P_{\tau} = \partial\HH^3$ sont les bords des plans totalement géodésiques de~$\HH^3$ (\cf figure~\ref{fig:cercles}, \`a gauche).
\end{exemple}

\begin{exemple}
Soit $G=\PSL(n,\mathbb{K})$ où $\mathbb{K}=\R$ ou~$\C$, et soit $\tau=\tau_n : \PSL(2,\R)\hookrightarrow G$ le plongement irréductible.
Si l'on identifie $\mathbb{K}^n$ à $\mathbb{K}[X,Y]_{n-1}$ comme dans l'exemple~\ref{ex:PSL(n,K)}, alors $\underline{\tau}$ envoie tout point $[\alpha:\beta]\in\PP^1(\R)$ sur le drapeau $(V_1\subset\dots\subset V_{n-1})$ où $V_i$ est le sous-espace de $\mathbb{K}[X,Y]_{n-1}$ formé des polynômes divisibles par $(-\beta X+\alpha Y)^{n-i}$.
En particulier, la premi\`ere coordonn\'ee de~$\underline{\tau}$ est le plongement $[\alpha:\beta]\mapsto [(-\beta X+\alpha Y)^{n-1}]$ de $\PP^1(\R)$ dans $\PP(\mathbb{K}[X,Y]_{n-1}) \simeq \PP(\mathbb{K}^n)$ (\og plongement de Veronese\fg).
Pour $\mathbb{K}=\R$ et $n=3$, son image est une conique du plan projectif $\PP(\R^3)$ ; l'image de~$\underline{\tau}$ est l'ensemble des drapeaux $(V_1,V_2)$ de~$\R^3$ o\`u le projectivis\'e de~$V_1$ est un point de la conique et le projectivis\'e de~$V_2$ est tangent \`a la conique en ce point (\cf figure~\ref{fig:cercles}, \`a droite).
\end{exemple}

\begin{figure}[h!]
\centering
\labellist
\small\hair 2pt
\pinlabel {$g\cdot\underline{\tau}(\PP^1(\R))$} [u] at -5 340
\pinlabel {$\underline{\tau}(\PP^1(\R))$} [u] at -60 190
\pinlabel {$\underline{\tau}(\PP^1(\R))$} [u] at 570 255
\pinlabel {$g\cdot\underline{\tau}(\PP^1(\R))$} [u] at 900 250
\endlabellist
\hspace{1cm}
\includegraphics[width=5cm]{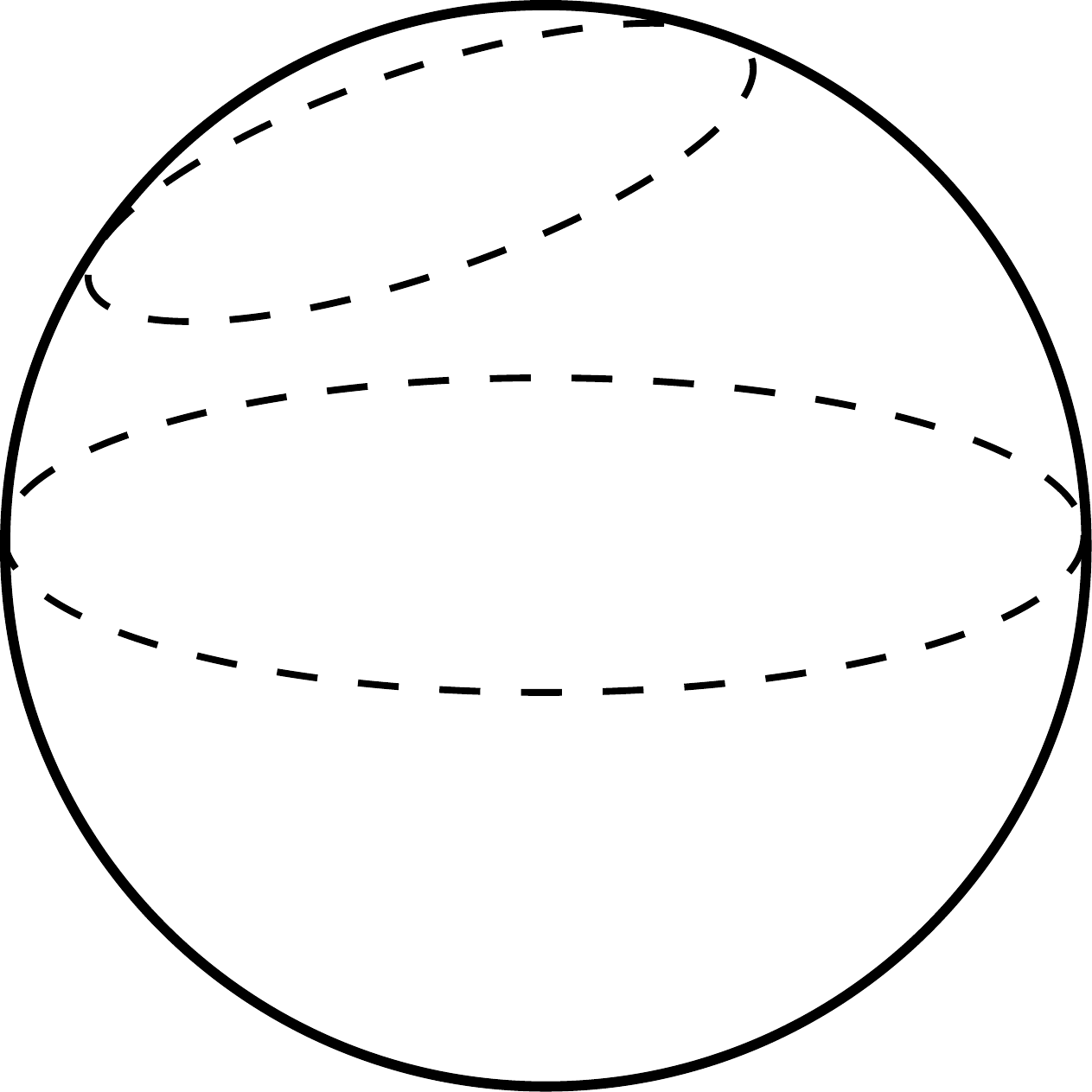}
\hspace{1.5cm}
\includegraphics[width=7cm]{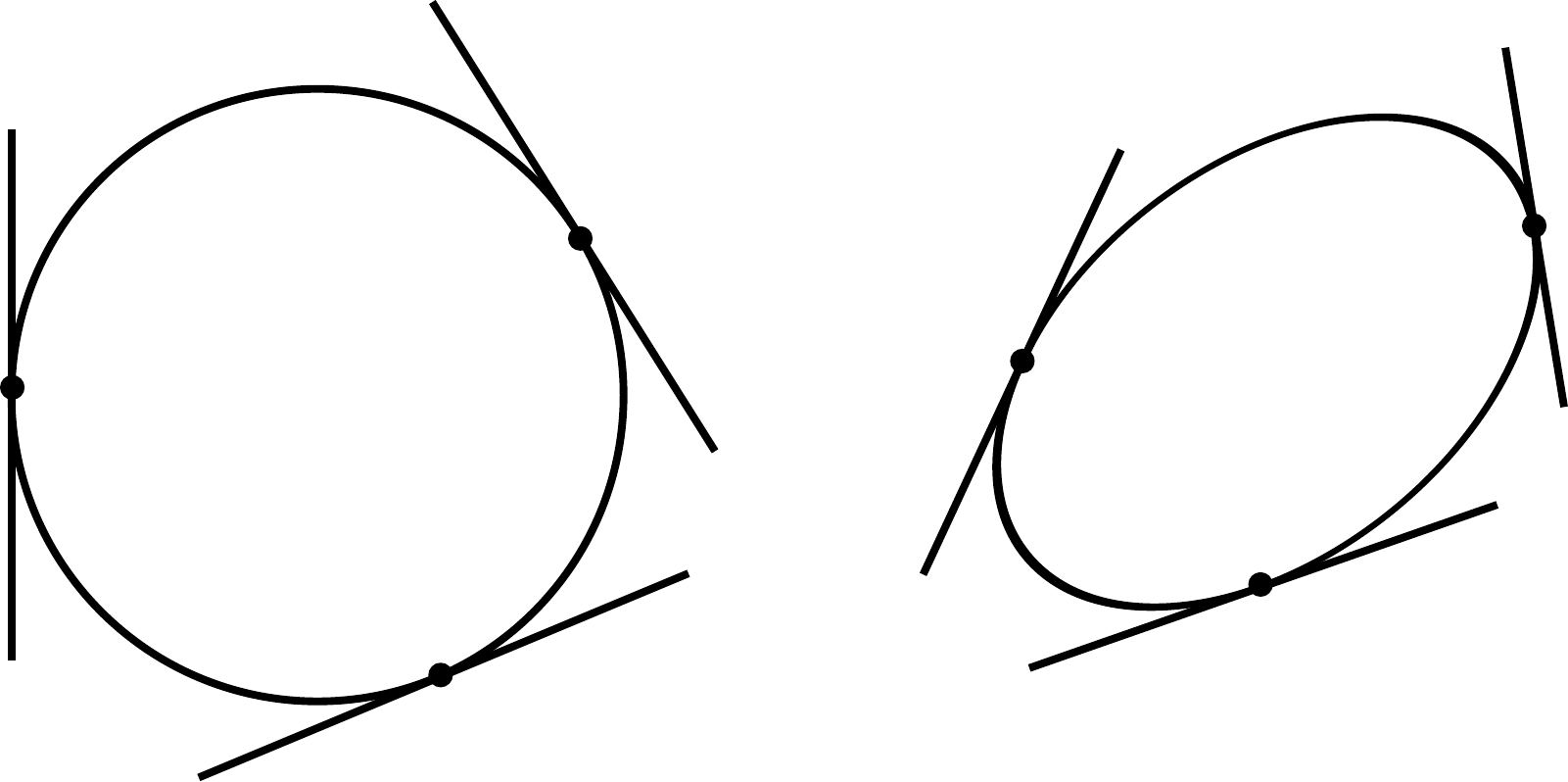}
\caption{Images de $\tau$-cercles dans $G/P_{\tau}$. \`A gauche : $G=\PSL(2,\C)$ et $G/P_{\tau} = \partial\HH^3$. \`A droite : $G=\PSL(3,\R)$ et $G/P_{\tau}$ est l'espace des drapeaux de~$\R^3$, repr\'esent\'es comme des couples $(\text{point}\in\text{droite projective})$ dans $\PP(\R^3)$.}
\label{fig:cercles}
\end{figure}

Une observation importante est que si $\pi_1(S)$ est un groupe de surface et $\rho_0 =\linebreak g(\tau\circ\varrho(\cdot))g^{-1} : \pi_1(S)\to G$ une repr\'esentation $\tau$-fuchsienne (d\'efinition~\ref{def:tau-fuchsien}), o\`u $g\in G$ et $\varrho : \pi_1(S)\to\PSL(2,\R)$, alors $\rho_0$ pr\'eserve le $\tau$-cercle $g\circ\underline{\tau} : \PP^1(\R)\to G/P_{\tau}$ : ce $\tau$-cercle est une application de bord pour~$\rho_0$ (d\'efinition~\ref{def:applic-bord}), faisant de $\rho_0$ une repr\'esentation $P_{\tau}$-anosovienne au sens du paragraphe~\ref{subsec:intro-Anosov}.

On suppose désormais que le centralisateur $Z^{\tau}$ de $\tau(\PSL(2,\R))$ dans~$G$ est compact, ce qui est vérifié dans le cadre du théorème~\ref{thm:quantitatif-PSL(n,C)} (et plus généralement dans le cadre~\ref{cadre} ci-dessous).
Pour $G=\PSL(n,\C)$ et $\tau=\tau_n$ le plongement irréductible, ce centralisateur $Z^{\tau}$ est même trivial.
La remarque suivante affirme que les $\tau$-cercles sont alors essentiellement les applications de bord des repr\'esentations $\tau$-fuchsiennes.

\begin{remarque} \label{rem:fuchsien-cercle}
Soit $\pi_1(S)$ un groupe de surface.
En supposant $Z^{\tau}$ compact, une représentation $\rho : \pi_1(S)\to G$ admet une application de bord de $\PP^1(\R)$ dans $G/P_{\tau}$ qui est un $\tau$-cercle si et seulement si $\rho$ est le produit d'une représentation $\tau$-fuchsienne $g(\tau\circ\varrho(\cdot))g^{-1}$, o\`u $g\in G$ et $\varrho : \pi_1(S)\to\PSL(2,\R)$, et d'une représentation de $\pi_1(S)$ à valeurs dans le groupe compact $g Z^{\tau} g^{-1}$.
\end{remarque}

\subsubsection{Applications $(\delta,\tau)$-sullivanniennes} \label{subsubsec:def-sullivan}

Munissons $\PP^1(\R)$ de sa distance riemannienne standard $d_{\PP^1(\R)}$, invariante par $\PSO(2)$, telle que $d_{\PP^1(\R)}(0,t) = |\!\arctan(t)|$ pour tout $t\in\R$.
Soit $K$ un sous-groupe compact maximal de~$G$ contenant $\tau(\PSO(2))$.
Munissons $G/P_{\tau}$ d'une distance riemannienne $d_{\tau}$, invariante par~$K$, telle que le $\tau$-cercle $\underline{\tau}$ de \eqref{eqn:underline-tau} soit isométrique.
Tout $\tau$-cercle $g\circ\underline{\tau} : \PP^1(\R)\to G/P_{\tau}$, où $g\in G$, est alors isométrique pour la distance riemannienne $d_{\tau}^g := d_{\tau}(g^{-1}\cdot,g^{-1}\cdot)$ sur~$G/P_{\tau}$.

\begin{definition}[\cite{klm18}] \label{def:sullivan}
Soient $G$ un groupe de Lie r\'eel semi-simple lin\'eaire, $\tau : \PSL(2,\R)\hookrightarrow G$ un plongement et $P_{\tau}$ le sous-groupe parabolique de~$G$ associ\'e.
Soit $\delta\geq 0$.
Une application $\xi : \PP^1(\R)\to G/P_{\tau}$ est \emph{$(\delta,\tau)$-sullivannienne} si pour tout $h\in\PSL(2,\R)$, il existe $g\in G$ tel que pour tout $x\in\PP^1(\R)$,
$$d_{\tau}^g\big(\xi(x), g\circ\underline{\tau}\circ h^{-1}(x)\big) \leq \delta.$$
\end{definition}

En d'autres termes, l'application $\xi$ est \og une $\delta$-approximation de $\tau$-cercle vu depuis n'importe quel triplet de points\fg\ : pour tout triplet $T = h\cdot (0,\infty,-1)$ de points deux à deux distincts positivement orientés de $\PP^1(\R)$, o\`u $h\in\PSL(2,\R)$, il existe $g\in G$ tel que le triplet $\xi(T)$ soit $\delta$-proche de $g\circ\underline{\tau}(0,\infty,-1)$ pour la distance $d_{\tau}^g$ et la courbe $\xi$ tout enti\`ere soit point par point $\delta$-proche du $\tau$-cercle $g\circ\underline{\tau}\circ h^{-1}$ pour $d_{\tau}^g$.

Pour $\delta=0$, une application $(\delta,\tau$)-sullivannienne est un $\tau$-cercle.
Plus $\delta>0$ est petit, mieux l'application approche un $\tau$-cercle, globalement comme localement (consid\'erer un triplet $T$ de points de $\PP^1(\R)$ tr\`es proches entre eux comme sur la figure~\ref{fig:Sullivan}).

\begin{figure}[h!]
\centering
\labellist
\small\hair 2pt
\pinlabel {$g\circ\underline{\tau}(\PP^1(\R))$} [u] at -35 285
\pinlabel {\textcolor{orange}{$\xi(\PP^1(\R))$}} [u] at 305 360
\pinlabel {\textcolor{orange}{$\xi(T)$}} [u] at 80 278
\pinlabel {$\underline{\tau}(\PP^1(\R))$} [u] at -45 185
\pinlabel {\textcolor{orange}{$g^{-1}\circ\xi(\PP^1(\R))$}} [u] at 440 190
\pinlabel {$\underline{\tau}(\infty)$} [u] at 185 220
\pinlabel {$\underline{\tau}(-1)$} [u] at 90 161
\pinlabel {$\underline{\tau}(0)$} [u] at 275 161
\endlabellist
\includegraphics[width=6.5cm]{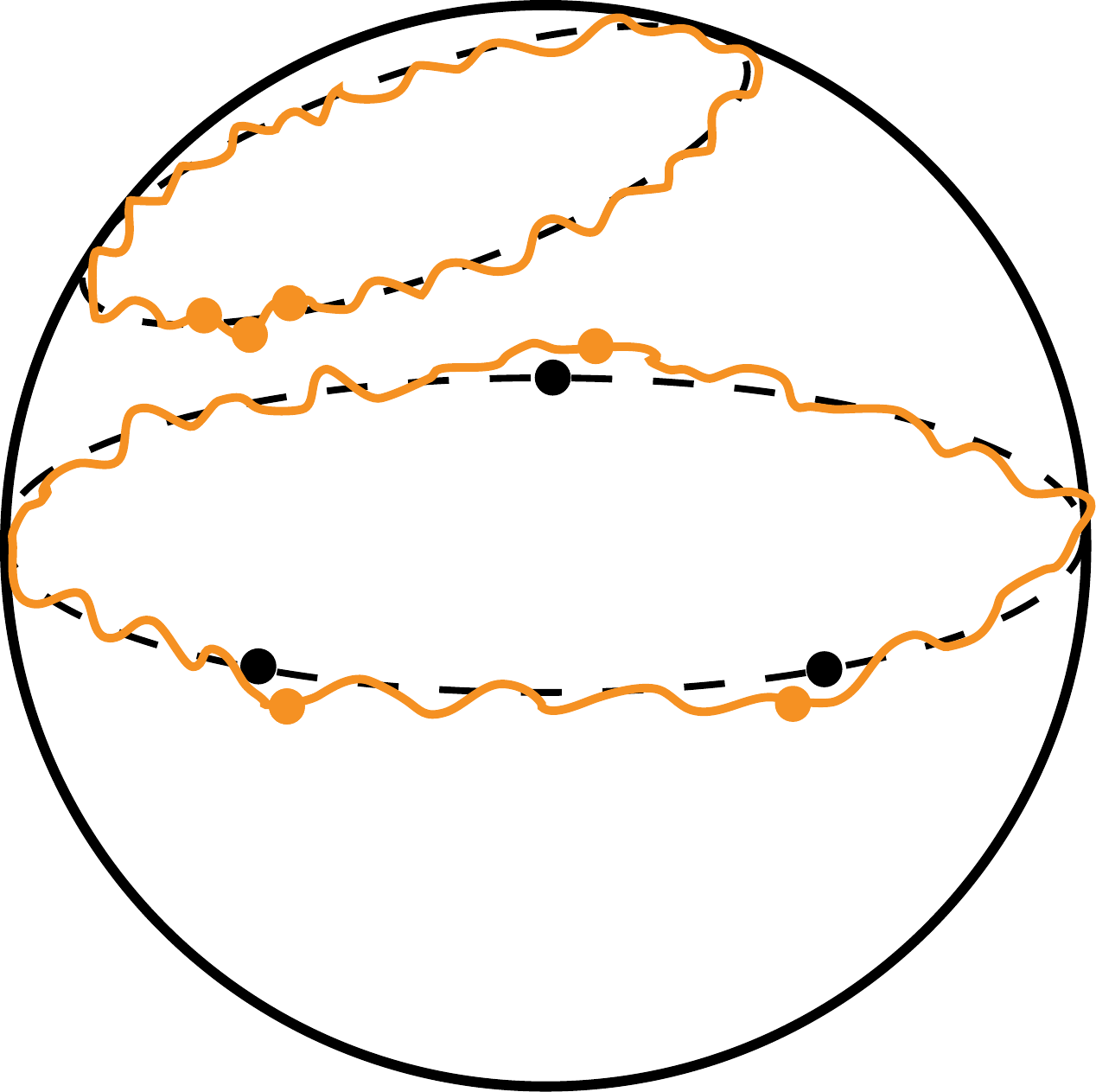}
\vspace{-0.3cm}
\caption{Application $(\delta,\tau)$-sullivannienne $\xi : \PP^1(\R)\to G/P_{\tau}$ pour $G=\PSL(2,\C)$ et $G/P_{\tau} = \partial\HH^3$. Le triplet $g^{-1}\circ\xi(T)$ est $\delta$-proche de $\underline{\tau}(0,\infty,-1)$ pour~$d_{\tau}$ et la courbe $g^{-1}\circ \xi$ tout enti\`ere est $\delta$-proche de $\underline{\tau}\circ h^{-1}$ pour~$d_{\tau}$.}
\label{fig:Sullivan}
\end{figure}

Cette définition est étroitement li\'ee aux birapports.
Par exemple, l'observation suivante est \'el\'ementaire, o\`u $[w:x:y:z]$ est le birapport (\`a valeurs dans $\mathbb{K}\cup\{\infty\}$) de quatre points $w,x,y,z\in\PP^1(\mathbb{K})$ pour $\mathbb{K}=\R$ ou~$\C$.

\begin{remarque}
Soient $G=\PSL(2,\C)$ et $\tau : \PSL(2,\R)\hookrightarrow G$ le plongement standard.
Il existe une fonction $\delta\mapsto k_{\delta}$ de $\R_+^*$ dans $]1,+\infty[$, tendant vers~$1$ en~$0$, telle que pour tout $\delta>0$ assez petit, toute application $(\delta,\tau)$-sullivannienne $\xi : \PP^1(\R)\to\PP^1(\C)=G/P_{\tau}$ soit $k_{\delta}$-quasi-sym\'etrique au sens suivant : pour tous $w,x,y,z\in\PP^1(\R)$ deux à deux distincts vérifiant $[w:\nolinebreak x:\nolinebreak y:\nolinebreak z]=-1$, on a $k_{\delta}^{-1} \leq | [\xi(w): \xi(x): \xi(y): \xi(z)] | \leq k_{\delta}$.
\end{remarque}

En effet, pour tous $w,x,y,z\in\PP^1(\R)$ deux à deux distincts, il existe $h\in\PSL(2,\R)$ envoyant $(w,x,y,z)$ sur $(0,\infty,-1,t)$ où $t := -1/[w:x:y:z] \in \R$.
Comme $\xi$ est $(\delta,\tau)$-sullivannienne, il existe $g\in G$ tel que, dans $G/P_{\tau} = \PP^1(\C)$ muni de sa distance $\mathrm{PSU}(2)$-invariante $d_{\tau}$, le quadruplet $g^{-1}\circ\xi(w,x,y,z)$ soit point par point $\delta$-proche du quadruplet $(0,\infty,-1,t)$.
Pour $\delta$ suffisamment petit et $t=1$, ceci implique que $| [g^{-1}\circ\xi(w):\nolinebreak g^{-1}\circ\nolinebreak\xi(x): g^{-1}\circ\xi(y): g^{-1}\circ\xi(z)] |$ appartient à $[k_{\delta}^{-1}, k_{\delta}]$ pour un certain $k_{\delta}>1$ ne dépendant que de~$\delta$, et tendant vers~$1$ lorsque $\delta$ tend vers~$0$.
On conclut par l'invariance du birapport par $G=\PSL(2,\C)$.

Il est probable que ce lien entre applications sullivanniennes et birapports se g\'en\'eralise \`a des groupes semi-simples $G$ de rang sup\'erieur, pour de bonnes notions de birapports (\cf par exemple \cite{bey21}).

L'observation facile suivante montre que, lorsque le centralisateur $Z^{\tau}$ de $\tau(\PSL(2,\R))$ dans~$G$ est compact, on peut quantifier le fait qu'une représentation soit proche d'être $\tau$-fuchsienne par l'existence d'une application de bord $(\delta,\tau)$-sullivannienne pour $\delta$ petit.

\begin{remarque}
Supposons $Z^{\tau}$ compact.
Pour toute surface hyperbolique compacte~$S$ et toute suite $(\rho_n)_{n\in\N}$ de représentations de $\pi_1(S)$ dans~$G$, si $\rho_n$ admet une application de bord $(\delta_n,\tau)$-sullivannienne pour tout~$n$ et si $\delta_n\to 0$, alors, quitte à conjuguer les~$\rho_n$, elles convergent vers une représentation $P_{\tau}$-anosovienne d'application de bord un $\tau$-cercle, comme à la remarque~\ref{rem:fuchsien-cercle}.
\end{remarque}

\subsubsection{Quelques propriétés des applications sullivanniennes} \label{subsec:prop-sullivan}

Les propriété suivantes seront utiles au paragraphe~\ref{subsec:dem-pi-1-inj-KLM}.

\begin{proposition}[\cite{klm18}] \label{prop:Sullivan}
Soient $G$ un groupe de Lie r\'eel semi-simple lin\'eaire, $\tau : \PSL(2,\R)\hookrightarrow G$ un plongement et $P_{\tau}$ le sous-groupe parabolique de~$G$ associ\'e.
On suppose que le centralisateur $Z^{\tau}$ de $\tau(\PSL(2,\R))$ dans~$G$ est compact.
Alors il existe $\delta_0,C,\alpha>0$ tels que
\begin{enumerate}
  \item\label{item:Sullivan-Holder} toute application $(\delta_0,\tau)$-sullivannienne $\xi : \PP^1(\R)\to G/P_{\tau}$ est $\alpha$-h\"old\'erienne ; plus pr\'ecis\'ement, pour tous $h\in\PSL(2,\R)$ et $g\in G$ comme \`a la d\'efinition~\ref{def:sullivan}, on~a $d_{\tau}^g(\xi(x),\xi(y)) \leq C\,d_{\PP^1(\R)}(h^{-1}(x),h^{-1}(y))^{\alpha}$ pour tous $x,y\in\PP^1(\R)$ ;
  \item pour tout groupe de surface $\pi_1(S)$ et toute repr\'esentation $\rho : \pi_1(S)\to G$, si $\rho$ admet une application de bord $(\delta,\tau)$-sullivannienne pour $\delta<\delta_0$, alors
  \begin{enumerate}
    \item\label{item:Sullivan-Anosov} $\rho$ est $P_{\tau}$-anosovienne (donc en particulier injective) ;
    \item\label{item:Sullivan-Anosov-deform} $\rho$ admet, pour tout $\varepsilon>0$, un voisinage dans $\Hom(\pi_1(S),G)$ form\'e enti\`erement de repr\'esentations $P_{\tau}$-anosoviennes d'application de bord $(\delta+\varepsilon,\tau)$-sullivannienne.
  \end{enumerate}
\end{enumerate}
\end{proposition}

Le point~\eqref{item:Sullivan-Holder} repose sur une version pr\'ecise, dans $G/P_{\tau}$, du lemme de Morse de \textcite{klp14,klp-morse}, qui requiert un contr\^ole fin de la convergence de parties imbriqu\'ees de $G/P_{\tau}$ (\og lunules\fg) et fait l'objet des parties 4 \`a~7 de \textcite{klm18}.
Le point~\eqref{item:Sullivan-Anosov} repose sur cette m\^eme propri\'et\'e d'imbrication (mais sans n\'ecessiter un contr\^ole aussi fin), dans l'esprit de la caract\'erisation des repr\'esentations anosoviennes en termes de multic\^ones de \textcite{bps19}.
Le point~\eqref{item:Sullivan-Anosov-deform} r\'esulte de~\eqref{item:Sullivan-Anosov} et du fait que l'application de bord d'une repr\'esentation $P_{\tau}$-anosovienne varie contin\^ument avec la repr\'esentation (\cf \cite{gw12}).

\subsection{Un cadre général pour le théorème principal} \label{subsec:cadre-general}

Dans ce paragraphe, nous décrivons les couples $(G,\tau)$, où $G$ est un groupe de Lie r\'eel semi-simple lin\'eaire et $\tau : \PSL(2,\R)\hookrightarrow G$ un plongement, pour lesquels le théorème principal et le théorème~\ref{thm:quantitatif-PSL(n,C)} sont démontrés.
Ces couples sont ceux qui satisfont deux conditions : l'une de \emph{régularité} et l'autre, que nous noterons (R), de \emph{retournement}.
Le lecteur peu familier avec la théorie de Lie pourra ignorer ce paragraphe et se concentrer sur l'exemple~\ref{ex:PSL(n,K)} de $(\PSL(n,\C),\tau_n)$, ou plus généralement sur les exemples du paragraphe~\ref{subsec:quantitatif-PSL(n,C)}, où les deux conditions sont satisfaites.

\subsubsection{La condition de régularité sur~$\mathsf{h}$} \label{subsubsec:cond-reg}

Soient $G$ un groupe de Lie r\'eel semi-simple lin\'eaire, $\tau : \PSL(2,\R)\hookrightarrow G$ un plongement et $\dd\tau : \psl(2,\R)\to\g$ le morphisme d'algèbres de Lie correspondant.
Posons
\begin{equation} \label{eqn:def-h}
\mathsf{h} := \dd\tau\left( \begin{pmatrix} 1 & 0 \\ 0 & -1 \end{pmatrix}\right).
\end{equation}
On demande que $\mathsf{h}$ soit \emph{régulier} au sens classique suivant.

Soit $K$ un sous-groupe compact maximal de~$G$ contenant $\tau(\PSO(2))$ ; c'est l'ensemble des points fixes d'une involution $\theta$ de~$G$, dite de Cartan.
On a la décomposition $\g = \mathfrak{k} + \mathfrak{k}^{\perp}$ où $\mathfrak{k}$ est l'algèbre de Lie de~$K$ et $\mathfrak{k}^{\perp}$ l'ensemble des points fixes de $-\dd\theta$ (l'orthogonal de~$\mathfrak{k}$ pour la forme de Killing).
L'élément $\mathsf{h}$ appartient à~$\mathfrak{k}^{\perp}$.
Soit $\mathfrak{a}$ un sous-espace abélien maximal~de~$\mathfrak{k}^{\perp}$ contenant~$\mathsf{h}$ (\emph{sous-espace de Cartan}).
Le groupe de Weyl $W = N_G(\mathfrak{a})/Z_G(\mathfrak{a})$ agit sur~$\mathfrak{a}$ ; l'ensemble des points qui sont fixés par au moins un élément non trivial de~$W$ est une union d'hyperplans appelés murs.
On dit que $\mathsf{h}$ est \emph{régulier} s'il n'appartient à aucun mur.
(Autrement dit, $\mathsf{h}$ appartient \`a l'\emph{int\'erieur} d'une chambre de Weyl $\mathfrak{a}^+$ de~$\mathfrak{a}$, o\`u une chambre de Weyl est par d\'efinition un c\^one convexe de~$\mathfrak{a}$ qui est l'adh\'erence d'une composante connexe du compl\'ementaire dans~$\mathfrak{a}$ de l'union des murs.)

\begin{exemple} \label{ex:PSL(n,K)-a}
Soient $G=\PSL(n,\mathbb{K})$, o\`u $\mathbb{K}=\R$ ou~$\C$, et $\tau = \tau_n : \PSL(2,\R)\hookrightarrow G$ le plongement irréductible (exemple~\ref{ex:PSL(n,K)}).
Dans la base $(X^{i-1}Y^{n-i})_{1\leq i\leq n}$ de $\mathbb{K}[X,Y]_{n-1} \simeq \mathbb{K}^n$, l'\'el\'ement $\mathsf{h}$ est diagonal de coefficients $(n-1,n-3,\dots,-(n-3),-(n-1))$.
L'unique sous-espace de Cartan $\mathfrak{a}$ de $\psl(n,\mathbb{K})$ contenant~$\mathsf{h}$ est l'espace des matrices diagonales dans cette base, de trace nulle.
Les murs de~$\mathfrak{a}$ sont les hyperplans donnés par l'égalité entre $i$-ième et $j$-ième coefficients diagonaux pour $1\leq i<j\leq n$.
L'élément $\mathsf{h}$ est ici régulier : ses coefficients sont deux à deux distincts.
\end{exemple}

\begin{remarques} \label{rem:h-reg}
\begin{enumerate}
  \item\label{item:h-reg-1} La r\'egularit\'e de~$\mathsf{h}$ implique que le centralisateur $Z^{\tau}$ de $\tau(\PSL(2,\R))$ dans~$G$ est compact.
  La réciproque est fausse : \cf exemple~\ref{ex:SO(4,2)}.
  \item L'\'el\'ement $\mathsf{h}$ est toujours r\'egulier lorsque $G$ est de rang r\'eel un.
  \item Des plongements $\tau$ différents peuvent avoir le m\^eme \'el\'ement $\mathsf{h}$ r\'egulier : \cf exemple~\ref{ex:PU(n,1)}.
\end{enumerate}
\end{remarques}

\subsubsection{La condition (R) sur $(G,\tau)$} \label{subsubsec:cond-R}

Soit $Z_G(\mathsf{h})$ le centralisateur de $\mathsf{h}$ dans~$G$, et soit $\underline{\tau} : \PP^1(\R)\hookrightarrow G/P_{\tau}$ comme en \eqref{eqn:underline-tau}.
Hamenst\"adt et Kahn--Labourie--Mozes consid\`erent la condition suppl\'ementaire suivante, dite de \emph{retournement} (\emph{flip} en anglais) :

\smallskip

\begin{itemize}
  \item[(R)] il existe un élément $j\in G$ d'ordre deux, appartenant à l'intersection du centre de $Z_G(\mathsf{h})$ et de la composante neutre de $Z_G(\mathsf{h})$, tel que $j\cdot\underline{\tau}(0,\infty,-1) = \underline{\tau}(0,\infty,1)$.
\end{itemize}

\smallskip

Cette condition exprime le fait que l'on peut utiliser une famille continue d'éléments de~$G$ pour \og retourner\fg\ $\underline{\tau}(0,\infty,-1)$ en $\underline{\tau}(0,\infty,1)$ en fixant $\underline{\tau}(0)$ et $\underline{\tau}(\infty)$ (\cf figure~\ref{fig:refl-I}, paragraphe~\ref{subsec:utiliser-hyp-retournement}).

\begin{exemple} \label{ex:PSL(n,C)-cond-R}
Soient $G=\PSL(n,\C)$ et $\tau = \tau_n$ le plongement irréductible.
Le centralisateur $Z_G(\mathsf{h})$ de $\mathsf{h}$ dans $G$ est le groupe des matrices diagonales de~$G$, c'est un groupe connexe et ab\'elien.
La condition~(R) est satisfaite en prenant $j = \tau\big(\big(\begin{smallmatrix} i & 0\\ 0 & -i\end{smallmatrix}\big)\big)$.
\end{exemple}

\subsubsection{Le r\'esultat général}

Dans toute la suite, on travaille dans le cadre suivant.

\begin{cadre} \label{cadre}
Soient $G$ un groupe de Lie r\'eel semi-simple lin\'eaire sans facteur compact et $\tau : \PSL(2,\R)\hookrightarrow G$ un plongement.
On suppose $\mathsf{h} := \dd\tau\big(\big(\begin{smallmatrix} 1 & 0\\ 0 & -1\end{smallmatrix}\big)\big) \in \g$ r\'egulier comme au paragraphe~\ref{subsubsec:cond-reg}, ce qui implique que le centralisateur $Z^{\tau}$ de $\dd\tau(\psl(2,\R))$ dans~$G$ est compact.
On note $P_{\tau}$ le sous-groupe parabolique minimal de~$G$ contenant l'image par~$\tau$ du groupe des matrices triangulaires sup\'erieures de $\PSL(2,\R)$ ; son alg\`ebre de Lie est la somme des sous-espaces propres de $\ad(\mathsf{h})$ associ\'es \`a des valeurs propres positives ou nulles.
\end{cadre}

On dit qu'un réseau $\Gamma$ de~$G$ est \emph{irréductible} si les seuls sous-groupes distingu\'es $G'$ de~$G$ tels que la projection de $\Gamma$ sur $G/G'$ soit discr\`ete sont les sous-groupes finis ou d'indice fini de~$G$.
Voici une version générale du théorème~\ref{thm:quantitatif-PSL(n,C)}.

\begin{theoreme} \label{thm:quantitatif}
Dans le cadre~\ref{cadre}, supposons que la condition (R) du paragraphe~\ref{subsubsec:cond-R} est vérifiée.
Soit $\Gamma$ un r\'eseau cocompact irr\'eductible de~$G$.
Pour tout $\delta>0$, on peut trouver un sous-groupe de surface de~$\Gamma$ qui admet une application de bord $(\delta,\tau)$-sullivannienne $\xi : \PP^1(\R)\to G/P_{\tau}$ (définition~\ref{def:sullivan}) ; en particulier, un tel sous-groupe est $P_{\tau}$-anosovien au sens du paragraphe~\ref{subsec:intro-Anosov} pour $\delta$ assez petit.
\end{theoreme}

Comme au paragraphe~\ref{subsec:quantitatif-PSL(n,C)}, le cas $G=\PSL(2,\C)$ résulte des travaux de \textcite{km12}, et le cas g\'en\'eral est trait\'e par \textcite{klm18}.
\textcite{ham15,ham20} construit dans ce contexte des sous-groupes de surface de~$\Gamma$ proches de groupes $\tau$-fuchsiens mais sans consid\'erer d'application de bord.

On s'attend \`a ce que le th\'eor\`eme~\ref{thm:quantitatif} reste vrai sans la condition~(R).
Voir \textcite{km15} pour $G=\PSL(2,\R)\times\PSL(2,\R)$ et le plongement diagonal $\tau : \PSL(2,\R)\hookrightarrow G$, pour des r\'eseaux cocompacts non irr\'eductibles.

\begin{remarque} \label{rem:h-non-reg} 
\textcite{klm18} n'ont en fait pas besoin de supposer que $\mathsf{h}$ est r\'egulier : il suffit que le centralisateur $Z^{\tau}$ de $\tau(\PSL(2,\R))$ dans~$G$ soit compact.
Lorsque $\mathsf{h}$ n'est pas r\'egulier, ils remplacent la condition~(R) par une condition plus forte : au lieu de l'intersection du centre et de la composante neutre de $Z_G(\mathsf{h})$, ils demandent que la composante neutre du centre de $Z_G(\mathsf{h})$ contienne un élément $j$ d'ordre deux tel que $j\cdot\underline{\tau}(0,\infty,-1) = \underline{\tau}(0,\infty,1)$.
Pour $\mathsf{h}$ non régulier, le sous-groupe parabolique $P_{\tau}$ de~$G$ est non minimal : \cf exemple~\ref{ex:SO(4,2)}.
Par souci de simplicit\'e, nous supposerons ici $\mathsf{h}$ r\'egulier.
\end{remarque}

\subsubsection{Exemples relatifs à la condition~(R)} \label{subsubsec:cond-R-ex}

On pourra utiliser les observations suivantes.

\begin{remarques} \label{rem:cond-R-satisf-ou-pas}
\begin{enumerate}
  \item\label{item:cond-R-oui} Si $\tau$ est la restriction \`a $\PSL(2,\R)$ d'un plongement $\tau_{\C} : \PSL(2,\C)\hookrightarrow G$ tel que le tore complexe $\exp\circ\dd\tau_{\C}\big(\C\big(\begin{smallmatrix} 1 & 0\\ 0 & -1\end{smallmatrix}\big)\big)$ soit contenu dans le centre de $Z_G(\mathsf{h})$, alors la condition~(R) est satisfaite pour $j = \dd\tau_{\C}\big(\big(\begin{smallmatrix} i & 0\\ 0 & -i\end{smallmatrix}\big)\big)$.
  \item\label{item:cond-R-non} Soit $Z_K(\mathsf{h})_0$ la composante neutre du centralisateur de $\mathsf{h}$ dans~$K$.
  Si $\mathsf{h}$ est régulier et si le centre de $Z_K(\mathsf{h})_0$ est trivial, ou plus généralement contenu dans le centralisateur $Z^{\tau}$ de $\tau(\PSL(2,\R))$ dans~$G$, alors la condition~(R) n'est pas satisfaite.
\end{enumerate}
\end{remarques}

Le point~\eqref{item:cond-R-oui} est immédiat.
Pour~\eqref{item:cond-R-non}, rappelons que si $\mathsf{h}$ est régulier, alors la composante neutre de $Z_G(\mathsf{h})$ est le produit direct de $A:=\exp(\mathfrak{a})$ et de $Z_K(\mathsf{h})_0$.
Les éléments de~$A$ fixent $\underline{\tau}(0)$ et $\underline{\tau}(\infty)$, et préservent $\underline{\tau}(\R_+^*)$.
Les éléments de $Z^{\tau}$ fixent $\underline{\tau}(\PP^1(\R))$ point par point.

\begin{exemple}
Pour $G=\PSL(2,\R)$ et $\tau$ l'identit\'e, ainsi que pour $G=\PSL(2,\R)\times\PSL(2,\R)$ et $\tau : \PSL(2,\R)\hookrightarrow G$ le plongement diagonal, la condition~(R) n'est pas satisfaite car $Z_K(\mathsf{h})_0$ est trivial (\cf remarque~\ref{rem:cond-R-satisf-ou-pas}.\eqref{item:cond-R-non}).
\end{exemple}

\begin{exemple} \label{ex:SO(n,1)-tau}
Pour $G=\SO(n,1)$, de rang r\'eel un, il n'y a \`a conjugaison pr\`es qu'un seul plongement $\tau : \PSL(2,\R)\hookrightarrow G$, obtenu en identifiant $\PSL(2,\R)$ \`a la composante neutre de $\SO(2,1)$ comme dans l'exemple~\ref{ex:SO(n,1)}.
Le groupe $Z_K(\mathsf{h})_0$ est isomorphe \`a $\SO(n-1)$.
Pour $n$ pair, la condition~(R) n'est pas satisfaite car le centre de $\SO(n-1)$ est trivial (\cf remarque~\ref{rem:cond-R-satisf-ou-pas}.\eqref{item:cond-R-non}).
Pour $n$ impair, la condition~(R) est satisfaite pour $j:=-\id\in\SO(n-\nolinebreak 1)\simeq Z_K(\mathsf{h})_0$.
\end{exemple}

\begin{exemple} \label{ex:PU(n,1)}
Soit $G = \PU(n,1)$ (c'est-à-dire $\U(n,1)$ divisé par son centre), de rang r\'eel un.
Pour $n\geq 2$, il existe \`a conjugaison pr\`es deux plongements $\tau : \PSL(2,\R)\hookrightarrow G$ :
\begin{itemize}
  \item le plongement dit \emph{totalement r\'eel} obtenu en identifiant $\PSL(2,\R)$ avec la composante neutre de $\PO(2,1)$ ;
  \item le plongement dit \emph{complexe} obtenu en identifiant $\PSL(2,\R)$ avec $\PU(1,1)$.
\end{itemize}
Dans les deux cas, le groupe $Z_K(\mathsf{h})_0$ est isomorphe \`a un quotient de $\U(n)$ par un sous-groupe fini, et le centre de $Z_K(\mathsf{h})_0$ est un tore compact de dimension un.
La condition~(R) est satisfaite pour le plongement totalement r\'eel, mais non satisfaite pour le plongement complexe car dans ce cas $Z_K(\mathsf{h})_0$ centralise $\tau(\PSL(2,\R))$ (\cf remarque~\ref{rem:cond-R-satisf-ou-pas}.\eqref{item:cond-R-non}).
\end{exemple}

\begin{exemple} \label{ex:PSL}
Pour $G=\PSL(n,\mathbb{K})$ o\`u $\mathbb{K}=\R$ ou~$\C$, l'hypoth\`ese de régularité de~$\mathsf{h}$ impose que $\tau : \PSL(2,\R)\hookrightarrow G$ soit, \`a conjugaison pr\`es, le plongement irr\'eductible~$\tau_n$ de l'exemple~\ref{ex:PSL(n,K)}.
La condition~(R) est satisfaite pour $\mathbb{K}=\C$ par la remarque~\ref{rem:cond-R-satisf-ou-pas}.\eqref{item:cond-R-oui} (\cf exemple~\ref{ex:PSL(n,C)-cond-R}), et non satisfaite pour $\mathbb{K}=\R$ par la remarque~\ref{rem:cond-R-satisf-ou-pas}.\eqref{item:cond-R-non}, car pour $\mathbb{K}=\R$ le groupe $Z_K(\mathsf{h})_0$ est trivial.
\end{exemple}

On montre de m\^eme que la condition~(R) est satisfaite si $G$ est un groupe de Lie simple complexe de centre trivial et $\dd\tau : \psl(2,\R)\hookrightarrow\g$ le plongement principal, alors qu'elle n'est jamais satisfaite si $G$ est un groupe de Lie simple r\'eel d\'eploy\'e.

\begin{exemple} \label{ex:SU(p,q)-bis}
Soit $G=\SU(p,q)$ ou $\Sp(p,q)$, o\`u $p>q$.
Soit $\tau : \PSL(2,\R)\hookrightarrow G$ le plongement obtenu en composant $\tau_{2q+1}$ avec l'inclusion naturelle de $\SU(q+1,q)$ dans~$G$, comme dans l'exemple~\ref{ex:SU(p,q)}.
L'\'el\'ement $\mathsf{h}$ est r\'egulier.
La condition~(R) est satisfaite par la remarque~\ref{rem:cond-R-satisf-ou-pas}.\eqref{item:cond-R-oui}.
\end{exemple}

\begin{exemple} \label{ex:SO(4,2)}
Soient $G=\SO(4,2)$ et $\tau : \PSL(2,\R)\simeq\SO(2,1)_0\hookrightarrow G$ le plongement obtenu en plongeant $\SO(2,1)$ diagonalement dans $\SO(2,1)\times\SO(2,1)\subset G$.
L'\'el\'ement $\mathsf{h}$ n'est pas r\'egulier, mais le centralisateur $Z^{\tau}$ de $\tau(\PSL(2,\R))$ dans~$G$ est compact.
La composante neutre de $Z_G(\mathsf{h})$ est isomorphe \`a $\GL^+(2,\R)\times\SO(2)$, et l'\'el\'ement $j=-\id$ du facteur $\SO(2)$ échange $\underline{\tau}(1)$ et $\underline{\tau}(-1)$ : la condition~(R) est satisfaite (version forte de la remarque~\ref{rem:h-non-reg}).
Le sous-groupe parabolique $P_{\tau}$ est ici le stabilisateur d'un plan totalement isotrope de~$\R^{4,2}$.
\end{exemple}

\subsubsection{Classification des plongements~$\tau$}

Pour $G$ fixé, il n'y a qu'un nombre fini de plongements $\tau$ possibles à conjugaison près.
On peut les classifier de la manière suivante.
Posons
\begin{equation} \label{eqn:h-e-f-0}
(\mathsf{h}_0,\mathsf{e}_0,\mathsf{f}_0) := \left( \begin{pmatrix} 1 & 0 \\ 0 & -1 \end{pmatrix}, \begin{pmatrix} 0 & 1 \\ 0 & 0 \end{pmatrix}, \begin{pmatrix} 0 & 0 \\ 1 & 0 \end{pmatrix} \right) \in \psl(2,\R)^3.
\end{equation}
Les plongements $\tau$ de $\PSL(2,\R)$ ou $\SL(2,\R)$ dans~$G$ sont en bijection avec les morphismes d'alg\`ebres de Lie $\dd\tau : \psl(2,\R)\to\g$ d'image non nulle, ou encore, via l'évaluation en $(\mathsf{h}_0,\mathsf{e}_0,\mathsf{f}_0)$, avec les \emph{$\ssl(2)$-triplets} de~$\g$, c'est-à-dire les triplets $(\mathsf{h},\mathsf{e},\mathsf{f})\in\g^3$ non nuls tels que
$$[\mathsf{h},\mathsf{e}]=2\mathsf{e}, \quad\quad [\mathsf{h},\mathsf{f}]=-2\mathsf{f} \quad\quad\mathrm{et}\quad\quad [\mathsf{e},\mathsf{f}]=\mathsf{h}.$$
Le th\'eor\`eme de Jacobson--Morozov affirme que tout \'el\'ement nilpotent non nul $\mathsf{e}$ de~$\g$ peut \^etre compl\'et\'e en un $\ssl(2)$-triplet, et un r\'esultat classique de Kostant dit que ceci induit une bijection entre l'ensemble des \'el\'ements nilpotents non nuls de~$\g$ modulo l'action adjointe de~$G$ et l'ensemble des $\ssl(2)$-triplets de~$\g$ modulo l'action adjointe~de~$G$.

\begin{remarque}
\textcite{klm18} expriment la condition~(R) du paragraphe~\ref{subsubsec:cond-R} en termes du $\ssl(2)$-triplet $(\mathsf{h},\mathsf{e},\mathsf{f}) = \dd\tau(\mathsf{h}_0,\mathsf{e}_0,\mathsf{f}_0)$ : il existe un élément $j\in G$ d'ordre deux, appartenant à l'intersection du centre de $Z_G(\mathsf{h})$ et de la composante neutre de $Z_G(\mathsf{h})$,  qui envoie (via l'action adjointe) $\mathsf{e}+\mathsf{f}$ sur $-(\mathsf{e}+\mathsf{f})$.
\end{remarque}

\section{Strat\'egie de d\'emonstration} \label{sec:strategie}

La strat\'egie g\'en\'erale de d\'emonstration du théorème principal et des th\'eor\`emes \ref{thm:quantitatif-PSL(n,C)} et~\ref{thm:quantitatif} remonte \`a \textcite{km12} ; nous la r\'esumons bri\`evement au paragraphe~\ref{subsec:strategie-KM}.
Elle fait intervenir la notion de structure hyperbolique $R$-parfaite, introduite au paragraphe~\ref{subsec:struct-R-parf}, et comporte trois grandes \'etapes, d\'ecrites aux paragraphes \ref{subsec:etape-geom} \`a~\ref{subsec:etape-comb}.

\subsection{Rappels de g\'eom\'etrie hyperbolique} \label{subsec:struct-R-parf}

Soit $S$ une surface compacte connexe orientée de genre $\mathtt{g}\geq 2$.
On peut la d\'ecomposer en une union disjointe de $3\mathtt{g}-3$ courbes fermées simples et de $2\mathtt{g}-2$ \emph{pantalons}, o\`u un pantalon est par d\'efinition une sous-surface ouverte hom\'eomorphe \`a une sph\`ere \`a trois trous.
On appelle \emph{d\'ecomposition en pantalons} l'ensemble $\mathcal{P}$ des $2\mathtt{g}-2$ pantalons ainsi obtenus ; les $3\mathtt{g}-3$ courbes fermées simples sont appelées \emph{courbes de bord}.
La d\'ecomposition est \emph{bipartie} si $\mathcal{P}$ est l'union disjointe de deux sous-ensembles $\mathcal{P}^+$ et~$\mathcal{P}^-$ tels que pour toute paire de pantalons adjacents le long d'une courbe de bord, l'un des pantalons appartient à~$\mathcal{P}^+$ et l'autre à~$\mathcal{P}^-$ ; dans ce cas, aucun pantalon n'est adjacent \`a lui-m\^eme.
Soit $\calG$ un graphe fini sur~$S$ avec~:
\begin{itemize}
  \item $2\mathtt{g}-2$ sommets : un à l'intérieur de chaque pantalon de~$\mathcal{P}$,
  \item $3\mathtt{g}-3$ arêtes : une pour chaque courbe de bord entre deux pantalons de~$\mathcal{P}$,
\end{itemize}
de sorte que $S$ se r\'etracte sur~$\calG$ (\cf figure~\ref{fig:decomp-pantalons}, \`a gauche).
Le graphe~$\calG$ permet d'associer, à toute courbe de bord orientée $a$ d'un pantalon $\Pi\in\mathcal{P}$, un élément privilégié de $\pi_1(\Pi)$ dans la classe de conjugaison d\'efinie par~$a$.
Si $\Pi^+\in\mathcal{P}^+$ et $\Pi^-\in\mathcal{P}^-$ sont adjacents le long de~$a$, on peut orienter $a$ et les autres courbes de bord $b^{\pm},c^{\pm}$ comme sur la figure~\ref{fig:decomp-pantalons}, \`a droite, de sorte que les éléments correspondants de $\pi_1(\Pi^{\pm})$ (notés par les mêmes lettres) vérifient $c^+ b^+ a = 1$ et $c^- b^- a = 1$.

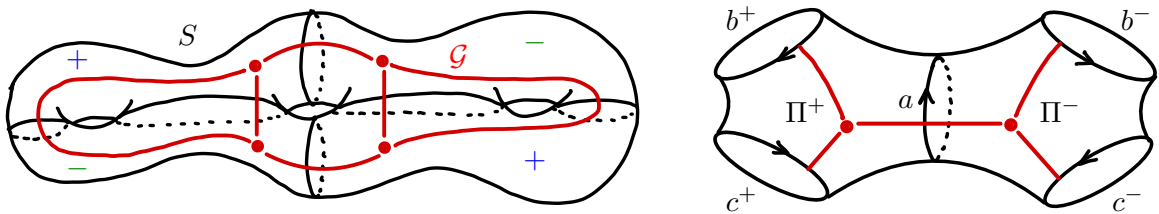
\begin{figure}[h!]
\centering
\labellist
\small\hair 2pt
\pinlabel {$S$} [u] at 140 370
\pinlabel {\textcolor{blue}{$+$}} [u] at 75 355
\pinlabel {\textcolor{mygreen}{$-$}} [u] at 75 290
\pinlabel {\textcolor{mygreen}{$-$}} [u] at 345 365
\pinlabel {\textcolor{blue}{$+$}} [u] at 345 295
\pinlabel {\textcolor{red}{$\mathcal{G}$}} [u] at 300 357
\pinlabel {$b^+$} [u] at 467 380
\pinlabel {$c^+$} [u] at 467 272
\pinlabel {$a$} [u] at 564 330
\pinlabel {$b^-$} [u] at 700 380
\pinlabel {$c^-$} [u] at 695 272
\pinlabel {$\Pi^+$} [u] at 505 325
\pinlabel {$\Pi^-$} [u] at 655 325
\endlabellist
\includegraphics[width=8.5cm]{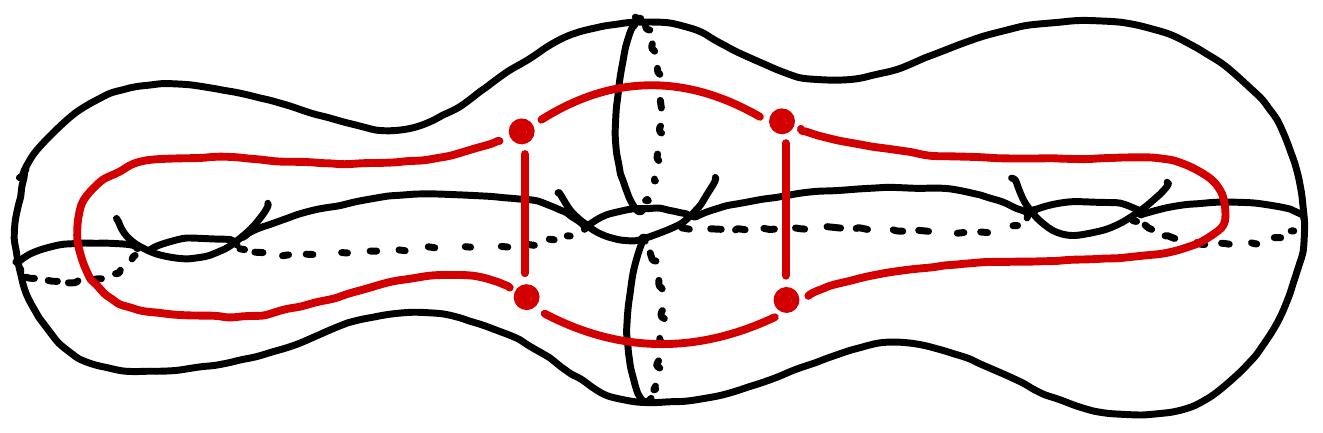}
\hspace{0.5cm}
\includegraphics[width=6cm]{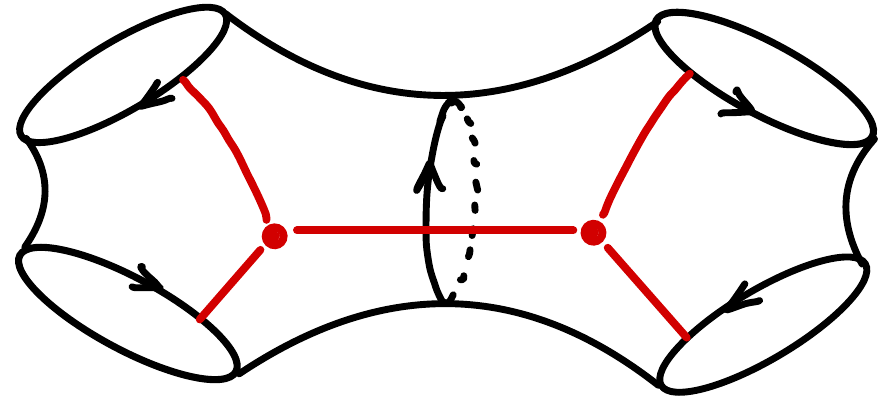}
\caption{\`A gauche : d\'ecomposition en pantalons bipartie de~$S$ avec un graphe associé~$\mathcal{G}$ ; \`a droite : configuration de deux pantalons adjacents}
\label{fig:decomp-pantalons}
\end{figure}

Soient $a_1,\dots,a_{3\mathtt{g}-3}$ les courbes de bord de la d\'ecomposition en pantalons~$\mathcal{P}$.
\`A toute repr\'esentation injective et discr\`ete $\varrho$ de $\pi_1(S)$ dans $\PSL(2,\R)$, on peut associer une \emph{structure hyperbolique marqu\'ee} sur~$S$, ce qui d\'efinit $6\mathtt{g}-6$ invariants (\cf \cite[\S\,7.6]{hub06}), \`a savoir, pour tout $i\in\{1,\dots,3\mathtt{g}-3\}$,
\begin{itemize}
  \item la longueur du repr\'esentant g\'eod\'esique de la courbe de bord~$a_i$ pour la structure hyperbolique sur~$S$ d\'efinie par~$\varrho$, c'est-\`a-dire la longueur de translation dans~$\HH^2$ de $\varrho(\alpha_i)$ o\`u $\alpha_i\in\pi_1(S)$ correspond \`a~$a_i$ ;
  \item un \og param\`etre de d\'ecalage\fg\ au niveau de~$a_i$, qui
  refl\`ete le fait que l'on peut changer la structure hyperbolique marqu\'ee en faisant \og tourner\fg\ l'un des deux pantalons adjacents à~$a_i$ par rapport \`a l'autre.
\end{itemize}
Les $3\mathtt{g}-3$ longueurs peuvent prendre n'importe quelles valeurs strictement positives, et les $3\mathtt{g}-3$ param\`etres de d\'ecalage n'importe quelles valeurs r\'eelles, lorsque $\varrho$ varie parmi les repr\'esentations injectives et discr\`etes de $\pi_1(S)$ dans $\PSL(2,\R)$.
Un r\'esultat classique de Fenchel et Nielsen affirme que ces $6\mathtt{g}-6$ invariants param\`etrent compl\`etement l'espace de Teichm\"uller de~$S$, c'est-\`a-dire l'espace des structures hyperboliques marqu\'ees sur~$S$, ou de mani\`ere \'equivalente l'espace des repr\'esentations injectives et discr\`etes de $\pi_1(S)$ dans $\PSL(2,\R)$ modulo conjugaison au but par $\PGL(2,\R)$.

L'id\'ee de \textcite{km12} est de consid\'erer le cas particulier suivant : pour $R>0$, on dit qu'une repr\'esentation injective et discr\`ete $\varrho : \pi_1(S)\to\PSL(2,\R)$ est \emph{$R$-parfaite} si les $3\mathtt{g}-3$ longueurs des courbes de bord sont toutes \'egales \`a~$2R$ et les $3\mathtt{g}-3$ param\`etres de d\'ecalage sont tous \'egaux \`a~$1$ (voir le paragraphe~\ref{subsubsec:pavage-hex} pour une interpr\'etation en termes de pavages de~$\HH^2$ par des hexagones \`a angles droits).
La valeur pr\'ecise~$1$ n'a pas d'importance particuli\`ere : ce qui compte est de choisir une valeur non nulle, ce qui rend vrai le fait important suivant (\cf \cite[Lem.\,2.7]{km12}).

\begin{fact} \label{fait:diam-borne}
Le diam\`etre de~$S$ munie d'une structure hyperbolique $R$-parfaite (c'est-\`a-dire le diam\`etre de $\varrho_R(\pi_1(S))\backslash\HH^2$) est uniform\'ement born\'e lorsque $R$ tend vers l'infini.
\end{fact}

\subsection{Bref r\'esum\'e de l'approche de Kahn et Markovi\'c} \label{subsec:strategie-KM}

Soient $\Gamma$ un r\'eseau cocompact de $G=\PSL(2,\C)$ et $M=\Gamma\backslash\HH^3$ la vari\'et\'e hyperbolique compacte de dimension trois correspondante.
Afin de r\'epondre affirmativement aux questions \ref{qu:ss-gpe-reseau}, \ref{qu:ss-gpe-qf-reseau} et~\ref{qu:ss-gpe-qf-proche-f}, \textcite{km12} construisent, pour tout $\varepsilon>0$ et tout $R>0$ assez grand par rapport à~$\varepsilon$ :
\begin{itemize}
  \item une surface compacte $S$ de genre au moins deux munie d'une d\'ecomposition en pantalons bipartie et d'un graphe fini comme au paragraphe~\ref{subsec:struct-R-parf} ;
  \item une immersion $f : S\to M$ envoyant chaque pantalon de~$S$ sur une sous-surface de~$M$ dont les courbes de bord sont des g\'eod\'esiques de longueurs complexes $\varepsilon$-proches de $2R$, avec des param\`etres de d\'ecalage $(\varepsilon/R)$-proches de~$1$
\end{itemize}
(ces longueurs complexes et param\`etres de d\'ecalage sont des g\'en\'eralisations naturelles des notions correspondantes du paragraphe~\ref{subsec:struct-R-parf}, \cf \cite[\S\,2]{ber13}).
L'immersion $f$ induit un morphisme de groupes $\rho : \pi_1(S)\to\pi_1(M)\simeq\Gamma$.
Soit $\varrho_R : \pi_1(S)\to\PSL(2,\R)$ une repr\'esentation $R$-parfaite au sens du paragraphe~\ref{subsec:struct-R-parf}.
Pour $\varepsilon>0$ assez petit, Kahn et Markovi\'c montrent, gr\^ace aux conditions sur l'immersion~$f$ (longueurs de bord et param\`etres de d\'ecalage), qu'elle induit une application de bord $(\varrho_R,\rho)$-\'equivariante de $\PP^1(\R)$ dans $\PP^1(\C)$ qui est $(1+C\varepsilon)$-quasi-sym\'etrique, pour une certaine constante $C>0$ ind\'ependante de~$\varepsilon$.
En particulier, pour $\varepsilon$ assez petit et $R$ assez grand par rapport \`a~$\varepsilon$, la repr\'esentation $\rho : \pi_1(S)\to\Gamma$ est injective et son image est \og proche\fg\ d'\^etre~fuchsienne.

La m\^eme strat\'egie est utilis\'ee par \textcite{km15} pour d\'emontrer la conjecture d'\textcite{ehr70}  du paragraphe~\ref{subsubsec:Ehrenpreis}.
En effet, consid\'erons deux surfaces de Riemann compactes de genre au moins deux, vues comme des surfaces hyperboliques $\Gamma_1\backslash\HH^2$ et $\Gamma_2\backslash\HH^2$ via le th\'eor\`eme d'uniformisation.
Appliqu\'ee \`a $M = \Gamma_i\backslash\HH^2$ pour $i\in\nolinebreak\{1,2\}$, la strat\'egie ci-dessus donne, pour tout $\varepsilon>0$ assez petit et tout $R>0$ assez grand par rapport à~$\varepsilon$, une surface compacte $S_i$ de genre au moins deux, une repr\'esentation $\rho_i : \pi_1(S_i)\to\pi_1(M) = \Gamma_i$, une repr\'esentation $R$-parfaite $\varrho_{R,i} : \pi_1(S_i)\to\PSL(2,\R)$ et une application de bord $(\varrho_{R,i},\rho_i)$-\'equivariante de $\PP^1(\R)$ dans lui-m\^eme qui est $(1+C\varepsilon)$-quasi-sym\'etrique, pour une certaine constante $C>0$ ind\'ependante de~$\varepsilon$.
Pour $\varepsilon$ assez petit la repr\'esentation $\rho_i$ est injective, ce qui implique que $\rho_i(\pi_1(S_i))\backslash\HH^2$ est un rev\^etement fini de $\Gamma_i\backslash\HH^2$.
Ce rev\^etement est quasi-conforme \`a une surface hyperbolique compacte $R$-parfaite, avec une constante de quasi-conformit\'e bien contr\^ol\'ee.
La conjecture d'Ehrenpreis s'en d\'eduit.

Notons que l'id\'ee de construire des rev\^etements de surfaces hyperboliques compactes \`a l'aide de pantalons immerg\'es bien recoll\'es remonte \`a \textcite{bow04}.

\subsection{\'Etape g\'eom\'etrique : conditions suffisantes d'injectivit\'e pour les repr\'esentations de groupes de surface} \label{subsec:etape-geom}

Expliquons \`a pr\'esent comment l'approche de Kahn--Markovi\'c a \'et\'e g\'en\'eralis\'ee par Hamenst\"adt et Kahn--Labourie--Mozes pour d\'emontrer le th\'eor\`eme principal et les th\'eor\`emes \ref{thm:quantitatif-PSL(n,C)} et~\ref{thm:quantitatif}.

Soit $S$ une surface compacte connexe orientée munie d'une d\'ecomposition en pantalons bipartie $\mathcal{P} = \mathcal{P}^+\sqcup\mathcal{P}^-$ et d'un graphe fini $\calG$ comme au paragraphe~\ref{subsec:struct-R-parf}.
Soient $G$ un groupe de Lie semi-simple et $\tau : \PSL(2,\R)\hookrightarrow G$ un plongement.

Au paragraphe~\ref{subsec:struct-R-parf} nous avons vu une notion de repr\'esentation $R$-parfaite de $\pi_1(S)$ dans $\PSL(2,\R)$.
Pour une telle repr\'esentation $\varrho_R : \pi_1(S)\to\PSL(2,\R)$ et pour des pantalons $\Pi^+\in\mathcal{P}^+$ et $\Pi^-\in\mathcal{P}^-$, on dira que la restriction de $\tau\circ\varrho_R$ à $\pi_1(\Pi^+)$ est une représentation \emph{$R$-parfaite} à valeurs dans~$G$ et que la restriction de $\tau\circ\varrho_R$ à $\pi_1(\Pi^-)$ est une représentation \emph{$(-R)$-parfaite} à valeurs dans~$G$.\footnote{Cette terminologie refl\`ete les orientations oppos\'ees consid\'er\'ees sur $\Pi^+$ et~$\Pi^-$, \cf figure~\ref{fig:decomp-pantalons},~\`a~droite.}

\'Etendant l'approche de Kahn--Markovi\'c (paragraphe~\ref{subsec:strategie-KM}) \`a des groupes $G$ plus g\'en\'eraux que $\PSL(2,\C)$, Hamenst\"adt et Kahn--Labourie--Mozes ont introduit, pour $\varepsilon,R>0$, les notions suivantes :
\begin{itemize}
  \item pour $\Pi\in\mathcal{P}$, une notion de \emph{représentation $(\varepsilon,\pm R)$-presque parfaite} de $\pi_1(\Pi)$ dans~$G$, à valeurs dans~$\Gamma$ : il s'agit d'une représentation qui est \og proche à $\varepsilon$ près\fg\ d'un conjugué d'une représentation $(\pm R)$-parfaite à valeurs dans~$G$, ce qui se traduit par l'existence d'une bonne \emph{donnée géométrique} associée à la représentation, décrivant une situation proche de celle du paragraphe~\ref{subsec:struct-R-parf} ;
  \item pour des pantalons $\Pi^+\in\mathcal{P}^+$ et $\Pi^-\in\mathcal{P}^-$ adjacents le long d'une courbe de bord~$a$, et pour des représentations $(\varepsilon,\pm R)$-presque parfaites $\rho^{\pm} : \pi_1(\Pi^{\pm})\to\Gamma$ telles que $\rho^+(a) = \rho^-(a)$, une notion pour $\rho^+$ et~$\rho^-$ d'\^etre \emph{$\varepsilon$-bien recollées} le long de~$a$ via leurs données géométriques : cela signifie qu'elles \og diff\`erent à $\varepsilon$ près par un décalage hyperbolique de~$1$\fg\ comme sur la figure~\ref{fig:pavage-hex} (\cf paragraphe~\ref{subsubsec:pavage-hex}), le long d'une copie de~$\HH^2$ induite par~$\tau$.
\end{itemize}
Les données géométriques sont exprim\'ees en termes d'\'el\'ements de $\Gamma\backslash G$, vus comme des \og rep\`eres de direction~$\tau$\fg\ dans l'espace tangent \`a l'espace localement sym\'etrique $\Gamma\backslash G/K$ (pour Hamenst\"adt) ou comme des raffinements de triplets de points sur un m\^eme $\tau$-cercle dans la vari\'et\'e de drapeaux $G/P_{\tau}$, modulo l'action de~$\Gamma$ (pour Kahn, Labourie et Mozes).
Les orbites de ces \'el\'ements par le flot $(\varphi_t)_{t\in\R}$ sur $\Gamma\backslash G$ correspondant \`a la multiplication \`a droite par $\tau\big(\big(\begin{smallmatrix} e^{t/2} & 0\\ 0 & e^{-t/2}\end{smallmatrix}\big)\big)$ doivent \og presque\fg\ se refermer modulo l'application de certaines sym\'etries (inversion ou rotation d'ordre trois) : \cf figure~\ref{fig:Triconn}, paragraphes~\ref{subsec:tau-P1-G/P}--\ref{subsec:pant-presque-parf-KLM}.

Une observation importante est que l'ensemble des classes de conjugaison de repr\'esentations $(\varepsilon,R)$-presque parfaites de $\pi_1(\Pi)$ dans le groupe discret~$\Gamma$ est fini.
Cela résulte du fait (\cf corollaire~\ref{cor:presque-parf->generique}.\eqref{item:nb-fini-presque-parf}) que si $\pi_1(\Pi) = \langle a,b,c \,|\, cba=1\rangle$ où $a,b,c$ correspondent aux courbes de bord de~$\Pi$, alors une telle représentation envoie, \`a conjugaison pr\`es, le triplet $(a,b,c)$ sur un triplet d'\'el\'ements de~$\Gamma$ proche de
$$\tau\left(\begin{pmatrix}e^{R/2} & 0\\ 0 & e^{-R/2}\end{pmatrix}, \begin{pmatrix}e^{-R/2} & e^{-R/2}-e^{R/2}\\ 0 & e^{R/2}\end{pmatrix}, \begin{pmatrix}e^{-R/2} & 0\\ e^{R/2}-e^{-R/2} & e^{R/2}\end{pmatrix}\right) \ \in G^3.$$
Un raisonnement analogue vaut pour les repr\'esentations $(\varepsilon,-R)$-presque parfaites.

La donnée géométrique associée à une représentation $(\varepsilon,\pm R)$-presque parfaite n'est pas unique, mais varie dans un espace continu (compact).

La premi\`ere \'etape de la d\'emonstration des th\'eor\`emes \ref{thm:quantitatif-PSL(n,C)} et~\ref{thm:quantitatif} consiste \`a donner les conditions suffisantes suivantes pour qu'une repr\'esentation $\rho : \pi_1(S)\to\Gamma$ soit injective (et donc que son image soit un sous-groupe de~$\Gamma$ isomorphe \`a $\pi_1(S)$).

\begin{proposition} \label{prop:pi-1-inj}
Dans le cadre du théorème~\ref{thm:quantitatif-PSL(n,C)}, ou plus généralement dans le cadre~\ref{cadre}, soient $\Gamma$ un r\'eseau cocompact irr\'eductible de~$G$ et $S$ une surface compacte munie d'une d\'ecomposition en pantalons bipartie $\mathcal{P} = \mathcal{P}^+\sqcup\mathcal{P}^-$ et d'un graphe fini $\mathcal{G}$ comme au paragraphe~\ref{subsec:struct-R-parf}.
Pour tout $\varepsilon>0$ assez petit et tout $R>0$ assez grand par rapport à~$\varepsilon$, si une représentation $\rho : \pi_1(S)\to\Gamma$ vérifie que
\begin{enumerate}
  \item\label{item:pi-1-inj-1} pour tout $\Pi^{\pm}\in\mathcal{P}^{\pm}$, la restriction de $\rho$ \`a $\pi_1(\Pi^{\pm})$ est $(\varepsilon/R,\pm R)$-presque parfaite,~et
  \item\label{item:pi-1-inj-2} on peut choisir les données géométriques associ\'ees de sorte que pour toute paire $(\Pi^+,\Pi^-)\in\mathcal{P}^+\times\mathcal{P}^-$ de pantalons adjacents, les restrictions de $\rho$ \`a $\pi_1(\Pi^+)$ et \`a $\pi_1(\Pi^-)$ soient $(\varepsilon/R)$-bien recoll\'ees,
\end{enumerate}
alors $\rho$ est injective.
\end{proposition}

Dans la version de \textcite{klm18}, les auteurs montrent de plus que $\rho$ admet une application de bord $(\delta,\tau)$-sullivannienne de paramètre $\delta$ arbitrairement petit : \cf proposition~\ref{prop:pi-1-inj-KLM}.

Dans la version de \textcite{ham15,ham20}, les conditions de $(\varepsilon/R,\pm R)$-presque perfection et de $(\varepsilon/R)$-bon recollement dans la proposition~\ref{prop:pi-1-inj} sont remplacées par des conditions plus simples mais plus fortes de $(e^{-\kappa R},\pm R)$-presque perfection et de $e^{-\kappa R}$-bon recollement, où $\kappa>0$ est une constante indépendante de~$R$.

\subsubsection{Injectivité selon \textcite{ham15,ham20}} \label{subsubsec:inj-H}

Soient $G/K$ l'espace riemannien sym\'etrique de~$G$, o\`u $K$ est un sous-groupe compact maximal de~$G$, et $M = \Gamma\backslash G/K$ l'espace localement sym\'etrique compact associ\'e \`a~$\Gamma$.
Soit $\rho : \pi_1(S)\to\Gamma$ une représentation vérifiant les conditions \eqref{item:pi-1-inj-1} et~\eqref{item:pi-1-inj-2} de la proposition~\ref{prop:pi-1-inj} pour $R>0$ assez grand.

Pour montrer que $\rho$ est injective, Hamenst\"adt suit une approche géométrique : à partir d'une structure hyperbolique $R$-parfaite sur~$S$ (\cf paragraphe~\ref{subsec:struct-R-parf}) et des données géométriques des restrictions de $\rho$ aux $\pi_1(\Pi)$ pour $\Pi\in\mathcal{P}$, elle construit :
\begin{itemize}
  \item une surface $S'$ homéomorphe à~$S$, d\'ecoup\'ee en un nombre fini de \og morceaux\fg\ (triangles, bandes ou anneaux), chaque morceau \'etant muni d'une structure hyperbolique ou euclidienne pour laquelle son bord est géodésique ;
  \item une application continue $f : S'\to M$, de morphisme induit $\rho = f_* : \pi_1(S) = \pi_1(S')\to\pi_1(M)=\Gamma$, qui est une immersion en restriction à chaque morceau ; elle se rel\`eve en une immersion par morceaux continue $\rho$-équivariante $\widetilde{f} : \widetilde{S'}\to G/K$.
\end{itemize}
La surface $S'$ est obtenue à partir de la surface hyperbolique $R$-parfaite $S$ en rajoutant, au niveau de chaque courbe de bord entre pantalons adjacents de~$\mathcal{P}$, un fin anneau euclidien (réduit à un cercle lorsque $G$ est de rang r\'eel un).
Les \og morceaux\fg\ sont ces anneaux et, pour chaque pantalon de~$\mathcal{P}$, deux triangles équilatéraux (dont le diamètre est borné lorsque $R$ tend vers l'infini) et trois bandes hyperboliques (dont la longueur est proche de~$R$) qui partitionnent le pantalon : \cf figure~\ref{fig:pant-immerge}.
La condition~\eqref{item:pi-1-inj-1} de la proposition~\ref{prop:pi-1-inj} permet de construire $f$ de sorte que sa restriction à chaque morceau soit presque isométrique, et la condition~\eqref{item:pi-1-inj-2} de sorte que les angles de recollements soient proches de~$\pi$ et que les décalages entre pantalons soient proches de décalages hyperboliques de longueur~$1$ comme au paragraphe~\ref{subsec:struct-R-parf}.

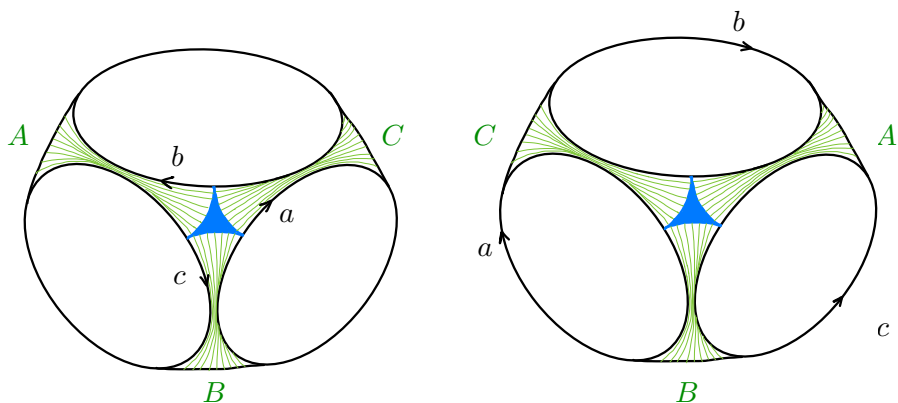
\begin{figure}[h!]
\centering
\labellist
\small\hair 2pt
\pinlabel {$b$} [u] at 183 452
\pinlabel {$a$} [u] at 278 400
\pinlabel {$c$} [u] at 185 345
\pinlabel {\textcolor{mygreen}{$A$}} [u] at 45 470
\pinlabel {\textcolor{mygreen}{$C$}} [u] at 370 470
\pinlabel {\textcolor{mygreen}{$B$}} [u] at 215 245
\pinlabel {$b$} [u] at 670 570
\pinlabel {$a$} [u] at 450 370
\pinlabel {$c$} [u] at 795 300
\pinlabel {\textcolor{mygreen}{$C$}} [u] at 450 470
\pinlabel {\textcolor{mygreen}{$A$}} [u] at 798 470
\pinlabel {\textcolor{mygreen}{$B$}} [u] at 625 245
\endlabellist
\includegraphics[width=5cm]{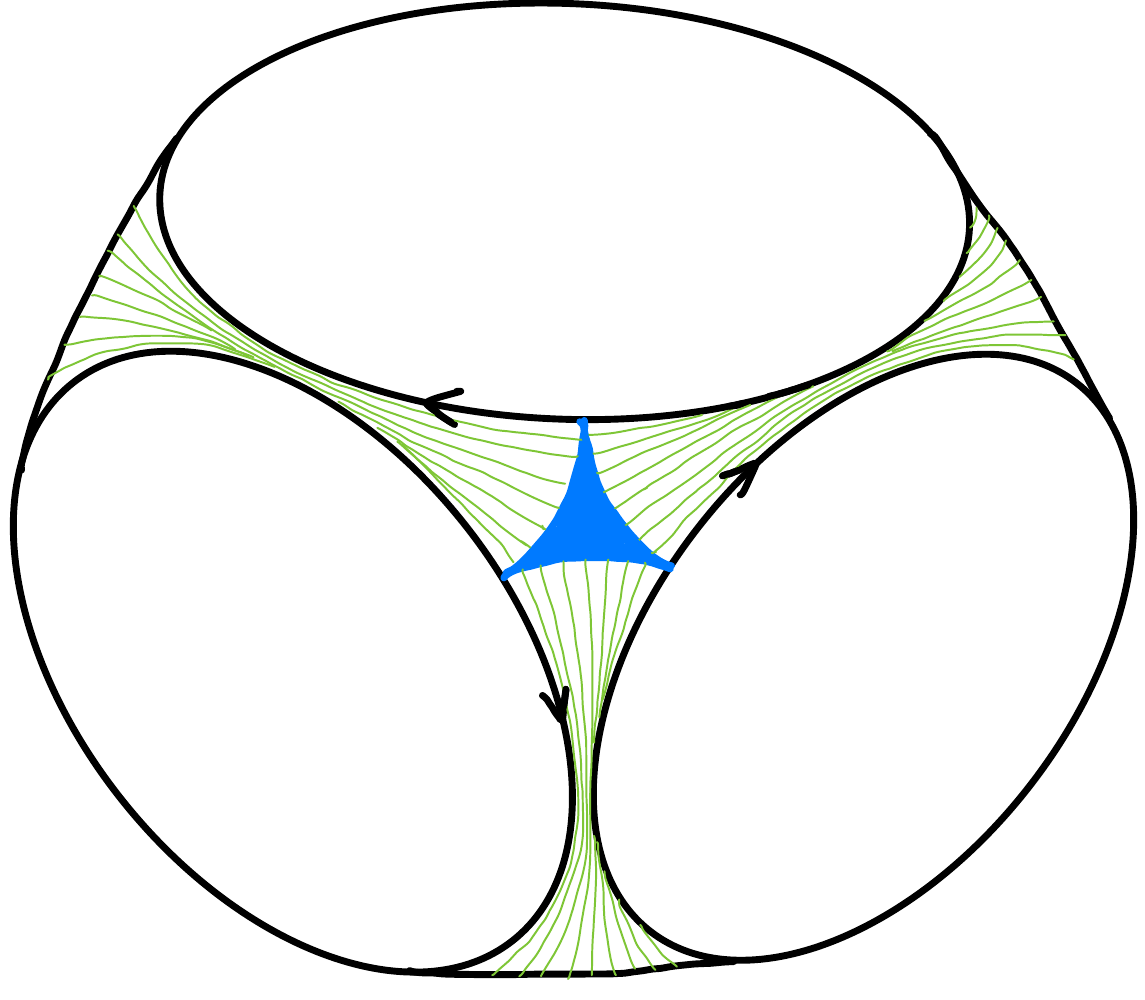}
\hspace{1cm}
\includegraphics[width=5cm]{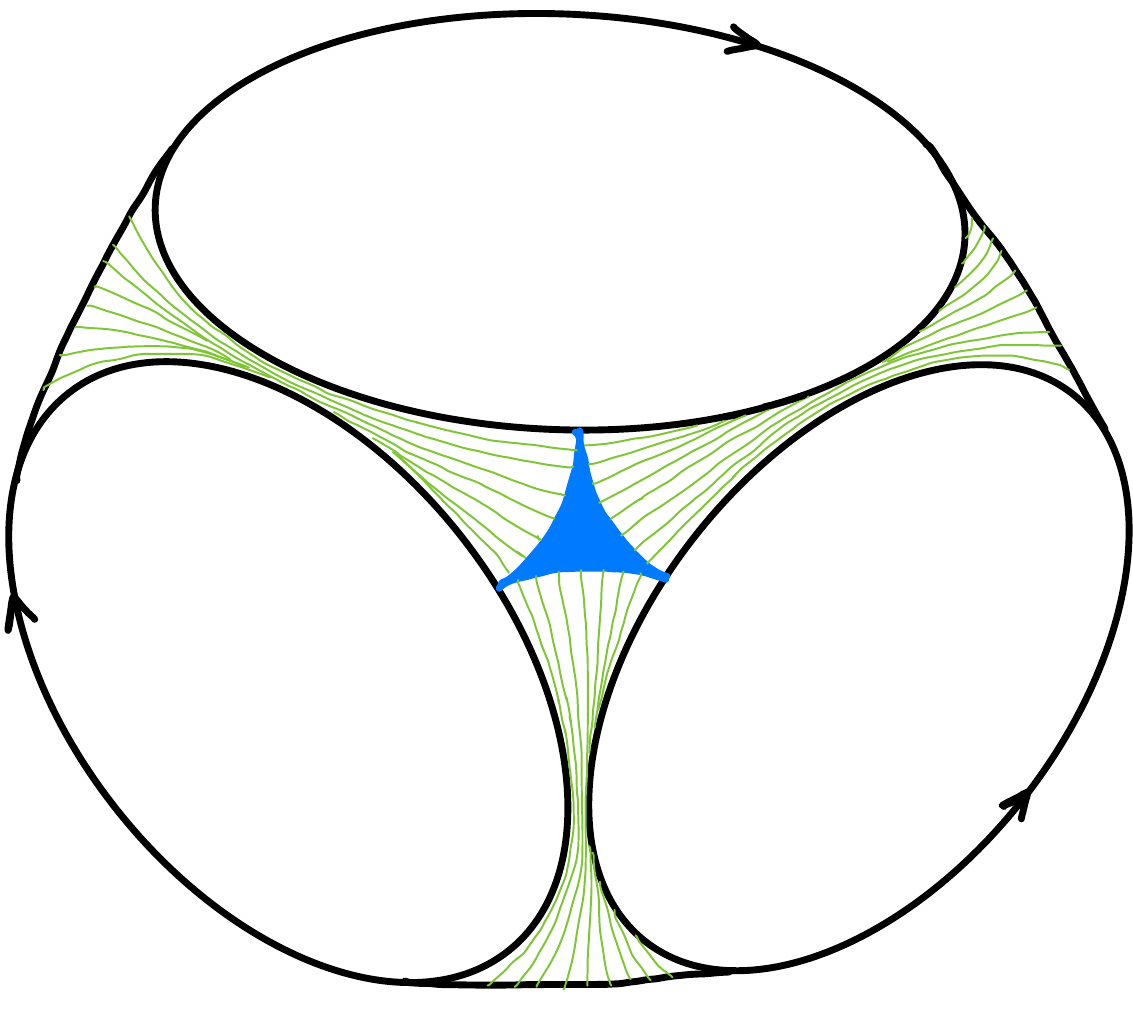}
\vspace{0.2cm}
\caption{D\'ecoupage d'un pantalon hyperbolique en deux triangles \'equilat\'eraux et trois bandes $A,B,C$ (vues de devant et de derri\`ere)}
\label{fig:pant-immerge}
\end{figure}

L'application $\widetilde{f}$ définit une métrique des chemins $\pi_1(S')$-invariante sur~$\widetilde{S'}$ : par définition, la distance entre deux points $x,y$ de~$\widetilde{S'}$ est la borne inférieure des longueurs, dans $G/K$, des images par~$\widetilde{f}$ de chemins de $x$ à $y$ dans~$\widetilde{S'}$.
Cette borne inférieure est en fait un minimum : l'espace métrique $\widetilde{S'}$ est géodésique par le théorème de Hopf--Rinow.
Par un contrôle fin de la géométrie de l'immersion par morceaux, Hamenst\"adt montre, pour $R>0$ suffisamment grand, que l'on peut trouver, pour toute droite g\'eod\'esique $\mathcal{L}$ de~$\widetilde{S'}$, une suite $(x_n)_{n\in\Z}$ de points de~$\mathcal{L}$ tels~que
\begin{itemize}
  \item la distance dans $\widetilde{S'}$ entre $x_n$ et $x_{n+1}$ soit uniformément majorée,
  \item la distance dans $G/K$ entre $\widetilde{f}(x_n)$ et $\widetilde{f}(x_{n+1})$ soit uniformément minorée,
  \item l'angle entre $[\widetilde{f}(x_n),\widetilde{f}(x_{n-1})]$ et $[\widetilde{f}(x_n),\widetilde{f}(x_{n+1})]$ soit suffisamment proche de~$\pi$,
  \item la direction de $[\widetilde{f}(x_n),\widetilde{f}(x_{n+1})]$ soit suffisamment proche de la direction r\'eguli\`ere $\mathsf{h} = \dd\tau\big(\big(\begin{smallmatrix} 1 & 0 \\ 0 & -1 \end{smallmatrix}\big)\big) \in \mathfrak{a}^+$.
\end{itemize}
(On note ici $\mathfrak{a}^+$ une chambre de Weyl d'un sous-espace de Cartan~$\mathfrak{a}$ comme au paragraphe~\ref{subsubsec:cond-reg}.
Rappelons que tout segment orient\'e $[z_1,z_2]$ dans $G/K$ est de la forme $[gK,g\exp(\mathsf{h}')K]$ o\`u $g\in G$ et $\mathsf{h}'\in\mathfrak{a}^+$ ; l'\'el\'ement $\mathsf{h}'$ modulo $\R_+^*$ est la \emph{direction} de $[z_1,z_2]$.)

Un r\'esultat de \textcite[Th.\,7.2 \& Cor.\,7.13]{klp14} (\og lemme de Morse\fg) implique alors que le chemin g\'eod\'esique par morceaux $\bigcup_{n\in\Z} [f(x_n),f(x_{n+1})]$ est une quasi-g\'eod\'esique de $G/K$, à distance bornée d'un plat de $G/K$.
En appliquant ceci \`a une droite géodésique $\mathcal{L}$ invariante par $\gamma\in\pi_1(S)\smallsetminus\{1\}$, on voit que la repr\'esentation $\rho$ est injective.
En fait, le lemme de Morse donne une uniformit\'e sur les quasi-g\'eod\'esiques, ce qui permet de voir que $\rho$ est $P_{\tau}$-anosovienne au sens du paragraphe~\ref{subsec:intro-Anosov}.

\subsubsection{Injectivité selon \textcite{klm18}}

Soit $G/P_{\tau}$ la variété de drapeaux associée à~$\tau$, comme aux paragraphes \ref{subsec:intro-Anosov}, \ref{subsec:quantitatif-PSL(n,C)} et~\ref{subsec:prop-sullivan}.
Pour d\'emontrer la proposition~\ref{prop:pi-1-inj} et son raffinement (proposition~\ref{prop:pi-1-inj-KLM}), Kahn, Labourie et Mozes observent que, grâce à la condition~\eqref{item:pi-1-inj-1}, pour tout $\gamma\in\pi_1(S)$ correspondant à une courbe de bord d'un pantalon de~$\mathcal{P}$, l'élément $\rho(\gamma)\in G$ admet un unique point fixe attractif dans $G/P_{\tau}$ (corollaire~\ref{cor:presque-parf->generique}.\eqref{item:presque-parf->generique}) ; on a donc une application $(\varrho_R,\rho)$-\'equivariante naturelle d'un sous-ensemble dense de $\PP^1(\R)$ vers $G/P_{\tau}$, qui au point fixe attractif dans $\PP^1(\R)$ de $\varrho_R(\gamma)$ associe le point fixe attractif dans $G/P_{\tau}$ de $\rho(\gamma)$.
Ils utilisent alors la condition~\eqref{item:pi-1-inj-2} pour montrer que cette application se prolonge en une application $(\varrho_R,\rho)$-\'equivariante \emph{continue} $\xi : \PP^1(\R)\to G/P_{\tau}$, avec un contrôle suffisant pour établir qu'à $\delta>0$ petit fixé, l'application $\xi$ est $(\delta,\tau)$-sullivannienne (d\'efinition~\ref{def:sullivan}) dès que $\varepsilon>0$ est suffisamment petit.
La proposition~\ref{prop:Sullivan}.\eqref{item:Sullivan-Anosov} assure alors que $\rho$ est $P_{\tau}$-anosovienne au sens du paragraphe~\ref{subsec:intro-Anosov}, donc injective.
Voir la partie~\ref{sec:geom-KLM} pour plus de d\'etails.

\subsection{\'Etape dynamique} \label{subsec:etape-dyn}

La deuxi\`eme \'etape de la d\'emonstration des th\'eor\`emes \ref{thm:quantitatif-PSL(n,C)} et~\ref{thm:quantitatif} con\-siste \`a \'etablir les propri\'et\'es d'existence suivantes, qui font intervenir les notions du paragraphe~\ref{subsec:etape-geom}.

\begin{proposition} \label{prop:pant-dans-reseau-general}
Dans le cadre~\ref{cadre}, soit $\Gamma$ un r\'eseau cocompact irr\'eductible de~$G$, et soient $\Pi^+$ et~$\Pi^-$ deux pantalons adjacents le long d'une courbe de bord~$a$ comme sur la figure~\ref{fig:decomp-pantalons}, \`a droite.
Il existe $C>0$ tel que pour tout $\varepsilon>0$ et tout $R>0$ assez grand par rapport à~$\varepsilon$,
\begin{enumerate}
  \item\label{item:pant-dans-reseau-1} il existe des repr\'esentations $(\varepsilon/R,\pm R)$-presque parfaites de $\pi_1(\Pi^{\pm})$ dans~$\Gamma$ ;
  \item\label{item:pant-dans-reseau-2} pour toute donnée géométrique associée à une repr\'esentation $(\varepsilon/R,R)$-presque parfaite de $\pi_1(\Pi^+)$ dans~$\Gamma$, on peut trouver une donnée géométrique associée à une repr\'esentation $(\varepsilon/R,-R)$-presque parfaite de $\pi_1(\Pi^-)$ dans~$\Gamma$ de sorte que les représentations soient $(C\varepsilon/R)$-bien recollées le long de~$a$ via ces données géométriques.
\end{enumerate}
\end{proposition}

On peut quantifier le point~\eqref{item:pant-dans-reseau-2} en utilisant des mesures sur un espace continu $\Geom_{\varepsilon,\pm R}$ qui paramètre, dans chacune des approches de \textcite{ham15,ham20} ou de \textcite{klm18}, les données géométriques des représentations $(\varepsilon/R,\pm R)$-presque parfaites de $\pi_1(\Pi^{\pm})$ dans~$\Gamma$ modulo conjugaison.

\begin{proposition} \label{prop:pant-dans-reseau-mesures}
Dans le cadre de la proposition~\ref{prop:pant-dans-reseau-general},
\begin{enumerate}
  \item[\namedlabel{item:exist-2'}{(2)'}] il existe des mesures $\mu_{\varepsilon,R}$ sur $\Geom_{\varepsilon,R}$ et $\mu_{\varepsilon,-R}$ sur $\Geom_{\varepsilon,-R}$, de même masse totale, telles que pour tout sous-ensemble mesurable $A$ de $\Geom_{\varepsilon,R}$, l'ensemble des \'el\'ements de $\Geom_{\varepsilon,-R}$ qui sont $(C\varepsilon/R)$-bien recoll\'es \`a au moins un \'el\'ement de~$A$ soit de $(\mu_{\varepsilon,-R})$-mesure sup\'erieure ou \'egale \`a $\mu_{\varepsilon,R}(A)$.
\end{enumerate}
\end{proposition}

Dans la partie~\ref{sec:dyn}, nous d\'etaillons la d\'emonstration des propositions \ref{prop:pant-dans-reseau-general} et~\ref{prop:pant-dans-reseau-mesures} en suivant l'approche d\'evelopp\'ee par \textcite{klm18}.
La d\'emonstration de Hamenst\"adt est analogue, mais avec un formalisme un peu diff\'erent, dû à sa définition différente des données géométriques.

Dans les deux approches, et d\'ej\`a chez \textcite{km12}, la d\'emonstration repose crucialement sur la propri\'et\'e de mélange suivante, appliqu\'ee \`a $\mathsf{h}' = \mathsf{h}$ comme au cadre~\ref{cadre}.
Cette propri\'et\'e, classique, provient de la d\'ecroissance exponentielle des coefficients matriciels des repr\'esentations temp\'er\'ees (\cf la partie~4 de \cite{ber13} ou l'appendice~B de \cite{klm18}).

\begin{fact} \label{fait:melange}
Soient $\mathfrak{a}$ un sous-espace de Cartan de~$\mathfrak{g}$ comme au paragraphe~\ref{subsubsec:cond-reg} et $\Gamma$ un r\'eseau irr\'eductible de~$G$.
Pour tout $\mathsf{h}'\in\mathfrak{a}$, le flot $(\varphi_t)_{t\in\R}$ sur $\Gamma\backslash G$ donné par la multiplication à droite par $\exp(t\mathsf{h}')$ est exponentiellement mélangeant : il existe $k\in\N$ et $C,q>0$ tel que pour toutes fonctions $\psi,\theta\in C^k(\Gamma\backslash G,\R)$ et tout $R>0$,
$$\left| \int_{[g]\in\Gamma\backslash G} \psi([g]) \, (\theta\circ\varphi_R)([g]) \, \dd [g] - \left( \int_{\Gamma\backslash G} \psi \right) \left( \int_{\Gamma\backslash G} \theta \right) \right| \leq C e^{-qR} \, \Vert\psi\Vert_{C^k} \, \Vert\theta\Vert_{C^k}.$$
\end{fact}

\textcite{km12}, \textcite{ham15}, et \textcite{klm18} n'ont en fait pas besoin du mélange exponentiel (d\'ecroissance en $e^{-qR}$), seulement d'un mélange polynomial (d\'ecroissance en $1/R^{\ell}$ pour un certain $\ell\geq 2$).

\subsection{\'Etape combinatoire} \label{subsec:etape-comb}

Pour conclure la d\'emonstration des th\'eor\`emes \ref{thm:quantitatif-PSL(n,C)} et~\ref{thm:quantitatif}, il s'agit de prendre les repr\'esentations de groupes de pantalons $(\varepsilon/R,\pm R)$-presque parfaites $(C\varepsilon/R)$-bien recollées des propositions \ref{prop:pant-dans-reseau-general} et~\ref{prop:pant-dans-reseau-mesures}, avec les données géométriques appropriées correspondantes, et de montrer qu'on peut les agencer de mani\`ere ad\'equate pour obtenir une repr\'esentation d'une surface \emph{compacte}~$S$, avec une d\'ecomposition en pantalons bipartie et un graphe fini associé, v\'erifiant les hypothèses de la proposition~\ref{prop:pi-1-inj}.
Pour cela, \textcite{km12} et \textcite{klm18} utilisent le lemme classique suivant (\cf la remarque~\ref{rem:lem-mariages-H} pour l'approche de Hamenst\"adt).

\begin{fact}[{Lemme des mariages de \cite{hal35}}] \label{fait:mariages}
Soient $\mathcal{E}^+$ et~$\mathcal{E}^-$ deux ensembles finis de m\^eme cardinal et $\mathcal{M} \subset \mathcal{E}^+\times\mathcal{E}^-$ un sous-ensemble.
Alors il existe une bijection $\psi : \mathcal{E}^+\to\mathcal{E}^-$ telle que $(x,\psi(x))\in\mathcal{M}$ pour tout $x\in\mathcal{E}^+$ d\`es que la condition suivante est v\'erifi\'ee :
\begin{equation} \label{eqn:cond-mariage}
\forall A\subset\mathcal{E}^+, \quad\quad \# \bigcup_{x\in A} \, \{ y\in\mathcal{E}^- ~|~ (x,y)\in\mathcal{M}\} \geq \# A.
\end{equation}
\end{fact}

On pense \`a $\mathcal{M}$ comme \`a l'ensemble des mariages possibles entre \'el\'ements de $\mathcal{E}^+$ et de~$\mathcal{E}^-$.
La condition \eqref{eqn:cond-mariage} dit que pour tout sous-ensemble $A$ de~$\mathcal{E}^+$, il existe au moins $\# A$ \'el\'ements de~$\mathcal{E}^-$ avec la propriété de pouvoir \^etre mari\'e \`a au moins un \'el\'ement de~$A$.

Dans notre contexte, $\mathcal{E}^{\pm}$ sera un ensemble fini obtenu en prenant certaines données géométriques de représentations $(\varepsilon/R,\pm R)$-parfaites de groupes de pantalons, avec certaines multiplicit\'es bien choisies données par les mesures $\mu_{\varepsilon,\pm R}$ de la proposition~\ref{prop:pant-dans-reseau-mesures}, et $\mathcal{M}$ correspondra \`a l'ensemble des paires $(C\varepsilon/R)$-bien recollées : \cf paragraphe~\ref{subsec:dem-existe-graphe}.
La bijection $\psi$ du lemme des mariages permet de construire un graphe fini biparti trivalent $\mathcal{G}$ de sommets $\mathcal{P} = \mathcal{P}^+ \sqcup \mathcal{P}^-$ de la manière suivante : les sommets $\mathcal{P}^{\pm}$ sont obtenus en prenant le quotient de $\mathcal{E}^{\pm}$ par la transformation d'ordre trois correspondant à la permutation cyclique des courbes de bord (\cf \eqref{eqn:tri}) ; pour tout $x\in\mathcal{E}^+\sqcup\mathcal{E}^-$, on relie la classe de~$x$ et celle de~$\psi(x)$ par une arête.
En épaississant ce graphe, on obtient une surface compacte~$S$ avec une décomposition en pantalons bipartie \'etiquet\'ee par~$\mathcal{P}$ et, pour tout pantalon $\Pi^{\pm}$ correspondant \`a un \'el\'ement de $\mathcal{P}^{\pm}$, une classe de conjugaison de représentations $(\varepsilon/R,\pm R)$-presque parfaites de $\pi_1(\Pi^{\pm})$ dans~$\Gamma$ ; de plus, les classes de deux pantalons adjacents admettent des représentants $(C\varepsilon/R)$-bien recollés.
On en déduit (\cf lemme~\ref{lem:recoller-repr}) une représentation $\rho$ de $\pi_1(S)$ dans~$\Gamma$ à laquelle la proposition~\ref{prop:pi-1-inj} s'applique, et qui est donc injective.

\section{\'Etape g\'eom\'etrique} \label{sec:geom-KLM}

Dans cette partie, nous pr\'esentons les notions de Kahn, Labourie et Mozes de \emph{représentation $(\varepsilon,R)$-presque parfaite} d'un groupe de pantalon (paragraphe~\ref{subsec:pant-presque-parf-KLM}) et de représentations \emph{$\varepsilon$-bien recoll\'ees} (paragraphe~\ref{subsec:recolle-KLM}).
Nous donnons les grandes lignes de leur d\'emonstration de la proposition~\ref{prop:pi-1-inj-KLM} ci-dessous, qui est une version plus précise de la proposition~\ref{prop:pi-1-inj}, faisant intervenir la notion d'application sullivannienne du paragraphe~\ref{subsec:sullivan}.

\smallskip

\emph{Dans toute la partie, on travaille dans le cadre~\ref{cadre} ; la condition~(R) du paragraphe~\ref{subsubsec:cond-R} n'a pas besoin d'être satisfaite.
On fixe un réseau cocompact irréductible $\Gamma$ de~$G$ et un sous-groupe compact maximal $K$ de~$G$ contenant $\tau(\PSO(2))$.}

\subsection{Préliminaires : $\tau$-triangles et tripodes} \label{subsec:tau-P1-G/P}

\subsubsection{$\tau$-triangles}

Rappelons que le groupe $\PSL(2,\R)$ agit simplement transitivement sur (et donc s'identifie \`a) l'ensemble $\mathcal{T}$ des triplets de points deux à deux distincts positivement orientés de $\PP^1(\R)$.
Ainsi, tout élément de $\mathcal{T}$ est de la forme $h\cdot (0,\infty,-1)$ où $h\in\PSL(2,\R)$ est unique.
(L'ensemble $\mathcal{T}$ s'identifie \'egalement \`a l'ensemble des repères de~$\HH^2$ (o\`u un rep\`ere est par d\'efinition une base orthonorm\'ee positivement orient\'ee d'un espace tangent $T_x\HH^2$ o\`u $x\in\HH^2$), ainsi qu'au fibr\'e unitaire tangent $T^1\HH^2$ de~$\HH^2$.)

Comme au paragraphe~\ref{subsubsec:cercles}, le plongement $\tau : \PSL(2,\R)\hookrightarrow G$ induit un plongement $\tau$-\'equivariant $\underline{\tau} : \PP^1(\R)\hookrightarrow G/P_{\tau}$.
On note encore $\underline{\tau}$ le plongement induit de $(\PP^1(\R))^n$ dans $(G/P_{\tau})^n$ pour $n\geq 2$.
La $G$-orbite de $\underline{\tau}(0,\infty)$ (pour l'action diagonale de~$G$) est un ouvert dense de $(G/P_{\tau})^2$ form\'e des couples de points dits \emph{transverses}, ou encore en position g\'en\'erique.

\begin{definition} \label{def:tau-triangle}
Un \emph{$\tau$-triangle de $G/P_{\tau}$} est un triplet de points deux à deux transverses de $G/P_{\tau}$ de la forme $g\cdot\underline{\tau}(0,\infty,-1)$ où $g\in G$.
\end{definition}

Notons que $g\cdot\underline{\tau}(T)$ est un $\tau$-triangle pour tout $g\in G$ et tout $T\in\mathcal{T}$, par \'equivariance de~$\underline{\tau}$.
Un $\tau$-triangle détermine de manière unique un $\tau$-cercle $g\circ\underline{\tau} : \PP^1(\R)\to G/P_{\tau}$ au sens du paragraphe~\ref{subsubsec:cercles}.

Par définition, le groupe $G$ agit transitivement sur l'espace des $\tau$-triangles de $G/P_{\tau}$.
Le stabilisateur de $\underline{\tau}(0,\infty,-1)$ est le centralisateur $Z^{\tau}$ de $\dd\tau(\psl(2,\R))$ dans~$G$ (suppos\'e compact, \cf cadre~\ref{cadre}).
Ainsi, $[g] \mapsto g\cdot\underline{\tau}(0,\infty,-1)$ est une bijection $G$-\'equivariante entre $G/Z^{\tau}$ et l'espace des $\tau$-triangles de $G/P_{\tau}$.

\begin{remarque}
L'espace des $\tau$-triangles de $G/P_{\tau}$ est en g\'en\'eral strictement inclus dans l'espace des triplets ordonn\'es de points deux \`a deux transverses de $G/P_{\tau}$.
Par exemple, pour $G=\PSL(n,\C)$ et $\tau : \PSL(2,\R)\hookrightarrow G$ le plongement irr\'eductible, ces espaces sont de dimensions complexes respectives $(n+1)(n-1)$ et $3n(n-1)/2$.
\end{remarque}

\subsubsection{Tripodes} \label{subsubsec:tripodes}

Afin de d\'efinir les repr\'esentations $(\varepsilon,R)$-presque parfaites,
on a besoin d'objets un peu plus pr\'ecis que les $\tau$-triangles de $G/P_{\tau}$, qui soient param\'etr\'es par $G$ plut\^ot que $G/Z^{\tau}$.
Pour cela, Kahn, Labourie et Mozes introduisent une notion de \emph{tripode}, qui est par d\'efinition un isomorphisme d'une copie abstraite $G_0$ du groupe~$G$ vers~$G$ ou, de mani\`ere \'equivalente, un automorphisme de~$G$.
En r\'ealit\'e, pour d\'emontrer les th\'eor\`emes \ref{thm:quantitatif-PSL(n,C)} et~\ref{thm:quantitatif} il n'est pas n\'ecessaire de consid\'erer tous les automorphismes de~$G$ : les automorphismes int\'erieurs (donn\'es par la conjugaison par un \'el\'ement de~$G$) suffisent.
Dans la suite de cet exposé, nous travaillerons donc directement avec le groupe~$G$ plutôt qu'avec les tripodes de Kahn, Labourie et Mozes.

\subsubsection{Transformations} \label{subsec:rot-inv-phi}

Kahn, Labourie et Mozes considèrent les transformations suivantes de l'espace $G/Z^{\tau}$ des $\tau$-triangles de $G/P_{\tau}$ (\cf figure~\ref{fig:rot-inv-phi-t}) :
\begin{itemize}
  \item l'involution $\inv$ (\og inversion\fg) qui envoie $g\cdot\underline{\tau}(0,\infty,-1)$ sur $g\cdot\underline{\tau}(\infty,0,1)$ ; elle correspond dans $G/Z^{\tau}$ à la multiplication \`a droite par $\tau((\begin{smallmatrix} 0 & 1\\ -1 & 0\end{smallmatrix}))$ ;
  \item le flot $(\varphi_t)_{t\in\R}$ qui envoie $g\cdot\underline{\tau}(0,\infty,-1)$ sur $g\cdot\underline{\tau}(0,\infty,-e^t)$ ; il correspond dans $G/Z^{\tau}$ à la multiplication \`a droite par $\tau\big(\big(\begin{smallmatrix} e^{t/2} & 0\\ 0 & e^{-t/2}\end{smallmatrix}\big)\big)$;
  \item la transformation $\rot$ (\og rotation\fg) d'ordre trois qui envoie $g\cdot\underline{\tau}(0,\infty,-1)$ sur $g\cdot\nolinebreak\underline{\tau}(\infty,-1,0)$ ; elle correspond dans $G/Z^{\tau}$ à la multiplication \`a droite\,par\,$\tau\big(\big(\begin{smallmatrix} 1 & 1\\ -1 & 0\end{smallmatrix})\big)$.
\end{itemize}
Ces transformations se relèvent en des transformations de~$G$, encore notées $\inv$, $\varphi_t$ et~$\rot$, données par la multiplication à droite par les mêmes éléments.
Elles commutent avec l'action de~$G$ par multiplication \`a gauche (que nous noterons parfois avec un point pour \'eviter toute confusion : $g_1\cdot g_2 := g_1g_2$).

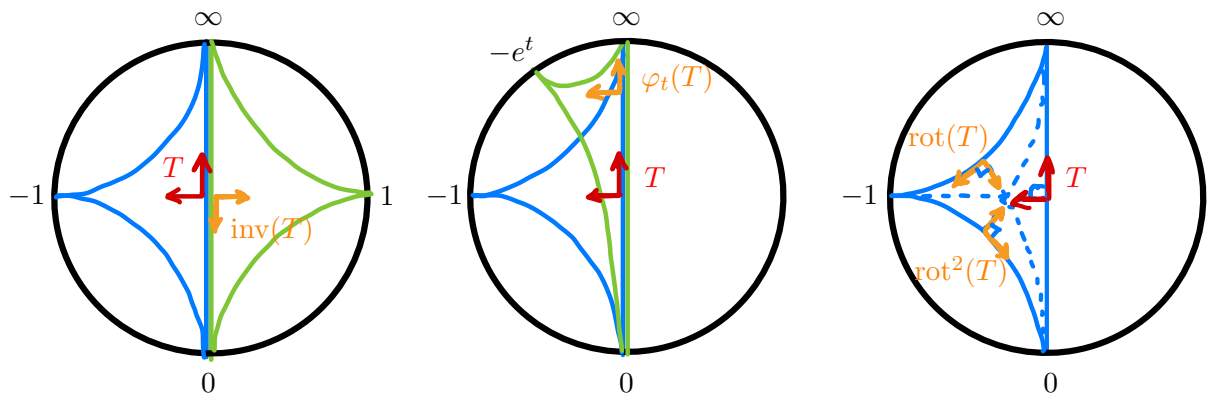
\begin{figure}[h!]
\centering
\labellist
\small\hair 2pt
\pinlabel {$\infty$} [u] at 88 538
\pinlabel {$\infty$} [u] at 219 538
\pinlabel {$\infty$} [u] at 352 538
\pinlabel {$-1$} [u] at 31 482
\pinlabel {$-1$} [u] at 162 482
\pinlabel {$-1$} [u] at 292 482
\pinlabel {$1$} [u] at 144 482
\pinlabel {$-e^t$} [u] at 183 528
\pinlabel {$0$} [u] at 88 423
\pinlabel {$0$} [u] at 219 423
\pinlabel {$0$} [u] at 352 423

\pinlabel {\textcolor{red}{$T$}} [u] at 77 491
\pinlabel {\textcolor{orange}{$\inv(T)$}} [u] at 108 471
\pinlabel {\textcolor{red}{$T$}} [u] at 228 488
\pinlabel {\textcolor{orange}{$\varphi_t(T)$}} [u] at 235 519
\pinlabel {\textcolor{red}{$T$}} [u] at 360 488
\pinlabel {\textcolor{orange}{$\rot(T)$}} [u] at 320 500
\pinlabel {\textcolor{orange}{$\rot^2(T)$}} [u] at 324 458
\endlabellist
\includegraphics[height=4.8cm]{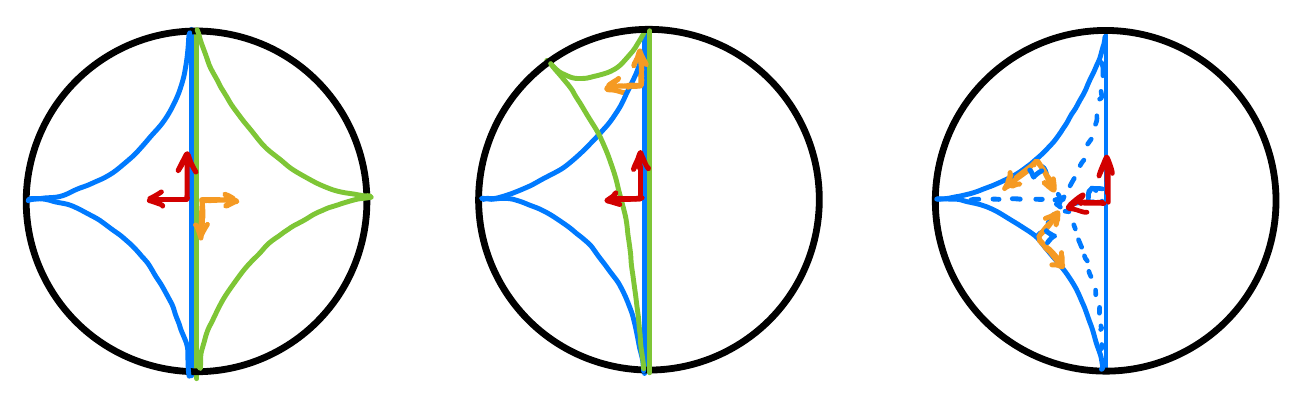}
\vspace{-0.5cm}
\caption{Les transformations $\inv$, $\varphi_t$ et $\rot$  de $G/Z^{\tau}$ proviennent de transformations $\inv$, $\varphi_t$ et $\rot$ de $\mathcal{T} \simeq \{ \text{repères de }\HH^2\}$
représentées ici ; en identifiant $\mathcal{T}$ au fibr\'e unitaire tangent $T^1\HH^2$, le flot $(\varphi_t)_{t\in\R}$ s'identifie au flot g\'eod\'esique}
\label{fig:rot-inv-phi-t}
\end{figure}

\subsubsection{Points de vue \'equivalents}

Le plongement $\tau : \PSL(2,\R)\hookrightarrow G$ induit un plongement $\tau$-\'equivariant $\underline{\underline{\tau}} : \HH^2\hookrightarrow\nolinebreak G/K$.
On appellera \emph{$\tau$-copie de~$\HH^2$} une surface totalement géodésique de $G/K$ de la forme $g\cdot\underline{\underline{\tau}}(\HH^2)$ o\`u $g\in G$.
Il est facile de voir (en utilisant par exemple la condition des \emph{syst\`emes de triplets de Lie}, \cf \cite[Ch.\,IV, \S\,7]{hel01}) que toute surface totalement g\'eod\'esique de $G/K$ est soit contenue dans un plat, soit égale à une $\tau'$-copie de~$\HH^2$ pour un certain plongement $\tau'$ de $\PSL(2,\R)$ (ou $\SL(2,\R)$) dans~$G$.

Les objets suivants s'identifient de mani\`ere $G$-\'equivariante :
\begin{itemize}
  \item les $\tau$-triangles de $G/P_{\tau}$,
  \item les triplets ordonn\'es de points du bord visuel de l'espace sym\'etrique $G/K$, correspondant \`a trois rayons g\'eod\'esiques de $G/K$
  contenus dans une m\^eme $\tau$-copie de~$\HH^2$, incidents en un point, et formant des angles de $2\pi/3$ en ce point,
  \item les triangles id\'eaux de $\tau$-copies de~$\HH^2$ dans $G/K$,
  \item les rep\`eres de $\tau$-copies de~$\HH^2$.
\end{itemize}
Chacune de ces classes d'objets est paramétrée par $G/Z^{\tau}$ où $Z^{\tau}$ est compact.

L\`a o\`u Kahn, Labourie et Mozes travaillent avec les $\tau$-triangles de $G/P_{\tau}$, Hamenst\"adt travaille avec les rep\`eres de $\tau$-copies de~$\HH^2$.
L\`a o\`u Kahn, Labourie et Mozes travaillent avec les tripodes, Hamenst\"adt travaille avec les \og $\tau$-rep\`eres de $G/K$\fg, c'est-\`a-dire avec une $G$-orbite de rep\`eres $(v_1,v_2,\dots,v_r)$ de $G/K$ obtenus en compl\'etant des rep\`eres $(v_1,v_2)$ de $\tau$-copies de~$\HH^2$, telle que le stabilisateur de cette $G$-orbite est trivial.

\subsection{Représentations presque parfaites selon Kahn, Labourie et Mozes} \label{subsec:pant-presque-parf-KLM}

\subsubsection{Intuition géométrique} \label{subsubsec:intuition-geom-KLM}

Pour tout élément hyperbolique $\alpha\in\PSL(2,\R)$, on note $\alpha^{\attract}$ (\resp $\alpha^{\repuls}$) son point fixe attractif (\resp répulsif) dans $\partial\HH^2=\PP^1(\R)$.

Soit $\Pi$ un pantalon de groupe fondamental $\pi_1(\Pi) = \langle a,b,c \,|\, cba = 1\rangle$, où $a,b,c$ correspondent aux courbes de bord de~$\Pi$.
Pour $R>0$, soit $\varrho : \pi_1(\Pi)\to\PSL(2,\R)$ une repr\'esentation \emph{$R$-parfaite} (\resp \emph{$(-R)$-parfaite}) comme au paragraphe~\ref{subsec:struct-R-parf}, c'est-\`a-dire l'holonomie d'une structure hyperbolique sur~$\Pi$ pour laquelle les trois courbes de bord sont de longueur~$2R$ et les points fixes $\varrho(a)^{\repuls},\varrho(a)^{\attract},\varrho(b)^{\repuls},\varrho(b)^{\attract},\varrho(c)^{\repuls},\varrho(c)^{\attract}\in\PP^1(\R)$ sont dans un ordre cyclique positif (\resp n\'egatif).
Considérons les éléments suivants de l'espace $\mathcal{T}$ des triplets de points deux à deux distincts positivement orientés de $\PP^1(\R)$ :
$$\left\{ \begin{array}{ccl}
T & := & (\varrho(a)^{\repuls}, \varrho(b)^{\repuls}, \varrho(c)^{\repuls}),\\
T' & := & (\varrho(a)^{\repuls}, \varrho(b^{-1}a^{-1})^{\repuls}, \varrho(b)^{\repuls})
\end{array}\right.
\ \text{\Big(\resp }
\left\{ \begin{array}{ccl}
T & := & (\varrho(b)^{\repuls}, \varrho(a)^{\repuls}, \varrho(c)^{\repuls}),\\
T' & := & (\varrho(b^{-1}a^{-1})^{\repuls}, \varrho(a)^{\repuls}, \varrho(b)^{\repuls})
\end{array}\right.
\hspace{-0.2cm}\text{\Big)}.
$$
Les éléments $T,T'\in\mathcal{T}$ correspondent à des triangles id\'eaux de~$\HH^2$ qui se projettent en des triangles idéaux de~$\Pi$, d'int\'erieurs disjoints, dont les côtés s'enroulent autour des courbes de bord et qui remplissent tout~$\Pi$, comme sur la figure~\ref{fig:decomp-tri}.
\begin{figure}[h!]
\centering
\includegraphics[width=4cm]{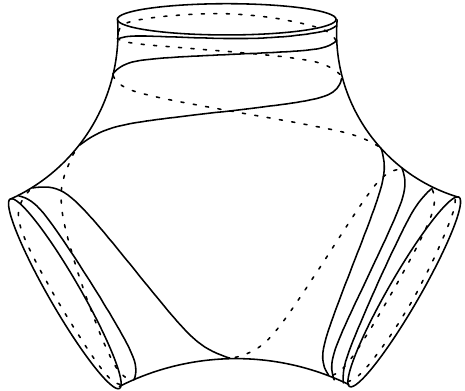}
\caption{Décomposition d'un pantalon hyperbolique en deux triangles idéaux dont les c\^ot\'es s'enroulent autour des courbes de bord}
\label{fig:decomp-tri}
\end{figure}

Adoptons la terminologie suivante, o\`u $\rot$, $\inv$ et $\varphi_t$ sont les transformations introduites au paragraphe~\ref{subsec:rot-inv-phi}, et o\`u les signes $\pm$ sont pris tous \'egaux \`a $+$ si la repr\'esentation est $R$-parfaite, et tous \'egaux \`a $-$ si elle est $(-R)$-parfaite.

\begin{definition} \label{def:triangles-realisants}
Un élément $\alpha\in\PSL(2,\R)$ est \emph{$(\pm R)$-r\'ealisé} par un couple $(T_1,T_2)$ d'éléments de $\mathcal{T}^2$ si $\rot^{\pm 1}\circ\inv\circ\varphi_{\pm R}(T_1) = T_2$ et $\rot^{\pm 1}\circ\inv\circ\varphi_{\pm R}(T_2) = \alpha\cdot T_1$.
\end{definition}

On vérifie alors (\cf figure~\ref{fig:tri-parfait}) que pour $(T,T')$ comme ci-dessus,
\begin{equation} \label{eqn:condRparf-realise}
\left\{ \begin{array}{l}
\varrho(a)\text{ est }(\pm R)\text{-r\'ealisé  par }(T,T'),\\
\varrho(b)\text{ est }(\pm R)\text{-r\'ealisé par }(\rot^{\pm 1}(T),\varrho(b)\cdot (\rot^2)^{\pm 1}(T')),\\
\varrho(c)\text{ est }(\pm R)\text{-r\'ealisé par }((\rot^2)^{\pm 1}(T),\varrho(a)^{-1}\cdot\rot^{\pm 1}(T')).
\end{array}\right.
\end{equation}
Une repr\'esentation $\varrho : \pi_1(\Pi)\to\PSL(2,\R)$ est $(\pm R)$-parfaite si et seulement s'il existe $(T,T')\in\mathcal{T}^2$ v\'erifiant ces conditions.
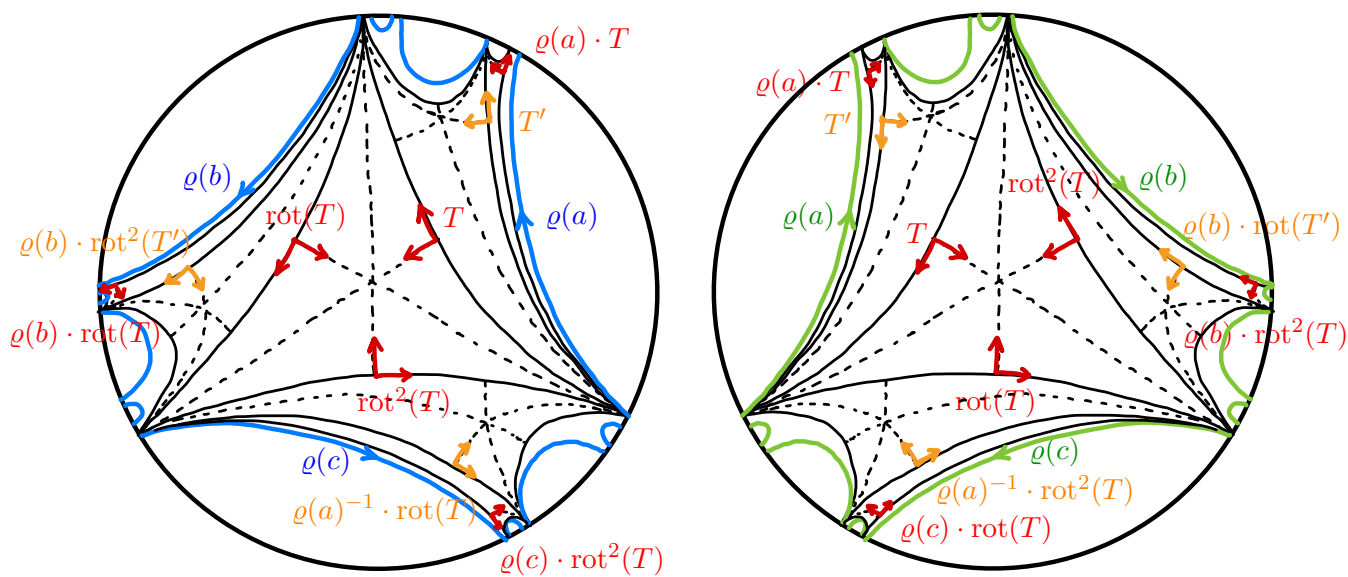
\begin{figure}[h!]
\centering
\labellist
\small\hair 2pt
\pinlabel {\textcolor{red}{$T$}} [u] at 255 410
\pinlabel {\textcolor{orange}{$T'$}} [u] at 295 460
\pinlabel {\textcolor{red}{$\varrho(a)\cdot T$}} [u] at 320 500
\pinlabel {\textcolor{red}{$\rot(T)$}} [u] at 180 410
\pinlabel {\textcolor{orange}{$\varrho(b)\cdot\rot^2(T')$}} [u] at 78 397
\pinlabel {\textcolor{red}{$\varrho(b)\cdot\rot(T)$}} [u] at 68 350
\pinlabel {\textcolor{red}{$\rot^2(T)$}} [u] at 230 315
\pinlabel {\textcolor{orange}{$\varrho(a)^{-1}\cdot\rot(T)$}} [u] at 220 260
\pinlabel {\textcolor{red}{$\varrho(c)\cdot\rot^2(T)$}} [u] at 320 235
\pinlabel {\textcolor{blue}{$\varrho(a)$}} [u] at 315 410
\pinlabel {\textcolor{blue}{$\varrho(b)$}} [u] at 130 430
\pinlabel {\textcolor{blue}{$\varrho(c)$}} [u] at 190 285
\pinlabel {\textcolor{red}{$T$}} [u] at 490 402
\pinlabel {\textcolor{orange}{$T'$}} [u] at 450 458
\pinlabel {\textcolor{red}{$\varrho(a)\cdot T$}} [u] at 432 478
\pinlabel {\textcolor{red}{$\rot^2(T)$}} [u] at 560 427
\pinlabel {\textcolor{orange}{$\varrho(b)\cdot\rot(T')$}} [u] at 665 405
\pinlabel {\textcolor{red}{$\varrho(b)\cdot\rot^2(T)$}} [u] at 667 350
\pinlabel {\textcolor{red}{$\rot(T)$}} [u] at 530 315
\pinlabel {\textcolor{orange}{$\varrho(a)^{-1}\cdot\rot^2(T)$}} [u] at 550 270
\pinlabel {\textcolor{red}{$\varrho(c)\cdot\rot(T)$}} [u] at 520 250
\pinlabel {\textcolor{mygreen}{$\varrho(a)$}} [u] at 435 410
\pinlabel {\textcolor{mygreen}{$\varrho(b)$}} [u] at 615 430
\pinlabel {\textcolor{mygreen}{$\varrho(c)$}} [u] at 560 290
\endlabellist
\includegraphics[width=7.5cm]{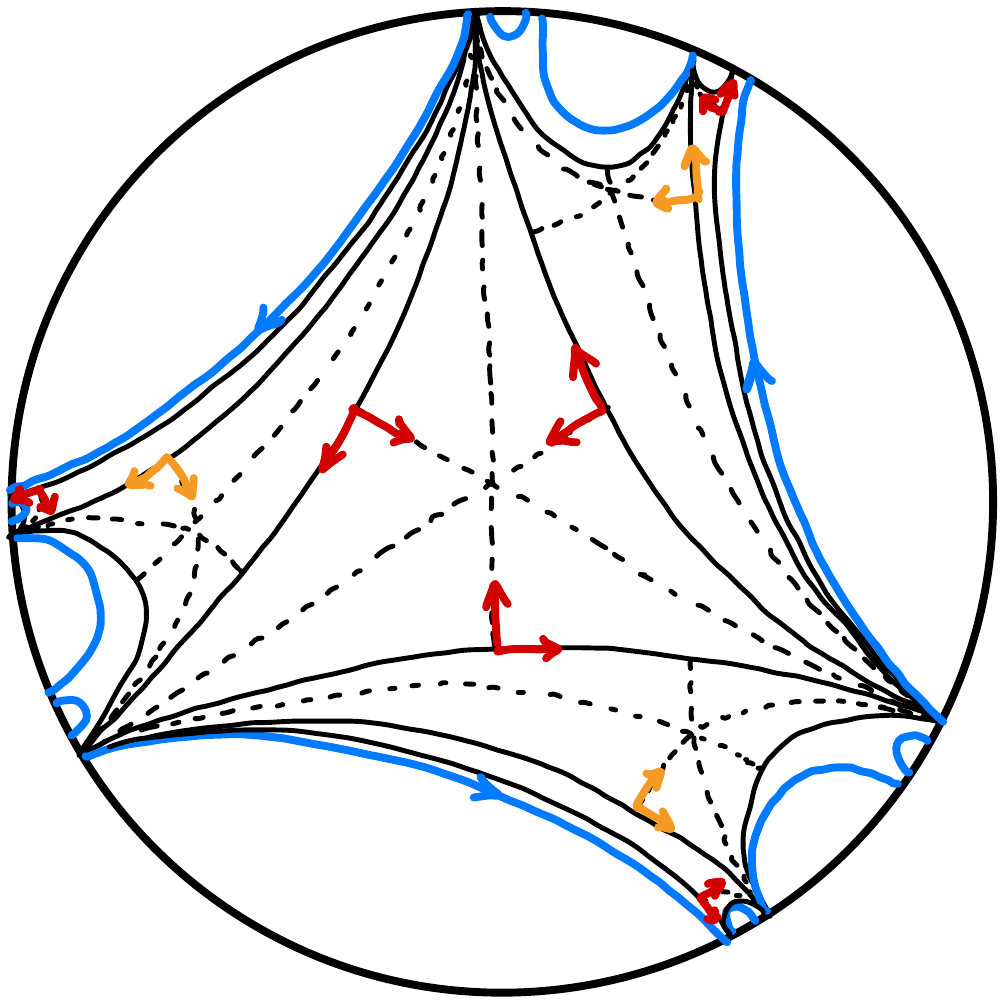}
\hspace{0.3cm}
\scalebox{-1}[1]{\includegraphics[width=7.5cm]{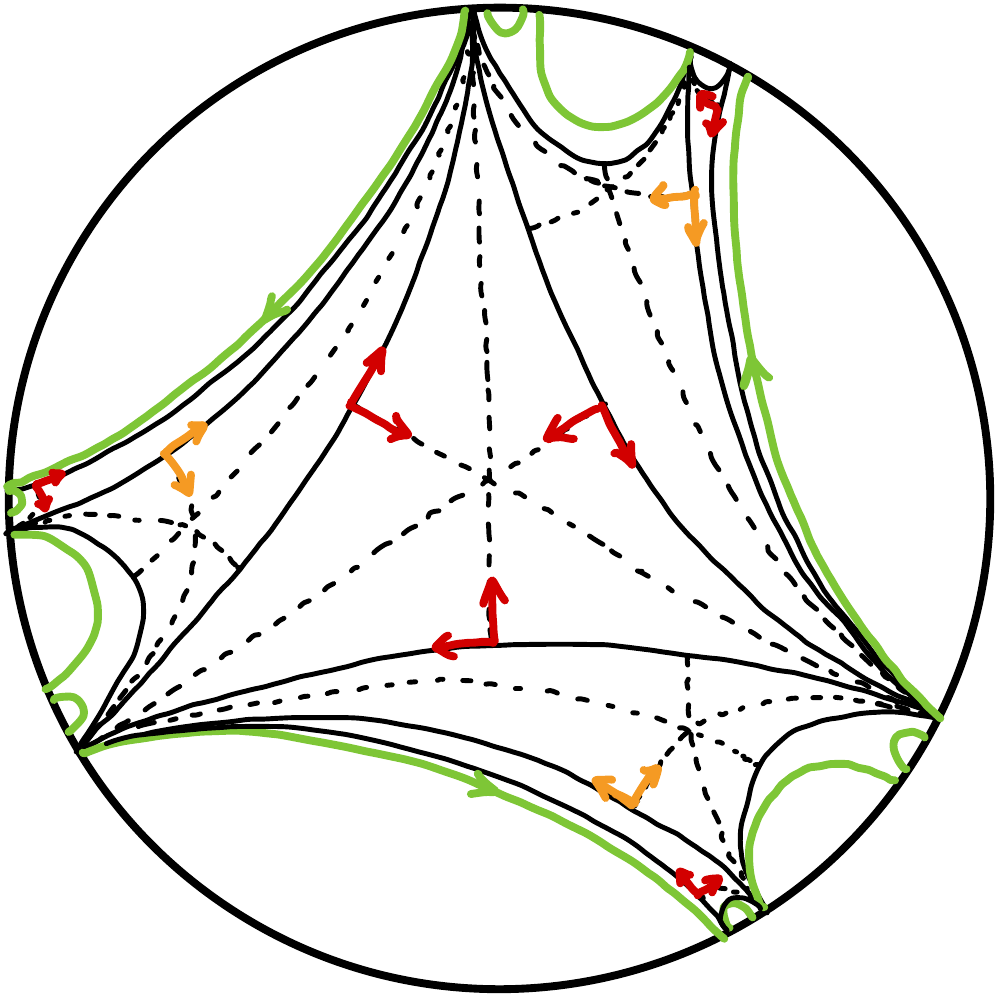}}
\caption{Repr\'esentations $R$-parfaite (\`a gauche) et $(-R)$-parfaite (\`a droite)}
\label{fig:tri-parfait}
\end{figure}

\begin{remarque} \label{rem:rho(a)-hyperbolique}
Si $\alpha$ est $(\pm R)$-r\'ealisé par $(T,T')$ au sens de la définition~\ref{def:triangles-realisants} et si l'on voit $T$ comme un \'el\'ement de $\PSL(2,\R)$, alors $T^{-1}\alpha\,T$ est \'egal \`a $\big(\begin{smallmatrix} e^R & 0\\ 1+e^{-R} & e^{-R}\end{smallmatrix}\big)$ (\resp $\big(\begin{smallmatrix} e^{-R} & 1+e^{-R}\\ 0 & e^{R}\end{smallmatrix}\big)$).
En particulier, $\alpha$ est hyperbolique de longueur de translation $2R$, et ses points fixes répulsif et attractif vérifient que $T^{-1}\cdot\alpha^{\repuls}\in\PP^1(\R)$ est \'egal \`a $0$ (\resp $\infty$), et $T^{-1}\cdot\alpha^{\attract}\in\PP^1(\R)$ est proche de $\infty$ (\resp $0$) pour $R$ grand.
\end{remarque}

Ainsi, lorsque $R$ tend vers~$0$, les points $\varrho(a)^{\attract},\varrho(b)^{\attract},\varrho(c)^{\attract}$ d'une repr\'esentation $(\pm R)$-parfaite $\varrho : \pi_1(\Pi)\to\PSL(2,\R)$ tendent respectivement vers $\varrho(a)^{\repuls},\varrho(b)^{\repuls},\varrho(c)^{\repuls}$; autrement dit, $\varrho(a),\varrho(b),\varrho(c)\in\PSL(2,\R)$ tendent vers des \'el\'ements paraboliques.
Lorsque $R$ tend vers $+\infty$, les points $\varrho(a)^{\attract},\varrho(b)^{\attract},\varrho(c)^{\attract}$ d'une repr\'esentation $(\pm R)$-parfaite $\varrho : \pi_1(\Pi)\to\PSL(2,\R)$ tendent respectivement vers $\varrho(b)^{\repuls},\varrho(c)^{\repuls},\varrho(a)^{\repuls}$ (\cf figure~\ref{fig:tri-parfait}).

\begin{remarque} \label{rem:repr-pmR-parf-PSL(2,R)}
Les repr\'esentations $R$-parfaites de $\pi_1(S)$ dans $\PSL(2,\R)$ sont les images des repr\'esentations $(-R)$-parfaites par la conjugaison par un \'el\'ement de $\PGL(2,\R)$ qui renverse l'orientation de~$\HH^2$.
\end{remarque}

\subsubsection{Représentations presque parfaites à valeurs dans~$G$} \label{subsubsec:pant-presque-parf-KLM-def}

Munissons $G$ d'une m\'etrique riemannienne invariante par multiplication \`a gauche par~$G$ (pour cela, on choisit une forme quadratique d\'efinie positive sur~$\g$ et on la pousse en avant par~$G$).
Quitte \`a la remplacer par sa moyenne par le groupe d'ordre trois engendr\'e par $\rot$, on suppose cette m\'etrique \'egalement invariante par $\rot$.
On note $d_G$ la distance correspondante.

On a une notion naturelle de repr\'esentation \emph{$(\pm R)$-parfaite} de $\pi_1(\Pi)$ dans~$G$, \`a savoir la composition d'une repr\'esentation $(\pm R)$-parfaite \`a valeurs dans $\PSL(2,\R)$ au sens ci-dessus, avec $g\,\tau(\cdot)\,g^{-1} : \PSL(2,\R)\hookrightarrow G$ o\`u $g\in G$.
La caractérisation \eqref{eqn:condRparf-realise} permet d'affaiblir cette notion de la mani\`ere suivante.

\begin{definition} \label{def:realise-presque-alpha}
Soient $\varepsilon,R>0$.
On dit qu'un \'el\'ement $\alpha\in G$ est \emph{$(\varepsilon,\pm R)$-presque réalisé} par un couple $(g,g')\in G^2$ s'il existe $(h,h')\in G^2$ v\'erifiant les quatre conditions suivantes :
\begin{align*}
d_G(g,h)<\varepsilon, \quad & d_G(\rot^{\pm 1}\circ\inv\circ\varphi_{\pm R}(h),g')<\varepsilon,\\
d_G(g',h')<\varepsilon, \quad & d_G(\rot^{\pm 1}\circ\inv\circ\varphi_{\pm R}(h'),\alpha g)<\varepsilon.
\end{align*}
\end{definition}

On considère les conditions suivantes sur un quintuplet $(\alpha,\beta,\gamma,g,g')\in G^5$ :
\begin{equation} \label{eqn:condRpresqueparf-realise}
\left\{ \begin{array}{l}
\alpha\text{ est }(\varepsilon,\pm R)\text{-presque réalisé par }(g,g'),\\
\beta \text{ est }(\varepsilon,\pm R)\text{-presque réalisé par }(\rot^{\pm 1}(g),\beta\cdot (\rot^2)^{\pm 1}(g')),\\
\gamma \text{ est }(\varepsilon,\pm R)\text{-presque réalisé par }((\rot^2)^{\pm 1}(g),\alpha^{-1}\cdot\rot^{\pm 1}(g')).
\end{array}\right.
\end{equation}

\begin{definition} \label{def:pant-presque-parfait-KLM}
Soit $\Pi$ un pantalon de groupe fondamental $\pi_1(\Pi) = \langle a,b,c \,|\, cba = 1\rangle$.
On dit qu'une repr\'esentation $\rho : \pi_1(\Pi)\to G$ est \emph{$(\varepsilon,\pm R)$-presque parfaite} s'il existe $(g,g')\in G^2$ tel que le quintuplet $Q := (\rho(a),\rho(b),\rho(c),g,g')$ vérifie \eqref{eqn:condRpresqueparf-realise}.
Ce quintuplet est une \emph{donnée géométrique} associée à la représentation $(\varepsilon,\pm R)$-presque parfaite~$\rho$.
\end{definition}

\begin{remarque} \label{rem:presque-parf-tri}
Ceci ne dépend pas de l'ordre choisi pour les courbes de bord : si $Q = (\rho(a),\rho(b),\rho(c),g,g')$ vérifie \eqref{eqn:condRpresqueparf-realise}, alors
\begin{align*}
\sym(Q) & := \big(\rho(b),\rho(c),\rho(a),\rot^{\pm 1}(g),\rho(b)\cdot (\rot^2)^{\pm 1}(g')\big)\\
\text{et}\quad \sym^2(Q) & := \big(\rho(c),\rho(a),\rho(b),(\rot^2)^{\pm 1}(g),\rho(a)^{-1}\cdot\rot^{\pm 1}(g')\big)
\end{align*}
aussi.
De plus, la notion de représentation $(\varepsilon,\pm R)$-presque parfaite est invariante par conjugaison au but par~$G$ : si $(\rho(a),\rho(b),\rho(c),g,g')$ vérifie \eqref{eqn:condRpresqueparf-realise}, alors $(x\rho(a)x^{-1}, x\rho(b)x^{-1}, x\rho(c)x^{-1}, xg, xg')$ aussi pour tout $x\in G$.
\end{remarque}

\begin{remarque}
Nous verrons \`a la proposition~\ref{prop:mesures-I}.\eqref{item:repr-pmR-parf} qu'en supposant la condition~(R) du paragraphe~\ref{subsubsec:cond-R} v\'erifi\'ee, les notions de repr\'esentations $(\varepsilon,R)$-presque parfaite et $(\varepsilon,-R)$-presque parfaite co\"incident, mais pour des \'el\'ements $(g,g')\in G^2$ qui diff\`erent par la multiplication \`a droite par un certain \'el\'ement de~$G$ envoyant $\mathsf{h}$ sur $-\mathsf{h}$ (ceci correspond au renversement d'orientation dans la remarque~\ref{rem:repr-pmR-parf-PSL(2,R)}).
\end{remarque}

Voir la figure~\ref{fig:Triconn} pour une illustration dans le cas des repr\'esentations presque parfaites \`a valeurs dans un r\'eseau $\Gamma$ de~$G$.

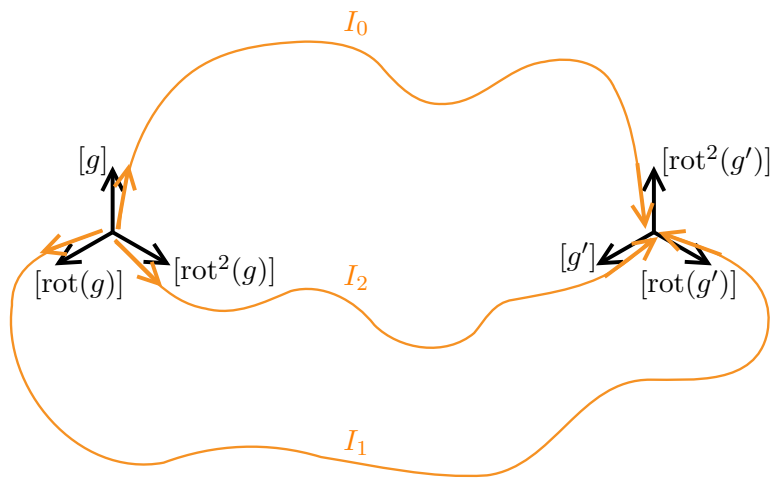
\begin{figure}[h!]
\centering
\labellist
\small\hair 2pt
\pinlabel {$[g]$} [u] at 60 230
\pinlabel {$[\rot(g)]$} [u] at 50 140
\pinlabel {$[\rot^2(g)]$} [u] at 155 150
\pinlabel {\textcolor{orange}{$I_0$}} [u] at 250 330
\pinlabel {\textcolor{orange}{$I_2$}} [u] at 250 145
\pinlabel {\textcolor{orange}{$I_1$}} [u] at 250 25
\pinlabel {$[\rot^2(g')]$} [u] at 510 230
\pinlabel {$[g']$} [u] at 410 160
\pinlabel {$[\rot(g')]$} [u] at 490 140
\endlabellist
\includegraphics[width=10cm]{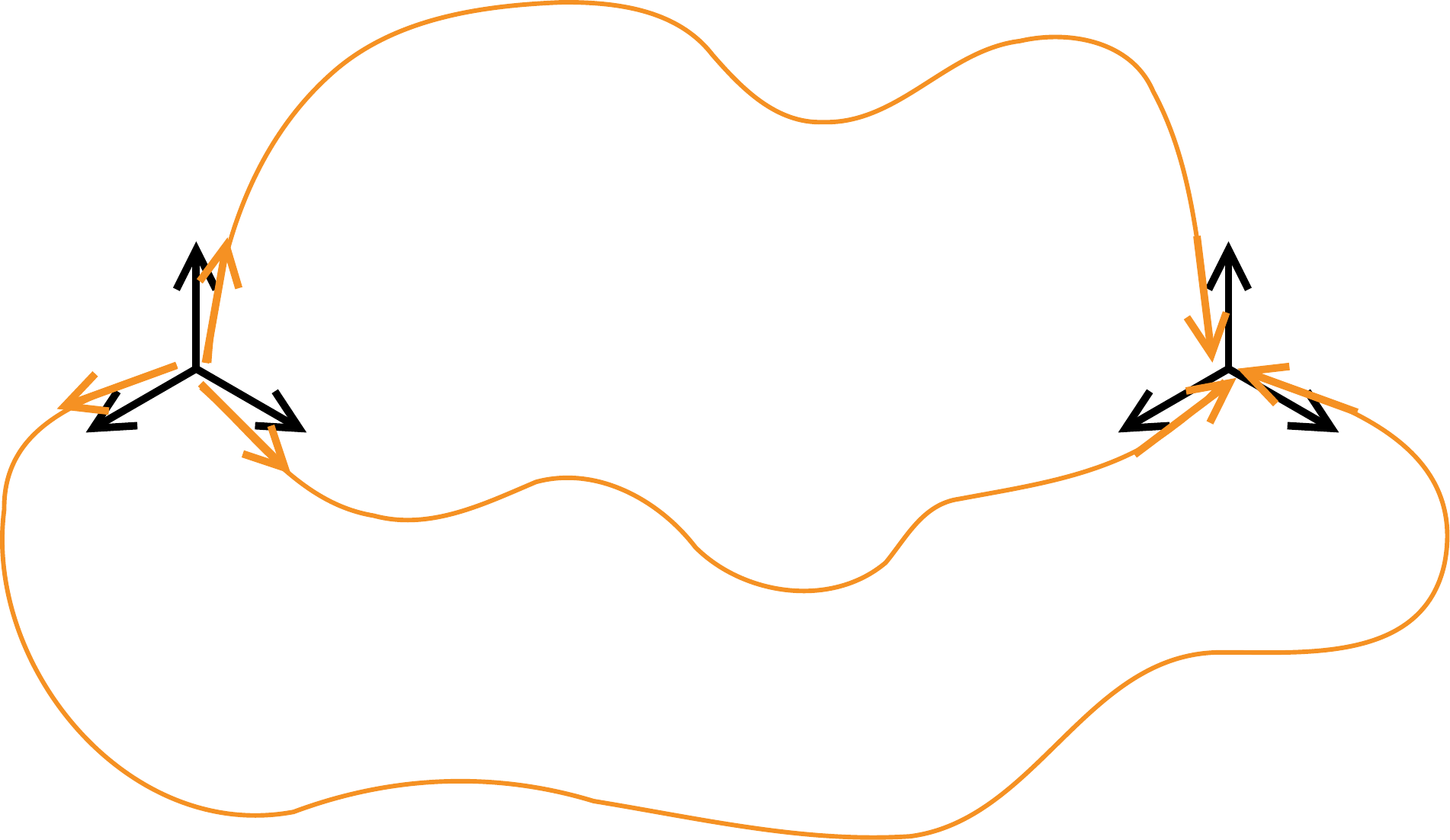}
\vspace{-0.2cm}
\caption{Donn\'ee g\'eom\'etrique $Q = (\rho(a),\rho(b),\rho(c),g,g')$ associ\'ee \`a une repr\'esentation $(\varepsilon,R)$-presque parfaite $\rho : \pi_1(\Pi)\to G$ \`a valeurs dans un r\'eseau $\Gamma$ de~$G$ : il existe dans $\Gamma\backslash G$ un segment $I_0$ (\resp $I_1$, \resp $I_2$) d'orbite du flot $(\varphi_t)_{t\in\R}$, de longueur~$R$, qui relie un \'el\'ement proche de $[g]$ (\resp $[\rot(g)]$, \resp $[\rot^2(g)]$) \`a un \'el\'ement proche de l'inverse $\inv$ de $[\rot^2(g')]$ (\resp $[\rot(g')]$, \resp $[g']$). L'\'el\'ement $\rho(a)$ (\resp $\rho(b)$, \resp $\rho(c)$) est \og presque r\'ealis\'e\fg\ dans~$\Gamma\backslash G$ par $I_0$ (\resp $I_1$, \resp $I_2$) suivi de $I_2$ (\resp $I_0$, \resp $I_1$) parcouru en sens~inverse.}
\label{fig:Triconn}
\end{figure}

\subsection{Rappels : proximalit\'e} \label{subsec:prox}

Soient $\mathfrak{a} \subset \mathfrak{k}^{\perp}$ un sous-espace de Cartan de~$\mathfrak{g}$, et $\mathfrak{a}^+$ une chambre de Weyl de~$\mathfrak{a}$ contenant~$\mathsf{h}$, comme au paragraphe~\ref{subsubsec:cond-reg}.
Soit $Z_K(\mathfrak{a})$ le centralisateur de $\mathfrak{a}$ dans~$K$.
Comme $\mathsf{h}$ est suppos\'e r\'egulier (\ie $\mathsf{h}\in\mathrm{Int}(\mathfrak{a}^+)$), le sous-groupe parabolique $P_{\tau}$ de~$G$ est minimal et pour tout \'el\'ement $g\in G$, les deux notions suivantes sont \'equivalentes :
\begin{itemize}
  \item $g$ est \emph{proximal dans $G/P_{\tau}$}, au sens o\`u il admet un unique point fixe attractif $g^{\attract}$ et un unique point fixe r\'epulsif $g^{\repuls}$ dans $G/P_{\tau}$,
  \item $g$ est \emph{loxodromique}, au sens o\`u il est conjugu\'e \`a un \'el\'ement de la forme $\exp(x)k$ o\`u $x\in\mathrm{Int}(\mathfrak{a}^+)$ et $k\in Z_K(\mathfrak{a})$ ; l'\'el\'ement $x$ est unique et l'on pose $\lambda(g):=x\in\mathrm{Int}(\mathfrak{a}^+)$.
\end{itemize}

Le centralisateur $Z_G(\alpha_0)$ de $\alpha_0:=\exp(\mathsf{h})$ dans~$G$ est le produit de $A:=\exp(\mathfrak{a})$ et de $Z_K(\mathfrak{a})$.
Pour un \'el\'ement $\alpha\in G$ proximal dans $G/P_{\tau}$ quelconque, conjugu\'e \`a $\exp(x)k$ o\`u $x\in\mathrm{Int}(\mathfrak{a}^+)$ et $k\in Z_K(\mathfrak{a})$, le centralisateur $Z_G(\alpha)$ de $\alpha$ dans~$G$ est conjugu\'e au produit de $A$ et du groupe $Z_K(\mathfrak{a})\cap Z_K(k)$, qui peut \^etre plus petit que $Z_K(\mathfrak{a})$.

\begin{exemple} \label{ex:PSLnC-centralisateur}
Soient $G=\PSL(n,\C)$ et $\tau : \PSL(2,\R)\hookrightarrow G$ le plongement irr\'eductible (\cf exemple~\ref{ex:PSL(n,K)-a}).
On a $K=\mathrm{PSU}(n)$.
Le sous-espace $\mathfrak{a}\subset\g$ est constitu\'e des matrices diagonales r\'eelles de trace nulle, et le groupe $A = \exp(\mathfrak{a})$ (\resp $Z_K(\mathfrak{a})$) des matrices diagonales de~$G$ \`a coefficients strictement positifs (\resp complexes de module un).
Comme $Z_K(\mathfrak{a})$ est ab\'elien, le centralisateur $Z_G(\alpha)$ de tout \'el\'ement $\alpha\in G$ proximal dans $G/P_{\tau}$ est conjugu\'e \`a $Z_G(\alpha_0) = A \, Z_K(\mathfrak{a}) \simeq (\C^*)^{n-1}$.
\end{exemple}

\begin{exemple} \label{ex:SO(n,1)-centralisateur}
Soient $G=\SO(n,1)$ et $\tau : \PSL(2,\R)\simeq\SO(2,1)_0\hookrightarrow G$ le plongement standard (\cf exemple~\ref{ex:SO(n,1)-tau}).
On a $K=\mathrm{S}(\OO(n)\times\OO(1))$.
Le groupe $A = \exp(\mathfrak{a})$ est isomorphe \`a $\R_+^*$, et $Z_K(\mathfrak{a})$ \`a $\mathrm{S}(\OO(n-1)\times\OO(1)\times\OO(1))$.
Pour $\alpha = \exp(x)k \in G$ o\`u $x\in\mathrm{Int}(\mathfrak{a}^+)$ et $k\in Z_K(\mathfrak{a})$, le centralisateur $Z_G(\alpha)$ de $\alpha$ dans~$G$ est strictement contenu dans $Z_G(\alpha_0)$ d\`es que $k$ n'est pas central dans $Z_K(\mathfrak{a})$, ce qui peut se produire d\`es que $n\geq 4$.
\end{exemple}

\subsection{Quelques observations utiles sur les pantalons presque parfaits} \label{subsec:presque-parf-prox}

Soit $\Vert\cdot\Vert_{\mathfrak{a}}$ la norme euclidienne sur~$\mathfrak{a}$ induite par la forme de Killing de~$\g$.
Munissons $G/P_{\tau}$ de la distance $d_{\tau}$ du paragraphe~\ref{subsubsec:def-sullivan}, invariante par $K\supset\tau(\PSO(2))$.
On renvoie au paragraphe~\ref{subsec:prox} pour la notion de proximalit\'e dans $G/P_{\tau}$ et la d\'efinition~de~$\lambda$.

\begin{lemme} \label{lem:realise-proximal}
Il existe $C>0$ tel que pour tout $\varepsilon>0$ assez petit, tout $R>0$ assez grand et tout $(g,g',\alpha)\in G^3$, si $\alpha$ est $(\varepsilon,R)$-presque (\resp $(\varepsilon,-R)$-presque) réalisé par $(g,g')$ au sens de la d\'efinition~\ref{def:realise-presque-alpha}, alors
\begin{itemize}
  \item $\alpha$ est proximal dans $G/P_{\tau}$ ;
  \item le point fixe attractif (\resp r\'epulsif) de $g^{-1}\alpha g$ dans $G/P_{\tau}$ est \`a distance $\leq C (\varepsilon+\nolinebreak e^{-R/2})$ de $\underline{\tau}(\infty)$ (\resp $\underline{\tau}(0)$) pour~$d_{\tau}$ ;
  \item $\Vert\lambda(\alpha) - \lambda(\exp(R\,\mathsf{h}))\Vert_{\mathfrak{a}} \leq C (\varepsilon+e^{-R/2})$ ; en particulier, la longueur de translation de $\alpha$ dans $G/K$ appartient \`a $[2R-C (\varepsilon+e^{-R/2}),2R+C (\varepsilon+e^{-R/2})]$ ;
  \item $g^{-1} g'$ appartient à $\mathcal{B}_{\varepsilon} \, \tau\big(\big(\begin{smallmatrix} e^{R/2} & 0\\ e^{-R/2} & e^{-R/2}\end{smallmatrix}\big)\big) \, \mathcal{B}_{\varepsilon}$ (\resp à $\mathcal{B}_{\varepsilon} \, \tau\big(\big(\begin{smallmatrix} e^{-R/2} & e^{-R/2}\\ 0 & e^{R/2}\end{smallmatrix}\big)\big) \, \mathcal{B}_{\varepsilon}$), où $\mathcal{B}_{\varepsilon}$ désigne la boule ferm\'ee de rayon~$\varepsilon$ centr\'ee en l'\'el\'ement neutre dans $(G,d_G)$.
\end{itemize}
\end{lemme}

\begin{proof}
Considérons les éléments $h_R := \big(\begin{smallmatrix} e^{R/2} & 0\\ e^{-R/2} & e^{-R/2}\end{smallmatrix}\big)$ et $h_{-R} :=\nolinebreak \big(\begin{smallmatrix} e^{-R/2} & e^{-R/2}\\ 0 & e^{R/2}\end{smallmatrix}\big)$ de $\PSL(2,\R)$.
G\'en\'eralisant la remarque~\ref{rem:rho(a)-hyperbolique}, on note que pour tous $\varepsilon,R>0$, si $\alpha$ est $(\varepsilon,\pm R)$-presque réalisé par $(g,g')\in G^2$, alors il existe $g_1,g_2,g_3,g_4\in\nolinebreak\mathcal{B}_{\varepsilon}$~tels~que
$$g^{-1} g' = g_1 \, \tau(h_{\pm R}) \, g_2 \quad\quad\mathrm{et}\quad\quad g^{-1}\alpha g = g_1 \, \tau(h_{\pm R}) \, (g_2 g_3) \, \tau(h_{\pm R}) \, g_4.$$
L'\'el\'ement $h_R\in\PSL(2,\R)$ est hyperbolique de point fixe r\'epulsif $0\in\PP^1(\R)$ et de point fixe attractif $e^{R/2}-1\in\PP^1(\R)$ \`a distance $\arctan(1/(e^{R/2}-1))\sim e^{-R/2}$ de $\infty$ pour la distance $\PSO(2)$-invariante de $\PP^1(\R)$.
De m\^eme, $h_{-R}$ est hyperbolique de point fixe attractif $\infty$ et de point fixe r\'epulsif \`a distance $\sim e^{-R/2}$ de~$0$.

Posons $x_0^+ := \underline{\tau}(\infty) \in G/P_{\tau}$ et notons $H_0^-$ l'ensemble des points de $G/P_{\tau}$ non tranverses \`a $\underline{\tau}(0)$.
Il existe $r>0$ tel que $\mathcal{V}_{3r}(x_0^+) \cap \mathcal{V}_{3r}(H_0^-) = \emptyset$, o\`u $\mathcal{V}_{\delta}(\cdot)$ d\'esigne le $\delta$-voisinage uniforme dans $(G/P_{\tau},d_{\tau})$ pour $\delta>0$.
On v\'erifie qu'il existe $C>0$ avec les propri\'et\'es suivantes :
\begin{itemize}
  \item pour tout $\varepsilon>0$ assez petit et tout $g\in\mathcal{B}_{\varepsilon}$, la constante de Lipschitz de $g$ dans $(G/P_{\tau},d_{\tau})$ est $\leq 1+C\varepsilon$ et l'on a $d_{\tau}(x,g\cdot x)\leq C\varepsilon$ pour tout $x\in G/P_{\tau}$ ;
  \item pour tout $R>0$ assez grand, $\tau(h_R)$ envoie le compl\'ementaire de $\mathcal{V}_r(H_0^-)$ dans $\mathcal{V}_{Ce^{-R/2}}(x_0^+)$ avec une constante de Lipschitz $\leq Ce^{-R}$.
\end{itemize}
Ainsi, pour tout $\varepsilon>0$ assez petit, tout $R>0$ assez grand et tous $g_1,g_2,g_3,g_4\in\nolinebreak\mathcal{B}_{\varepsilon}$, l'\'el\'ement $g_1\,\tau(h_R)\,g_2g_3\,\tau(h_R)\,g_4\in G$ envoie le compl\'ementaire de $\mathcal{V}_{2r}(H_0^-)$ dans $\mathcal{V}_{C(\varepsilon+e^{-R/2})}(x_0^+)$ avec une constante de Lipschitz $\leq (1+C\varepsilon)^4\,C^2\,e^{-2R}$, donc de mani\`ere uniform\'ement contractante si $\varepsilon$ est assez petit et $R$ assez grand.
On en d\'eduit que $g_1\,\tau(h_R)\,g_2g_3\,\tau(h_R)\,g_4$ est proximal dans $G/P_{\tau}$, de point fixe attractif \`a distance $\leq C(\varepsilon+e^{-R/2})$ de $x_0^+=\underline{\tau}(\infty)$.
En raisonnant de m\^eme, on voit que son point fixe r\'epulsif est \`a distance $\leq C(\varepsilon+e^{-R/2})$ de $\underline{\tau}(0)$.
La borne sup\'erieure pour $\Vert\lambda(g_1\,\tau(h_R)\,g_2g_3\,\tau(h_R)\,g_4) - \lambda(\exp(R\,\mathsf{h}))\Vert_{\mathfrak{a}}$ est obtenue par des raisonnements analogues, tenant compte de mani\`ere plus pr\'ecise de la constante de Lipschitz, dans divers espaces projectifs associ\'es \`a~$\tau$, comme dans \textcite{ben97}.
Ceci prouve les propri\'et\'es voulues de $g^{-1}\alpha g=g_1\,\tau(h_R)\,g_2g_3\,\tau(h_R)\,g_4$ lorsque $\alpha$ est $(\varepsilon,R)$-presque réalisé par $(g,g')\in G^2$.
Le cas o\`u $\alpha$ est $(\varepsilon,-R)$-presque réalisé par $(g,g')$ est analogue.
\end{proof}

\begin{corollaire} \label{cor:presque-parf->generique}
Si $\varepsilon>0$ est assez petit et $R>0$ assez grand, alors
\begin{enumerate}
  \item\label{item:presque-parf->generique} toute repr\'esentation $(\varepsilon,\pm R)$-presque parfaite $\rho : \pi_1(\Pi)\to G$ est \emph{$\tau$-g\'en\'erique} au sens o\`u $\rho(a)$, $\rho(b)$ et $\rho(c)$ sont proximaux dans $G/P_{\tau}$ de points fixes attractifs $\rho(a)^{\attract}, \rho(b)^{\attract}, \rho(c)^{\attract} \in G/P_{\tau}$ deux \`a deux transverses ;
  \item\label{item:nb-fini-presque-parf} pour tout r\'eseau cocompact $\Gamma$ de~$G$, il n'y a qu'un nombre fini de classes de conjugaison de repr\'esentations $(\varepsilon,\pm R)$-presque parfaites de $\pi_1(\Pi)$ \`a valeurs dans~$\Gamma$.
\end{enumerate}
\end{corollaire}

\begin{proof}
\eqref{item:presque-parf->generique} Les points $\underline{\tau}(\infty)$, $\underline{\tau}(-1)$ et $\underline{\tau}(0)$ sont transverses et la transversalité est une condition ouverte.
Par conséquent, il existe $\delta>0$ tel que pour tous $x,y\in G/P_{\tau}$, si $x$ est $\delta$-proche de l'un des points $\underline{\tau}(\infty)$, $\underline{\tau}(-1)$ ou $\underline{\tau}(0)$, et si $y$ est $\delta$-proche d'un autre de ces points, alors $x$ et~$y$ sont transverses.
Soit $(\rho(a),\rho(b),\rho(c),g,g')$ une donnée géométrique associée à~$\rho$, comme à la définition~\ref{def:pant-presque-parfait-KLM}.
D'apr\`es le lemme~\ref{lem:realise-proximal}, si $\varepsilon$ est assez petit et $R$ assez grand, alors $g^{-1}\rho(a)g$ (\resp $g^{-1}\rho(b)g$, \resp $g^{-1}\rho(c)g$) est proximal dans $G/P_{\tau}$, de point fixe attractif $\delta$-proche de $\underline{\tau}(\infty)$ (\resp $\underline{\tau}(-1)$, \resp $\underline{\tau}(0)$).
On conclut en remarquant que $(g^{-1}\rho(d)g)^{\attract} = g^{-1}\cdot\rho(d)^{\attract}$ pour tout $d\in\{a,b,c\}$ et que l'action de~$G$ pr\'eserve la transversalit\'e.

\eqref{item:nb-fini-presque-parf} Soit $\Gamma$ un r\'eseau cocompact de~$G$ : il existe une partie compacte $K_1$ de~$G$ telle que $G = \Gamma K_1$.
D'apr\`es le lemme~\ref{lem:realise-proximal}, pour tout $\varepsilon>0$ assez petit et tout $R>0$ assez grand, il existe une partie compacte $K_2$ de~$G$ telle que pour tout $(g,g',\alpha)\in G^3$, si $\alpha$ est $(\varepsilon,\pm R)$-presque réalisé par $(g,g')$ au sens de la d\'efinition~\ref{def:realise-presque-alpha}, alors $g^{-1}\alpha g\in K_2$.
En particulier, si $\rho : \pi_1(\Pi)\to\Gamma$ est $(\varepsilon,\pm R)$-presque parfaite, de donnée géométrique $(\rho(a),\rho(b),\rho(c),g,g')$ comme à la définition~\ref{def:pant-presque-parfait-KLM}, alors $g^{-1}\rho(a)g$ et $g^{-1}\rho(b)g$ appartiennent tous deux au compact $K'_2 := K_2 \cup \tau((\begin{smallmatrix} 1 & 1\\ -1 & 0\end{smallmatrix}))^{\pm 1} K_2 \tau((\begin{smallmatrix} 1 & 1\\ -1 & 0\end{smallmatrix}))^{\mp 1}$.
Si l'on \'ecrit $g = x k_1$ o\`u $x\in\Gamma$ et $k_1\in K_1$, alors $x^{-1}\rho(a)x$ et $x^{-1}\rho(b)x$ appartiennent tous deux \`a $K_1 K'_2 K_1^{-1} \cap \Gamma$, qui est fini car $\Gamma$ est discret dans~$G$.
\end{proof}

Le lemme~\ref{lem:realise-proximal} implique que pour $\varepsilon>0$ assez petit et $R$ assez grand, l'image de toute repr\'esentation $(\varepsilon,\pm R)$-presque parfaite $\rho : \pi_1(\Pi)\to G$ est un groupe \emph{$\varepsilon$-Schottky} au sens de \textcite[Def.\,4.1]{ben97}.
En particulier, $\rho$ est injective et discr\`ete, c'est un plongement quasi-isom\'etrique (\cf \cite{ben97}) et plus pr\'ecis\'ement une repr\'esentation $P_{\tau}$-anosovienne au sens du paragraphe~\ref{subsec:intro-Anosov} (\cf \cite{clss17,klp-morse}).

\subsection{Application \og pied\fg\ et bons recollements selon Kahn, Labourie, Mozes} \label{subsec:recolle-KLM}

Afin de d\'efinir les bons recollements de pantalons presque parfaits, Kahn, Labourie et Mozes introduisent une notion d'application \og pied\fg, d\'efinie de la mani\`ere suivante.

Soit $\alpha\in G$ un \'el\'ement proximal dans $G/P_{\tau}$.
Soit $L_{\alpha}$ l'ensemble des \'el\'ements $g\in G$ dont le $\tau$-triangle $g\cdot\underline{\tau}(0,\infty,-1)$ de $G/P_{\tau}$ associ\'e est de la forme $(\alpha^{\repuls},\alpha^{\attract},\cdot)$.
Choisissons $r>0$ assez petit de sorte que le $r$-voisinage uniforme $\mathcal{U}_{\alpha}$ de $L_{\alpha}$ dans $(G,d_G)$ soit hom\'eomorphe au produit direct de $L_{\alpha}$ avec une boule.

\begin{definition} \label{def:pied}
L'\emph{application \og pied\fg}\ associ\'ee \`a l'\'el\'ement proximal $\alpha\in G$ est la projection orthogonale $\Psi_{\alpha} : \mathcal{U}_{\alpha} \to L_{\alpha}$ sur $L_{\alpha}$.
\end{definition}

On peut choisir le m\^eme $r>0$ pour tous les \'el\'ements~$\alpha$ par la remarque suivante.

\begin{remarque} \label{rem:pied-alpha-0-alpha}
Pour $\alpha_0=\exp(\mathsf{h})$, l'ensemble $L_{\alpha_0}$ est le centralisateur $Z_G(\alpha_0) = A Z_K(\mathfrak{a})$ de $\alpha_0$ dans~$G$ (\cf paragraphe~\ref{subsec:prox}). 
Pour un \'el\'ement proximal quelconque $\alpha\in G$, on a $L_{\alpha} = g_{\alpha}L_{\alpha_0}$ et $\mathcal{U}_{\alpha} = g_{\alpha}\mathcal{U}_{\alpha_0}$ et $\Psi_{\alpha} = g_{\alpha}\Psi_{\alpha_0}(g_{\alpha}^{-1}\cdot)$ o\`u $g_{\alpha}\in G$ envoie $(\alpha_0^{\repuls},\alpha_0^{\attract}) = (\underline{\tau}(0),\underline{\tau}(\infty))$ sur $(\alpha^{\repuls},\alpha^{\attract})$.
Le stabilisateur $\mathrm{stab}_G(\alpha^{\repuls},\alpha^{\attract})$ de $(\alpha^{\repuls},\alpha^{\attract})\in G/P_{\tau}\times G/P_{\tau}$ dans~$G$ est \'egal \`a $g_{\alpha}L_{\alpha_0}g_{\alpha}^{-1} = L_{\alpha}g_{\alpha}^{-1}$ ; il agit sur $\mathcal{U}_{\alpha}$ et $L_{\alpha}$ par multiplication \`a gauche et $\Psi_{\alpha}$ est \'equivariante pour ces actions.
\end{remarque}

\begin{lemme} \label{lem:presque-parf-U-alpha}
\begin{enumerate}
  \item\label{item:N-delta} Pour tout \'el\'ement proximal $\alpha\in G$, les ensembles
  $$\mathcal{N}_{\delta} := \big\{ g\in G ~\big|~ d_{\tau}\big(g^{-1}\cdot\alpha^{\repuls},\underline{\tau}(0)\big)\leq\delta\text{ et }d_{\tau}\big(g^{-1}\cdot\alpha^{\attract},\underline{\tau}(\infty)\big)\leq\delta\big\},$$
  pour $\delta>0$, forment une base de voisinages $\mathrm{stab}_G(\alpha^{\repuls},\alpha^{\attract})$-invariants de $L_{\alpha}$ dans~$G$.
  \item\label{item:presque-parf-U-alpha} En particulier, si $\varepsilon>0$ est assez petit et $R>0$ assez grand, alors pour toute repr\'esentation $(\varepsilon,\pm R)$-presque parfaite $\rho : \pi_1(\Pi)\to G$, de donnée géométrique $(\rho(a),\rho(b),\rho(c),g,g')$ comme à la définition~\ref{def:pant-presque-parfait-KLM}, on a $g\in\mathcal{U}_{\rho(a)^{\pm 1}}$.
\end{enumerate}
\end{lemme}

\begin{proof}
\eqref{item:N-delta} Soit $g_{\alpha}\in G$ tel que $g_{\alpha}^{-1}\cdot (\alpha^{\repuls},\alpha^{\attract}) = (\underline{\tau}(0),\underline{\tau}(\infty))$.
L'espace homog\`ene $\mathrm{stab}_G(\alpha^{\repuls},\alpha^{\attract})\backslash G$ s'identifie \`a l'ensemble des couples de points transverses de $G/P_{\tau}$ via $[g]\mapsto g^{-1}\cdot (\alpha^{\repuls},\alpha^{\attract})$.
Par cons\'equent, une base de voisinages de $[g_{\alpha}]$ dans $\mathrm{stab}_G(\alpha^{\repuls},\alpha^{\attract})\backslash G$ est donn\'ee par les ensembles
$$[\mathcal{N}_{\delta}] := \big\{ [g]\in \mathrm{stab}_G(\alpha^{\repuls},\alpha^{\attract})\backslash G ~\big|~ d_{\tau}\big(g^{-1}\cdot\alpha^{\repuls},\underline{\tau}(0)\big)\leq\delta\text{ et }d_{\tau}\big(g^{-1}\cdot\alpha^{\attract},\underline{\tau}(\infty)\big)\leq\delta\big\},$$
pour $\delta>0$.
Les images r\'eciproques $\mathcal{N}_{\delta}$ des $[\mathcal{N}_{\delta}]$ par la projection naturelle $G\to\mathrm{stab}_G(\alpha^{\repuls},\alpha^{\attract})\backslash G$ forment alors une base de voisinages $\mathrm{stab}_G(\alpha^{\repuls},\alpha^{\attract})$-invariants de $L_{\alpha}$ dans~$G$.

\eqref{item:presque-parf-U-alpha} Cons\'equence imm\'ediate de \eqref{item:N-delta}, de la d\'efinition de $\mathcal{U}_{\rho(a)^{\pm 1}}$ et du lemme~\ref{lem:realise-proximal}.
\end{proof}

Notons que le flot $(\varphi_t)_{t\in\R}$ du paragraphe~\ref{subsec:rot-inv-phi} préserve~$L_{\alpha}$, alors que l'involution $\inv$ échange $L_{\alpha}$ et~$L_{\alpha^{-1}}$.

Dans la définition suivante, Kahn, Labourie et Mozes ont en tête des couples $(\Pi^+,\Pi^-)$ de pantalons adjacents le long d'une courbe de bord~$a$ comme sur la figure~\ref{fig:decomp-pantalons},~\`a~droite.

\begin{definition} \label{def:bon-recoll-KLM}
Soient $\Pi^+$ et~$\Pi^-$ deux pantalons, de groupes fondamentaux $\pi_1(\Pi^{\pm}) = \langle a^{\pm},b^{\pm},c^{\pm} \,|\, c^{\pm} b^{\pm} a^{\pm} = 1\rangle$.
Soient $\varepsilon,R>0$ et $\rho^{\pm} : \pi_1(\Pi^{\pm})\to G$ des repr\'esentations $(\varepsilon,\pm R)$-presque parfaites de données géométriques $Q^{\pm} = (\rho^{\pm}(a^{\pm}), \rho^{\pm}(b^{\pm}), \rho^{\pm}(c^{\pm}),$ $g^{\pm}, {g'}^{\pm})$.
Pour $\varepsilon'>\nolinebreak 0$, on dit que $\rho^+$ et~$\rho^-$ (ou leurs images) sont \emph{$\varepsilon'$-bien recoll\'ees le long de $a^+$ et~$a^-$} via $Q^+$ et~$Q^-$ si $\rho^+(a^+) = \rho^-(a^-) =: \alpha \in G$, si $g^{\pm}\in\mathcal{U}_{\alpha^{\pm 1}}$ et si
\begin{equation} \label{eqn:presque-flip-decalage}
d_G\big(\Psi_{\alpha}(g^+), \varphi_1\circ\inv\circ\Psi_{\alpha^{-1}}(g^-)\big) < \varepsilon'.
\end{equation}
\end{definition}

Lorsque $R$ est grand, l'\'el\'ement $g^{\pm}$ est tr\`es proche de son image par $\Psi_{\alpha^{\pm 1}}$ (lemmes \ref{lem:realise-proximal} et~\ref{lem:presque-parf-U-alpha}.\eqref{item:N-delta}), et l'in\'egalit\'e \eqref{eqn:presque-flip-decalage} signifie donc que $g^-$ est \og presque\fg\ obtenu \`a partir de~$g^+$ par inversion et d\'ecalage de~$1$ (\cf figure~\ref{fig:recollement}).
En pratique, on prendra des représentations $(\varepsilon/R,\pm R)$-presque parfaites qui sont $\varepsilon'$-bien recollées pour $\varepsilon'=C\varepsilon/R$ o\`u $C>0$ est une constante ne d\'ependant que de~$G$.

\begin{figure}[h!]
\centering
\labellist
\small\hair 2pt
\pinlabel {$\rho^+(b^+)$} [u] at 155 232
\pinlabel {$\rho^+(c^+)$} [u] at 155 132
\pinlabel {$\rho^+(a) = \rho^-(a)$} [u] at 227 220
\pinlabel {$\rho^-(b^-)$} [u] at 293 207
\pinlabel {$\rho^-(c^-)$} [u] at 290 128
\pinlabel {\textcolor{red}{$g^+$}} [u] at 200 188
\pinlabel {\textcolor{orange}{${g'}^+$}} [u] at 208 253
\pinlabel {\textcolor{red}{$g^-$}} [u] at 245 183
\pinlabel {\textcolor{orange}{${g'}^-$}} [u] at 243 247
\endlabellist
\includegraphics[width=8cm]{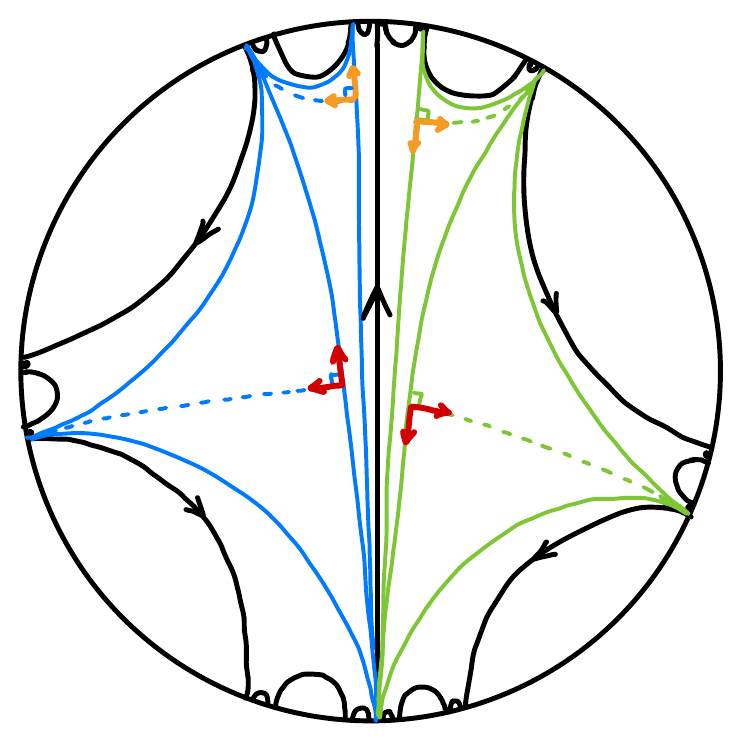}
\vspace{-0.2cm}
\caption{Représentations $(\varepsilon,\pm R)$-presque parfaites $\rho^{\pm}$ qui sont bien recollées au sens de Kahn, Labourie et Mozes}
\label{fig:recollement}
\end{figure}
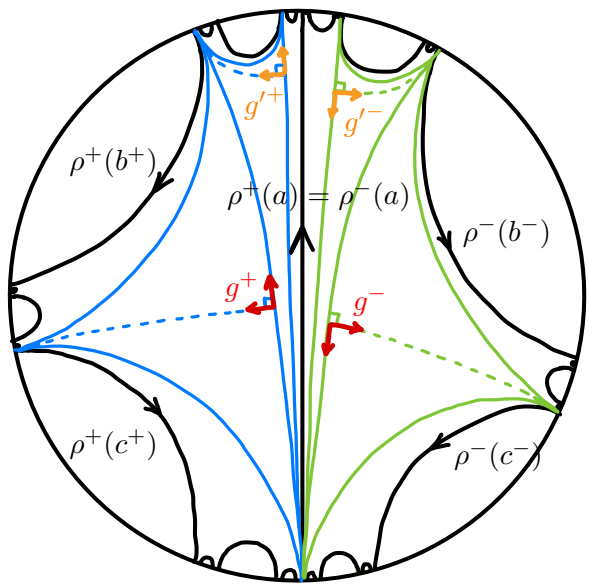

\begin{remarque} \label{rem:recolle-seulement-une-courbe}
En utilisant le lemme~\ref{lem:realise-proximal}, on vérifie que pour tout $\delta>0$, si $\varepsilon,\varepsilon'>0$ sont assez petits et $R>0$ assez grand, alors les images par $(g^-)^{-1}$ des points $\rho^{\pm}(a^{\pm})^{\repuls}$, $\rho^{\pm}(a^{\pm})^{\attract}$, $\rho^+(b^+)^{\repuls}$, $\rho^+(b^+)^{\attract}$, $\rho^+(c^+)^{\repuls}$, $\rho^+(c^+)^{\attract}$, $\rho^-(b^-)^{\repuls}$, $\rho^-(b^-)^{\attract}$, $\rho^-(c^-)^{\repuls}$ et $\rho^-(c^-)^{\attract}$ sont à distance $\leq\delta$ pour~$d_{\tau}$, respectivement, de $\underline{\tau}(0)$, $\underline{\tau}(\infty)$, $\underline{\tau}(\infty)$, $\underline{\tau}(-e)$, $\underline{\tau}(-e)$, $\underline{\tau}(0)$, $\underline{\tau}(\infty)$, $\underline{\tau}(1)$, $\underline{\tau}(1)$ et $\underline{\tau}(0)$ (\cf figure~\ref{fig:recollement}).
\end{remarque}

\subsection{Injectivit\'e des représentations de $\pi_1(S)$ dont les restrictions aux pantalons sont presque parfaites et bien recollées} \label{subsec:dem-pi-1-inj-KLM}

Kahn, Labourie et Mozes établissent la version plus précise suivante de la proposition~\ref{prop:pi-1-inj}.
On note $\sigma_{\Pi} : \{ a_{\Pi},b_{\Pi},c_{\Pi}\}\to\{ a_{\Pi},b_{\Pi},c_{\Pi}\}$ la permutation cyclique envoyant $a_{\Pi}$ sur~$b_{\Pi}$, et l'on renvoie à la remarque~\ref{rem:presque-parf-tri} pour la définition de $\sym$.

\begin{proposition} \label{prop:pi-1-inj-KLM}
Dans le cadre~\ref{cadre}, soit $\Gamma$ un r\'eseau cocompact irr\'eductible de~$G$, et soit $S$ une surface compacte de genre au moins deux avec une d\'ecomposition en pantalons bipartie $\mathcal{P}$ et un graphe~$\mathcal{G}$ comme au paragraphe~\ref{subsec:struct-R-parf}, donnant une présentation $\pi_1(\Pi) = \langle a_{\Pi}, b_{\Pi}, c_{\Pi} \,|\, c_{\Pi} b_{\Pi} a_{\Pi} = 1\rangle$ pour tout $\Pi\in\mathcal{P}$.
Pour tout $\delta>0$, tout $\varepsilon>0$ assez petit par rapport à~$\delta$ et tout $R>0$ assez grand par rapport à~$\varepsilon$, si une repr\'esentation $\rho : \pi_1(S)\to\Gamma$ vérifie que
\begin{enumerate}
  \item\label{item:pi-1-inj-1-KLM} pour tout $\Pi^{\pm}\in\mathcal{P}^{\pm}$, la restriction $\rho_{\Pi^{\pm}}$ de $\rho$ à $\pi_1(\Pi^{\pm})$ est $(\varepsilon/R,\pm R)$-presque parfaite comme à la définition~\ref{def:pant-presque-parfait-KLM}, et
  \item\label{item:pi-1-inj-2-KLM} on peut choisir les données géométriques $Q_{\Pi}$ des~$\rho_{\Pi}$ (définition~\ref{def:pant-presque-parfait-KLM}) de telle sorte que pour tous $\Pi^+,\Pi^-\in\mathcal{P}$ adjacents le long d'une courbe de bord correspondant à $\sigma_{\Pi^+}^{i^+}(a_{\Pi^+})\in\pi_1(\Pi^+)$ et $\sigma_{\Pi^-}^{i^-}(a_{\Pi^-})\in\pi_1(\Pi^-)$ o\`u $i^+,i^-\in\Z/3\Z$, les restrictions $\rho_{\Pi^+}$ et $\rho_{\Pi^-}$ sont $(\varepsilon/R)$-bien recoll\'ees le long de $\sigma_{\Pi^+}^{i^+}(a_{\Pi^+})$ et $\sigma_{\Pi^-}^{i^-}(a_{\Pi^-})$ via $\sym^{i^+}(Q_{\Pi^+})$ et $\sym^{i^-}(Q_{\Pi^-})$, comme à la définition~\ref{def:bon-recoll-KLM},
 \end{enumerate}
alors $\rho$ est injective et admet une application de bord $\xi : \PP^1(\R)\to G/P_{\tau}$ qui est $(\varrho_R,\rho)$-\'equivariante (pour une repr\'esentation $R$-parfaite~$\varrho_R$ comme au paragraphe~\ref{subsec:struct-R-parf}) et $(\delta,\tau)$-sullivannienne.
\end{proposition}

Nous esquissons à présent les idées de la démonstration.
Ces considérations de Kahn, Labourie et Mozes sur les applications de bord sullivanniennes, en lien avec les représentations anosoviennes en rang supérieur, constituent l'une des nouveautés importantes par rapport à l'approche originale de Kahn et Markovi\'c.

\subsubsection{Pavages de $\HH^2$ par des hexagones} \label{subsubsec:pavage-hex}

\`A toute repr\'esentation injective et discr\`ete $\varrho : \pi_1(S)\to\PSL(2,\R)$ est associ\'e un pavage $\varrho(\pi_1(S))$-invariant $\mathrm{Hex}_{\varrho}$ de~$\HH^2$ par des hexagones \`a angles droits, obtenu en consid\'erant les axes de translation dans~$\HH^2$ des images par~$\varrho$ des \'el\'ements de $\pi_1(S)$ correspondant aux courbes de bord des pantalons de~$\mathcal{P}$, et en rajoutant la perpendiculaire commune \`a toute paire d'axes \emph{adjacents} (c'est-\`a-dire non s\'epar\'es par un autre axe).
Ainsi, chaque hexagone est bord\'e par trois segments d'axe de translation (\og c\^ot\'es principaux\fg) et trois perpendiculaires communes.

Si $\varrho_R : \pi_1(S)\to\PSL(2,\R)$ est $R$-parfaite, o\`u $R>1$, alors les c\^ot\'es principaux des hexagones de $\mathrm{Hex}_{\varrho_R}$ sont tous de longueur~$R$, et deux hexagones adjacents le long d'un c\^ot\'e principal y sont toujours d\'ecal\'es de~$1$ (\cf figure~\ref{fig:pavage-hex}).

Tout hexagone $H$ de $\mathrm{Hex}_{\varrho_R}$ d\'efinit un triplet $(a,b,c)$ d'\'el\'ements de $\pi_1(S)$, unique \`a permutation cyclique pr\`es, tel que les c\^ot\'es principaux de~$H$ sont contenus dans les axes de translation $A_{\varrho_R(a)},A_{\varrho_R(b)},A_{\varrho_R(c)}$ de $\varrho_R(a),\varrho_R(b),\varrho_R(c)$, et les points fixes $\varrho_R(a)^{\repuls}$, $\varrho_R(a)^{\attract}$, $\varrho_R(b)^{\repuls}$, $\varrho_R(b)^{\attract}$, $\varrho_R(c)^{\repuls}$ et $\varrho_R(c)^{\attract}$ dans $\PP^1(\R)$ sont dans un ordre cyclique positif.
On dira que $(H,(a,b,c))$ et $(H,(b,c,a))$ et $(H,(c,a,b))$ sont des \emph{hexagones marqu\'es} de $\mathrm{Hex}_{\varrho_R}$.
Comme au paragraphe~\ref{subsubsec:intuition-geom-KLM}, les points fixes de $\varrho_R(a),\varrho_R(b),\varrho_R(c),\varrho_R(b^{-1}a^{-1})$ d\'efinissent une application $(\varrho_R,\varrho_R)$-\'equivariante $\underline{H} \mapsto (h_{\varrho_R,\underline{H}},h'_{\varrho_R,\underline{H}})$ de l'ensemble des hexagones marqu\'es de $\mathrm{Hex}_{\varrho_R}$ vers $\PSL(2,\R)^2$, telle que si $\underline{H} = (H,(a,b,c))$, alors $a,b,c$ et $(T,T') := (h_{\varrho_R,\underline{H}},h'_{\varrho_R,\underline{H}})$ v\'erifient \eqref{eqn:condRparf-realise} pour $\varrho = \varrho_R|_{\langle a,b,c\rangle}$.

\begin{remarque} \label{rem:hH-proche-de-hT}
Le fait~\ref{fait:diam-borne} assure que tout point de~$\HH^2$ est \`a distance uniform\'ement born\'ee (ind\'ependante de $R$ et~$\varrho_R$) du centre d'un hexagone du pavage $\mathrm{Hex}_{\varrho_R}$.
On en d\'eduit l'existence d'une constante $C'>0$ (ind\'ependante de $R$ et~$\varrho_R$) telle que pour tout $h\in\PSL(2,\R)$ et tout $R>1$ on puisse trouver un hexagone marqu\'e $\underline{H}$ de $\mathrm{Hex}_{\varrho_R}$ tel que $d_{\tau}(\tau(h^{-1})\cdot z,\tau(h^{-1})\cdot z') \leq C'\,d_{\tau}(\tau(h_{\varrho_R,\underline{H}}^{-1})\cdot z,\tau(h_{\varrho_R,\underline{H}}^{-1})\cdot z')$ pour tous $z,z'\in G/P_{\tau}$.
\end{remarque}

Suivant Kahn, Labourie et Mozes, on dira qu'une suite $(a_n)_{n\in\N}$ d'éléments de $\pi_1(S)$ est \emph{acceptable} si chaque $a_n$ correspond \`a une courbe de bord d'un pantalon de~$\mathcal{P}$ et pour tout $n\geq 1$, les axes de translation $A_{\varrho_R(a_{n-1})}$ et $A_{\varrho_R(a_n)}$ d'une part, et $A_{\varrho_R(a_n)}$ et $A_{\varrho_R(a_{n+1})}$ d'autre part, bordent des hexagones du pavage $\mathrm{Hex}_{\varrho_R}$ avec des c\^ot\'es qui co\"incident sur une longueur de $R-1$ dans~$\HH^2$ (\cf figure~\ref{fig:pavage-hex}).
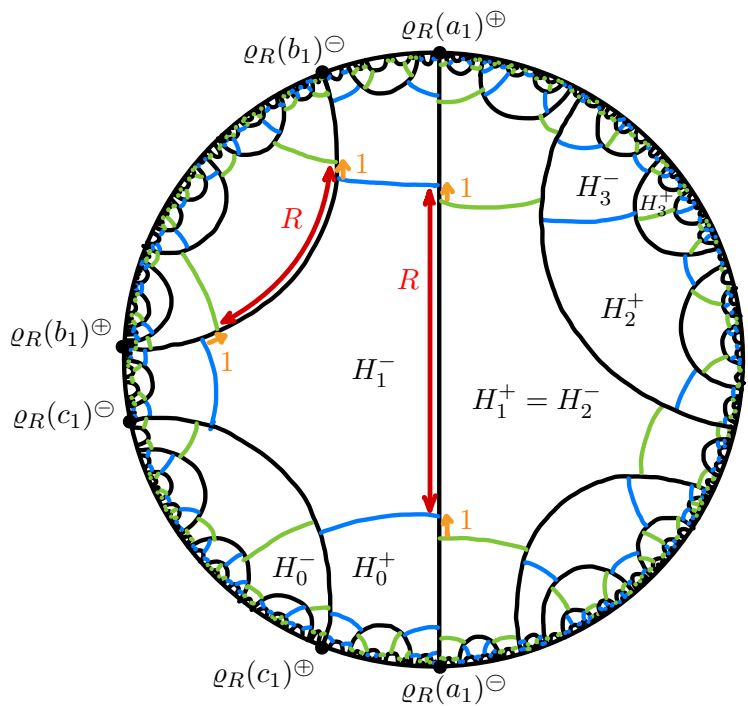
\begin{figure}[h!]
\centering
\labellist
\small\hair 2pt
\pinlabel {\textcolor{orange}{$1$}} [u] at 141 453
\pinlabel {\textcolor{orange}{$1$}} [u] at 185 518
\pinlabel {\textcolor{orange}{$1$}} [u] at 219 510
\pinlabel {\textcolor{orange}{$1$}} [u] at 219 401
\pinlabel {\textcolor{red}{$R$}} [u] at 162 500
\pinlabel {\textcolor{red}{$R$}} [u] at 200 480
\pinlabel {$H_0^-$} [u] at 163 385
\pinlabel {$H_0^+$} [u] at 189 385
\pinlabel {$H_1^-$} [u] at 189 450
\pinlabel {$H_1^+=H_2^-$} [u] at 242 440
\pinlabel {$H_2^+$} [u] at 270 470
\pinlabel {$H_3^-$} [u] at 262 510
\pinlabel {{\tiny $H_3^+$}} [u] at 281 505
\pinlabel {$\varrho_R(a_1)^{\attract}$} [u] at 215 563
\pinlabel {$\varrho_R(b_1)^{\repuls}$} [u] at 163 555
\pinlabel {$\varrho_R(b_1)^{\attract}$} [u] at 87 462
\pinlabel {$\varrho_R(c_1)^{\repuls}$} [u] at 88 435
\pinlabel {$\varrho_R(c_1)^{\attract}$} [u] at 153 350
\pinlabel {$\varrho_R(a_1)^{\repuls}$} [u] at 215 345
\endlabellist
\includegraphics[width=8.5cm]{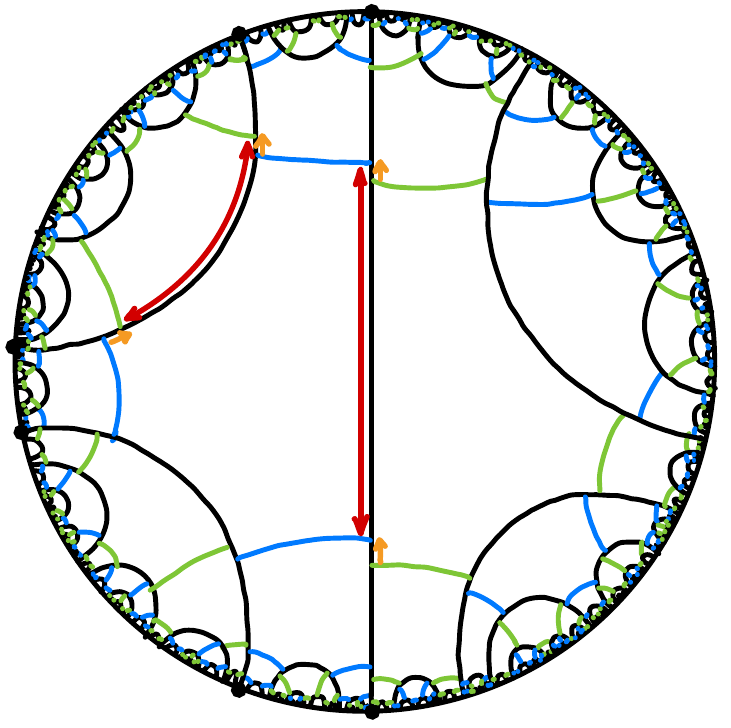}
\vspace{0.2cm}
\caption{Pavage $\mathrm{Hex}_{\varrho_R}$ de $\HH^2$ par des hexagones \`a angles droits, invariant par une repr\'esentation $R$-parfaite $\varrho_R : \pi_1(S)\to\PSL(2,\R)$. On a repr\'esent\'e, pour plusieurs valeurs de~$n$, des hexagones $H_n^-$ et $H_n^+$ ayant des c\^ot\'es qui co\"incident sur une longueur de $R-1$ le long d'un axe de translation $A_{\varrho_R(a_n)}$ o\`u $a_n\in\pi_1(S)$ : la suite $(a_n)$ est \emph{$\varrho_R$-acceptable}. On a \'egalement repr\'esent\'e les points fixes attractifs et r\'epulsifs de $\varrho_R(a_1),\varrho_R(b_1),\varrho_R(c_1)$ pour trois \'el\'ements $a_1,b_1,c_1\in\pi_1(S)$ tels que $(H_1^-,(a_1,b_1,c_1))$ est un \emph{hexagone marqu\'e} de $\mathrm{Hex}_{\varrho_R}$.}
\label{fig:pavage-hex}
\end{figure}
On dira qu'un point $x\in\PP^1(\R)$ est \emph{$\varrho_R$-accessible} si c'est la limite, dans la compactification $\overline{\HH^2} = \HH^2 \sqcup \PP^1(\R)$ de~$\HH^2$, d'une suite $(A_{\varrho_R(a_n)})_{n\in\N}$ où $(a_n)_{n\in\N}$ est acceptable ; dans ce cas, si $A_{\varrho_R(a_0)}$ et $A_{\varrho_R(a_1)}$ bordent un hexagone~$H$ du pavage $\mathrm{Hex}_{\varrho_R}$, on dira que $x$ est \emph{$\varrho_R$-accessible \`a partir de~$H$}.
Kahn, Labourie et Mozes font l'observation suivante.

\begin{lemme} \label{lem:acc-theta-dense}
Il existe une fonction $\vartheta : \R_+^*\to\R_+^*$, tendant vers~$0$ en $+\infty$, telle que pour tout $R>1$, toute repr\'esentation $R$-parfaite $\varrho_R : \pi_1(S)\to\PSL(2,\R)$ et tout hexagone marqu\'e $\underline{H} = (H,(a,b,c))$ de $\mathrm{Hex}_{\varrho_R}$, l'image par $h_{\varrho_R,\underline{H}}^{-1}$ de l'ensemble des points $\varrho_R$-accessibles \`a partir de~$H$ est $\vartheta(R)$-dense dans $\PP^1(\R)$ pour $d_{\PP^1(\R)}$.
\end{lemme}

De mani\`ere analogue \`a l'application $\underline{H}\mapsto (h_{\varrho_R,\underline{H}},h'_{\varrho_R,\underline{H}})$, \`a toute repr\'esentation $\rho : \pi_1(S)\to G$ v\'erifiant la condition \eqref{item:pi-1-inj-1-KLM} de la proposition~\ref{prop:pi-1-inj-KLM} pour un certain $\varepsilon>0$, on peut associer une application $(\varrho_R,\rho)$-\'equivariante $\underline{H}\mapsto (g_{\rho,\underline{H}},g'_{\rho,\underline{H}})$ de l'ensemble des hexagones marqu\'es de $\mathrm{Hex}_{\varrho_R}$ vers $G^2$, telle que si $\underline{H} = (H,(a,b,c))$, alors $(\rho(a),\rho(b),\rho(c),g_{\rho,\underline{H}},g'_{\rho,\underline{H}})$ est une donnée géométrique associée à la repr\'esentation $(\varepsilon/R,R)$-presque parfaite $\rho|_{\langle a,b,c\rangle}$ comme à la définition~\ref{def:pant-presque-parfait-KLM}.

\subsubsection{Applications de bord partielles et d\'eformations}

Le c\oe ur de la d\'emonstration de la proposition~\ref{prop:pi-1-inj-KLM} est form\'e des deux r\'esultats ci-dessous.
Le premier, comme la proposition~\ref{prop:Sullivan} sur laquelle repose le second, fait intervenir une version pr\'ecise, dans $G/P_{\tau}$, du lemme de Morse de \textcite{klp14,klp-morse}, d\'ej\`a mentionn\'ee au paragraphe~\ref{subsec:prop-sullivan}.

Rappelons (corollaire~\ref{cor:presque-parf->generique}.\eqref{item:presque-parf->generique}) que pour tout $\varepsilon>0$ et tout $R>0$ assez grand par rapport \`a~$\varepsilon$, si une repr\'esentation $\rho : \pi_1(S)\to G$ v\'erifie la condition~\eqref{item:pi-1-inj-1-KLM} de la proposition~\ref{prop:pi-1-inj-KLM}, alors pour tout $a\in\pi_1(S)$ correspondant à une courbe de bord d'un pantalon de~$\mathcal{P}$, l'élément $\rho(a)$ est proximal dans $G/P_{\tau}$ : il admet un unique point fixe attractif $\rho(a)^{\attract}\in\nolinebreak G/P_{\tau}$.
Soit $C'>0$ la constante de la remarque~\ref{rem:hH-proche-de-hT}.

\begin{proposition} \label{prop:appl-bord-partielle}
Pour tout $\delta>0$, tout $\varepsilon>0$ assez petit par rapport \`a~$\delta$, tout $R>1$ assez grand par rapport \`a~$\varepsilon$, toute repr\'esentation $R$-parfaite $\varrho_R : \pi_1(S)\to\PSL(2,\R)$ et toute repr\'esentation $\rho : \pi_1(S)\to G$ v\'erifiant les conditions \eqref{item:pi-1-inj-1-KLM} et~\eqref{item:pi-1-inj-2-KLM} de la proposition~\ref{prop:pi-1-inj-KLM}, il existe une unique application $(\varrho_R,\rho)$-\'equivariante $\xi$ de l'ensemble des points $\varrho_R$-accessibles de $\PP^1(\R)$ vers $G/P_{\tau}$ v\'erifiant la propri\'et\'e suivante : pour toute suite acceptable $(a_n)_{n\in\N}$ d'éléments de $\pi_1(S)$ et tout hexagone marqu\'e $\underline{H} = (H,(a,b,c))$ de $\mathrm{Hex}_{\varrho_R}$ o\`u $H$ est bord\'e par $A_{\varrho_R(a_0)}$ et~$A_{\varrho_R(a_1)}$, si $(A_{\varrho_R(a_n)})$ converge dans~$\overline{\HH^2}$ vers un point $x\in\PP^1(\R)$, alors $(\rho(a_n)^{\attract})_{n\in\N}$ converge dans $G/P_{\tau}$ vers $\xi(x)$ et
$$d_{\tau}\big(g_{\rho,\underline{H}}^{-1}\circ\xi(x),\underline{\tau}\circ h_{\varrho_R,\underline{H}}^{-1}(x)\big) \leq \frac{\delta}{2C'}.$$
\end{proposition}

\begin{proposition} \label{prop:Sullivan-deform}
Soit $\varrho : \pi_1(S)\to\PSL(2,\R)$ une repr\'esentation injective et discr\`ete.
Pour tout $\delta>0$ assez petit, il existe $\theta>0$ tel que pour toute famille continue $(\rho_t)_{t\in [0,1]}\subset\Hom(\pi_1(S),G)$ de repr\'esentations et toutes applications $(\varrho,\rho_t)$-\'equivariantes $\xi_t : \PP^1(\R)\to G/P_{\tau}$, si $\xi_0$ est $(\delta,\tau)$-sullivannienne et si pour tout $t\in ]0,1]$, les deux hypothèses suivantes sont satisfaites :
\begin{enumerate}
  \item\label{item:Sullivan-deform-1} (\og coh\'erence d'attraction\fg) pour tout $x\in\PP^1(\R)$, il existe une suite $(a_n)\in\pi_1(S)^{\N}$ telle que $\rho_t(a_n)$ soit proximal dans $G/P_{\tau}$ pour tout~$n$ et telle que $(\varrho(a_n)^{\attract})_{n\in\N}$ converge dans $\PP^1(\R)$ vers~$x$ et $(\rho_t(a_n)^{\attract})_{n\in\N}$ converge dans $G/P_{\tau}$ vers $\xi_t(x)$,
  \item\label{item:Sullivan-deform-2} (\og condition sullivannienne sur un sous-ensemble suffisamment dense\fg) pour tout $h\in\PSL(2,\R)$, il existe $g\in G$ et un sous-ensemble $D$ de $\PP^1(\R)$ tels que $h^{-1}(D)$ soit $\theta$-dense dans $\PP^1(\R)$ et que $d_{\tau}^g(\xi_t(x), g\circ\underline{\tau}\circ h^{-1}(x)) \leq \delta/2$ pour tout $x\in D$,
\end{enumerate}
alors $\xi_t$ est $(\delta,\tau)$-sullivannienne pour tout $t\in [0,1]$.
\end{proposition}

\begin{proof}[D\'emonstration de la proposition~\ref{prop:Sullivan-deform}]
Soient $\alpha,\delta_0,C>0$ donn\'es par la proposition~\ref{prop:Sullivan}.\eqref{item:Sullivan-Holder}.
Prenons $\delta<\delta_0$ et $\theta$ tel que $\theta + C\,\theta^{\alpha} \leq \delta/2$, et montrons que l'ensemble $E$ des $t\in [0,1]$ tels que $\xi_t$ soit $(\delta,\tau)$-sullivannienne est ouvert et ferm\'e dans~$[0,1]$.

Pour montrer que $E$ est ferm\'e, on observe que les applications $(\delta,\tau)$-sullivanniennes forment une famille \'equicontinue, par la proposition~\ref{prop:Sullivan}.\eqref{item:Sullivan-Holder} : ainsi, d'après le théorème d'Arzelà--Ascoli, si $(t_m)\in E^{\N}$ converge vers~$t$, on peut supposer apr\`es extraction que $(\xi_{t_m})_{m\in\N}$ converge vers une application $\xi'_t : \PP^1(\R)\to G/P_{\tau}$.
Cette application $\xi'_t$ est encore $(\delta,\tau)$-sullivannienne, et $(\varrho,\rho_t)$-\'equivariante.
D'après la proposition~\ref{prop:Sullivan}.\eqref{item:Sullivan-Anosov}, la représentation $\rho_t$ est $P_{\tau}$-anosovienne d'application de bord~$\xi'_t$.
En particulier (propri\'et\'e g\'en\'erale des repr\'esentations anosoviennes), pour toute suite $(a_n)_{n\in\N}$ d'\'el\'ements non triviaux de $\pi_1(S)$ telle que $(\varrho(a_n)^{\attract})_{n\in\N}$ converge dans $\PP^1(\R)$ vers $x\in\PP^1(\R)$, l'\'el\'ement $\rho_t(a_n)$ est proximal dans $G/P_{\tau}$ pour tout~$n$ et $(\rho_t(a_n)^{\attract})_{n\in\N}$ converge dans $G/P_{\tau}$ vers $\xi'_t(x)$.
L'hypothèse \eqref{item:Sullivan-deform-1} de la proposition implique alors $\xi'_t = \xi_t$.
Ainsi, $t\in E$.

V\'erifions que $E$ est ouvert dans $[0,1]$.
D'apr\`es la proposition~\ref{prop:Sullivan}.\eqref{item:Sullivan-Anosov-deform}, tout $t\in E$ admet un voisinage dans $[0,1]$ form\'e d'\'el\'ements $s$ tels que $\xi_s$ soit $(\delta_0,\tau)$-sullivannienne.
Gr\^ace \`a l'hypothèse \eqref{item:Sullivan-deform-2} ci-dessus, on en d\'eduit que $\xi_s$ est en fait $(\delta,\tau)$-sullivannienne, \ie $s\in E$.
En effet, soit $h\in\PSL(2,\R)$.
D'apr\`es l'hypothèse \eqref{item:Sullivan-deform-2} il existe $g\in G$ et, pour tout $x\in\nolinebreak\PP^1(\R)$, un point $y\in\PP^1(\R)$ tels que $d_{\tau}(\underline{\tau}\circ h^{-1}(x),\underline{\tau}\circ\nolinebreak h^{-1}(y)) = d_{\PP^1(\R)}(h^{-1}(x),h^{-1}(y)) \leq \theta$ et $d_{\tau}(g^{-1}\circ\xi_s(y), \underline{\tau}\circ h^{-1}(y)) \leq \delta/2$.
D'apr\`es la proposition~\ref{prop:Sullivan}.\eqref{item:Sullivan-Holder}, on a $d_{\tau}(g^{-1}\circ\xi_s(x),g^{-1}\circ\xi_s(y)) \leq C\,d_{\PP^1(\R)}(h^{-1}(x),h^{-1}(y))^{\alpha} \leq C\,\theta^{\alpha}$, et l'on conclut par in\'egalit\'e triangulaire.
\end{proof}

\subsubsection{D\'emonstration de la proposition~\ref{prop:pi-1-inj-KLM}}

Fixons $\delta>0$ (que l'on peut supposer tr\`es petit) et prenons $\varepsilon>0$ assez petit par rapport à~$\delta$ et $R>1$ assez grand par rapport \`a~$\varepsilon$ comme dans la proposition~\ref{prop:appl-bord-partielle}.
Prenons de plus $R$ assez grand de sorte que le r\'eel $\vartheta(R)$ du lemme~\ref{lem:acc-theta-dense} soit inf\'erieur \`a $\theta/C'$ o\`u $\theta$ et~$C'$ sont donn\'es respectivement par la proposition~\ref{prop:Sullivan-deform} et la remarque~\ref{rem:hH-proche-de-hT}.

Supposons dans un premier temps qu'il existe une famille continue $(\rho_t)_{t\in [0,1]}\subset\Hom(\pi_1(S),G)$ telle que $\rho_0$ soit $\tau$-fuchsienne, $\rho_1=\rho$ et pour tout $t\in [0,1]$, la repr\'esentation $\rho_t$ v\'erifie les conditions \eqref{item:pi-1-inj-1-KLM} et~\eqref{item:pi-1-inj-2-KLM} de la proposition~\ref{prop:pi-1-inj-KLM}.
Soit $\mathrm{Hex}_{\varrho_R}$ le pavage $\varrho_R(\pi_1(S))$-invariant de~$\HH^2$ du paragraphe~\ref{subsubsec:pavage-hex} et, pour tout $t\in [0,1]$, soit $\underline{H}\mapsto (h_{\varrho_R,\underline{H}},h'_{\varrho_R,\underline{H}})$ (\resp $\underline{H}\mapsto (g_{\rho_t,\underline{H}},g'_{\rho_t,\underline{H}})$) une application $(\varrho_R,\varrho_R)$-\'equivariante (\resp $(\varrho_R,\rho_t)$-\'equivariante) de l'ensemble des hexagones marqu\'es de $\mathrm{Hex}_{\varrho_R}$ vers $\PSL(2,\R)^2$ (\resp $G^2$) comme au paragraphe~\ref{subsubsec:pavage-hex}.
Soit $\xi_t$ l'application $(\varrho_R,\rho_t)$-\'equivariante de la proposition~\ref{prop:appl-bord-partielle}, allant de l'ensemble des points $\varrho_R$-accessibles de $\PP^1(\R)$ vers $G/P_{\tau}$ ; on la prolonge en une application $(\varrho_R,\rho_t)$-équivariante $\xi_t : \PP^1(\R)\to G/P_{\tau}$ v\'erifiant la condition~\eqref{item:Sullivan-deform-1} de la proposition~\ref{prop:Sullivan-deform}.
La remarque~\ref{rem:hH-proche-de-hT}, le lemme~\ref{lem:acc-theta-dense} et la proposition~\ref{prop:appl-bord-partielle} assurent que la condition~\eqref{item:Sullivan-deform-2} de la proposition~\ref{prop:Sullivan-deform} est v\'erifi\'ee.
En effet, pour $t\in [0,1]$ et $h\in\PSL(2,\R)$, soit $\underline{H} = (H,(a,b,c))$ un hexagone marqu\'e de $\mathrm{Hex}_{\varrho_R}$ donn\'e par la remarque~\ref{rem:hH-proche-de-hT}.
Soit $g := g_{\rho_t,\underline{H}} \, \tau(h_{\varrho_R,\underline{H}}^{-1}\,h) \in G$ et soit $D$ l'ensemble des points $\varrho_R$-accessibles \`a partir de~$H$.
Pour tout $x\in D$, on a
\begin{align*}
d_{\tau}^g\big(\xi_t(x), g\circ\underline{\tau}\circ h^{-1}(x)\big) & = d_{\tau}\big(g^{-1}\circ\xi_t(x), \underline{\tau}\circ h^{-1}(x)\big)\\
& = d_{\tau}\big(\tau(h^{-1} h_{\varrho_R,\underline{H}}) \, g_{\rho_t,\underline{H}}^{-1}\circ\xi_t(x), \tau(h^{-1} h_{\varrho_R,\underline{H}})\circ\underline{\tau}\circ h_{\varrho_R,\underline{H}}^{-1}(x)\big)\\
& \leq C' \, d_{\tau}\big(g_{\rho_t,\underline{H}}^{-1}\circ\xi_t(x), \underline{\tau}\circ h_{\varrho_R,\underline{H}}^{-1}(x)\big) \,\leq\, \delta/2,
\end{align*}
o\`u les deux in\'egalit\'es utilisent respectivement la remarque~\ref{rem:hH-proche-de-hT} et la proposition~\ref{prop:appl-bord-partielle}.
Or, l'ensemble $h_{\varrho_R,\underline{H}}^{-1}(D)$ est $(\theta/C')$-dense dans $\PP^1(\R)$ d'apr\`es le lemme~\ref{lem:acc-theta-dense} et notre choix de~$R$, donc l'ensemble $h^{-1}(D)$ est $\theta$-dense dans $\PP^1(\R)$ d'apr\`es la remarque~\ref{rem:hH-proche-de-hT}.
On peut alors appliquer la proposition~\ref{prop:Sullivan-deform} : on obtient que pour tout $t\in [0,1]$, l'application $(\varrho_R,\rho_t)$-\'equivariante $\xi_t$ est $(\delta,\tau)$-sullivannienne.

En g\'en\'eral, Kahn, Labourie et Mozes ne construisent pas de famille continue $(\rho_t)_{t\in [0,1]}\subset\Hom(\pi_1(S),G)$ telle que $\rho_0$ soit $\tau$-fuchsienne et $\rho_1=\rho$.
Ils proc\`edent plut\^ot par approximation, en consid\'erant une suite $(S_n)_{n\in\N}$ de surfaces hyperboliques compactes obtenues en doublant des surfaces hyperboliques $S'_n$ compactes \`a bord qui sont des unions de pantalons isom\'etriques \`a des pantalons de~$\mathcal{P}$, et telles que le rev\^etement universel de~$S$ co\"incide avec celui de~$S'_n$ sur une boule de rayon~$n$.
La repr\'esentation $\rho : \pi_1(S)\to G$ d\'efinit pour tout~$n$ une repr\'esentation $\rho^{(n)} : \pi_1(S_n)\to G$ dont les restrictions aux pantalons de~$S_n$ sont encore $(\varepsilon/R,\pm R)$-presque parfaites et $(\varepsilon/R)$-bien recoll\'ees.
Pour tout~$n$, en utilisant le fait que $\pi_1(S'_n)$ est un groupe libre, ils construisent une famille continue $(\rho^{(n)}_t)_{t\in [0,1]}\subset\Hom(\pi_1(S_n),G)$ telle que $\rho^{(n)}_0$ soit $\tau$-fuchsienne, $\rho^{(n)}_1=\rho^{(n)}$ et pour tout $t\in [0,1]$, la repr\'esentation $\rho^{(n)}_t$ v\'erifie les conditions \eqref{item:pi-1-inj-1-KLM} et~\eqref{item:pi-1-inj-2-KLM} de la proposition~\ref{prop:pi-1-inj-KLM}.
D'apr\`es le paragraphe pr\'ec\'edent, $\rho^{(n)}$ admet donc une application de bord $(\delta,\tau)$-sullivannienne $\xi^{(n)} : \PP^1(\R)\to G/P_{\tau}$.
La famille $(\xi^{(n)})_{n\in\N}$ est \'equicontinue d'apr\`es la proposition~\ref{prop:Sullivan}.\eqref{item:Sullivan-Holder}.
En proc\'edant par convergence, on obtient alors une application de bord $(\delta,\tau)$-sullivannienne $\xi : \PP^1(\R)\to G/P_{\tau}$ pour~$\rho$.

\section{\'Etape dynamique : utilisation du mélange} \label{sec:dyn}

Dans cette partie, nous donnons les grandes lignes de la d\'emonstration des propositions \ref{prop:pant-dans-reseau-general} et~\ref{prop:pant-dans-reseau-mesures} selon le point de vue de \textcite{klm18}.
L'espace $\mathrm{Geom}_{\varepsilon,\pm R}$ de la proposition~\ref{prop:pant-dans-reseau-mesures} est ici un espace $\Triconn_{\varepsilon,\pm R}$ de \og couples triconnect\'es\fg\ que nous introduisons au paragraphe~\ref{subsec:Triconn}.

\smallskip

\emph{Dans toute la partie, on travaille dans le cadre~\ref{cadre}.
On suppose à présent la condition~(R) du paragraphe~\ref{subsubsec:cond-R} vérifiée.
On fixe un réseau cocompact irréductible $\Gamma$ de~$G$.
Par le lemme de Selberg, quitte à remplacer $\Gamma$ par un sous-groupe d'indice fini, on peut le supposer sans torsion ; c'est ce que nous ferons ici.}

\subsection{Fonctions poids sur $G\times G$} \label{subsec:poids}

Rappelons la notion de \emph{repr\'esentation $(\varepsilon,R)$-presque parfaite} (définition~\ref{def:pant-presque-parfait-KLM}), qui fait intervenir la fonction $\rot\circ\inv\circ\varphi_R$.
Pour d\'emontrer la proposition~\ref{prop:pant-dans-reseau-general}.\eqref{item:pant-dans-reseau-1}, l'id\'ee est de montrer, en utilisant le mélange (fait~\ref{fait:melange}), que pour tous \'el\'ements $[x],[y]\in\Gamma\backslash G$, il existe \og beaucoup\fg\ d'\'el\'ements $[g]\in\Gamma\backslash G$ proches de~$[x]$ tels que $\rot\circ\inv\circ\varphi_R([g])$ soit proche de~$[y]$.

Plus pr\'ecis\'ement, fixons une fonction cloche $\chi\in C^{\infty}(\R,\R^+)$, de support $[-1,1]$.
Pour tout $\varepsilon>0$, on obtient une fonction cloche $\chi_{\varepsilon}\in C^{\infty}(G,\R^+)$ de support la boule de rayon~$\varepsilon$ autour de l'\'el\'ement neutre $\id\in G$, d'int\'egrale~$1$, en posant
$$\chi_{\varepsilon}(g) := \frac{1}{\int_G \chi(d_G(\cdot,\id)^2/\varepsilon^2)} \, \chi(d_G(g,\id)^2/\varepsilon^2).$$

Pour $R>0$, on d\'efinit des fonctions \og poids\fg\ $W_{\varepsilon,R},W_{\varepsilon,-R} : G\times G\to\R^+$, \`a support compact dans $G\times G$, par
\begin{equation} \label{eqn:W-epsilon-R}
W_{\varepsilon,\pm R}(x,y) = \int_{g\in G} \chi_{\varepsilon/R}(x^{-1}g) \, \chi_{\varepsilon/R}\big(y^{-1}\,(\rot^{\pm 1}\circ\inv\circ\varphi_{\pm R})(g)\big) \, \dd g.
\end{equation}
La fonction $W_{\varepsilon,\pm R}$ mesure la proportion d'\'el\'ements $g\in G$ tels que $g$ soit $(\varepsilon/R)$-proche de~$x$ et $\rot^{\pm 1}\circ\inv\circ\varphi_{\pm R}(g)$ soit $(\varepsilon/R)$-proche de~$y$.
Elle est invariante sous l'action diagonale de~$G$ : on a $W_{\varepsilon,\pm R}(gx,gy)=W_{\varepsilon,\pm R}(x,y)$ pour tous $g,x,y\in G$, car $\rot$, $\inv$ et $\varphi_R$ correspondent \`a des multiplications \`a droite.
Elle induit une fonction $(\Gamma\times\Gamma)$-invariante $w_{\varepsilon,\pm R} : G\times G\to\R^+$ donn\'ee par
\begin{equation} \label{eqn:w-eps-R}
w_{\varepsilon,\pm R}(x,y) = \sum_{\gamma\in\Gamma} W_{\varepsilon,\pm R}(x,\gamma y).
\end{equation}
(Cette somme est finie car à $x$ fixé, la fonction $W_{\varepsilon,\pm R}(x,\cdot) : G\to\R^+$ est \`a support compact et $\Gamma$ est discret dans~$G$ ; ainsi, $w_{\varepsilon,\pm R}$ est bien d\'efinie et lisse.)
La fonction $w_{\varepsilon,\pm R}$ passe au quotient en une fonction $\Gamma\backslash G \times \Gamma\backslash G \to \R^+$, que l'on notera encore $w_{\varepsilon,\pm R}$.
Celle-ci mesure la proportion d'\'el\'ements $[g]\in\Gamma\backslash G$ proches de~$[x]$ tels que $\rot^{\pm 1}\circ\inv\circ\varphi_{\pm R}([g])$ soit proche de~$[y]$, ou encore la proportion d'\'el\'ements $g\in G$ proches de~$x$ pour lesquels il existe $\gamma\in\Gamma$ tel que $\gamma\cdot (\rot^{\pm 1}\circ\inv\circ\varphi_{\pm R})(g)$ soit proche de~$y$.

Le lemme suivant reprend l'approche de Margulis dans sa th\`ese.
On n'utilise ici qu'une vitesse de mélange polynomiale (\cf paragraphe~\ref{subsec:etape-dyn}), qui donne une estim\'ee en $\varepsilon/R^2$ et suffit pour d\'emontrer les propositions \ref{prop:pant-dans-reseau-general} et~\ref{prop:pant-dans-reseau-mesures} ; en utilisant pleinement le mélange exponentiel (fait~\ref{fait:melange}), on obtiendrait une estim\'ee en $e^{-qR/2}$.


\begin{lemme} \label{lem:w-presque-1}
Pour tout $\varepsilon>0$, tout $R>0$ assez grand par rapport \`a~$\varepsilon$ et tous $[x],[y]\in\Gamma\backslash G$, on a
\begin{equation} \label{eqn:w-proche-de-1}
\big|w_{\varepsilon,\pm R}([x],[y]) - 1\big| \leq \frac{\varepsilon}{R^2}.
\end{equation}
\end{lemme}

\begin{proof}
Pour tout $\varepsilon>0$, la fonction $X_{\varepsilon} : G\times G\to\R^+$ donn\'ee par $X_{\varepsilon}(x,g) = \sum_{\gamma\in\Gamma} \chi_{\varepsilon}(x^{-1}\gamma g)$ est de classe $C^{\infty}$ et passe au quotient en une fonction $X_{\varepsilon} : \Gamma\backslash G\times\Gamma\backslash G\to\R^+$.
Soit $k\in\N$ donn\'e par le fait~\ref{fait:melange}.
Il n'est pas difficile de v\'erifier qu'il existe $D>0$ tel que $\Vert\chi_{\varepsilon}\Vert_{C^k} \leq D\,\varepsilon^{-k-D}$ pour tout $\varepsilon>0$ ; on en d\'eduit qu'il existe $D'>0$ tel que $\Vert X_{\varepsilon}(x,\cdot)\Vert_{C^k}, \Vert X_{\varepsilon}(x,(\rot\circ\inv)(\cdot))\Vert_{C^k} \leq D'\,\varepsilon^{-k-D'}$ pour tout $\varepsilon>0$ et tout $[x]\in\Gamma\backslash G$.

Soit $\mathcal{D}$ un domaine fondamental mesurable relativement compact de $G$ pour l'action de~$\Gamma$ (rappelons que l'on a suppos\'e $\Gamma$ sans torsion).
Soient $\varepsilon,R>0$.
Pour tous $x,y\in G$, on a
\begin{eqnarray*}
w_{\varepsilon,R}([x],[y]) & = & \sum_{\gamma\in\Gamma} \int_{g\in G} \chi_{\varepsilon/R}(x^{-1}g) \, \chi_{\varepsilon/R}\big(y^{-1}\gamma^{-1}\,(\rot\circ\inv\circ\varphi_R)(g)\big) \, \dd g\\
& = & \sum_{\gamma,\gamma'\in\Gamma} \int_{g\in\mathcal{D}} \chi_{\varepsilon/R}(x^{-1}\gamma'g) \, \chi_{\varepsilon/R}\big(y^{-1}\gamma^{-1}\gamma'\,(\rot\circ\inv\circ\varphi_R)(g)\big) \, \dd g\\
& = & \int_{[g]\in\Gamma\backslash G} X_{\varepsilon/R}(x,[g]) \, X_{\varepsilon/R}(y,(\rot\circ\inv\circ\varphi_R)([g])) \, \dd [g].
\end{eqnarray*}
Si l'on pose $\psi=X_{\varepsilon/R}(x,\cdot)$ et $\theta=X_{\varepsilon/R}(y,(\rot\circ\inv)(\cdot))$, alors $\int_{\Gamma\backslash G} \psi = \int_{\Gamma\backslash G} \theta = \int_G \chi_{\varepsilon/R} = 1$, et le fait~\ref{fait:melange} donne
$$\big|w_{\varepsilon,R}([x],[y]) - 1\big| \leq C\,e^{-qR} \, \Vert\psi\Vert_{C^k} \, \Vert\theta\Vert_{C^k} \leq C\,{D'}^2\,e^{-qR} \Big(\frac{\varepsilon}{R}\Big)^{-2k-2D'}.$$
En particulier, on obtient \eqref{eqn:w-proche-de-1} d\`es que $R$ est assez grand par rapport \`a~$\varepsilon$.
Un raisonnement analogue vaut pour $w_{\varepsilon,-R}$.
\end{proof}

\begin{corollaire} \label{cor:W>0}
Pour tout $\varepsilon>0$, tout $R>0$ assez grand par rapport \`a~$\varepsilon$ et tous $x,y\in G$, il existe $\gamma\in\Gamma$ tel que $W_{\varepsilon,R}(x,\gamma y)>0$.
\end{corollaire}

\subsection{Fonctions poids sur $G^4$ et repr\'esentations presque parfaites : d\'emonstration de la proposition~\ref{prop:pant-dans-reseau-general}.\eqref{item:pant-dans-reseau-1}} \label{subsec:lemme-fermeture}

Pour tous $\varepsilon,R>0$, on d\'efinit des fonctions $W^{\mathrm{tri}}_{\varepsilon,\pm R} : G^4\to\R^+$ par
\begin{equation} \label{eqn:W-epsilon-R-tri}
\left\{\hspace{-0.2cm}\begin{array}{ccl}
W^{\mathrm{tri}}_{\varepsilon,R}(x,y_0,y_1,y_2) & \hspace{-0.3cm} = \hspace{-0.3cm} & W_{\varepsilon,R}(x,y_0) \, W_{\varepsilon,R}(\rot^2(x),\rot(y_1)) \, W_{\varepsilon,R}(\rot(x),\rot^2(y_2)),\\
W^{\mathrm{tri}}_{\varepsilon,-R}(x,y_0,y_1,y_2) & \hspace{-0.3cm} = \hspace{-0.3cm} & W_{\varepsilon,-R}(x,\rot^2(y_1)) \, W_{\varepsilon,-R}(\rot(x),\rot(y_0)) \, W_{\varepsilon,-R}(\rot^2(x),y_2)
\end{array}\right. \hspace{-0.5cm}
\end{equation}
pour tous $x,y_0,y_1,y_2\in G$, où $W_{\varepsilon,\pm R}$ est donnée par \eqref{eqn:W-epsilon-R}.
Ces fonctions sont invariantes par l'action diagonale de $G$ sur~$G^4$ par multiplication \`a gauche.
Le lemme suivant fait le lien avec les données géométriques de repr\'esentations $(\varepsilon,\pm R)$-presque parfaites (d\'efinition~\ref{def:pant-presque-parfait-KLM}).

\begin{lemme} \label{lem:triconn->pant}
Soit $\Pi$ un pantalon de groupe fondamental $\pi_1(\Pi) = \langle a,b,c \,|\, cba = 1\rangle$, soit $\rho : \pi_1(\Pi)\to G$ une repr\'esentation, soient $\varepsilon,R>0$ et soient $g,g'\in G$.
Alors
\begin{enumerate}
  \item\label{item:triconn->pant+} la représentation $\rho$ est $(\varepsilon,R)$-presque parfaite de donnée géométrique\\ $(\rho(a),\rho(b),\rho(c),g,g')$ si et seulement si $W^{\mathrm{tri}}_{\varepsilon R,R}(g,g',\rho(a)^{-1}g',\rho(b)g')>0$ ;
  \item\label{item:triconn->pant-} la représentation $\rho$ est $(\varepsilon,-R)$-presque parfaite de donnée géométrique\\ $(\rho(a),\rho(b),\rho(c),g,g')$ si et seulement si $W^{\mathrm{tri}}_{\varepsilon R,-R}(g,g'',\rho(a)g'',\rho(ba)g'')>0$,\\ où l'on pose $g'' := \rho(a)^{-1}\circ\rot(g')$.
\end{enumerate}
\end{lemme}

\begin{proof}
Traitons le point~\eqref{item:triconn->pant+} ; le point~\eqref{item:triconn->pant-} est analogue.
Par d\'efinition \eqref{eqn:W-epsilon-R} de $W_{\varepsilon R,R}$, on a $W_{\varepsilon R,R}(g,g')>0$ si et seulement s'il existe $h\in G$ tel que $d_G(h,g)<\varepsilon$ et $d_G(\rot\circ\inv\circ\varphi_R(h),g')<\varepsilon$.
De m\^eme, on a $W_{\varepsilon R,R}(\rot^2(g),\rho(a)^{-1}\cdot\rot(g'))>0$ si et seulement s'il existe $h''\in G$ tel que
$$d_G(h'',\rot^2(g))<\varepsilon \quad\mathrm{et}\quad d_G(\rot\circ\inv\circ\varphi_R(h''),\rho(a)^{-1}\cdot\rot(g'))<\varepsilon\ ;$$
en posant $h':=\rho(a)\cdot\inv\circ\varphi_R(h'')$, ceci est \'equivalent \`a
$$d_G(g',h')<\varepsilon \quad\mathrm{et}\quad d_G(\rot\circ\inv\circ\varphi_R(h'),\rho(a)\cdot g)<\varepsilon$$
(car $d_G$ est invariante par $G$ et $\rot$, et $(\inv\circ\varphi_R)^2=\mathrm{id}_G$).
Ainsi, on a\linebreak $W_{\varepsilon R,R}(g,g')\,W_{\varepsilon R,R}(\rot^2(g),\rho(a)^{-1}\cdot\rot(g'))>\nolinebreak 0$ si et seulement si $\rho(a)$ est $(\varepsilon,R)$-presque réalisé par $(g,g')$ (d\'efinition~\ref{def:realise-presque-alpha}).
De m\^eme, $W_{\varepsilon R,R}(g,g')\,W_{\varepsilon R,R}(\rot(g),\rho(b)\cdot\rot^2(g'))>0$ (\resp $W_{\varepsilon R,R}(\rot^2(g),\rho(a)^{-1}\cdot\rot(g'))\,W_{\varepsilon R,R}(\rot(g),\rho(b)\cdot\rot^2(g'))>0$) si et seulement si $\rho(b)$ (\resp $\rho(c)$) est $(\varepsilon,R)$-presque réalisé par $(\rot(g),\rho(b)\cdot\rot^2(g'))$ (\resp $(\rot^2(g),\rho(a)^{-1}\cdot\rot(g'))$).
On en d\'eduit que $(\rho(a),\rho(b),\rho(c),g,g')$ satisfait \eqref{eqn:condRpresqueparf-realise} avec les signes $+$ si et seulement si $W^{\mathrm{tri}}_{\varepsilon R,R}(g,g',\rho(a)^{-1}g',\rho(b)g')>0$.
\end{proof}

\begin{proof}[D\'emonstration de la proposition~\ref{prop:pant-dans-reseau-general}.\eqref{item:pant-dans-reseau-1}]
D'apr\`es le corollaire~\ref{cor:W>0}, si $R$ est assez grand par rapport \`a~$\varepsilon$, alors pour tous $g,g''\in G$ on peut trouver $\gamma_0,\gamma_1,\gamma_2\in\Gamma$ tels que $W^{\mathrm{tri}}_{\varepsilon,R}(g,\gamma_0g'',\gamma_1g'',\gamma_2g'')>0$.
D'apr\`es le lemme~\ref{lem:triconn->pant}.\eqref{item:triconn->pant+}, la repr\'esentation de $\pi_1(\Pi^+) = \langle a,b^+,c^+ \,|\, c^+ b^+ a = 1\rangle$ dans~$\Gamma$ envoyant $(a,b^+)$ sur $(\gamma_0\gamma_1^{-1},\gamma_2\gamma_0^{-1})$ est $(\varepsilon/R,R)$-presque parfaite.
De m\^eme, les lemmes \ref{lem:w-presque-1} et~\ref{lem:triconn->pant}.\eqref{item:triconn->pant-} permettent de trouver des repr\'esentations $(\varepsilon/R,-R)$-presque parfaites de $\pi_1(\Pi^-)$ dans~$\Gamma$.
\end{proof}

\subsection{Couples triconnect\'es et mesures $\mu_{\varepsilon,\pm R}$ de la proposition~\ref{prop:pant-dans-reseau-mesures}} \label{subsec:Triconn}

Soit $\Pi$ un pantalon de groupe fondamental $\pi_1(\Pi) = \langle a,b,c \,|\, cba = 1\rangle$.
Dans ce paragraphe, nous introduisons l'espace $\Triconn_{\varepsilon,\pm R}$ qui paramètre les données géométriques associées aux représentations $(\varepsilon/R,\pm R)$-presque parfaites de $\pi_1(\Pi)$ dans~$\Gamma$ modulo l'action de~$\Gamma$.
Nous introduisons des mesures $\mu_{\varepsilon,\pm R}$ dont nous montrerons plus loin (proposition~\ref{prop:mu-convient}) qu'elles satisfont la conclusion de la proposition~\ref{prop:pant-dans-reseau-mesures}.

\begin{definition} \label{def:Triconn}
Un \emph{couple triconnect\'e\footnote{Cette terminologie est justifi\'ee par la figure~\ref{fig:Triconn} (paragraphe~\ref{subsubsec:pant-presque-parf-KLM-def}) associ\'ee au lemme~\ref{lem:triconn->geom}.} dans $\Gamma\backslash G$} est un \'el\'ement de
\begin{equation} \label{eqn:Triconn}
\Triconn := \{[(x,y_0,y_1,y_2)]\in\Gamma\backslash G^4 ~|~ y_1,y_2\in\Gamma y_0\},
\end{equation}
o\`u $\Gamma$ agit diagonalement sur~$G^4$ par multiplication \`a gauche.
Pour $\varepsilon,R>0$, on note $\Triconn_{\varepsilon,\pm R}$ l'ensemble des couples triconnect\'es $[(x,y_0,y_1,y_2)]$ tels que $W^{\mathrm{tri}}_{\varepsilon,\pm R}(x,y_0,y_1,y_2)>0$ (\cf \eqref{eqn:W-epsilon-R-tri}).
\end{definition}

Rappelons (remarque~\ref{rem:presque-parf-tri}) que $\Gamma$ agit sur l'ensemble des représentations $(\varepsilon/R,\pm R)$-presque parfaites de $\pi_1(\Pi)$ dans~$\Gamma$ par conjugaison au but, ce qui se traduit par une action de $\Gamma$ sur les données géométriques correspondantes donn\'ee par $x\cdot (\alpha,\beta,\gamma,g,g') = (x\alpha x^{-1}, x\beta x^{-1}, x\gamma x^{-1}, xg, xg')$ pour tout $x\in\Gamma$.
Le lemme~\ref{lem:triconn->pant} se reformule ainsi.

\begin{lemme} \label{lem:triconn->geom}
Soit $\Pi$ un pantalon de groupe fondamental $\pi_1(\Pi) = \langle a,b,c \,|\, cba = 1\rangle$ et soient $\varepsilon,R>0$.
Alors
\begin{enumerate}
  \item l'ensemble des donn\'ees g\'eom\'etriques associ\'ees \`a des repr\'esentations $(\varepsilon/R,R)$-presque parfaites de $\pi_1(\Pi)$ dans~$\Gamma$, modulo l'action de~$\Gamma$, est param\'etr\'e par $\Triconn_{\varepsilon,R}$ via $[(\rho(a),\rho(b),\rho(c),g,g')] \mapsto [(g,g',\rho(a)^{-1}g',\rho(b)g')]$ ;
  \item l'ensemble des donn\'ees g\'eom\'etriques associ\'ees \`a des repr\'esentations $(\varepsilon/R,-R)$-presque parfaites de $\pi_1(\Pi)$ dans~$\Gamma$, modulo l'action de~$\Gamma$, est param\'etr\'e~par $\Triconn_{\varepsilon,-R}$ via $[(\rho(a),\rho(b),\rho(c),g,g')] \mapsto [(g,g'',\rho(a)g'',\rho(ba)g'')]$ o\`u $g'' := \rho(a)^{-1}\circ\rot(g')$.
\end{enumerate}
On peut prendre $\mathrm{Geom}_{\varepsilon,R} := \Triconn_{\varepsilon,R}$ dans la proposition~\ref{prop:pant-dans-reseau-mesures}.
\end{lemme}

Notons $\mathcal{A}_{\varepsilon,\pm R}$ l'ensemble des classes de conjugaison dans~$\Gamma$ d'éléments $\rho(a)$ où $\rho : \pi_1(\Pi)\to\Gamma$ est $(\varepsilon/R,\pm R)$-presque parfaite ; c'est un ensemble fini d'après le corollaire~\ref{cor:presque-parf->generique}.\eqref{item:nb-fini-presque-parf}.
Pour tout $\alpha\in\Gamma$, l'ensemble
\begin{equation} \label{eqn:Triconn-alpha}
\Triconn_{\varepsilon,\pm R}^{\alpha} := \{[(x,y_0,y_1,y_2)]\in\Triconn_{\varepsilon,\pm R} ~|~ y_0=\alpha^{\pm 1}y_1\}
\end{equation}
ne dépend que de la classe de conjugaison $[\alpha]$ de $\alpha$ dans~$\Gamma$.
L'ensemble $\Triconn_{\varepsilon,\pm R}$ est l'union disjointe de ses sous-ensembles $\Triconn_{\varepsilon,\pm R}^{\alpha}$ où $[\alpha]$ parcourt $\mathcal{A}_{\varepsilon,\pm R}$.
Tout \'el\'ement de $\Triconn_{\varepsilon,R}^{\alpha}$ (\resp $\Triconn_{\varepsilon,-R}^{\alpha}$) est par d\'efinition de la forme $[(g,g',\alpha^{-1} g',\beta g')]$ (\resp $[(g,g',\alpha g',\beta\alpha g')]$) o\`u $(g,g')\in G^2$ et $\beta\in\Gamma$, et définit une représentation $(\varepsilon/R,R)$-presque parfaite (\resp $(\varepsilon/R,-R)$-presque parfaite) de $\pi_1(\Pi)$ par $(a,b)\mapsto (\alpha,\beta)$ modulo conjugaison au but par~$\Gamma$.

D'après la remarque~\ref{rem:presque-parf-tri}, l'ensemble des représentations $(\varepsilon/R,\pm R)$-parfaites est invariant par la symétrie naturelle d'ordre trois $\sym$ de $\Hom(\pi_1(\Pi),G)$ donnée par $\sym(\rho)(a,b,c) = (\rho(b),\rho(c),\rho(a))$.
Cette symétrie induit une symétrie d'ordre trois de l'ensemble des classes de conjugaison modulo~$\Gamma$ de représentations $(\varepsilon/R,\pm R)$-parfaites, qui elle-même se relève en une symétrie d'ordre trois de $\Triconn_{\varepsilon,\pm R}$, que nous noterons encore $\sym$ et qui est donnée par
\begin{equation} \label{eqn:tri-Triconn}
\left\{ \begin{array}{cccl}
\sym([(g, g', \alpha^{-1}g', \beta g')]) & \!\!=\!\! & [(x, x', \beta^{-1}x', \alpha^{-1}\beta^{-1}x')] & \text{sur }\Triconn_{\varepsilon,R},\\
\sym([(g, g', \alpha g', \beta\alpha g')]) & \!\!=\!\! & [(y, y', \beta y', \alpha^{-1}y')] & \text{sur }\Triconn_{\varepsilon,-R},
\end{array}\right.\hspace{-0.5cm}
\end{equation}
où $(x,x') := (\rot(g), \beta\circ\rot^2(g'))$ et $(y,y') := (\rot^2(g), \alpha\circ\rot(g'))$.

L'observation suivante est une cons\'equence directe des lemmes \ref{lem:realise-proximal} et~\ref{lem:triconn->geom} et de la d\'efinition~\ref{def:pant-presque-parfait-KLM}.

\begin{remarque} \label{rem:Triconn-rel-compact}
Pour tout $R>0$ assez grand par rapport \`a~$\varepsilon$ et tout $[\alpha]\in\mathcal{A}_{\varepsilon,\pm R}$, l'adhérence de $\Triconn_{\varepsilon,\pm R}^{\alpha}$ dans $\Gamma\backslash G^4$ est compacte ; ainsi, l'adhérence de $\Triconn_{\varepsilon,\pm R}$ dans $\Gamma\backslash G^4$, c'est-à-dire l'image dans $\Gamma\backslash G^4$ du support de $W^{\mathrm{tri}}_{\varepsilon,\pm R}$, est compacte.
\end{remarque}

On définit une mesure $\mu_{\varepsilon,\pm R}$ sur $\Triconn_{\varepsilon,\pm R}$ de la manière suivante.
La mesure de Haar de~$G$ induit une mesure naturelle sur $\Gamma\backslash G^2$.
L'espace $\Triconn$ de \eqref{eqn:Triconn} se projette, en considérant les deux premières coordonnées, sur $\Gamma\backslash G^2$, avec des fibres dénombrables paramétrées par~$\Gamma$.
Soit $\lambda_{\Triconn}$ la mesure localement finie sur $\Triconn$ dont le poussé en avant est la mesure naturelle de $\Gamma\backslash G^2$ et dont la restriction à chaque fibre est la mesure de comptage.
On définit, sur $\Triconn_{\varepsilon,\pm R}$, la mesure
\begin{equation} \label{def:mu-epsilon-R}
\mu_{\varepsilon,\pm R} := W^{\mathrm{tri}}_{\varepsilon,\pm R}\,\lambda_{\Triconn}
\end{equation}
(\cf \eqref{eqn:W-epsilon-R-tri}).
Elle est finie d'apr\`es la remarque~\ref{rem:Triconn-rel-compact}.
On vérifie également qu'elle est invariante par la symétrie d'ordre trois $\sym$ de $\Triconn_{\varepsilon,\pm R}$ de \eqref{eqn:tri-Triconn}.

Soient $\Pi^+$ et~$\Pi^-$ deux pantalons de groupes fondamentaux
$$\pi_1(\Pi^{\pm}) = \langle a^{\pm}, b^{\pm}, c^{\pm} \,|\, c^{\pm} b^{\pm} a^{\pm} = 1\rangle.$$
Pour $\varepsilon'>0$, on dit que deux éléments $\mathtt{T}^+\in\Triconn_{\varepsilon,R}$ et $\mathtt{T}^-\in\Triconn_{\varepsilon,-R}$ sont \emph{$\varepsilon'$-bien recoll\'es} s'ils définissent des représentations $(\varepsilon/R,\pm R)$-presque parfaites de $\pi_1(\Pi^{\pm})$ dans~$G$ qui, après conjugaison par~$\Gamma$, sont $\varepsilon'$-bien recollées le long de $a^+$ et~$a^-$ via des données géométriques associées à $\mathtt{T}^+$ et~$\mathtt{T}^-$, au sens de la définition~\ref{def:bon-recoll-KLM}.

Notre but est d\'esormais d'expliquer le r\'esultat suivant, qui implique les propositions \ref{prop:pant-dans-reseau-mesures} (avec $\mathrm{Geom}_{\varepsilon,R} = \Triconn_{\varepsilon,R}$) et~\ref{prop:pant-dans-reseau-general}.\eqref{item:pant-dans-reseau-2}.

\begin{proposition} \label{prop:mu-convient}
En supposant la condition~(R) v\'erifi\'ee, il existe $C>0$ tel que pour tout $\varepsilon>0$ et tout $R>0$ assez grand par rapport à~$\varepsilon$, les mesures $\sym$-invariantes $\mu_{\varepsilon,\pm R}$ sur $\Triconn_{\varepsilon,\pm R}$ de \eqref{def:mu-epsilon-R} satisfont la conclusion de la proposition~\ref{prop:pant-dans-reseau-mesures} : on a $\mu_{\varepsilon,R}(\Triconn_{\varepsilon,R}) = \mu_{\varepsilon,-R}(\Triconn_{\varepsilon,-R})$ et pour tout sous-ensemble mesurable $A$ de $\Triconn_{\varepsilon,R}$, l'ensemble des \'el\'ements de $\Triconn_{\varepsilon,-R}$ qui sont $(C\varepsilon/R)$-bien recoll\'es \`a au moins un \'el\'ement de~$A$ est de $(\mu_{\varepsilon,-R})$-mesure sup\'erieure ou \'egale \`a $\mu_{\varepsilon,R}(A)$.
\end{proposition}

\subsection{Reformulation en termes de la distance de Lévy--Prokhorov} \label{subsec:reformulation-LP}

Soient $\varepsilon,R>0$.
Comme au paragraphe~\ref{subsec:Triconn} précédent, notons $\mathcal{A}_{\varepsilon,\pm R}$ l'ensemble fini des classes de conjugaison dans~$\Gamma$ d'éléments $\rho(a)$ où $\rho : \pi_1(\Pi^{\pm})\to\Gamma$ est $(\varepsilon/R,\pm R)$-presque parfaite.
Supposons $R$ assez grand par rapport \`a~$\varepsilon$ de sorte que pour tout $[\alpha]\in\mathcal{A}_{\varepsilon,\pm R}$, l'\'el\'ement $\alpha\in\Gamma\subset G$ soit proximal dans $G/P_{\tau}$ (corollaire~\ref{cor:presque-parf->generique}.\eqref{item:presque-parf->generique}), de points fixes attractif $\alpha^{\attract}$ et répulsif~$\alpha^{\repuls}$.
Comme au paragraphe~\ref{subsec:recolle-KLM}, notons $L_{\alpha}$ l'ensemble des \'el\'ements $g\in G$ tels que $g\cdot\underline{\tau}(0,\infty) = (\alpha^{\repuls},\alpha^{\attract})$ ; c'est une classe à gauche du centralisateur $Z_G(\alpha_0)$ de $\alpha_0 = \exp(\mathsf{h})$ dans~$G$ (remarque~\ref{rem:pied-alpha-0-alpha}).
Le centralisateur $Z_{\Gamma}(\alpha)$ de $\alpha$ dans~$\Gamma$ est contenu dans le stabilisateur $\mathrm{stab}_G(\alpha^{\repuls},\alpha^{\attract})$ qui agit sur $L_{\alpha}$ par multiplication à gauche.
Dans ce paragraphe, nous reformulons la proposition~\ref{prop:mu-convient} en termes de mesures sur $Z_{\Gamma}(\alpha)\backslash L_{\alpha}$.

Notons que, comme $\Gamma$ est un r\'eseau cocompact de~$G$ et $\mathsf{h}$ est r\'egulier (cadre~\ref{cadre}), le quotient $Z_{\Gamma}(\alpha)\backslash L_{\alpha}$ est compact (\cf \cite[Th.\,4.2]{wol62}).

\begin{exemple} \label{ex:PSLnC-pied}
Soient $G=\PSL(n,\C)$ et $\tau : \PSL(2,\R)\hookrightarrow G$ le plongement irr\'eductible (\cf exemples \ref{ex:PSL} et~\ref{ex:PSLnC-centralisateur}).
Pour $\Gamma$ sans torsion et $\alpha\in\Gamma$ proximal dans $G/P_{\tau}$, les groupes $Z_G(\alpha)$ et $Z_{\Gamma}(\alpha)$ sont isomorphes respectivement \`a $(\C^*)^{n-1}$ et $\Z^{n-1}$, et $Z_{\Gamma}(\alpha)\backslash L_{\alpha} \simeq Z_{\Gamma}(\alpha)\backslash Z_G(\alpha) \simeq \mathbb{T}^{2(n-1)}$ est un tore compact.
\end{exemple}

On peut voir l'ensemble $\Triconn_{\varepsilon,\pm R}^{\alpha}$ de \eqref{eqn:Triconn-alpha} comme un sous-ensemble de\linebreak $Z_{\Gamma}(\alpha)\backslash G^4$ plut\^ot que de $\Gamma\backslash G^4$ : en effet, si deux \'el\'ements de l'ensemble\linebreak $\{ (x,y_0,y_1,y_2)\in G^4 \,|\, y_1=\alpha^{\pm}y_0,\ y_2\in\Gamma y_0\}$ ont m\^eme image dans $\Gamma\backslash G^4$, alors ils ont d\'ej\`a m\^eme image dans $Z_{\Gamma}(\alpha)\backslash G^4$.
D'apr\`es les lemmes \ref{lem:presque-parf-U-alpha}.\eqref{item:presque-parf-U-alpha} et~\ref{lem:triconn->geom}, si $R$ est assez grand par rapport \`a~$\varepsilon$, alors pour tout $[(g,g',\alpha^{-1}g',\beta g')] \in \Triconn_{\varepsilon,R}^{\alpha}$ (\resp pour tout $[(g,g',\alpha g',\beta\alpha g')] \in \Triconn_{\varepsilon,-R}^{\alpha}$), l'\'el\'ement $g$ appartient au domaine de d\'efinition $\mathcal{U}_{\alpha}$ (\resp $\mathcal{U}_{\alpha^{-1}}$) de l'application \og pied\fg\ $\Psi_{\alpha} : \mathcal{U}_{\alpha}\to L_{\alpha}$ (\resp $\Psi_{\alpha^{-1}} : \mathcal{U}_{\alpha^{-1}}\to L_{\alpha^{-1}}$), \cf d\'efinition~\ref{def:pied}.
L'application $\Psi_{\alpha^{\pm 1}}$ est équivariante par rapport aux actions de $Z_{\Gamma}(\alpha) = Z_{\Gamma}(\alpha^{-1})$ par multiplication à gauche, donc induit une application
\begin{equation} \label{eqn:pied-tri}
\Psi_{\alpha^{\pm 1}}^{\mathrm{tri}} : \Triconn_{\varepsilon,\pm R}^{\alpha} \longrightarrow Z_{\Gamma}(\alpha)\backslash L_{\alpha^{\pm 1}}
\end{equation}
donnée par $\Psi_{\alpha}^{\mathrm{tri}}[(g,g',\alpha^{-1}g',\beta g')] = [\Psi_{\alpha}(g)]$ (\resp $\Psi_{\alpha^{-1}}^{\mathrm{tri}}[(g,g',\alpha g',\beta\alpha g')] = [\Psi_{\alpha^{-1}}(g)]$).
Soit $\mu_{\varepsilon,\pm R}$ la mesure finie sur $\Triconn_{\varepsilon,\pm R}$ de \eqref{def:mu-epsilon-R}.
Elle se restreint en une mesure finie $\mu_{\varepsilon,\pm R}^{\alpha}$ sur $\Triconn_{\varepsilon,\pm R}^{\alpha}$, qui induit sur $Z_{\Gamma}(\alpha)\backslash L_{\alpha^{\pm 1}}$ une mesure finie
\begin{equation} \label{eqn:mesures-m}
m_{\varepsilon,\pm R}^{\alpha} := (\Psi_{\alpha^{\pm 1}}^{\mathrm{tri}})_* \, \mu_{\varepsilon,\pm R}^{\alpha}.
\end{equation}

Rappelons (\cf paragraphe~\ref{subsec:recolle-KLM}) que le flot $(\varphi_t)_{t\in\R}$ du paragraphe~\ref{subsec:rot-inv-phi} préserve~$L_{\alpha}$, alors que l'involution $\inv$ échange $L_{\alpha}$ et~$L_{\alpha^{-1}}$.
Ainsi, $Z_{\Gamma}(\alpha)\backslash L_{\alpha^{-1}} = (\varphi_1\circ\inv)(Z_{\Gamma}(\alpha)\backslash L_{\alpha})$.

La proposition~\ref{prop:mu-convient} se reformule en termes de proximit\'e des mesures $m_{\varepsilon,R}^{\alpha}$ et $(\varphi_1\circ\nolinebreak\inv)_*\,m_{\varepsilon,-R}^{\alpha}$ sur $Z_{\Gamma}(\alpha)\backslash L_{\alpha}$ pour la distance suivante, dite de \emph{L\'evy--Prokhorov}.

\begin{definition} \label{def:LP}
Soient $\mu_1,\mu_2$ deux mesures finies sur un espace m\'etrique~$X$, telles que $\mu_1(X)=\mu_2(X)$.
On pose
$$d_{LP}(\mu_1,\mu_2) := \inf \big\{\delta\geq 0 ~|~ \mu_1(A)\leq\mu_2(\mathcal{V}_{\delta}(A))\ \,\forall A\subset X\big\},$$
o\`u $\mathcal{V}_{\delta}(A)$ d\'esigne le $\delta$-voisinage uniforme de $A$ dans~$X$.
\end{definition}

Ceci d\'efinit une distance $d_{LP}$ sur les mesures sur~$X$.
(Pour v\'erifier la sym\'etrie, notons que si $d_{LP}(\mu_1,\mu_2)\leq\delta$, alors pour tout $A\subset X$ on a
$$\mu_2(A) \leq \mu_2(X) - \mu_2(\mathcal{V}_{\delta}(X\smallsetminus\mathcal{V}_{\delta}(A))) \leq \mu_1(X) - \mu_1(X\smallsetminus\mathcal{V}_{\delta}(A)) = \mu_1(\mathcal{V}_{\delta}(A)),$$
d'o\`u $d_{LP}(\mu_2,\mu_1)\leq\delta$.)
Dans notre cas $X$ sera $Z_{\Gamma}(\alpha)\backslash L_{\alpha}$ muni de la distance induite par~$d_G$ par restriction et passage au quotient, o\`u $\alpha\in\Gamma$ est proximal dans $G/P_{\tau}$.

La proposition~\ref{prop:mu-convient} se reformule ainsi.

\begin{proposition} \label{prop:mesures-LP-proches}
En supposant la condition~(R) v\'erifi\'ee, il existe $C>0$ avec la propri\'et\'e suivante : pour tout $\varepsilon>0$, tout $R>0$ assez grand par rapport \`a~$\varepsilon$ et tout $[\alpha]\in\mathcal{A}_{\varepsilon,R}$, les mesures finies $m_{\varepsilon,R}^{\alpha}$ et $(\varphi_1\circ\inv)_*\,m_{\varepsilon,-R}^{\alpha}$ sur $Z_{\Gamma}(\alpha)\backslash L_{\alpha}$ (\cf \eqref{eqn:mesures-m}) ont m\^eme masse totale et l'on a
\begin{equation} \label{eqn:mes-LP-proches}
d_{LP}\big(m_{\varepsilon,R}^{\alpha}, (\varphi_1\circ\inv)_*\,m_{\varepsilon,-R}^{\alpha}\big) \leq \frac{C\varepsilon}{R}.
\end{equation}
\end{proposition}

\begin{proof}[D\'emonstration du fait que la proposition~\ref{prop:mesures-LP-proches} implique la proposition~\ref{prop:mu-convient}. --- \emph{Rappelons}]
\noindent
(\cf  paragraphe~\ref{subsec:Triconn}) que $\Triconn_{\varepsilon,\pm R}$ est l'union disjointe de ses sous-ensembles $\Triconn_{\varepsilon,\pm R}^{\alpha}$ o\`u $[\alpha]$ parcourt l'ensemble fini $\mathcal{A}_{\varepsilon,\pm R}$.
Comme toute mesure a m\^eme masse totale que ses pouss\'es en avant, on a
$$\left\{ \begin{array}{l}
\mu_{\varepsilon,R}(\Triconn_{\varepsilon,R}) = \sum_{[\alpha]} \mu_{\varepsilon,R}^{\alpha}(\Triconn_{\varepsilon,R}^{\alpha}) = \sum_{[\alpha]} m_{\varepsilon,R}^{\alpha}(Z_{\Gamma}(\alpha)\backslash L_{\alpha}),\\
\mu_{\varepsilon,-R}(\Triconn_{\varepsilon,-R}) = \sum_{[\alpha]} m_{\varepsilon,-R}^{\alpha}(Z_{\Gamma}(\alpha)\backslash L_{\alpha^{-1}}) = \sum_{[\alpha]} ((\varphi_1\circ\inv)_*\,m_{\varepsilon,-R}^{\alpha})(Z_{\Gamma}(\alpha)\backslash L_{\alpha}).
\end{array}\right.$$
Ainsi, le fait que les mesures $m_{\varepsilon,R}^{\alpha}$ et $(\varphi_1\circ\inv)_*\,m_{\varepsilon,-R}^{\alpha}$ sur $Z_{\Gamma}(\alpha)\backslash L_{\alpha}$ aient m\^eme masse totale comme dans la proposition~\ref{prop:mesures-LP-proches} implique l'\'egalit\'e $\mu_{\varepsilon,R}(\Triconn_{\varepsilon,R}) = \mu_{\varepsilon,-R}(\Triconn_{\varepsilon,-R})$ de la proposition~\ref{prop:mu-convient}.

Soit $C>0$ la constante de la proposition~\ref{prop:mesures-LP-proches}.
Pour tout sous-ensemble mesurable $A$ de $\Triconn_{\varepsilon,R}$, l'ensemble des \'el\'ements de $\Triconn_{\varepsilon,-R}$ qui sont $(C\varepsilon/R)$-bien recoll\'es \`a au moins un \'el\'ement de~$A$ est l'union disjointe sur $[\alpha]\in\mathcal{A}_{\varepsilon,R}$ des ensembles
$$B_{\alpha} := (\varphi_1\circ\inv\circ\Psi_{\alpha^{-1}}^{\mathrm{tri}})^{-1}\big(\mathcal{V}_{C\varepsilon/R}\big(\Psi_{\alpha}^{\mathrm{tri}}\big(A\cap\Triconn_{\varepsilon,R}^{\alpha}\big)\big)\big) \subset \Triconn_{\varepsilon,-R}^{\alpha},$$
o\`u $\mathcal{V}_{C\varepsilon/R}(\cdot)$ d\'esigne le $(C\varepsilon/R)$-voisinage uniforme dans $Z_{\Gamma}(\alpha)\backslash L_{\alpha}$.
Par d\'efinition \eqref{eqn:mesures-m} de $m_{\varepsilon,-R}^{\alpha}$, on a
$$\mu_{\varepsilon,-R}^{\alpha}(B_{\alpha}) = \big((\varphi_1\circ\inv)_*\,m_{\varepsilon,-R}^{\alpha}\big)\big(\mathcal{V}_{C\varepsilon/R}\big(\Psi_{\alpha}^{\mathrm{tri}}\big(A\cap\Triconn_{\varepsilon,R}^{\alpha}\big)\big)\big)$$
qui, par la proposition~\ref{prop:mesures-LP-proches}, est sup\'erieur ou \'egal \`a $m_{\varepsilon,R}^{\alpha}( \Psi_{\alpha}^{\mathrm{tri}}(A\cap\nolinebreak\Triconn_{\varepsilon,R}^{\alpha}))$ lorsque $R$ est assez grand par rapport \`a~$\varepsilon$.
Par d\'efinition \eqref{eqn:mesures-m} de $m_{\varepsilon,R}^{\alpha}$, on a donc\linebreak $\mu_{\varepsilon,-R}^{\alpha}(B_{\alpha}) \geq \mu_{\varepsilon,R}^{\alpha}(A\cap\Triconn_{\varepsilon,R}^{\alpha})$.
Ainsi, la $(\mu_{\varepsilon,-R})$-mesure de l'ensemble des \'el\'ements de $\Triconn_{\varepsilon,-R}$ qui sont $(C\varepsilon/R)$-bien recoll\'es \`a au moins un \'el\'ement de~$A$ est égale à $\sum_{[\alpha]} \mu_{\varepsilon,-R}^{\alpha}(B_{\alpha}) \geq \sum_{[\alpha]} \mu_{\varepsilon,R}^{\alpha}(A\cap\Triconn_{\varepsilon,R}^{\alpha}) = \mu_{\varepsilon,R}(A)$.
\end{proof}

\subsection{Utilisation de la condition de retournement~(R)} \label{subsec:utiliser-hyp-retournement}

Nous avons vu au paragraphe \ref{subsec:reformulation-LP} précédent que pour démontrer la proposition~\ref{prop:pant-dans-reseau-mesures}, il suffit de démontrer la proposition~\ref{prop:mesures-LP-proches}.
Expliquons à présent comment la condition~(R) du paragraphe~\ref{subsubsec:cond-R} intervient dans la d\'emonstration de la proposition~\ref{prop:mesures-LP-proches}.

La condition~(R) affirme que $Z_G(\mathsf{h})$ contient, dans sa composante neutre, un \'el\'ement central $j$ d'ordre deux tel qui échange $\underline{\tau}(0,\infty,-1)$ et $\underline{\tau}(0,\infty,1)$.
Notons $\refl : G\to G$ la multiplication \`a droite par~$j$ ; elle commute \`a $\inv$.
Posons
\begin{equation} \label{eqn:def-I}
I := \inv\circ\refl : G \longrightarrow G.
\end{equation}
Les involutions $\refl$ et~$I$ de~$G$ passent au quotient en des involutions de l'espace $G/Z^{\tau}$ des $\tau$-triangles de $G/P_{\tau}$ du paragraphe~\ref{subsec:tau-P1-G/P} :
\begin{itemize}
  \item $\refl$ (\og r\'eflexion\fg) envoie $g\cdot\underline{\tau}(0,\infty,-1)$ sur $g\cdot\underline{\tau}(0,\infty,1)$,
  \item $I$ envoie $g\cdot\underline{\tau}(0,\infty,-1)$ sur $g\cdot\underline{\tau}(\infty,0,-1)$
\end{itemize}
(\cf figure~\ref{fig:refl-I}).
(Par exemple, si $\tau$ est la restriction \`a $\PSL(2,\R)$ d'un plongement $\tau_{\C} : \PSL(2,\C)\hookrightarrow G$ comme \`a la remarque~\ref{rem:cond-R-satisf-ou-pas}.\eqref{item:cond-R-oui}, alors $\refl$ et~$I$ correspondent \`a la multiplication \`a droite respectivement par $\tau_{\C}((\begin{smallmatrix} i & 0\\ 0 & -i\end{smallmatrix}))$ et $\tau_{\C}((\begin{smallmatrix} 0 & i\\ i & 0\end{smallmatrix}))$.)

\begin{figure}[h!]
\centering
\labellist
\small\hair 2pt
\pinlabel {$\underline{\tau}(-1)$} [u] at -30 190
\pinlabel {$\underline{\tau}(0)$} [u] at 165 115
\pinlabel {$\underline{\tau}(\infty)$} [u] at 205 255
\pinlabel {$\underline{\tau}(1)$} [u] at 395 190
\pinlabel {\textcolor{orange}{$\left(\!\tau_{\C}\!\begin{pmatrix} e^{i\theta} & 0\\ 0 & e^{-i\theta}\end{pmatrix}\!\right)_{\theta\in\R/\pi\Z}$}} [u] at 210 318
\endlabellist
\includegraphics[width=6cm]{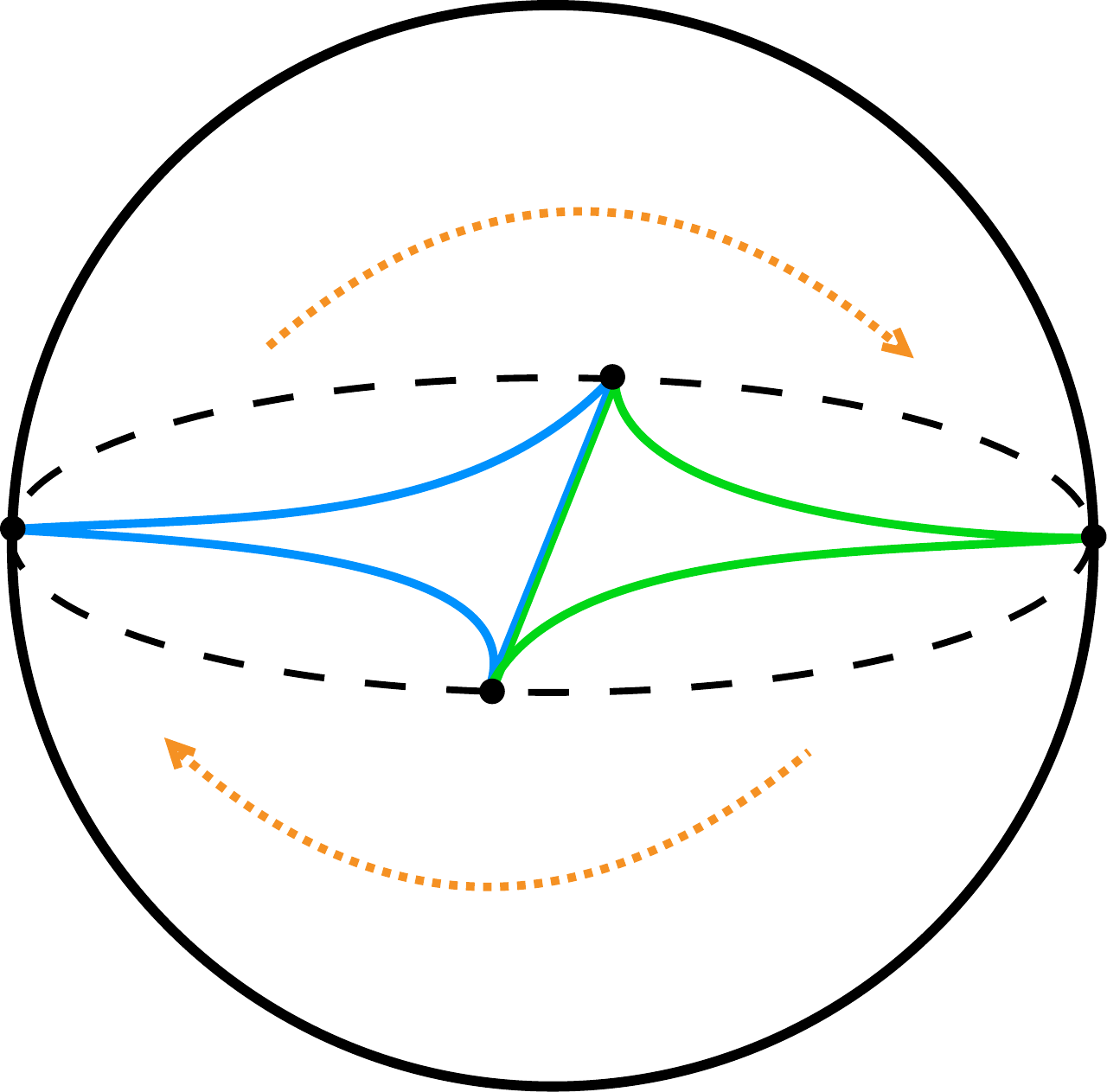}
\vspace{-0.2cm}
\caption{La transformation $\refl$ \'echange les $\tau$-triangles $\underline{\tau}(0,\infty,-1)$ et $\underline{\tau}(0,\infty,1)$ de $G/P_{\tau}$ ; la transformation~$I$ \'echange les $\tau$-triangles $\underline{\tau}(0,\infty,-1)$ et $\underline{\tau}(\infty,0,-1)$ de $G/P_{\tau}$. Ici $G=\PSL(2,\C)$ et $G/P_{\tau} = \partial\HH^3$. La condition de retournement~(R) est satisfaite : le tore compact $\big(\begin{smallmatrix} e^{i\theta} & 0\\ 0 & e^{-i\theta}\end{smallmatrix}\big)_{\theta\in\R/\pi\Z}$ permet de \og retourner contin\^ument\fg\ $\underline{\tau}(0,\infty,-1)$ en $\underline{\tau}(0,\infty,1)$ en fixant $\underline{\tau}(0)$ et $\underline{\tau}(\infty)$.}
\label{fig:refl-I}
\end{figure}

Notons que $\rot$ et $I$ engendrent un groupe de transformations de~$G$ qui est relativement compact, car il préserve l'ensemble $\{\underline{\tau}(0),\underline{\tau}(\infty),\underline{\tau}(-1)\}$ et le sous-groupe $Z^{\tau}$ de~$G$ qui fixe ces trois points est compact.
Quitte \`a remplacer la distance riemannienne $G$-invariante \`a gauche $d_G$ par sa moyenne par le groupe engendré par $\rot$ et~$I$, \emph{on supposera d\'esormais que $d_G$ est invariante par $\rot$ et par~$I$}.
Pour tout élément proximal $\alpha\in G$, l'involution $\refl$ de~$G$ (\resp le flot $(\varphi_t)_{t\in\R}$ du paragraphe~\ref{subsec:rot-inv-phi}) préserve l'ensemble~$L_{\alpha}$, et induit une involution (\resp un flot) sur $\Lambda\backslash L_{\alpha}$ pour tout sous-groupe $\Lambda$ de $Z_G(\alpha)$.

\subsubsection{Existence de l'\'el\'ement~$j$} \label{subsubsec:R-utile-1}

L'existence d'un \'el\'ement $j\in Z_G(\mathsf{h})$ d'ordre deux qui échange $\underline{\tau}(0,\infty,-1)$ et $\underline{\tau}(0,\infty,1)$, d\'efinissant des involutions $\refl : G\to G$ (multiplication \`a droite par~$j$) et $I := \inv\circ\refl : G\to G$ comme ci-dessus, permet d'oublier les mesures $m_{\varepsilon,-R}^{\alpha}$ pour ne plus travailler qu'avec les mesures $m_{\varepsilon,R}^{\alpha}$ de \eqref{eqn:mesures-m}.

\begin{proposition} \label{prop:mesures-I}
En supposant la condition~(R) v\'erifi\'ee, soient $\varepsilon,R>0$.
\begin{enumerate}
  \item\label{item:repr-pmR-parf} Une repr\'esentation $\rho : \pi_1(\Pi)\to G$ est $(\varepsilon,R)$-presque parfaite de donnée géométrique $(\rho(a), \rho(b), \rho(c), g, g')$ si et seulement si $\rho$ est $(\varepsilon,-R)$-presque parfaite de donnée géométrique $(\rho(a), \rho(b), \rho(c), I(g), I(g'))$.
  En particulier, $\mathcal{A}_{\varepsilon,R} = \mathcal{A}_{\varepsilon,-R}$.
  \item\label{item:mesures-I} Pour $R$ assez grand par rapport \`a~$\varepsilon$ comme au paragraphe~\ref{subsec:reformulation-LP} et pour $[\alpha]\in\mathcal{A}_{\varepsilon,R}$, on a $m_{\varepsilon,-R}^{\alpha} = I_* \, m_{\varepsilon,R}^{\alpha}$.
  En particulier, $m_{\varepsilon,R}^{\alpha}$ et $m_{\varepsilon,-R}^{\alpha}$ ont m\^eme masse totale.
\end{enumerate}
\end{proposition}

\begin{proof}
\eqref{item:repr-pmR-parf} Comme $d_G$ est invariante par $I = \inv\circ\refl$, l'application cloche $\chi_{\varepsilon/R} : G\to\R^+$ du paragraphe~\ref{subsec:poids} v\'erifie $\chi_{\varepsilon/R}(I(x)^{-1} I(g)) = \chi_{\varepsilon/R}(x^{-1} g)$ pour tous $x,g\in G$.
De plus, on a $\varphi_R\circ\inv = \inv\circ\varphi_{-R}$ pour tout $R\in\R$ et $I\circ\rot = \rot^2\circ I$, et $\varphi_R$ commute avec $\refl$.
On en d\'eduit aisément, par un changement de variable, que $W_{\varepsilon,R}(I(x),I(y)) = W_{\varepsilon,-R}(x,y)$ pour tous $x,y\in G$, puis que
\begin{equation} \label{eqn:W-I}
W^{\mathrm{tri}}_{\varepsilon,R}\circ I^{\mathrm{tri}} = W^{\mathrm{tri}}_{\varepsilon,-R},
\end{equation}
o\`u $I^{\mathrm{tri}} : \Triconn\to\Triconn$ est l'involution donn\'ee par
\begin{equation} \label{eqn:def-I-tri}
I^{\mathrm{tri}}(x,y_0,y_1,y_2) = \big(I(x),\rot\circ I(y_1),\rot\circ I(y_0),\rot\circ I(y_2)\big).
\end{equation}
On conclut en appliquant le lemme~\ref{lem:triconn->pant}.

\eqref{item:mesures-I} D'après \eqref{def:mu-epsilon-R} et \eqref{eqn:W-I}, on a $I^{\mathrm{tri}}_*\mu_{\varepsilon,R}^{\alpha} = \mu_{\varepsilon,-R}^{\alpha}$.
D'autre part, on a $I(L_{\alpha}) = L_{\alpha^{-1}}$ et, comme $\Psi_{\alpha}$ est une projection orthogonale (d\'efinition~\ref{def:pied}) pour la m\'etrique riemannienne associ\'ee \`a~$d_G$, qui est par hypoth\`ese invariante par~$I$, on a $I(\mathcal{U}_{\alpha}) = \mathcal{U}_{\alpha^{-1}}$ et $I\circ\Psi_{\alpha} = \Psi_{\alpha^{-1}}\circ I : \mathcal{U}_{\alpha}\to L_{\alpha^{-1}}$.
On en d\'eduit $I\circ\Psi_{\alpha}^{\mathrm{tri}} = \Psi_{\alpha^{-1}}^{\mathrm{tri}}\circ I^{\mathrm{tri}}$, puis
$$m_{\varepsilon,-R}^{\alpha} = (\Psi_{\alpha^{-1}}^{\mathrm{tri}})_* \mu_{\varepsilon,-R}^{\alpha} = (\Psi_{\alpha^{-1}}^{\mathrm{tri}})_* I^{\mathrm{tri}}_* \mu_{\varepsilon,R}^{\alpha} = I_* (\Psi_{\alpha}^{\mathrm{tri}})_* \mu_{\varepsilon,R}^{\alpha} = I_* \, m_{\varepsilon,R}^{\alpha}. \qedhere$$
\end{proof}

Pour d\'emontrer la proposition~\ref{prop:mesures-LP-proches}, il suffit donc de montrer que la mesure $m_{\varepsilon,R}^{\alpha}$ sur $Z_{\Gamma}(\alpha)\backslash L_{\alpha}$ est proche, pour la distance de L\'evy--Prokhorov, de son pouss\'e en avant par $\varphi_1\circ\inv\circ I = \varphi_1\circ\refl$.

\subsubsection{Tores compacts} \label{subsubsec:R-utile-2}

La condition~(R) requiert que l'\'el\'ement $j$ soit central dans $Z_G(\mathsf{h})$ et appartienne \`a la composante neutre de $Z_G(\mathsf{h})$.
Ceci permet d'obtenir une action d'un certain tore compact $\mathbb{S}_{\alpha}$, ce qui est utile de manière générale pour contrôler la distance de Lévy--Prokhorov (\cf lemme~\ref{lem:LP-tore} ci-dessous).

\begin{lemme} \label{lem:refl-C-alpha}
Il existe un entier $N\geq 1$, ne dépendant que de~$G$, avec la propriété suivante : en supposant la condition~(R) v\'erifi\'ee, pour tout réseau cocompact irréductible sans torsion $\Gamma$ de~$G$ et tout élément proximal $\alpha\in\Gamma$, il existe un sous-groupe $\Lambda_{\alpha}$ d'indice $\leq N$ de $Z_{\Gamma}(\alpha)$ et un tore compact $\mathbb{S}_{\alpha}$ de $Z_G(\Lambda_{\alpha})_0$ tels que l'involution $\refl : L_{\alpha}\to L_{\alpha}$ corresponde \`a la multiplication \`a gauche par un \'el\'ement de~$\mathbb{S}_{\alpha}$.
\end{lemme}

Notons que $Z_G(\Lambda_{\alpha})_0$ est bien contenu dans le groupe $\mathrm{stab}_G(\alpha^{\repuls},\alpha^{\attract})$ qui agit \`a gauche sur~$L_{\alpha}$.
En effet, le sous-groupe d'indice fini $\Lambda_{\alpha}$ de $Z_{\Gamma}(\alpha)$ contient une puissance non nulle de~$\alpha$.

\begin{proof}
Soient $\Gamma$ un réseau cocompact irréductible sans torsion de~$G$ et $\alpha\in\Gamma$ un élément proximal.
Comme \`a la remarque~\ref{rem:pied-alpha-0-alpha}, soit $g_{\alpha}\in G$ envoyant $(\underline{\tau}(0),\underline{\tau}(\infty))$ sur $(\alpha^{\repuls},\alpha^{\attract})$.
On a $L_{\alpha} = g_{\alpha} Z_G(\mathsf{h})$ et $Z_G(\alpha) \subset g_{\alpha} Z_G(\mathsf{h}) g_{\alpha}^{-1}$, où $Z_G(\mathsf{h}) = A Z_K(\mathsf{h})$.
Comme $\Gamma$ est sans torsion, la projection de $Z_{\Gamma}(\alpha)$ sur $g_{\alpha} A g_{\alpha}^{-1}$ est injective, donc $Z_{\Gamma}(\alpha)$ est abélien.
La projection de $Z_{\Gamma}(\alpha)$ sur $g_{\alpha} Z_K(\mathsf{h}) g_{\alpha}^{-1}$ est contenue dans un sous-groupe abélien maximal $\mathbb{S}'_{\alpha}$ de $g_{\alpha} Z_K(\mathsf{h}) g_{\alpha}^{-1}$.
Or, il existe $N\geq 1$, ne dépendant que de~$G$, tel que $\mathbb{S}'_{\alpha}$ admette un sous-groupe d'indice $\leq N$ qui est un tore maximal $\mathbb{S}_{\alpha}$ de $g_{\alpha} Z_K(\mathsf{h})_0 g_{\alpha}^{-1}$ (\cf \cite{mil64}).
On note $\Lambda_{\alpha}$ l'intersection de $Z_{\Gamma}(\alpha)$ avec l'image réciproque de $\mathbb{S}_{\alpha}$ par la projection sur $g_{\alpha} Z_K(\mathsf{h}) g_{\alpha}^{-1}$ : c'est un sous-groupe d'indice $\leq N$ de $Z_{\Gamma}(\alpha)$.
Par construction, $Z_G(\Lambda_{\alpha})_0$ contient le tore maximal $\mathbb{S}_{\alpha}$ de $g_{\alpha} Z_K(\mathsf{h}) g_{\alpha}^{-1}$.

L'involution $\refl : L_{\alpha} \to L_{\alpha}$ correspond \`a la multiplication \`a droite par $j\in G$, qui appartient par d\'efinition (\cf condition~(R), paragraphe~\ref{subsubsec:cond-R}) au centre de $Z_G(\mathsf{h})$, et même (\cf paragraphe~\ref{subsubsec:cond-R-ex}) au centre de $Z_K(\mathsf{h})_0$.
Elle correspond donc à l'action de $j_{\alpha} := g_{\alpha}jg_{\alpha}^{-1}$ sur $L_{\alpha}$ par multiplication à gauche.
L'élément $j_{\alpha}$ appartient au centre de $g_{\alpha} Z_K(\mathsf{h})_0 g_{\alpha}^{-1}$, donc au tore maximal $\mathbb{S}_{\alpha}$ (\cf \cite[Cor.\,IV.4.47]{kna02}).
\end{proof}

La distance de Lévy--Prokhorov se comporte bien par revêtement fini de degré borné (\cf \cite[\S\,18]{klm18}), ce qui implique le fait suivant.

\begin{remarque} \label{rem:mesures-LP-proches-Lambda-alpha}
Pour que les mesures finies $m_{\varepsilon,R}^{\alpha}$ et $(\varphi_1\circ\inv)_*\,m_{\varepsilon,-R}^{\alpha}$ sur $Z_{\Gamma}(\alpha)\backslash L_{\alpha}$ satisfassent la conclusion de la proposition~\ref{prop:mesures-LP-proches}, il suffit que leurs relevés à $\Lambda_{\alpha}\backslash L_{\alpha}$ la satisfassent, où $\Lambda_{\alpha}$ est donné par le lemme~\ref{lem:refl-C-alpha}.
(Par définition, le relevé de $m_{\varepsilon,R}^{\alpha}$ est la mesure sur $\Lambda_{\alpha}\backslash L_{\alpha}$ dont le poussé en avant par le revêtement fini $\Lambda_{\alpha}\backslash L_{\alpha} \to Z_{\Gamma}(\alpha)\backslash L_{\alpha}$ est $m_{\varepsilon,R}^{\alpha}$ et dont la restriction à chaque fibre est la mesure de probabilité uniforme.)
\end{remarque}

Dans la suite, on travaille avec $\Lambda_{\alpha}$ plutôt que $Z_{\Gamma}(\alpha)$ afin de pouvoir utiliser le tore $\mathbb{S}_{\alpha}$ du lemme~\ref{lem:refl-C-alpha}.
Le lecteur pourra penser en première approximation $\Lambda_{\alpha} = Z_{\Gamma}(\alpha)$, motivé notamment par l'exemple suivant.

\begin{exemple} \label{ex:PSLnC-C-alpha}
Soient $G=\PSL(n,\C)$ et $\tau : \PSL(2,\R)\hookrightarrow G$ le plongement irr\'eductible (\cf exemples \ref{ex:PSL} et~\ref{ex:PSLnC-centralisateur}).
Pour $\Gamma$ sans torsion et $\alpha\in\Gamma$ proximal, on peut prendre $\Lambda_{\alpha} = Z_{\Gamma}(\alpha)$ et $\mathbb{S}_{\alpha} = g_{\alpha} Z_K(\mathsf{h}) g_{\alpha}$ où $g_{\alpha}\in G$ envoie $(\underline{\tau}(0),\underline{\tau}(\infty))$ sur $(\alpha^{\repuls},\alpha^{\attract})$.
On a $C_{\alpha}^{\Gamma} := Z_G(\Lambda_{\alpha})_0 = Z_G(\alpha)$ et $C_{\alpha}^{\Gamma}/(C_{\alpha}^{\Gamma} \cap \Lambda_{\alpha}) \simeq \mathbb{T}^{2(n-1)}$ (\cf exemple~\ref{ex:PSLnC-pied}).
\end{exemple}

\subsubsection{Stratégie de démonstration de la proposition~\ref{prop:mesures-LP-proches}}

Prenons $R$ assez grand par rapport \`a~$\varepsilon$ comme au paragraphe~\ref{subsec:reformulation-LP}.
Pour $[\alpha]\in\mathcal{A}_{\varepsilon,R}$, soit $\Lambda_{\alpha}$ le sous-groupe d'indice fini de $Z_{\Gamma}(\alpha)$ donn\'e par le lemme~\ref{lem:refl-C-alpha}.
Le groupe $C_{\alpha}^{\Gamma} := Z_G(\Lambda_{\alpha})_0$ agit sur $X := \Lambda_{\alpha}\backslash L_{\alpha}$ par multiplication \`a gauche.
Cette action se factorise en une action libre du groupe quotient $C_{\alpha}^{\Gamma}/(C_{\alpha}^{\Gamma}\cap \Lambda_{\alpha})$.
Pour démontrer la proposition~\ref{prop:mesures-LP-proches}, l'idée est de~:
\begin{enumerate}
  \item\label{item:mesures-LP-proches-etape-1} trouver une mesure finie $n_{\varepsilon,R}^{\alpha}$ sur $Z_{\Gamma}(\alpha)\backslash L_{\alpha}$, invariante par $Z_G(Z_{\Gamma}(\alpha))$, telle que $m_{\varepsilon,R}^{\alpha}$ soit absolument continue par rapport à~$n_{\varepsilon,R}^{\alpha}$, de d\'eriv\'ee de Radon--Nikodym proche de~$1$ : c'est l'objet de la proposition~\ref{prop:mesures-proches-I}, qui utilise le mélange (fait~\ref{fait:melange}) via le lemme~\ref{lem:w-presque-1} ; en relevant $n_{\varepsilon,R}^{\alpha}$, on obtient une mesure finie sur $X = \Lambda_{\alpha}\backslash L_{\alpha}$, invariante par $C_{\alpha}^{\Gamma} = Z_G(\Lambda_{\alpha})_0$, par rapport à laquelle le relevé de $m_{\varepsilon,R}^{\alpha}$ est absolument continu, de d\'eriv\'ee de Radon--Nikodym proche de~$1$ ;
  \item\label{item:mesures-LP-proches-etape-2} observer que pour un bon sous-groupe à un paramètre $(\alpha_t)_{t\in\R}\subset C_{\alpha}^{\Gamma}$ contenant~$\alpha$, l'élément $\varphi_1$ agit sur~$X$ \og presque\fg\ comme~$\alpha_1$ (proposition~\ref{prop:alpha-t-presque-conj}) ; considérer alors le tore compact $\mathbb{T}_{\alpha}$ de $C_{\alpha}^{\Gamma}/(C_{\alpha}^{\Gamma}\cap\Lambda_{\alpha})$ engendré par les images de $(\alpha_t)_{t\in\R}$ et du tore $\mathbb{S}_{\alpha}$ du lemme~\ref{lem:refl-C-alpha}, et montrer en utilisant \eqref{item:mesures-LP-proches-etape-1} et un argument général (lemme~\ref{lem:LP-tore}) que la mesure $m_{\varepsilon,R}$ relevée à $X = \Lambda_{\alpha}\backslash L_{\alpha}$ est proche pour la distance de L\'evy--Prokhorov de sa moyenne par~$\mathbb{T}_{\alpha}$ ; en déduire, en utilisant la remarque~\ref{rem:mesures-LP-proches-Lambda-alpha}, que $m_{\varepsilon,R}$ est proche pour la distance de L\'evy--Prokhorov de la mesure $(\varphi_1\circ\refl)_* m_{\varepsilon,R}^{\alpha}$, qui est égale à $(\varphi_1\circ\inv)_*\,m_{\varepsilon,-R}^{\alpha}$ d'après la proposition~\ref{prop:mesures-I}.
\end{enumerate}

\subsection{Mesures proches au sens de la d\'eriv\'ee de Radon--Nikodym} \label{subsec:biconn}

La première étape de la démonstration de la proposition~\ref{prop:mesures-LP-proches} consiste à établir le résultat suivant, qui utilise le mélange (fait~\ref{fait:melange}) via le lemme~\ref{lem:w-presque-1}.

\begin{proposition} \label{prop:mesures-proches-I}
En supposant la condition~(R) v\'erifi\'ee, pour tout $\varepsilon>0$, tout $R>0$ assez grand par rapport \`a~$\varepsilon$ et tout $[\alpha]\in\mathcal{A}_{\varepsilon,R}$, il existe une mesure finie $n_{\varepsilon,R}^{\alpha}$ sur $Z_{\Gamma}(\alpha)\backslash L_{\alpha}$, invariante par $Z_G(Z_{\Gamma}(\alpha))$, telle que la mesure $m_{\varepsilon,R}^{\alpha}$ de \eqref{eqn:mesures-m} soit absolument continue par rapport \`a $n_{\varepsilon,R}^{\alpha}$ et que sa d\'eriv\'ee de Radon--Nikodym v\'erifie
  $$\bigg\Vert\frac{\dd m_{\varepsilon,R}^{\alpha}}{\dd n_{\varepsilon,R}^{\alpha}} - 1\bigg\Vert_{\infty} \leq \frac{\varepsilon}{R^2}.$$
\end{proposition}

La proposition~\ref{prop:mesures-proches-I} implique que, sous la condition~(R), les mesures $m_{\varepsilon,R}^{\alpha}$ et $(\varphi_1\circ\nolinebreak\inv)_*\,m_{\varepsilon,-R}^{\alpha}$ sont absolument continues l'une par rapport \`a l'autre, avec une d\'eriv\'ee de Radon--Nikodym proche de~$1$.
En effet, d'apr\`es la proposition~\ref{prop:mesures-I} on a $(\varphi_1\circ\inv)_*\,m_{\varepsilon,-R}^{\alpha} = (\varphi_1\circ\refl)_*\,m_{\varepsilon,R}^{\alpha}$.
En raisonnant comme à la fin de la démonstration du lemme~\ref{lem:refl-C-alpha}, on voit que l'action de $\varphi_1\circ\refl$ sur $Z_{\Gamma}(\alpha)\backslash L_{\alpha}$ correspond à la multiplication à gauche par un élément de $Z_G(Z_{\Gamma}(\alpha))$.
La mesure $(\varphi_1\circ\inv)_*\,m_{\varepsilon,-R}^{\alpha}$ est donc elle aussi absolument continue par rapport \`a $n_{\varepsilon,R}^{\alpha}$ et sa d\'eriv\'ee de Radon--Nikodym v\'erifie $\Vert\dd((\varphi_1\circ\inv)_*\,m_{\varepsilon,-R}^{\alpha})/\dd n_{\varepsilon,R}^{\alpha} - 1\Vert_{\infty} \leq \varepsilon/R^2$.
Pour $\varepsilon/R^2\leq 1/2$, on en d\'eduit que $(\varphi_1\circ\refl)_*\,m_{\varepsilon,R}^{\alpha}$ est absolument continue par rapport \`a $m_{\varepsilon,R}^{\alpha}$ et que
$$\bigg\Vert\frac{\dd((\varphi_1\circ\refl)_*\,m_{\varepsilon,R}^{\alpha})}{\dd m_{\varepsilon,R}^{\alpha}} - 1\bigg\Vert_{\infty} \leq \frac{4\varepsilon}{R^2}.$$

Afin de construire une mesure $n_{\varepsilon,R}^{\alpha}$ vérifiant la conclusion de la proposition~\ref{prop:mesures-proches-I}, pour tous $\varepsilon,R>0$, on d\'efinit une fonction $W^{\mathrm{bi}}_{\varepsilon,R} : G^3\to\R$ par
$$W^{\mathrm{bi}}_{\varepsilon,R}(x,y_0,y_1) = W_{\varepsilon,R}(x,y_0) \, W_{\varepsilon,R}(\rot^2(x),\rot(y_1)),$$
o\`u $x,y_0,y_1\in G$.
Elle est invariante par l'action diagonale de $G$ par multiplication \`a gauche.
On appelle \emph{couple biconnect\'e dans $\Gamma\backslash G$} tout \'el\'ement de
$$\Biconn := \{[(x,y_0,y_1)]\in\Gamma\backslash G^3 ~|~ y_1\in\Gamma y_0\},$$
o\`u $\Gamma$ agit diagonalement sur~$G^3$ par multiplication \`a gauche.
%
%
Pour $\varepsilon,R>0$ et $\alpha\in\Gamma$, on pose
$$\Biconn_{\varepsilon,R}^{\alpha} := \{[(x,y_0,y_1)]\in\Biconn ~|~ W^{\mathrm{bi}}_{\varepsilon,R}(x,y_0,y_1) > 0 \ \mathrm{et}\ y_1=\alpha^{-1}y_0\}.$$
Comme pour les couples triconnect\'es, on peut voir $\Biconn_{\varepsilon,R}^{\alpha}$ comme un sous-ensemble de $Z_{\Gamma}(\alpha)\backslash G^3$.
En particulier, $Z_G(Z_{\Gamma}(\alpha))$ agit sur $\Biconn_{\varepsilon,R}^{\alpha}$ par multiplication \`a gauche.
On a une projection naturelle $\pi : \Triconn\to\Biconn$, qui envoie $\Triconn_{\varepsilon,R}^{\alpha}$ sur $\Biconn_{\varepsilon,R}^{\alpha}$ de mani\`ere $Z_G(Z_{\Gamma}(\alpha))$-\'equivariante.
Pour $R$ assez grand par rapport \`a~$\varepsilon$, l'application \og pied\fg\ $\Psi_{\alpha} : \mathcal{U}_{\alpha}\to L_{\alpha}$ de la d\'efinition~\ref{def:pied} induit (par analogie avec $\Psi_{\alpha}^{\mathrm{tri}}$, \cf \eqref{eqn:pied-tri}) une application $\Psi_{\alpha}^{\mathrm{bi}} : \Biconn_{\varepsilon,R}^{\alpha}\to Z_{\Gamma}(\alpha)\backslash L_{\alpha}$ qui \`a $[(g,g',\alpha^{-1}g')]$ associe $[\Psi_{\alpha}(g)]$.
Cette application $\Psi_{\alpha}^{\mathrm{bi}}$ est $Z_G(Z_{\Gamma}(\alpha))$-\'equivariante car $\Psi_{\alpha}$ est $Z_G(\alpha)$-\'equivariante (remarque~\ref{rem:pied-alpha-0-alpha}).
Le diagramme suivant commute :
$$\xymatrix{\Triconn_{\varepsilon,R}^{\alpha} \ar[r]^{\pi} \ar@/_2pc/[rr]_{\Psi_{\alpha}^{\mathrm{tri}}} & \Biconn_{\varepsilon,R}^{\alpha} \ar[r]^{\Psi_{\alpha}^{\mathrm{bi}}} & Z_{\Gamma}(\alpha)\backslash L_{\alpha}}.$$

\begin{proof}[D\'emonstration de la proposition~\ref{prop:mesures-proches-I}]
Comme $\Triconn$, l'espace $\Biconn$ se projette, en considérant les deux premières coordonnées, sur $\Gamma\backslash G^2$, avec des fibres dénombrables paramétrées par~$\Gamma$.
Soit $\lambda_{\Biconn}$ la mesure localement finie sur $\Biconn$ dont le poussé en avant est la mesure naturelle de $\Gamma\backslash G^2$ (induite par la mesure de Haar de~$G$) et dont la restriction à chaque fibre est la mesure de comptage.
On définit, sur $\Biconn_{\varepsilon,R}$, la mesure $\nu_{\varepsilon,R} := W^{\mathrm{bi}}_{\varepsilon,R}\,\lambda_{\Biconn}$.
Soit $\nu_{\varepsilon,R}^{\alpha}$ sa restriction \`a $\Biconn_{\varepsilon,R}^{\alpha}$ pour $\alpha\in\Gamma$.
Montrons que la mesure $n_{\varepsilon,R}^{\alpha} := (\Psi_{\alpha}^{\mathrm{bi}})_* \nu_{\varepsilon,R}^{\alpha}$ convient.
Comme la fonction $W^{\mathrm{bi}}_{\varepsilon,R} : G^3\to\R$ est invariante par l'action diagonale de~$G$, la mesure $\nu_{\varepsilon,R}^{\alpha}$ est invariante par $Z_G(Z_{\Gamma}(\alpha))$, donc $n_{\varepsilon,R}^{\alpha} = (\Psi_{\alpha}^{\mathrm{bi}})_* \nu_{\varepsilon,R}^{\alpha}$ aussi par $Z_G(Z_{\Gamma}(\alpha))$-\'equivariance de $\Psi_{\alpha}^{\mathrm{bi}}$.
Soit $f_{\varepsilon,R}^{\alpha} : \Biconn_{\varepsilon,R}^{\alpha}\to\R$ la fonction lisse donn\'ee par
$$f_{\varepsilon,R}^{\alpha}([g,g',\alpha^{-1}g']) = w_{\varepsilon,R}([\rot(g)],[\rot^2(g')]) = \sum_{\beta\in\Gamma} W_{\varepsilon,R}(\rot(g),\rot^2\circ\beta(g'))$$
(\cf \eqref{eqn:w-eps-R}).
On a $\pi_* \mu_{\varepsilon,R}^{\alpha} = f_{\varepsilon,R}^{\alpha}\,\nu_{\varepsilon,R}^{\alpha}$, o\`u $\mu_{\varepsilon,R}^{\alpha}$ est la restriction \`a $\Triconn_{\varepsilon,R}^{\alpha}$ de la mesure $\mu_{\varepsilon,R}$ de \eqref{def:mu-epsilon-R}, et pour $R$ assez grand par rapport \`a~$\varepsilon$ on a $\Vert f_{\varepsilon,R}^{\alpha} -\nolinebreak 1\Vert_{\infty} \leq \varepsilon/R^2$ par le lemme~\ref{lem:w-presque-1}.
La mesure $m_{\varepsilon,R}^{\alpha} = (\Psi_{\alpha}^{\mathrm{tri}})_*\,\mu_{\varepsilon,R}^{\alpha} = (\Psi_{\alpha}^{\mathrm{bi}})_*\,(f_{\varepsilon,R}^{\alpha}\,\nu_{\varepsilon,R}^{\alpha})$ est absolument continue par rapport \`a $n_{\varepsilon,R}^{\alpha} = (\Psi_{\alpha}^{\mathrm{bi}})_*\,\nu_{\varepsilon,R}^{\alpha}$, et sa d\'eriv\'ee de Radon--Nikodym $h_{\varepsilon,R}^{\alpha} = \frac{\dd m_{\varepsilon,R}^{\alpha}}{\dd n_{\varepsilon,R}^{\alpha}}$ v\'erifie $\Vert h_{\varepsilon,R}^{\alpha} -\nolinebreak 1\Vert_{\infty} \leq \Vert f_{\varepsilon,R}^{\alpha} - 1\Vert_{\infty}$.
En effet, pour toute fonction $\psi\in L^1(Z_{\Gamma}(\alpha)\backslash L_{\alpha},\R)$, on a
\begin{align*}
& \left| \int_{Z_{\Gamma}(\alpha)\backslash L_{\alpha}} \psi \, (h_{\varepsilon,R}^{\alpha} - 1) \, \dd n_{\varepsilon,R}^{\alpha} \right| = \left| \int_{\Biconn_{\varepsilon,R}^{\alpha}} (\psi\circ\Psi_{\alpha}^{\mathrm{bi}}) \, (f_{\varepsilon,R}^{\alpha} - 1) \, \dd\nu_{\varepsilon,R}^{\alpha} \right|\\
& \leq \Vert f_{\varepsilon,R}^{\alpha} - 1\Vert_{\infty} \, \left| \int_{\Biconn_{\varepsilon,R}^{\alpha}} (\psi\circ\Psi_{\alpha}^{\mathrm{bi}}) \, \dd\nu_{\varepsilon,R}^{\alpha} \right| \ =\ \Vert f_{\varepsilon,R}^{\alpha} - 1\Vert_{\infty} \, \left| \int_{Z_{\Gamma}(\alpha)\backslash L_{\alpha}} \psi \, \dd n_{\varepsilon,R}^{\alpha} \right|. \qedhere
\end{align*}
\end{proof}

\subsection{Mesures proches au sens de la distance de L\'evy--Prokhorov} \label{subsec:dem-prop-mesures-LP-proches}

On souhaite \`a pr\'esent d\'eduire la proposition~\ref{prop:mesures-LP-proches} de la proposition~\ref{prop:mesures-proches-I}.

\subsubsection{Premier ingrédient}

On utilise les observations générales suivantes sur la distance de Lévy--Prokhorov (d\'efinition~\ref{def:LP}).
La démonstration du lemme~\ref{lem:LP-tore} est esquissée au paragraphe~\ref{subsec:dem-LP-tore} ci-dessous.

\begin{remarques} \label{rem:LP-bouge-peu}
Soient $(X,d)$ un espace m\'etrique, $\mu,\nu$ deux mesures finies sur~$X$ et $f,g : X\to X$ deux applications mesurables.
\begin{enumerate}
  \item\label{item:LP-isom} Si $f$ est une isom\'etrie bijective de $(X,d)$, alors $d_{LP}(f_*\mu,f_*\nu) = d_{LP}(\mu,\nu)$.
  \item\label{item:LP-f-g} En g\'en\'eral, $d_{LP}(f_*\mu,g_*\mu) \leq \sup_{x\in X} d(f(x),g(x))$.
  En effet, supposons $\delta :=\linebreak \sup_{x\in X} d(f(x),g(x)) < +\infty$.
  Pour tout $A\subset X$ et tout $x\in X$, si $f(x)\in A$, alors $g(x)\in\mathcal{V}_{\delta}(A)$, d'o\`u $(f_*\mu)(A) = \mu(f^{-1}(A)) \leq \mu(g^{-1}(\mathcal{V}_{\delta}(A))) = (g_*\mu)(\mathcal{V}_{\delta}(A))$.
\end{enumerate}
\end{remarques}

\begin{lemme} \label{lem:LP-tore}
Pour tout $k\geq 1$, il existe une constante $C_k>0$ avec la propriété suivante : pour toute variété riemannienne $X$ munie d'une action par isom\'etries d'un tore compact $\mathbb{T}$ de dimension~$k$, pr\'eservant une mesure $n$ sur~$X$, pour tout $\delta>0$ et toute fonction $h : X\to\R_+^*$, si la moyenne $\overline{h} := \int_{g\in\mathbb{T}} h\circ g \, \dd g$ de~$h$ pour la mesure de probabilité de Haar de~$\mathbb{T}$ vérifie $e^{-\delta}\,\overline{h} \leq h \leq e^{\delta}\,\overline{h}$, alors
$$d_{LP}(hn,\overline{h}n) \leq \delta \cdot C_k \sup_{x\in X} \mathrm{diam}(\mathbb{T}\cdot x).$$
\end{lemme}

\subsubsection{Second ingrédient}

Rappelons que l'on a suppos\'e $\Gamma$ sans torsion.
Pour $[\alpha]\in\nolinebreak\mathcal{A}_{\varepsilon,R}$, soit $\Lambda_{\alpha}$ le sous-groupe d'indice fini de $Z_{\Gamma}(\alpha)$ du lemme~\ref{lem:refl-C-alpha}, et soit $C_{\alpha}^{\Gamma} = Z_G(\Lambda_{\alpha})_0$.
Soit $(\varphi_t)_{t\in\R}$ le flot sur~$G$ du paragraphe~\ref{subsec:rot-inv-phi}.
On utilise le résultat suivant, qui provient d'un raffinement du lemme~\ref{lem:realise-proximal}, et dont la démonstration est esquissée au paragraphe~\ref{subsec:dem-alpha-presque-conj-exp-2Rh} ci-dessous.

\begin{proposition} \label{prop:alpha-t-presque-conj}
Il existe $C'>0$ tel que pour tout $\varepsilon>0$, tout $R>0$ assez grand par rapport \`a~$\varepsilon$ et tout $[\alpha]\in\mathcal{A}_{\varepsilon,R}$, il existe un sous-groupe à un paramètre $(\alpha_t)_{t\in\R}$ du centre de $C_{\alpha}^{\Gamma}$ tel que $\alpha_{2R} = \alpha$ et $d_G(\varphi_t(x),\alpha_tx) \leq C' (\varepsilon/R+e^{-R})$ pour tous $x\in L_{\alpha}$ et $t\in [0,2R]$.
\end{proposition}

\subsubsection{Esquisse de démonstration de la proposition~\ref{prop:mesures-LP-proches}}

Soit $\varepsilon>0$, soit $R>0$ assez grand par rapport \`a~$\varepsilon$ comme au paragraphe~\ref{subsec:reformulation-LP}, et soit $[\alpha]\in\mathcal{A}_{\varepsilon,R}$.
On considère le quotient $X = \Lambda_{\alpha}\backslash L_{\alpha}$ de $L_{\alpha}$ par $\Lambda_{\alpha}$, muni de la m\'etrique riemannienne induite par celle de~$G$ par restriction \`a~$L_{\alpha}$ et passage au quotient.
Le groupe $C_{\alpha}^{\Gamma} = Z_G(\Lambda_{\alpha})_0$ agit sur~$X$ par multiplication à gauche, et cette action se factorise en une action du groupe quotient $C_{\alpha}^{\Gamma}/(C_{\alpha}^{\Gamma}\cap\Lambda_{\alpha})$.

Soit $\mathbb{T}_{\alpha} \subset C_{\alpha}^{\Gamma}/(C_{\alpha}^{\Gamma}\cap\Lambda_{\alpha})$ le tore compact engendré par les images dans $C_{\alpha}^{\Gamma}/(C_{\alpha}^{\Gamma}\cap\Lambda_{\alpha})$ du tore compact $\mathbb{S}_{\alpha} \subset C_{\alpha}^{\Gamma}$ du lemme~\ref{lem:refl-C-alpha} et du sous-groupe \`a un param\`etre $(\alpha_t)_{t\in\R}$ de la proposition~\ref{prop:alpha-t-presque-conj}.
On vérifie qu'il existe $C''>0$, indépendant de $\varepsilon$ et~$R$, tel que le diamètre des orbites de $\mathbb{T}_{\alpha}$ dans~$X$ est borné par $C''R$ d\`es que $R$ est assez grand par rapport \`a~$\varepsilon$.

Soit $\hat{m}_{\varepsilon,R}^{\alpha}$ le relevé à $X = \Lambda_{\alpha}\backslash L_{\alpha}$ de la mesure $m_{\varepsilon,R}^{\alpha}$ sur $Z_{\Gamma}(\alpha)\backslash L_{\alpha}$.
La proposition~\ref{prop:mesures-proches-I} implique, pour $R$ assez grand par rapport \`a~$\varepsilon$, l'existence d'une mesure finie $\hat{n}_{\varepsilon,R}^{\alpha}$ sur~$X$, invariante par $C_{\alpha}^{\Gamma}$, telle que $\hat{m}_{\varepsilon,R}^{\alpha} = h \hat{n}_{\varepsilon,R}^{\alpha}$ où $h : X\to\R_+^*$ v\'erifie $\Vert h - 1\Vert_{\infty} \leq \varepsilon/R^2$.
La moyenne $\overline{h} = \int_{g\in\mathbb{T}_{\alpha}} h\circ g\,\dd g$ de $h$ par~$\mathbb{T}_{\alpha}$ v\'erifie alors $e^{-3\varepsilon/R^2}\,\overline{h} \leq h \leq e^{3\varepsilon/R^2}\,\overline{h}$ d\`es que $R$ est assez grand par rapport \`a~$\varepsilon$.
D'apr\`es le lemme~\ref{lem:LP-tore} il existe $C_G>0$, ne d\'ependant que de la dimension maximale d'un tore compact de~$G$, tel que
\begin{equation} \label{eqn:d-LP-m-moyenne-n}
d_{LP}\big(\hat{m}_{\varepsilon,R}^{\alpha},\overline{h} \hat{n}_{\varepsilon,R}^{\alpha}\big) = d_{LP}\big(h \hat{n}_{\varepsilon,R}^{\alpha},\overline{h}n_{\varepsilon,R}^{\alpha}\big) \leq \frac{3\varepsilon}{R^2} \, C_G \, \sup_{x\in X} \mathrm{diam}(\mathbb{T}_{\alpha}\cdot x) \leq 3 C_G C'' \frac{\varepsilon}{R}.
\end{equation}
Or, d'apr\`es la proposition~\ref{prop:mesures-I} on a $(\varphi_1\circ\inv)_*\,\hat{m}_{\varepsilon,-R}^{\alpha} = (\varphi_1\circ\refl)_*\,\hat{m}_{\varepsilon,R}^{\alpha}$.
D'apr\`es le lemme~\ref{lem:refl-C-alpha} l'involution $\refl : X\to X$ correspond \`a l'action d'un \'el\'ement de $\mathbb{S}_{\alpha}$ donc de $\mathbb{T}_{\alpha}$, et d'apr\`es la proposition~\ref{prop:alpha-t-presque-conj}, pour $R$ assez grand par rapport \`a~$\varepsilon$, l'élément $\alpha_1\in\mathbb{T}_{\alpha}$ envoie tout $x\in X$ \`a distance $\leq 2C'\varepsilon/R$ de $\varphi_1(x)$, o\`u $C'>0$ est ind\'ependant de $\varepsilon$ et~$R$.
Par in\'egalit\'e triangulaire, on a $d_{LP}(\hat{m}_{\varepsilon,R}^{\alpha}, (\varphi_1\circ\refl)_*\,\hat{m}_{\varepsilon,R}^{\alpha}) \leq d_{LP}(\hat{m}_{\varepsilon,R}^{\alpha}, \overline{\hat{h}}\hat{n}_{\varepsilon,R}^{\alpha}) + d_{LP}(\overline{\hat{h}}\hat{n}_{\varepsilon,R}^{\alpha}, (\alpha_1\circ\refl)_*\,\hat{m}_{\varepsilon,R}^{\alpha}) + d_{LP}((\alpha_1\circ\refl)_*\,\hat{m}_{\varepsilon,R}^{\alpha}, (\varphi_1\circ\refl)_*\,\hat{m}_{\varepsilon,R}^{\alpha})$.
Le premier terme de droite est born\'e par $3 C_G C'' \varepsilon/R$ d'apr\`es \eqref{eqn:d-LP-m-moyenne-n}.
Le deuxi\`eme terme est \'egal au premier par la remarque~\ref{rem:LP-bouge-peu}.\eqref{item:LP-isom}, en utilisant le fait que $\alpha_1\circ\refl\in\mathbb{T}_{\alpha}$ agit sur~$X$ comme une isom\'etrie et pr\'eserve $\overline{\hat{h}}\hat{n}_{\varepsilon,R}^{\alpha}$.
Le troisi\`eme terme est born\'e par $2C'\varepsilon/R$, d'apr\`es la remarque~\ref{rem:LP-bouge-peu}.\eqref{item:LP-f-g} et le fait que $\alpha_1$ envoie tout $x\in X$ \`a distance $\leq 2C'\varepsilon/R$ de $\varphi_1(x)$.
On obtient ainsi $d_{LP}(\hat{m}_{\varepsilon,R}^{\alpha}, (\varphi_1\circ\refl)_*\,\hat{m}_{\varepsilon,R}^{\alpha}) \leq (6C_GC'' + 2C') \varepsilon/R$.
On conclut en utilisant la remarque~\ref{rem:mesures-LP-proches-Lambda-alpha}.

\subsection{D\'emonstration du lemme~\ref{lem:LP-tore}} \label{subsec:dem-LP-tore}

Nous traitons ici le cas o\`u $X$ est le tore euclidien $\mathbb{T}^{\ell} = (\R/\Z)^{\ell}$ muni de sa mesure de Haar~$n$ et $\mathbb{T}^k$ est un sous-tore, pour $1\leq k\leq\ell$.
Le cas g\'en\'eral utilise les m\^emes id\'ees : \cf \textcite[\S\,18]{klm18}.

Comme la distance de L\'evy--Prokhorov entre deux mesures est invariante par multiplication des deux mesures par une m\^eme constante, on peut supposer que $n$ est une mesure de probabilit\'e.
On raisonne en trois \'etapes :
\begin{enumerate}
  \item\label{item:LP-tore-1} pour $k=\ell=1$, si $\overline{h} \leq e^{2\delta}\,h$, alors $d_{LP}(hn,\overline{h}n) \leq \delta$ ;
  \item\label{item:LP-tore-2} pour $k=1$ et $\ell\geq 1$ quelconque, si $\overline{h} \leq e^{2\delta}\,h$, alors $d_{LP}(hn,\overline{h}n) \leq \delta$ ;
  \item\label{item:LP-tore-3} pour $1\leq k\leq\ell$ quelconques, si $e^{-2\delta}\,\overline{h} \leq h \leq e^{2\delta}\,\overline{h}$, alors $d_{LP}(hn,\overline{h}n) \leq (2^k - 1)\,\delta$.
\end{enumerate}

\noindent
$\bullet$ D\'emonstration de~\eqref{item:LP-tore-1} : Supposons $\overline{h} \leq e^{2\delta}\,h$ et soit $A\subset X=\mathbb{T}^1$.
Si $\mathcal{V}_{\delta}(A)=\mathbb{T}^1$, on~a
$$(hn)(\mathcal{V}_{\delta}(A)) = \int_{\mathbb{T}^1} h \dd n = \overline{h} \geq \overline{h}\,n(A).$$
Si $\mathcal{V}_{\delta}(A)\subsetneq\mathbb{T}^1$, alors $n(A) \leq n(\mathcal{V}_{\delta}(A)) - 2\delta \leq (1 - 2\delta) \, n(\mathcal{V}_{\delta}(A))$ car $X=\mathbb{T}^1$, d'o\`u
$$(hn)(\mathcal{V}_{\delta}(A)) \geq e^{-2\delta} \, \overline{h}\,n(\mathcal{V}_{\delta}(A)) \geq \frac{e^{-2\delta}}{1-2\delta} \, \overline{h}\,n(A) \geq \overline{h}\,n(A).$$

\noindent
$\bullet$ D\'emonstration de~\eqref{item:LP-tore-2} : \'Ecrivons $X=\mathbb{T}^{\ell}$ comme le produit de $\mathbb{T}^k = \mathbb{T}^1$ et d'un facteur $\mathbb{T}^{\ell-1}$.
Toute mesure $m$ sur~$X$ se d\'esint\`egre en des mesures $m_x$ sur $\mathbb{T}^1\times\{x\}$ pour $x\in\mathbb{T}^{\ell-1}$.
D'apr\`es~\eqref{item:LP-tore-1} on a $d_{LP}((hn)_x,(\overline{h}n)_x) \leq \delta$ pour tout $x\in\mathbb{T}^{\ell-1}$.
Par int\'egration, pour tout $A\subset X$, en utilisant le fait que $(\{x\}\times\mathbb{T}^{\ell-1})\cap\mathcal{V}_{\delta}(A)$ contient le $\delta$-voisinage de $(\{x\}\times\mathbb{T}^{\ell-1})\cap A$ dans $\{x\}\times\mathbb{T}^{\ell-1}$, on voit que $(hn)(\mathcal{V}_{\delta}(A)) \geq (\overline{h}n)(A)$.

\smallskip
\noindent
$\bullet$ D\'emonstration de~\eqref{item:LP-tore-3} : On raisonne par r\'ecurrence sur~$k$.
Le cas $k=1$ est contenu dans~\eqref{item:LP-tore-2}.
Supposons le r\'esultat vrai pour $k-1$.
\'Ecrivons notre facteur $\mathbb{T}^k$ comme $\mathbb{T}^{k-1}\times\mathbb{T}^1$, et soit $\widetilde{h}$ la moyenne de~$h$ par $\mathbb{T}^{k-1}$.
Par hypoth\`ese on a $e^{-2\delta}\,\overline{h} \leq h \leq e^{2\delta}\,\overline{h}$, d'o\`u $e^{-2\delta}\,\overline{h} \leq \widetilde{h} \leq e^{2\delta}\,\overline{h}$ en prenant la moyenne par $\mathbb{T}^{k-1}$.
On en d\'eduit $e^{-4\delta}\,\widetilde{h} \leq h \leq e^{4\delta}\,\widetilde{h}$, et donc $d_{LP}(hn,\widetilde{h}n) \leq (2^{k-1} - 1)\,2\delta$ par hypoth\`ese de r\'ecurrence.
D'autre part, comme $\overline{h}$ est la moyenne de $\widetilde{h}$ par~$\mathbb{T}^1$, on a $d_{LP}(\widetilde{h}n,\overline{h}n)\leq\delta$ par~\eqref{item:LP-tore-2}.
Par in\'egalit\'e triangulaire, on obtient
$$d_{LP}(hn,\overline{h}n) \leq d_{LP}(hn,\widetilde{h}n) + d_{LP}(\widetilde{h}n,\overline{h}n) \leq (2^{k-1} - 1)\,2\delta + \delta = (2^k - 1)\,\delta.$$

\subsection{D\'emonstration de la proposition~\ref{prop:alpha-t-presque-conj}} \label{subsec:dem-alpha-presque-conj-exp-2Rh}

Le lemme~\ref{lem:realise-proximal} affirme que pour tout $\varepsilon>0$ assez petit, tout $R>0$ assez grand et tout $(g,g',\alpha)\in G^3$, si $\alpha$ est $(\varepsilon,R)$-presque réalisé par $(g,g')$ au sens de la d\'efinition~\ref{def:realise-presque-alpha}, alors $g^{-1}\alpha g$ est proche de $\exp(R\,\mathsf{h})$ au sens où les points fixes attractifs, les points fixes répulsifs et les images par~$\lambda$ (projections de Jordan/Lyapounov) de ces deux éléments sont proches.
Le lemme suivant affirme que dans ce contexte, l'élément $\Psi_{\alpha}(g)^{-1}\alpha\Psi_{\alpha}(g)$ est proche de $\exp(R\,\mathsf{h})$ au sens de la distance $G$-invariante à gauche $d_G$, où $\Psi_{\alpha} : \mathcal{U}_{\alpha}\to L_{\alpha}$ est l'application \og pied\fg\ de la définition~\ref{def:pied}.

\begin{lemme} \label{lem:alpha-presque-conj-exp-2Rh-Psi}
Il existe $C''>0$ tel que pour tout $\varepsilon>0$ assez petit, tout $R>0$ assez grand et tout $(g,g',\alpha)\in G^3$, si $\alpha$ est $(\varepsilon,R)$-presque réalisé par $(g,g')$ au sens de la d\'efinition~\ref{def:realise-presque-alpha}, alors
$$d_G(\varphi_{2R}(\Psi_{\alpha}(g)), \alpha\,\Psi_{\alpha}(g)) \leq C'' (\varepsilon+e^{-R}).$$
\end{lemme}

Pour démontrer le lemme~\ref{lem:alpha-presque-conj-exp-2Rh-Psi}, Kahn, Labourie et Mozes utilisent les propriétés suivantes de contraction et de dilatation du flot $(\varphi_t)_{t\in\R}$.
Soit $U_{\tau}^+$ le radical unipotent du sous-groupe parabolique $P_{\tau}$ de~$G$ du paragraphe~\ref{subsec:intro-Anosov}, de sorte que $P_{\tau} = Z_K(\mathfrak{a}) A U_{\tau}^+$.
Soit $U_{\tau}^-$ le conjugué de $U_{\tau}^+$ par $\tau((\begin{smallmatrix} 0 & 1\\ -1 & 0\end{smallmatrix}))$, de sorte que $P_{\tau}^- := Z_K(\mathfrak{a}) A U_{\tau}^-$ soit un sous-groupe parabolique de~$G$ opposé à~$P_{\tau}$.
Les classes \`a gauche de $A$ (\resp $U_{\tau}^+$, \resp $P_{\tau}$, \resp $U_{\tau}^-$, \resp $P_{\tau}^-$) forment un feuilletage de~$G$, dit central (\resp stable, \resp central stable, \resp instable, \resp central instable) ; pour $t>0$, la transformation $\varphi_t$ envoie feuille sur feuille de mani\`ere isom\'etrique (\resp uniform\'ement contractante, \resp contractante au sens large, \resp uniform\'ement dilatante, \resp dilatante au sens large).
Pour $x,y\in G$ suffisamment proches, la feuille stable de~$x$ intersecte la feuille centrale instable de~$y$ en un unique point, proche de $x$ et~$y$.

\begin{exemple}
Pour $G=\PSL(2,\R)$, vu comme $T^1\HH^2$ comme au paragraphe~\ref{subsec:tau-P1-G/P}, le flot $(\varphi_t)_{t\in\R}$ est le flot géodésique.
La feuille du feuilletage central stable (\resp central instable) contenant $x\in T^1\HH^2$ est l'ensemble des vecteurs unitaires tangents pointant vers le même point à l'infini $\eta_x^+\in\partial\HH^2$ (\resp $\eta_x^-\in\partial\HH^2$) dans le futur (\resp passé) que~$x$.
La feuille du feuilletage stable (\resp instable) contenant~$x$ est l'ensemble des vecteurs unitaires tangents de la feuille centrale stable (\resp centrale instable) de~$x$ qui sont basés sur le même horocycle centré en $\eta_x^+$ (\resp $\eta_x^-$) que~$x$.
\end{exemple}

Kahn, Labourie et Mozes utilisent ces propriétés dynamiques du flot $(\varphi_t)_{t\in\R}$ pour établir l'existence de constantes $C_1,C_2>0$ telles que~:
\begin{enumerate}
  \item\label{item:alpha-presque-conj-1} pour tout $\varepsilon>0$ assez petit, tout $R>0$ assez grand et tout $(g,g',\alpha)\in G^3$, si $\alpha$ est $(\varepsilon,R)$-presque r\'ealis\'e par $(g,g')$ (d\'efinition~\ref{def:realise-presque-alpha}), alors il existe $z\in G$ tel que $d_G(z,g)\leq C_1 (\varepsilon+e^{-R})$ et $d_G(\varphi_{2R}(z),\alpha g)\leq C_1 (\varepsilon+e^{-R})$~;
  \item\label{item:alpha-presque-conj-2} pour tout $\delta>0$ assez petit, tout $R>0$ et tous $\alpha,z\in G$, si $\alpha$ est proximal dans $G/P_{\tau}$, de points fixes attractif $\alpha^{\attract}$ et répulsif $\alpha^{\repuls}$, si $d_{\tau}({z'}^{-1}\cdot\alpha^{\repuls},\underline{\tau}(0))\leq\delta$ et si $d_{\tau}({z'}^{-1}\cdot\alpha^{\attract},\underline{\tau}(\infty))\leq\delta$ pour tout $z'\in\{z,\varphi_{2R}(z)\}$, alors
  $$d_G(\Psi_{\alpha}\circ\varphi_{2R}(z), \varphi_{2R}\circ\Psi_{\alpha}(z)) \leq C_2\delta.$$
\end{enumerate}
Ils en déduisent alors le lemme~\ref{lem:alpha-presque-conj-exp-2Rh-Psi} en utilisant le lemme~\ref{lem:realise-proximal} et le fait que l'application \og pied\fg\ $\Psi_{\alpha}$, définie comme la projection orthogonale d'un petit voisinage de $L_{\alpha}$ sur~$L_{\alpha}$, est lipschitzienne ; sa constante de Lipschitz ne dépend pas de~$\alpha$ d'après la remarque~\ref{rem:pied-alpha-0-alpha}.

\begin{proof}[D\'emonstration de la proposition~\ref{prop:alpha-t-presque-conj}]
Pour tout $[\alpha]\in\mathcal{A}_{\varepsilon,R}$, par définition, il existe $(g,g')\in G$ tel que $\alpha$ est $(\varepsilon/R,R)$-presque réalisé par $(g,g')$.
Le lemme~\ref{lem:alpha-presque-conj-exp-2Rh-Psi} affirme que si $R$ est assez grand par rapport à~$\varepsilon$, alors il existe $y\in L_{\alpha}$ tel que $d_G(\varphi_{2R}(y), \alpha y) = d_G(y \exp(R\,\mathsf{h}), \alpha y) \leq C'' r_{\varepsilon,R}$, où $r_{\varepsilon,R} := (\varepsilon/R+\nolinebreak e^{-R})$ ; autrement dit, comme $d_G$ est $G$-invariante à gauche, on a $d_G(\exp(-R\,\mathsf{h}) y^{-1}\alpha y, \id) \leq C'' r_{\varepsilon,R}$.

Soit $v\in Z_{\mathfrak{k}}(\mathfrak{a})+\mathfrak{a}$ de norme minimale tel que $\exp(2Rv) = \exp(-R\,\mathsf{h})y^{-1}\alpha y$.
Pour tout $t\in [0,2R]$, on a $d_G(\exp(tv), \id) \leq d_G(\exp(2Rv), \id) \leq C'' r_{\varepsilon,R}$.
On considère le sous-groupe à un paramètre
$$(\alpha_t)_{t\in\R} := \Big(y\exp\!\Big(t\Big(\frac{\mathsf{h}}{2} + v\Big)\Big)y^{-1}\Big)_{t\in\R}$$
de $Z(Z_G(\alpha))_0 \subset Z(C_{\alpha}^{\Gamma})$.
On a $\alpha_{2R} = \alpha$ et $d_G(\varphi_t(y), \alpha_t y) = d_G(\exp(tv), \id)  \leq C'' r_{\varepsilon,R}$ pour tout $t\in [0,2R]$.

En écrivant $L_{\alpha}$ comme une classe à gauche de $A Z_K(\mathfrak{a})$ comme à la remarque~\ref{rem:pied-alpha-0-alpha}, on voit que pour tout $x\in L_{\alpha}$ et tout $t\in [0,2R]$, on a $x^{-1}y\in A Z_K(\mathfrak{a})$ et
$$d_G(\varphi_t(x), \alpha_tx) = d_G((x^{-1}y) \exp(tv) (x^{-1}y)^{-1}, \id).$$
Les éléments $(x^{-1}y) \exp(tv) (x^{-1}y)^{-1}$ et $\exp(tv)$ de $A Z_K(\mathfrak{a})$ ne diffèrent que selon leurs composantes dans $Z_K(\mathfrak{a})$, qui sont conjuguées par un élément de $Z_K(\mathfrak{a})$ indépendant de~$t$ (à savoir la composante de $x^{-1}y$ dans $Z_K(\mathfrak{a})$).
On en déduit l'existence d'une constante $C'>0$ telle que $d_G(\varphi_t(x), \alpha_t x) \leq (C'/C'') \, d_G(\varphi_t(y), \alpha_t y) \leq C' r_{\varepsilon,R}$ pour tous $x\in L_{\alpha}$ et $t\in [0,2R]$.
\end{proof}

\section{Conclusion : \'Etape combinatoire} \label{sec:conclusion}

Dans cette partie nous concluons la d\'emonstration des th\'eor\`emes \ref{thm:quantitatif-PSL(n,C)} et~\ref{thm:quantitatif} \`a partir des propositions \ref{prop:pi-1-inj} et~\ref{prop:pant-dans-reseau-mesures}.
Nous présentons l'approche de \textcite{klm18}, qui utilise (comme \cite{km12}) le lemme des mariages de Hall (fait~\ref{fait:mariages}).

\begin{remarque} \label{rem:lem-mariages-H}
L'approche de Hamenst\"adt, plus pr\'ecise, ne requiert pas le lemme des mariages.
Le point clé est une version plus forte de la proposition~\ref{prop:mesures-LP-proches}, qui revient essentiellement \`a l'existence, pour tout $\alpha\in\Gamma$, d'un homéomorphisme $\psi_{\alpha}$ de $Z_{\Gamma}(\alpha)\backslash L_{\alpha}$ tel que $(\psi_{\alpha})_*m_{\varepsilon,R}^{\alpha} = (\varphi_1\circ\inv)_*\,m_{\varepsilon,-R}^{\alpha}$ et tel que la distance de tout point de $Z_{\Gamma}(\alpha)\backslash L_{\alpha}$ \`a son image par~$\psi_{\alpha}$ soit uniform\'ement (exponentiellement) petite par rapport à~$R$.
\end{remarque}

\smallskip

\emph{Dans toute la partie, on travaille dans le cadre~\ref{cadre}.
On suppose la condition~(R) du paragraphe~\ref{subsubsec:cond-R} vérifiée.
On fixe un réseau cocompact irréductible $\Gamma$ de~$G$, que l'on suppose comme précédemment sans torsion.}

\subsection{Id\'ee de la d\'emonstration : graphes enrubann\'es} \label{subsec:graphes-enrubannes}

Comme expliqué au paragraphe~\ref{subsec:etape-comb}, pour conclure la d\'emonstration des th\'eor\`emes \ref{thm:quantitatif-PSL(n,C)} et~\ref{thm:quantitatif}, il s'agit de prendre les repr\'esentations de groupes de pantalons $(\varepsilon/R,\pm R)$-presque parfaites $(C\varepsilon/R)$-bien recollées des propositions \ref{prop:pant-dans-reseau-general} et~\ref{prop:pant-dans-reseau-mesures}, avec les données géométriques appropriées correspondantes, et de montrer qu'on peut les agencer de mani\`ere ad\'equate pour obtenir une repr\'esentation d'une surface compacte~$S$, avec une d\'ecomposition en pantalons bipartie et un graphe $\mathcal{G}$ comme au paragraphe~\ref{subsec:struct-R-parf}, v\'erifiant les hypothèses de la proposition~\ref{prop:pi-1-inj} ou de son raffinement~\ref{prop:pi-1-inj-KLM}.

Pour trouver la surface~$S$ et la décomposition en pantalons, l'idée est de commencer par construire le graphe $\mathcal{G}$.
Plus précisément, $\mathcal{G}$ sera un \emph{graphe enrubanné}, c'est-à-dire un graphe muni d'un ordre cyclique sur l'ensemble des ar\^etes en chaque sommet.
Le point est que tout graphe enrubanné fini trivalent $\mathcal{G}$ peut être \'epaissi (en rempla\c{c}ant chaque ar\^ete par un cylindre) pour obtenir une surface compacte $S$ de genre au moins deux.
Les milieux des ar\^etes de~$\mathcal{G}$ s'\'epaississent en des courbes ferm\'ees simples de~$S$ qui d\'efinissent une d\'ecomposition en pantalons de~$S$.
On peut voir $\mathcal{G}$ comme un graphe sur~$S$, contenant exactement un sommet à l'intérieur de chaque pantalon, avec une arête pour chaque courbe de bord entre deux pantalons, comme au paragraphe~\ref{subsec:struct-R-parf}.
Le graphe $\mathcal{G}$ est biparti si et seulement si la d\'ecomposition en pantalons l'est.

\subsection{Le point de vue de Kahn, Labourie et Mozes} \label{subsec:conclusion-KLM}

Fixons deux pantalons $\Pi^+$ et~$\Pi^-$ de groupes fondamentaux
$$\pi_1(\Pi^{\pm}) = \langle a^{\pm}, b^{\pm}, c^{\pm} \,|\, c^{\pm} b^{\pm} a^{\pm} = 1\rangle.$$
Pour $\varepsilon,R>0$, soit $\bs{\mathcal{C}}_{\varepsilon,\pm R}$ l'ensemble des classes de conjugaison modulo~$\Gamma$ de repr\'esentations $(\varepsilon/R,\pm R)$-presque parfaites de $\pi_1(\Pi^{\pm})$ dans~$\Gamma$.
Rappelons qu'il est fini d'après le corollaire~\ref{cor:presque-parf->generique}.\eqref{item:nb-fini-presque-parf}.
Il admet une symétrie naturelle $\sym$ d'ordre trois induite par la symétrie $\sym$ de $\mathrm{Hom}(\pi_1(\Pi^{\pm}),G)$ donnée par
\begin{equation} \label{eqn:tri}
\sym(\rho)(a^{\pm},b^{\pm},c^{\pm}) = (\rho(b^{\pm}),\rho(c^{\pm}),\rho(a^{\pm}))
\end{equation}
pour tout $\rho\in\mathrm{Hom}(\pi_1(\Pi^{\pm}),G)$ (\cf remarque~\ref{rem:presque-parf-tri}).
Cette sym\'etrie est sans point fixe pour $R$ assez grand par rapport \`a~$\varepsilon$ d'apr\`es le corollaire~\ref{cor:presque-parf->generique}.\eqref{item:presque-parf->generique}.

D'autre part, rappelons que l'ensemble $\Triconn_{\varepsilon,\pm R}$ du paragraphe~\ref{subsec:Triconn} paramètre les données géométriques associées aux représentations $(\varepsilon/R,\pm R)$-presque parfaites de $\pi_1(\Pi^{\pm})$ dans~$\Gamma$ modulo l'action de~$\Gamma$.
Il est muni lui aussi d'une symétrie $\sym$ d'ordre trois, donnée par \eqref{eqn:tri-Triconn}.
On a une projection naturelle $\varpi : \Triconn_{\varepsilon,\pm R}\to\bs{\mathcal{C}}_{\varepsilon,\pm R}$, qui induit une projection naturelle de $\Triconn_{\varepsilon,\pm R}/\langle\sym\rangle$ vers $\bs{\mathcal{C}}_{\varepsilon,\pm R}/\langle\sym\rangle$, encore notée~$\varpi$.
Comme au paragraphe~\ref{subsec:Triconn}, pour $\varepsilon'>0$, la définition~\ref{def:bon-recoll-KLM} induit une notion d'éléments de $\Triconn_{\varepsilon,R}$ et $\Triconn_{\varepsilon,-R}$ qui sont \emph{$\varepsilon'$-bien recollés}.

En utilisant les mesures $\mu_{\varepsilon,\pm R}$ de la proposition~\ref{prop:mu-convient} et le lemme des mariages (fait~\ref{fait:mariages}), Kahn, Labourie et Mozes établissent le résultat suivant.

\begin{proposition} \label{prop:existe-graphe-enrubanne}
Supposons la condition~(R) vérifiée.
Soit $C>0$ la constante de la proposition~\ref{prop:mu-convient}.
Pour tout $\varepsilon>0$ assez petit et tout $R>0$ assez grand par rapport \`a~$\varepsilon$, il existe
\begin{itemize}
  \item un graphe enrubann\'e fini $\mathcal{G}$, trivalent, biparti, de sommets $\mathcal{P} = \mathcal{P}^+\sqcup\mathcal{P}^-$,
  \item une application $E : \mathcal{P}^{\pm}\to\Triconn_{\varepsilon,\pm R}/\langle\sym\rangle$ (\og étiquetage\fg)
\end{itemize}
tels que pour tout sommet $v\in\mathcal{P}^{\pm}$, de sommets adjacents $v_1,v_2,v_3\in\mathcal{P}^{\mp}$, on puisse trouver $\mathtt{T}\in\Triconn_{\varepsilon,\pm R}$ et $\mathtt{T}_1,\mathtt{T}_2,\mathtt{T}_3\in\Triconn_{\varepsilon,\mp R}$ vérifiant $E(v)=[\mathtt{T}]$ et $E(v_i)=[\mathtt{T}_i]$, tels que $\sym^i(\mathtt{T})$ et $\mathtt{T}_i$ soient $(C\varepsilon/R)$-bien recollés pour tout $i\in\Z/3\Z$.
\end{proposition}

Adoptons la terminologie suivante : pour une classe de conjugaison $\bs{c}\in\bs{\mathcal{C}}_{\varepsilon,\pm R}$ de représentations de $\pi_1(\Pi^{\pm})$, et pour un pantalon $\Pi$ quelconque, disons qu'une représentation $\rho_{\Pi} : \pi_1(\Pi)\to\Gamma$ est \emph{de type $\bs{c}$} si l'on peut identifier les courbes de bord de~$\Pi$ à $a^{\pm},b^{\pm},c^{\pm}$ de sorte que $\rho_{\Pi}$ définisse une représentation de $\pi_1(\Pi^{\pm})$ dans la classe de conjugaison~$\bs{c}$.
Ceci ne dépend que de la classe de $\bs{c}$ dans $\bs{\mathcal{C}}_{\varepsilon,\pm R}/\langle\sym\rangle$.

Le graphe enrubann\'e $\mathcal{G}$ de la proposition~\ref{prop:existe-graphe-enrubanne} d\'efinit une surface compacte $S$ avec une décomposition en pantalons bipartie comme au paragraphe~\ref{subsec:graphes-enrubannes}, et l'on en d\'eduit une repr\'esentation de $\pi_1(S)$ dans~$\Gamma$ par le lemme suivant.

\begin{lemme} \label{lem:recoller-repr}
Soit $S$ une surface compacte de genre au moins deux avec une d\'ecomposition en pantalons bipartie $\mathcal{P}=\mathcal{P}^+\sqcup\mathcal{P}^-$.
Soit $\mathcal{G}$ un graphe fini sur~$S$, avec un sommet à l'intérieur de chaque pantalon et une arête pour chaque courbe de bord entre deux pantalons, comme au paragraphe~\ref{subsec:struct-R-parf}.
Pour $\varepsilon,R>0$, soit $E : \mathcal{P}^{\pm}\to\Triconn_{\varepsilon,\pm R}/\langle\sym\rangle$ un étiquetage comme à la proposition~\ref{prop:existe-graphe-enrubanne}.
Alors il existe une représentation $\rho : \pi_1(S)\to\Gamma$, unique à conjugaison par~$\Gamma$ près, telle que pour tout $\Pi\in\mathcal{P}^{\pm}$, la restriction de $\rho$ à $\pi_1(\Pi)$ soit de type $\varpi\circ E(\Pi) \in \bs{\mathcal{C}}_{\varepsilon,\pm R}/\langle\sym\rangle$.
\end{lemme}

Notons que $\Triconn_{\varepsilon,\pm R} \subset \Triconn_{C\varepsilon,\pm R}$ pour $C\geq 1$.
La proposition~\ref{prop:pi-1-inj} implique alors que pour $\varepsilon>0$ assez petit et $R>0$ assez grand par rapport \`a~$\varepsilon$, la repr\'esentation $\rho : \pi_1(S)\to\Gamma$ du lemme~\ref{lem:recoller-repr} est injective.
De plus, à $\delta>0$ fixé, la proposition~\ref{prop:pi-1-inj-KLM} assure que si $\varepsilon>0$ est assez petit par rapport à~$\delta$ et $R>0$ assez grand par rapport \`a~$\varepsilon$, alors $\rho$ admet une application de bord $(\delta,\tau)$-sullivannienne de $\PP^1(\R)$ dans $G/P_{\tau}$, ce qui démontre les th\'eor\`emes \ref{thm:quantitatif-PSL(n,C)} et~\ref{thm:quantitatif}.

\subsection{D\'emonstration de la proposition~\ref{prop:existe-graphe-enrubanne}} \label{subsec:dem-existe-graphe}

D'apr\`es la proposition~\ref{prop:mu-convient}, pour tout $\varepsilon>0$ et tout $R>0$ assez grand par rapport \`a~$\varepsilon$, il existe des mesures $\mu_{\varepsilon,\pm R}$ sur $\Triconn_{\varepsilon,\pm R}$, invariantes par la transformation $\sym$ de \eqref{eqn:tri-Triconn}, v\'erifiant $\mu_{\varepsilon,R}(\Triconn_{\varepsilon,R}) = \mu_{\varepsilon,-R}(\Triconn_{\varepsilon,-R})$ et telles que pour tout sous-ensemble mesurable $A$ de $\Triconn_{\varepsilon,R}$, l'ensemble des \'el\'ements de $\Triconn_{\varepsilon,-R}$ qui sont $(C\varepsilon/R)$-bien recoll\'es \`a au moins un \'el\'ement de~$A$ est de $(\mu_{\varepsilon,-R})$-mesure sup\'erieure ou \'egale \`a $\mu_{\varepsilon,R}(A)$.

Les espaces $\Triconn_{\varepsilon,R}$ et $\Triconn_{\varepsilon,-R}$ étant compacts, on peut remplacer les mesures $\mu_{\varepsilon,R}$ et~$\mu_{\varepsilon,-R}$ par des mesures $\mu^f_{\varepsilon,R}$ et~$\mu^f_{\varepsilon,-R}$ de même masse totale, de \emph{supports finis}, qui sont encore invariantes par $\sym$ et vérifient encore que pour tout $A\subset\Triconn_{\varepsilon,R}$, l'ensemble des \'el\'ements de $\Triconn_{\varepsilon,-R}$ qui sont $(C\varepsilon/R)$-bien recoll\'es \`a au moins un \'el\'ement de~$A$ est de $\mu^f_{\varepsilon,-R}$-mesure sup\'erieure ou \'egale \`a $\mu^f_{\varepsilon,R}(A)$.
Autrement dit, en notant $\mathcal{F}^+$ (\resp $\mathcal{F}^-$) le support (fini) de $\mu^f_{\varepsilon,R}$ (\resp $\mu^f_{\varepsilon,-R}$) et $\mathcal{F}$ l'union disjointe de $\mathcal{F}^+$ et~$\mathcal{F}^-$, le syst\`eme d'in\'equations
\begin{equation} \label{eqn:ineq-mu}
\displaystyle
\left\{ \begin{array}{ccll}
\mu(\mathtt{T}) & = & \mu(\sym(\mathtt{T})) & \quad\quad\forall \mathtt{T} \in \mathcal{F}, \vspace{0.2cm}\\
\sum_{\mathtt{T}^+\in\mathcal{F}^+} \mu(\mathtt{T}^+) & = & \sum_{\mathtt{T}^-\in\mathcal{F}^-} \mu(\mathtt{T}^-), & \vspace{0.2cm}\\
\sum_{\mathtt{T}^-\in\mathcal{F}^-(A)} \mu(\mathtt{T}^-) & \geq & \sum_{\mathtt{T}^+\in A} \mu(\mathtt{T}^+) & \quad\quad \forall A\subset\mathcal{F}^+
\end{array} \right.
\end{equation}
admet une solution $\mu : \mathcal{F}\to\R^+$ non nulle, o\`u $\mathcal{F}^-(A)$ désigne l'ensemble des \'el\'ements de $\mathcal{F}^- \subset \Triconn_{\varepsilon,-R}$ qui sont $(C\varepsilon/R)$-bien recoll\'es \`a au moins un \'el\'ement de $A \subset \mathcal{F}^+ \subset \Triconn_{\varepsilon,R}$.
Le syst\`eme \eqref{eqn:ineq-mu} \'etant \`a coefficients entiers, ceci implique l'existence d'une solution \emph{rationnelle}, et donc (en multipliant par un entier assez grand) l'existence d'une solution \emph{enti\`ere} non nulle $\mu : \mathcal{F}\to\N$.

Soit $\mathcal{E}^{\pm}$ l'ensemble obtenu en prenant $\mu(\mathtt{T})\in\N$ copies de chaque \'el\'ement $\mathtt{T}\in\mathcal{F}^{\pm}$.
La premi\`ere ligne de \eqref{eqn:ineq-mu} assure que la transformation d'ordre trois (sans point fixe) $\sym$ de $\Triconn_{\varepsilon,\pm R}$ induit une transformation d'ordre trois (sans point fixe) $\sym$ de $\mathcal{E}^{\pm}$.
La deuxi\`eme ligne assure que $\#\mathcal{E}^+ = \#\mathcal{E}^-$.
Soit $\mathcal{M} \subset \mathcal{E}^+\times\mathcal{E}^-$ le sous-ensemble correspondant aux couples $(\mathtt{T}^+,\mathtt{T}^-) \in \Triconn_{\varepsilon,R}\times\Triconn_{\varepsilon,-R}$ qui sont $(C\varepsilon/R)$-bien recollés.
La troisi\`eme ligne de \eqref{eqn:ineq-mu} assure que la condition \eqref{eqn:cond-mariage} est v\'erifi\'ee.
D'apr\`es le lemme des mariages de Hall (fait~\ref{fait:mariages}), il existe une bijection $\psi : \mathcal{E}^+\to\mathcal{E}^-$ telle que tout couple de la forme $(x,\psi(x))$ o\`u $x\in\mathcal{E}^+$ corresponde \`a un couple $(C\varepsilon/R)$-bien recollé.

Soit $\mathcal{G}$ le graphe fini biparti de sommets $\mathcal{P} = \mathcal{P}^+\sqcup\mathcal{P}^-$ o\`u $\mathcal{P}^{\pm} := \mathcal{E}^+/\langle\sym\rangle$, pour lequel on met une ar\^ete entre les images de $x$ et $\psi(x)$ pour tout $x\in\mathcal{E}^+$.
Ce graphe est trivalent car $\sym$ est d'ordre trois sans point fixe.
C'est un graphe enrubann\'e : $\sym$ d\'efinit un ordre cyclique sur les ar\^etes en chaque sommet.
La projection naturelle de $\mathcal{E}^{\pm}$ vers~$\Triconn_{\varepsilon,\pm R}$ induit une application $E : \mathcal{P}^{\pm}\to\Triconn_{\varepsilon,\pm R}/\langle\sym\rangle$ v\'erifiant les conclusions de la proposition~\ref{prop:existe-graphe-enrubanne}.

\subsection{D\'emonstration du lemme~\ref{lem:recoller-repr}} \label{subsec:constr-rho}

On utilise l'observation suivante.

\begin{remarque} \label{rem:action-libre-Gamma}
Pour $\varepsilon/R$ assez petit, si $\rho^+ : \pi_1(\Pi^+)\to\Gamma$ est $(\varepsilon/R,R)$-presque parfaite, si $\rho^- : \pi_1(\Pi^-)\to\Gamma$ est $(\varepsilon/R,-R)$-presque parfaite et si les représentations $\rho^+$ et~$\rho^-$ sont $(C\varepsilon/R)$-bien recollées le long de $a^+$ et~$a^-$, alors pour $\gamma\in\Gamma\smallsetminus\{\id\}$ les représentations $\rho^+$ et $\gamma\rho^-(\cdot)\gamma^{-1}$ ne sont pas $(C\varepsilon/R)$-bien recollées le long de $a^+$ et~$a^-$.
\end{remarque}

En effet, soit $\gamma\in\Gamma$.
Si $\rho^+$ et $\gamma\rho^-(\cdot)\gamma^{-1}$ sont $(C\varepsilon/R)$-bien recollées le long de $a^+$ et~$a^-$, alors $\gamma\rho^-(a^-)\gamma^{-1} = \rho^+(a^+) = \rho^-(a^-) =: \alpha$, donc $\gamma$ appartient au centralisateur $Z_{\Gamma}(\alpha)$ de $\alpha$ dans~$\Gamma$.
Or, l'application \og pied\fg\ $\Psi_{\alpha} : \mathcal{U}_{\alpha}\to L_{\alpha}$ (d\'efinition~\ref{def:pied}) est $Z_{\Gamma}(\alpha)$-équivariante et tout élément non trivial du groupe discret sans torsion $Z_{\Gamma}(\alpha)$ déplace les points de~$L_{\alpha}$ d'au moins une certaine distance, indépendante de $\varepsilon$ et~$R$.
La condition \eqref{eqn:presque-flip-decalage} sur $(g^+,g^-)$ pour $\varepsilon' = C\varepsilon/R$ ne peut donc pas être satisfaite par $(g^+,\gamma g^-)$ pour $\gamma\in Z_{\Gamma}(\alpha)\smallsetminus\{\id\}$ si $\varepsilon/R$ est assez petit.

Dans la suite, on note $\tildeTriconn_{\varepsilon,\pm R}\subset G^4$ la pré-image de $\Triconn_{\varepsilon,\pm R}$ par la projection naturelle $G^4\to\Gamma\backslash G^4$.
La symétrie d'ordre trois $\sym$ de $\Triconn_{\varepsilon,\pm R}$ se relève en une symétrie d'ordre trois de $\tildeTriconn_{\varepsilon,\pm R}$, encore notée $\sym$.
La projection $\varpi : \Triconn_{\varepsilon,\pm R}\to\bs{\mathcal{C}}_{\varepsilon,\pm R}$ du paragraphe~\ref{subsec:conclusion-KLM} se relève en une projection $\widetilde{\varpi}$ de $\tildeTriconn_{\varepsilon,\pm R}$ vers l'ensemble des représentations $(\varepsilon/R,\pm R)$-presque parfaites de $\pi_1(\Pi^{\pm})$ dans~$\Gamma$, telle que $\widetilde{\varpi}\circ\sym = \sym\circ\widetilde{\varpi}$.
On dit que des éléments de $\tildeTriconn_{\varepsilon,R}$ et $\tildeTriconn_{\varepsilon,-R}$ sont \emph{$(C\varepsilon/R)$-bien recollés} si leurs images par~$\widetilde{\varpi}$ le sont.

\smallskip

Dans le cadre du lemme~\ref{lem:recoller-repr}, le groupe fondamental $\pi_1(\mathcal{G})$ est un groupe libre non ab\'elien.
On peut voir $S$ comme un \'epaississement de~$\calG$, ce qui donne une injection $\pi_1(\calG)\hookrightarrow\pi_1(S)$.
On peut voir $\calG$ comme l'image d'une r\'etraction $r : S\to\calG$, ce qui donne une surjection $r_* : \pi_1(S)\twoheadrightarrow\pi_1(\calG)$.
La composition de l'injection $\pi_1(\calG)\hookrightarrow\pi_1(S)$ avec la surjection~$r_*$ est l'identit\'e de $\pi_1(\calG)$.
On en d\'eduit
$$\pi_1(S) = (\mathrm{Ker}\,r_*) \rtimes \pi_1(\calG),$$
et il existe un rev\^etement infini $\widehat{S}$ de~$S$ tel que $\mathrm{Ker}\,r_*$ s'identifie \`a $\pi_1(\widehat{S})$ et $\pi_1(\calG)$ au groupe de Galois du rev\^etement $\widehat{S}\to S$.
Le graphe trivalent $\calG$ sur~$S$ se rel\`eve en un arbre trivalent $\widetilde{\calG}$ (rev\^etement universel de~$\calG$) sur~$\widehat{S}$ (\cf figure~\ref{fig:revetement-S}) ; notons $\widetilde{\mathcal{P}} = \widetilde{\mathcal{P}}^+ \sqcup \widetilde{\mathcal{P}}^-$ l'ensemble de ses sommets.
\begin{figure}[h!]
\centering
\labellist
\small\hair 2pt
\pinlabel {$\widehat{S}$} [u] at 320 520
\pinlabel {$S$} [u] at 370 200
\pinlabel {\textcolor{red}{$\widetilde{\calG}$}} [u] at 160 420
\pinlabel {\textcolor{red}{$\calG$}} [u] at 153 195
\endlabellist
\includegraphics[width=6cm]{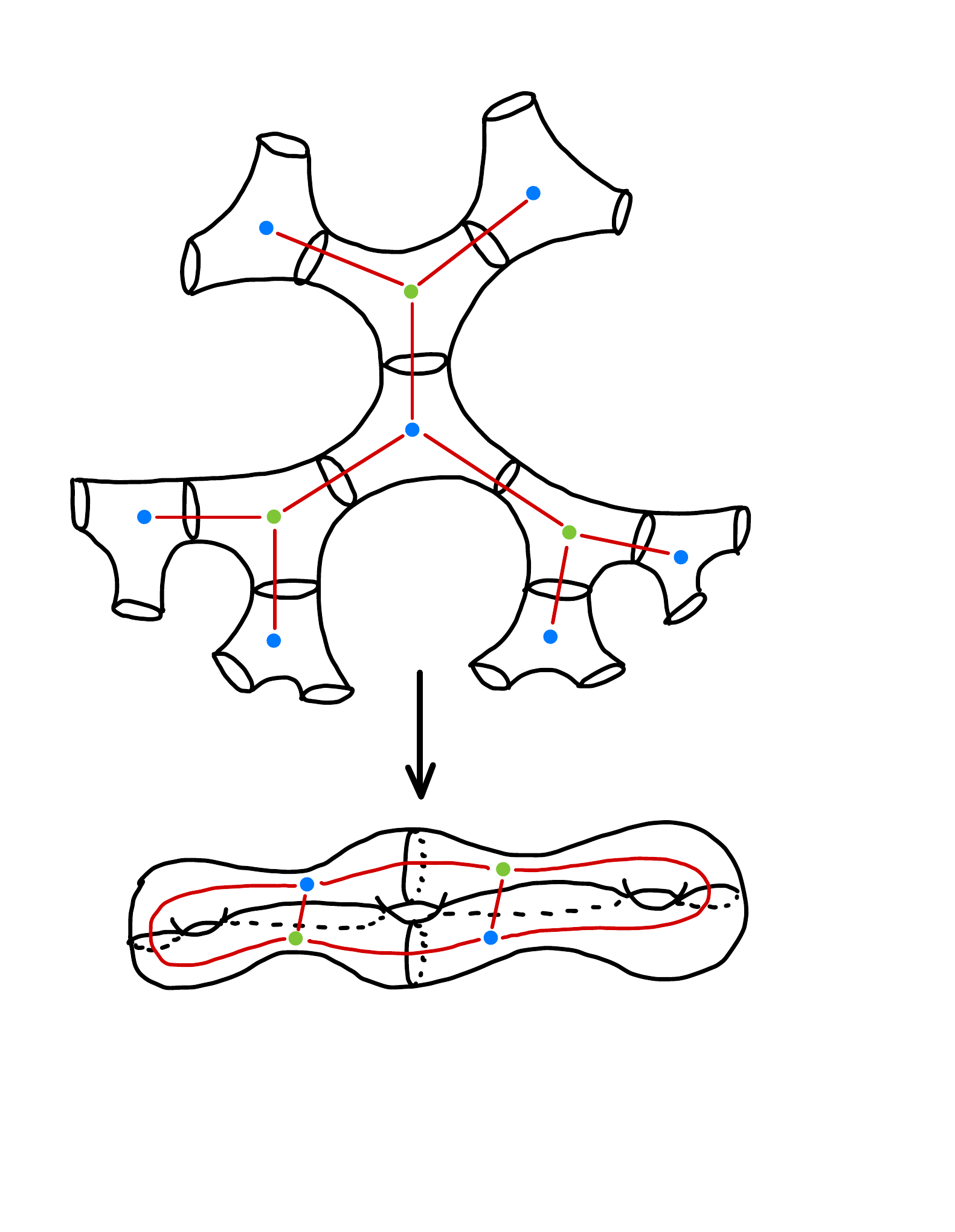}
\vspace{-0.3cm}
\caption{Le rev\^etement $\widehat{S}$ de~$S$}
\label{fig:revetement-S}
\end{figure}
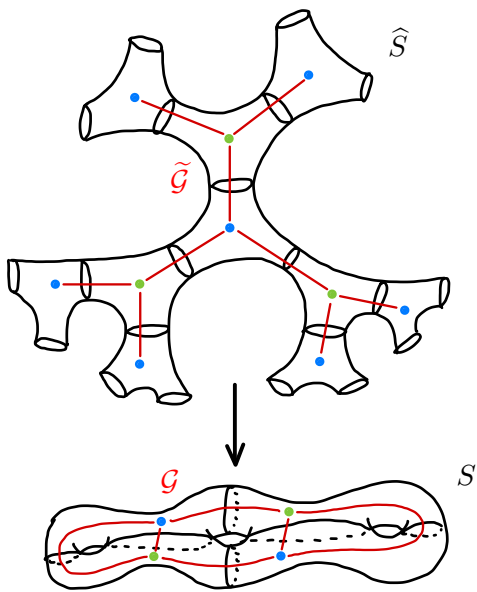
Pour toute ar\^ete $\widetilde{A}$ de~$\widetilde{\calG}$, notons $a_{\widetilde{A}}$ un \'el\'ement de $\pi_1(\widehat{S})$ correspondant \`a la courbe de bord entre les deux pantalons de~$\widehat{S}$ d\'efinis par~$\widetilde{A}$, orient\'ee de sorte que si $\widetilde{A}_1,\widetilde{A}_2,\widetilde{A}_3$ sont incidentes dans cet ordre en un m\^eme sommet, alors $a_{\widetilde{A}_3} a_{\widetilde{A}_2} a_{\widetilde{A}_1} = 1$.
Le groupe $\pi_1(\widehat{S})$ admet alors la pr\'esentation par le syst\`eme g\'en\'erateur $\{ a_{\widetilde{A}} \,|\,\widetilde{A}\text{ ar\^ete de }\widetilde{\calG}\}$ et les relations $a_{\widetilde{A}_3} a_{\widetilde{A}_2} a_{\widetilde{A}_1} = 1$ pour toutes ar\^etes $\widetilde{A}_1,\widetilde{A}_2,\widetilde{A}_3$ incidentes dans cet ordre en un m\^eme sommet.

Fixons un sommet initial de~$\calG$, appartenant à un pantalon $\Pi_0\in\mathcal{P}^+$, et un relev\'e $\widetilde{\Pi}_0\in\widetilde{\mathcal{P}}^+$ dans $\widetilde{\calG}$.
Choisissons un représentant $[\widetilde{\mathtt{T}}_{\widetilde{\Pi}_0}]\in\tildeTriconn_{\varepsilon,R}/\langle\sym\rangle$ de $E(\Pi_0)\in\Triconn_{\varepsilon,R}/\langle\sym\rangle$.
Par hypothèse, pour tout sommet $\widetilde{\Pi}_1\in\widetilde{\mathcal{P}}^-$ adjacent à~$\widetilde{\Pi}_0$, se projetant en $\Pi_1\in\mathcal{P}^-$, il existe un représentant $[\widetilde{\mathtt{T}}_{\widetilde{\Pi}_1}]\in\tildeTriconn_{\varepsilon,-R}/\langle\sym\rangle$ de $E(\Pi_1)\in\Triconn_{\varepsilon,-R}/\langle\sym\rangle$ qui est \emph{compatible} avec~$[\widetilde{\mathtt{T}}_{\widetilde{\Pi}_0}]$ au sens où il existe $i_0,i_1\in\Z/3\Z$ tels que $\sym^{i_0}(\widetilde{\mathtt{T}}_{\widetilde{\Pi}_0})$ et $\sym^{i_1}(\widetilde{\mathtt{T}}_{\widetilde{\Pi}_1})$ soient $(C\varepsilon/R)$-bien recollés~; ce représentant est unique d'après la remarque~\ref{rem:action-libre-Gamma}.
De même, pour tout sommet $\widetilde{\Pi}_2\in\widetilde{\mathcal{P}}^+$ adjacent à~$\widetilde{\Pi}_1$, se projetant en $\Pi_2\in\mathcal{P}^+$, il existe un unique représentant $[\widetilde{\mathtt{T}}_{\widetilde{\Pi}_2}]\in\tildeTriconn_{\varepsilon,R}/\langle\sym\rangle$ de $E(\Pi_2)\in\Triconn_{\varepsilon,R}/\langle\sym\rangle$ qui soit compatible avec~$[\widetilde{\mathtt{T}}_{\widetilde{\Pi}_1}]$.
On continue.
Comme $\widetilde{\calG}$ est un arbre, il n'y a pas d'autre condition de compatibilit\'e \`a v\'erifier.
En raisonnant de proche en proche, on obtient ainsi une famille $([\widetilde{\mathtt{T}}_{\widetilde{\Pi}^+}])_{\widetilde{\Pi}^+\in\widetilde{\mathcal{P}}^+}$ d'éléments de $\tildeTriconn_{\varepsilon,R}/\langle\sym\rangle$ paramétrés par les sommets $\widetilde{\mathcal{P}}^+$, et une famille $([\widetilde{\mathtt{T}}_{\widetilde{\Pi}^-}])_{\widetilde{\Pi}^-\in\widetilde{\mathcal{P}}^-}$ d'éléments de $\tildeTriconn_{\varepsilon,-R}/\langle\sym\rangle$ paramétrés par les sommets $\widetilde{\mathcal{P}}^-$, de sorte que les éléments au-dessus de deux sommets adjacents soient compatibles.
Pour toute arête $\widetilde{A}$ de~$\widetilde{\calG}$ entre $\widetilde{\Pi}^+$ et~$\widetilde{\Pi}^-$, si $\sym^{i^+}(\widetilde{\mathtt{T}}_{\widetilde{\Pi}^+})$ et $\sym^{i^-}(\widetilde{\mathtt{T}}_{\widetilde{\Pi}^+})$ sont bien recollés où $i^+,i^-\in\Z/3\Z$, on note $\rho_{\widehat{S}}(a_{\widetilde{A}})\in\Gamma$ l'image de $a_{\widetilde{A}}$ par $\widetilde{\varpi}(\sym^{i^+}(\widetilde{\mathtt{T}}_{\widetilde{\Pi}^+}))$.
\'Etant donn\'ee la pr\'esentation par g\'en\'erateurs et relations de $\pi_1(\widehat{S})$ d\'ecrite ci-dessus, on obtient ainsi une repr\'esentation $\rho_{\widehat{S}} : \pi_1(\widehat{S})\to\Gamma$.

Ceci nous donne \'egalement une \emph{repr\'esentation d'holonomie} $\rho_{\calG} : \pi_1(\calG)\to\Gamma$, d\'efinie de mani\`ere unique.
En effet, pour tout $f\in\pi_1(\calG)$, il existe un unique \'el\'ement $\rho_{\calG}(f)\in\Gamma$ tel que $[\widetilde{\mathtt{T}}_{f\cdot\widetilde{\Pi}_0}] = \rho_{\calG}(f)\cdot [\widetilde{\mathtt{T}}_{\widetilde{\Pi}_0}]$.
Par unicit\'e de la construction (remarque~\ref{rem:action-libre-Gamma}), on a $[\widetilde{\mathtt{T}}_{f\cdot\widetilde{\Pi}}] = \rho_{\calG}(f)\cdot [\widetilde{\mathtt{T}}_{\widetilde{\Pi}}]$ pour tout sommet $\tilde{\Pi}$ de~$\widetilde{\calG}$, et $\rho_{\calG}(ff') = \rho_{\calG}(f) \, \rho_{\calG}(f')$ pour tous $f,f'\in\pi_1(\calG)$.

Par construction, pour tout $f\in\pi_1(\calG)$ et toute ar\^ete $\widetilde{A}$ de~$\calG$, on a $\rho_{\widehat{S}}(f a_{\widetilde{A}} f^{-1}) = \rho_{\calG}(f) \, \rho_{\widehat{S}}(a_{\widetilde{A}}) \, \rho_{\calG}(f)^{-1}$.
Comme les $a_{\widetilde{A}}$ engendrent $\pi_1(\widehat{S})$, on en d\'eduit $\rho_{\widehat{S}}(f\gamma f^{-1}) = \rho_{\calG}(f) \, \rho_{\widehat{S}}(\gamma) \, \rho_{\calG}(f)^{-1}$ pour tout $\gamma\in\pi_1(\widehat{S})$.
Ainsi, les repr\'esentations $\rho_{\widehat{S}} : \pi_1(\widehat{S})\to\Gamma$ et $\rho_{\calG} : \pi_1(\calG)\to\Gamma$ se combinent en une repr\'esentation
$$\rho : \pi_1(S) = \pi_1(\widehat{S}) \rtimes \pi_1(\calG) \longrightarrow \Gamma.$$
Cette repr\'esentation est unique \'etant donn\'e notre choix initial de représentant $[\widetilde{\mathtt{T}}_{\widetilde{\Pi}_0}]$ de $E(\Pi_0)$.
Changer le choix initial pour $[\widetilde{\mathtt{T}}_{\widetilde{\Pi}_0}]$ revient \`a conjuguer $\rho : \pi_1(S)\to\Gamma$ par un \'el\'ement de~$\Gamma$.

Ceci conclut la d\'emonstration du lemme~\ref{lem:recoller-repr}, et donc des th\'eor\`emes \ref{thm:quantitatif-PSL(n,C)} et~\ref{thm:quantitatif}.

\printbibliography


\end{document}
